\documentclass[reqno, 12pt]{amsart}
\usepackage{amsmath}
\usepackage{amsthm}
\usepackage{amsfonts}
\usepackage{amssymb}
\usepackage{mathrsfs}
\usepackage{epsfig}
\usepackage{slashed}
\usepackage{mathtools}
\usepackage{enumerate}
\usepackage{fullpage}
\usepackage{xcolor}
\usepackage{comment}
\usepackage[hidelinks]{hyperref}
\usepackage{bm}
\usepackage[
top    = 2.5cm,
bottom = 2.5cm,
left   = 1.5cm,
right  = 1.5cm]{geometry}
\usepackage{soul}
 \newcommand{\mb}{\mathbf}
\newcommand{\mc}{\mathcal}

\renewcommand{\Re}{\mathrm{Re}\,}

\newcommand{\rg}{\mathrm{rg}\,}

\newcommand{\N}{\mathbb{N}}
\newcommand{\R}{\mathbb{R}}
\newcommand{\C}{\mathbb{C}}

\newcommand{\B}{\mathbb{B}}

\newcommand{\supp}{\mathrm{supp}\,}

\renewcommand{\O}{\mathcal{O}}

\newcommand{\Lf}{\textup{\textbf{L}}}

\newcommand{\hf}{\textup{\textbf{h}}}

\newcommand{\Sf}{\textup{\textbf{S}}}

\newcommand{\If}{\textup{\textbf{I}}}

\newcommand{\Rf}{\textup{\textbf{R}}}

\newcommand{\ff}{\textup{\textbf{f}}}
\newcommand{\Pf}{\textup{\textbf{P}}}

\newcommand\restr[2]{{  \left.\kern-\nulldelimiterspace #1 \vphantom{\big|}  \right|_{#2} }}

\newtheorem{lemma}{Lemma}[section]
\newtheorem{theorem}[lemma]{Theorem}
\newtheorem{corollary}[lemma]{Corollary}
\newtheorem{proposition}[lemma]{Proposition}
\theoremstyle{remark}
\newtheorem{remark}[lemma]{Remark}

\theoremstyle{definition}
\newtheorem{definition}[lemma]{Definition}
\newtheorem*{definition*}{Definition}

\numberwithin{equation}{section}

\title[]{Stability of global self-similar solutions to the cubic wave equation and the wave maps equation}

\author{Akansha Sanwal}
\address{Universit\"at Innsbruck, Institut f\"ur Mathematik, Technikerstra{\ss}e 13, 6020 Innsbruck,
	Austria \& New York University Abu Dhabi, Abu Dhabi, United Arab Emirates}
\email{as23982@nyu.edu}

\author{Birgit Sch\"orkhuber}
\address{Universit\"at Innsbruck, Institut f\"ur Mathematik, Technikerstra{\ss}e 13, 6020 Innsbruck,
Austria}
\email{birgit.schoerkhuber@uibk.ac.at}

\author{David Wallauch}
\address{EPFL SB MATH PDE, Batiment MA, Station 8, CH-1015 Lausanne}
\email{david.wallauch@epfl.ch}

\begin{document}

\begin{abstract}
We study the long-time stability of global self-similar solutions to two energy supercritical nonlinear wave equations, namely, the cubic nonlinear wave equation in $6$ dimensions and the co-rotational wave maps equation in $4$  dimensions. We prove the stability of self-similar solutions under perturbations that are small in the critical Sobolev spaces. The proof is based on Strichartz estimates for wave equations with potentials in similarity variables.

\end{abstract}

\maketitle

\section{Introduction}

\subsection{The focusing nonlinear wave equation}
We consider the cubic wave equation
\begin{equation}
  \label{eq:cNLW}
  \begin{cases}
    \partial_t^2 u - \Delta_x u = u^3
      \\[3pt]
    \bigl(u(t_0,\cdot),\,\partial_t u(t_0,\cdot)\bigr) = (f,g)
  \end{cases}
\end{equation}
for some $t_0 \in \R$, $u: I \times \R^d \to \R$, $d\geqslant 5$ and $I \subset \R$ an interval containing $t_0$. Eq.~\eqref{eq:cNLW} is invariant under the following scaling:
\begin{equation*}
	u\to u_{\lambda}:=\frac{1}{\lambda}u\Big(\frac{t}{\lambda},\frac{x}{\lambda}\Big), \quad \lambda >0,
\end{equation*}
with the initial data scaled accordingly. With $\dot{H}^s(\R^d)$ denoting the homogeneous Sobolev space, we have
\begin{equation}
	\label{eq:cNLWScaling}
	\|u_{\lambda}(t,\cdot)\|_{\dot{H}^s(\R^d)} = \lambda^{\frac{d}{2}-1-s}\|u(t/\lambda, \cdot)\|_{\dot{H}^s(\R^d)},
	\end{equation}
	implying that the scaling leaves the $\dot{H}^s(\R^d)$ norm invariant for $s=s_c=\frac{d}{2}-1$. Consequently, from the conserved energy:
	\begin{equation*}
		E[u,\partial_t u](t):=\frac{1}{2}\int_{\R^d} |\nabla u(t,x)|^2 dx +\frac{1}{2}\int_{\R^d} (\partial_tu(t,x))^2 dx -\frac{1}{4}\int_{\R^d} |u(t,x)|^{4}dx,
	\end{equation*}
	we deduce that \eqref{eq:cNLW} is energy supercritical for $d\geqslant 5$. \\
    
For the cubic wave equation, radially symmetric self-similar solutions, i.e., solutions that remain invariant under the scaling \eqref{eq:cNLWScaling}, take the form
\begin{equation*}
	u(t,x)=\frac{1}{t}U\Big(\frac{|x|}{t}\Big), 
\end{equation*}
for some function $U$ which is referred to as the profile of the solution. The trivial profile is given by $U = c_{d}$ for a suitable constant $c_d > 0$. However, in supercritical dimensions, the existence of genuinely non-trivial profiles can be expected, see e.g.  \cite{BizonMaisonWasserman, DaiDuyckaerts, Kycia}. A fully explicit example  was recently found  by Glogi\'c  and the second author \cite{GlogicSchoerkhuber} for $d\geqslant 5$,

\begin{equation}
	\label{eq:ExplicitU}
	\quad U(\rho)=\frac{2\sqrt{2(d-1)(d-4)}}{d-4+3\rho^2}.
\end{equation}

Obviously, the corresponding self-similar solution blows up in a backward evolution at time $t=0$.  However, as any self-similar solution, it still serves as a perfectly smooth forward in time solution on any time interval $[T,\infty)$ for any $T>0$. In dimensions $5 \leqslant d \leqslant 12$, numerical work by Kycia \cite{Kycia} suggests the existence of an infinite, countable family of smooth self-similar profiles. A proof of this fact for $d=5$ is contained in the work of Dai and Duyckaerts  \cite{DaiDuyckaerts}. For higher space dimensions, at least finitely many non-trivial profiles can be expected, as the explicit example above shows.   In this article, we aim to understand the long-time ($t\to \infty$) stability of self-similar solutions.

\begin{definition}\label{Def:Admiss_Selfsim}
A function  $u_1^*$ is called an \emph{admissible self-similar solution} to the cubic nonlinear wave equation, if 
it solves
\begin{align}\label{NLW3}
	\partial_t^2 u-\Delta_xu=u^3,\qquad (t,x)\in(0,\infty)\times \R^d
\end{align} 
and can be written as $u_1^*(t,x)=t^{-1} U_*(\frac{|x|}{t})$ for some $0\neq U_*\in C^\infty ([0,\infty))$ that satisfies
\begin{align*}
\lim_{\rho\to\infty} U_*(\rho)=0.
\end{align*}
\end{definition}

We note that the profile in \eqref{eq:ExplicitU}  has decay of the order $\frac{1}{\rho^{2}}$ at infinity. This is in contrast to the solutions of \cite{DaiDuyckaerts}, which decay like $\frac{1}{\rho}$. In fact, by classical results of Kavian and Weissler \cite[Theorem 3.1]{KavianWeissler}, we have the following characterisation, which completely determines the possible behaviour of non-constant  smooth self-similar profiles at infinity.

\begin{lemma}{\cite{KavianWeissler}}\label{Le:DecaySelfSim}
Let $u(t,x)=t^{-1}U\big(\frac{|x|}{t}\big)$ with $0\neq U
\in C^{2}([0,\infty))$ solve Eq.~\eqref{NLW3}
for $t>0$.
Assume further that $\displaystyle \lim_{\rho \to \infty} U(\rho )=0$.
Then there exists a constant $L\neq 0$ such that either 
\begin{align*}
\lim_{\rho \to \infty} \rho  U(\rho )=L, \quad 
\lim_{\rho \to \infty} \rho ^2 U'(\rho )=-L
\end{align*}
or 
\begin{align*}
\lim_{\rho \to \infty} \rho^2 U(\rho)=L, \quad 
\lim_{\rho \to \infty} \rho ^3 U'(\rho )=-2L.
\end{align*}
\end{lemma}

Obviously, the profile in \eqref{eq:ExplicitU} falls into the second category. However, we will see that for the stability analysis,  it suffices to assume $U(\rho) = \mc O(\rho^{-1})$.

\subsubsection{Main results}
We formulate the main results for Eq.~\eqref{eq:cNLW}. We believe that our analysis holds for all supercritical dimensions $d\geqslant 5$. However, we choose to work at $d=6$ where the critical exponent is of integer order, i.e.,  $s_c = 2$, so that the technical details do not obscure the analysis. Also, we restrict ourselves to radial symmetry. To state our first main theorem, we define 
\[\Omega:=\{(t,x)\in \R_{\geqslant 1}\times\R^6: |x|>t\}.
\]
\begin{theorem}
	\label{thm:MainTheoremNLW}
	Let $d=6$ and $u_1^*$ be an admissible self-similar solution to the cubic nonlinear wave equation.
	 There exist constants $\delta, \delta'>0$ such that the following holds. Let $(\varphi_0,\varphi_1)$ be radial Schwartz functions that satisfy
     
	\begin{equation}\label{Eq:DataSmall}
		\| (\varphi_0, \varphi_1) \|_{\dot{H}^2(\R^6) \times \dot{H}^1(\R^6)} \leqslant \delta.
	\end{equation}
    
	There exists a unique smooth solution $u: \Omega \to \R$ to  Eq.~\eqref{eq:cNLW} satisfying the initial condition
    
        \begin{align}\label{NLW3:pert_data}
u(1,\cdot) = u_1^*(1,\cdot) + \varphi_0, \quad  \partial_t u(1,\cdot) = \partial_t u_1^*(1,\cdot) + \varphi_1.
\end{align}

    Moreover,
	
    \begin{equation}\label{Esti: stbaility1}
		\| (u, \partial_t u) - (u_1^*, \partial_tu_1^*)\|_{L^\infty_{t}((1,\infty)) (\dot{H}^2(\R^6\setminus\B^6_{t}(0))\times \dot{H}^1(\R^6\setminus\B^6_{t}(0)))} \leqslant \delta'.
	\end{equation}
    
\end{theorem}
To extend  \eqref{Esti: stbaility1} to all of $(1,\infty)\times \R^6$, we need to introduce the notion of resonances. 
\begin{definition}\label{def:Resonance}
We say that $\lambda\in \mathbb{C}$ is a resonance of the radial potential $V \in C_{rad}^\infty(\R^6)$ if there exists a radial function $f_\lambda\in C^1_{c,rad}(\R^6)\cap C^\infty(\B_1^6), f_{\lambda} \not\equiv 0,  ~\mathrm{supp}~f_{\lambda} \subset \overline{\B_1^6}$ with 
\begin{equation}\label{Eq:resonances}
    \sum_{i,j=1}^d (\delta_{ij}-\xi_i\xi_j)
    \partial_i\partial_j f_\lambda(\xi)
    +2(\lambda-2)\,\xi\cdot\nabla f_\lambda(\xi)
    -\bigl[(\lambda-1)(\lambda-2)-V(\xi)\bigr]f_\lambda(\xi)
    =0
\end{equation}
for all $\xi \in \B^6$.
\end{definition}
Essentially, \eqref{Eq:resonances} is the spectral equation for the linearised evolution in similarity variables. The absence of a threshold resonance allows us to extend our previous result to all of $\R^6$ as follows.
\begin{theorem}\label{Thm: NLW2}
	Let $d=6$ and $u_1^*$ be an admissible self-similar solution to Eq.~\eqref{eq:cNLW}. Assume that $0$ is not a resonance of the potential $V=3(u_1^*)^2$. Then,
	 there exist constants $\delta, \delta'>0$, $N' \in \N_0$ and functions $\hf_1,\dots \hf_{N'} \in C^1_{c,rad}(\R^6)\times C_{c,rad}(\R^6)$ such that the following holds. Let $(\varphi_0,\varphi_1)$ be radial Schwartz functions that satisfy
     
	\begin{equation}\label{Eq:DataSmall2}
		\| (\varphi_0, \varphi_1) \|_{\dot{H}^2(\R^6) \times \dot{H}^1(\R^6)} \leqslant \delta.
	\end{equation}
    Then, there exist constants $\mu_1\dots \mu_{N'} \in \R$ such that the corrected initial data 
    $(\widetilde \varphi_0, \widetilde\varphi_1)$, where
       \[ (\widetilde \varphi_0, \widetilde\varphi_1)=( \varphi_0,\varphi_1)-\sum_{j=1}^{N'}\mu_j \hf_j\]
	lead to a global strong solution $u: \R_{\geqslant 1}\times\R^6 \to \R$ to  Eq.~\eqref{eq:cNLW} relative to $u_1^*$ satisfying the initial condition
    
        \begin{align}\label{NLW3:pert_data_2}
u(1,\cdot) = u_1^*(1,\cdot) + \widetilde \varphi_0, \quad  \partial_t u(1,\cdot) = \partial_t u_1^*(1,\cdot) +\widetilde \varphi_1.
\end{align} 
    Moreover,
    \begin{equation}\label{Esti: stbaility_NLw}
		\| (u, \partial_t u) - (u_1^*, \partial_tu_1^*)\|_{L^\infty_{t}((1,\infty)) (\dot{H}^2(\R^6)\times \dot{H}^1(\R^6))} \leqslant \delta'.
	\end{equation}
    \end{theorem}
    
Some remarks are in order. 
\begin{remark}\label{Remark:Main}
Lemma \ref{Le:DecaySelfSim} shows that self-similar solutions have infinite critical Sobolev norm: depending on the case, it is either $u_1^*(t,\cdot)$ or $\partial_t u_1^*(t,\cdot)$ that logarithmically fails $\dot H^{s_c}(\R^d)$ or  $\dot H^{s_c-1}(\R^d)$, respectively, for $t > 0$. As a result, the solution $u$ in Theorem \ref{Thm: NLW2} has infinite $\dot{H}^2 \times \dot{H}^1$-norm.  The correction terms may have only finite regularity across the unit sphere, so the theorem does not assert smoothness across the forward light cone. Hence, our underlying solution concept is understood relative to the self-similar background: We call a radially symmetric function $u: \R_{\geqslant 1}\times\R^6 \to \R$ a global strong solution to  Eq.~\eqref{eq:cNLW} relative to $u_1^*$, if 
\[(u, \partial_t u) = (u_1^*, \partial_tu_1^*) + (\varphi, \partial_t \varphi)\]
and $$(\varphi, \partial_t \varphi) \in  C\bigl([1,\infty);\dot H^2_{\mathrm{rad}}(\R^6)
\times\dot H^1_{\mathrm{rad}}(\R^6)\bigr)$$
is a strong solution of 
\[ \partial_t^2 \varphi - \Delta_x \varphi = (u_1^*+\varphi)^3-(u_1^*)^3.\]
Moreover, the perturbation constructed in
Theorem~\ref{Thm: NLW2} satisfies the global Strichartz bound $\varphi\in L^2\bigl((1,\infty);\dot W^{1,4}(\R^6)\bigr)$.
\end{remark}

\subsubsection{Related results}\label{Sec:Results_Wave} The local well-posedness of the initial value problem 
\begin{equation}
	\label{eq:NLW}
	\partial_t^2 u -\Delta _x u = |u|^{p-1} u 
\end{equation}
for supercritical $p > \frac{d+2}{d-2}$, $d \geqslant 3$, in the critical Sobolev space is classical $\dot H^{s_c}(\R^d) \times  \dot{H}^{s_c-1}(\R^d)$, see \cite{LinBla95}. Moreover, in this setting, global existence for small initial data can be obtained. As mentioned above, self-similar solutions are not contained in the critical Sobolev space. However, as shown by Planchon \cite{Pla2000}, this can be remedied by working in Besov spaces. In particular, \cite{Pla2000} proves the existence of global solutions for data with sufficiently small Besov norm, which includes small homogeneous data.  In contrast, for \textit{large} initial data, the understanding of the evolution is still rather limited in the supercritical case. As discussed above, smooth non-trivial self-similar profiles provide explicit examples for both finite-time blowup and non-dispersive global solutions (in one time direction). Their existence has been investigated for various combinations of $p$ and $d$ in  \cite{BizonMaisonWasserman,Kycia} and more recently in \cite{DaiDuyckaerts}. In general, countably infinite families of smooth profiles are expected to exist for $\frac{d+2}{d-2} < p < p_{JL}$, where $p_{JL}:= 1+ \frac{4}{d-4-2\sqrt{d-1}}$ for $d \geqslant 11$ denotes the Joseph-Lundgren exponent. Further examples of large global solutions  have been constructed by Krieger and Schlag \cite{KriegerSchlag} for $d =3$, $p=7$ in the radial case. Their solutions have infinite critical Sobolev norm precisely due to their (self-similar) decay at infinity. Moreover, the solutions are  stable with respect to a certain class of perturbations.\\ 

Concerning finite-time blowup, open (or finite co-dimension) sets of initial data have been constructed such that the corresponding solutions display specific self-similar blowup behaviour, see e.g. \cite{DonningerSchoerkhuber1, DonSch16, GlogicSchoerkhuber, CsoGloSch24, Ost24, Wallauch2024} and the references therein. In a different direction, Collot  \cite{Col18} constructed blowup bubbles concentrating a soliton in $d \geqslant  11$ for odd $p > p_{JL}$. 
 \\

We now proceed to the second problem studied in this paper. 

\subsection{Co-rotational wave maps}
We consider the wave maps equation for $\mc U: I \times \R^n \to \mathbb{S}^n \subset \R^{n+1}$ which reads
\begin{equation}
	\label{eq:WaveMaps}
	\partial^{\mu} \partial_{\mu} \mc U + (\partial^{\mu} \mc U \cdot \partial_{\mu} \mc U) \mc U =0.
\end{equation} 
for $\mu \in \{0,\dots, n\}$.
We are interested in the co-rotational setting, namely, maps $\mc U$ which respect a combined rotation of the domain and the target. This co-rotational ansatz is given by:

\begin{equation*}
	\mc U(t,x)=\begin{bmatrix}
		\sin(|x|u(t,|x|))\frac{x}{|x|}\\
		\cos(|x|u(t,|x|))
	\end{bmatrix}
\end{equation*}

for a smooth radial function $u=u(t,r): I \times [0,\infty) \to \R$. 
For such maps, Eq.~\eqref{eq:WaveMaps} reduces to a $(n+2)$-dimensional radial equation for $u$:

\begin{equation}
	\label{eq:WaveMapsCR}
	\left (\partial_t^2 -\partial_r^2 -\frac{n+1}{r}\partial_r \right )u(t,r) + \frac{(n-1)\sin(2ru(t,r)) - 2(n-1)ru(t,r)}{2r^3} =0.
\end{equation}

It is well-known that \eqref{eq:WaveMapsCR} admits self-similar solutions in the supercritical case $n \geqslant 3$, see \cite{Sha88,ShaTah94}, with an explicit example \cite{TurokSpergel, BizBir15} given by
\begin{equation}
	\label{eq:WaveMapsSSS}
	u^*_{1}(t,r) = \frac{1}{t} \phi\Big(\frac{r}{t}\Big), \quad \phi(\rho) = \frac{2}{\rho} \arctan\Big(\frac{\rho}{\sqrt{n-2}}\Big),
\end{equation}
which is referred as the ground state solution. 
Moreover, countable families of non-explicit profiles are known to exist in $3 \leqslant n \leqslant 6$, see \cite{Biz00, BieBizMal17}.\\

Similar to Definition \ref{Def:Admiss_Selfsim}, we define admissible self-similar solutions to the wave maps equation. Furthermore, given Lemma \ref{lem: wmadmissible} we do not need to make any decay assumptions. 

\begin{definition}
A function $u^*\neq 0$ is called an \emph{admissible self-similar solution} to the wave maps equation if it solves Eq.~\eqref{eq:WaveMapsCR} on $(0,\infty)^2$ and can be written as $u^*(t,r)=t^{-1} U_*(\frac{r}{t})$ for some $U_*\in C^\infty ([0,\infty))$.
\end{definition}

In the formulation of the main result, we restrict ourselves to $d := n+2 = 6$ and consider Eq.~\eqref{eq:WaveMapsCR} as a radial wave equation on $\R^6$. 

\begin{theorem}
\label{thm:MainTheoremWM}
	Let $n=4$ and $u_1^*$ be an admissible self-similar solution to \eqref{eq:WaveMapsCR}. There exist constants $\delta, \delta'>0$ such that the following holds.  Let $(\varphi_0,\varphi_1)$ be radial Schwartz functions that satisfy
    \begin{equation}\label{Eq:DataSmall3}
		\| (\varphi_0, \varphi_1) \|_{\dot{H}^2(\R^6) \times \dot{H}^1(\R^6)} \leqslant \delta.
	\end{equation}

	Then, there exists a unique smooth solution $u: \Omega \to \R$ to  Eq.~\eqref{eq:WaveMapsCR} satisfying the initial condition
    \begin{align}\label{NLW3:pert_data_3}
u(1,\cdot) = u_1^*(1,\cdot) + \varphi_0, \quad  \partial_t u(1,\cdot) = \partial_t u_1^*(1,\cdot) + \varphi_1.
\end{align}

    Moreover,
	
    \begin{equation}\label{Esti: stbaility2}
		\| (u, \partial_t u) - (u_1^*, \partial_tu_1^*)\|_{L^\infty_{t}((1,\infty)) (\dot{H}^2(\R^6\setminus\B^6_{t}(0))\times \dot{H}^1(\R^6\setminus\B^6_{t}(0)))} \leqslant \delta'.
	\end{equation}
    \end{theorem}
    
As before, the absence of a $0$ resonance allows us to extend this result. 
\begin{theorem}
\label{thm:MainTheoremWM2}
	Let $n=4$ and $u_1^*$ be an admissible self-similar solution to \eqref{eq:WaveMapsCR} and
assume
that $0$ is not a resonance of the potential $V=-\frac{3}{\rho^2}\big[\cos(2\rho u_{\mathrm{WM}_1}^*)-1\big]$. Then,
	 there exist constants $\delta, \delta'>0$, $N'' \in \N_0$ and functions $\hf_1,\dots \hf_{N''} \in C^1_{c,rad}(\R^6)\times C_{c,rad}(\R^6)$ such that the following holds. Let $(\varphi_0,\varphi_1)$ be radial Schwartz functions that satisfy
     
	\begin{equation}\label{Eq:DataSmall4}
		\| (\varphi_0, \varphi_1) \|_{\dot{H}^2(\R^6) \times \dot{H}^1(\R^6)} \leqslant \delta.
	\end{equation}
    Then, there exist constants $\mu_1\dots \mu_{N''} \in \R$ such that the corrected initial data 
    $(\widetilde \varphi_0, \widetilde\varphi_1)$, where
    \[ (\widetilde \varphi_0, \widetilde\varphi_1)=(\varphi_0, \varphi_1)-\sum_{j=1}^{N''}\mu_j \hf_j\]
     lead to a global strong solution $u: \R_{\geqslant 1}\times\R^6 \to \R$ to Eq.~\eqref{eq:WaveMapsCR} relative to $u_1^*$ satisfying the initial condition
         \begin{align}\label{NLW3:pert_data_4}
u(1,\cdot) = u_1^*(1,\cdot) + \widetilde \varphi_0, \quad  \partial_t u(1,\cdot) = \partial_t u_1^*(1,\cdot) +\widetilde \varphi_1.
\end{align}
    Moreover,
    \begin{equation}\label{Esti: stbaility3}
		\| (u, \partial_t u) - (u_1^*, \partial_tu_1^*)\|_{L^\infty_{t}((1,\infty)) (\dot{H}^2(\R^6)\times \dot{H}^1(\R^6))} \leqslant \delta'.
	\end{equation}
    \end{theorem}

Again, the solution is understood in the sense of Remark \ref{Remark:Main} with suitable adjustment for the nonlinearity. The following result, which is a consequence of Proposition \ref{lemma:UnstableSpectrumOfLWM}, shows that the case $N'' = 0$ occurs. 

\begin{corollary}
\label{thm:MainTheoremWM2_stable}
Let $n = 4$ and let $u^*_1$ denote the ground state self-similar solution given in Eq.~\eqref{eq:WaveMapsSSS}. Then $N'' = 0$ in Theorem \ref{thm:MainTheoremWM2}. In particular, the solution is orbitally stable under small perturbations and the corresponding unique solution is smooth on  $\R_{\geqslant 1}\times\R^6 \to \R$ for smooth $(\varphi_0,\varphi_1)$.
\end{corollary}

\subsubsection{Related results for the wave maps equation}\label{Sec:Results_WaveMaps} 
We give a short summary of the results which are important in the context of the paper. In particular, we restrict ourselves to the supercritical co-rotational setting. There, local well-posedness in the critical space has first been established in \cite{ShaTah94}. For more general results and a short history we refer the reader, e.g. to \cite{CanHer18} and the references therein.
In the spirit of \cite{Pla2000}, Germain \cite{Ger08} studied the supercritical co-rotational equation problem in Besov spaces. Beyond the small data theory, he also considers large homogeneous data, from which self-similar solutions are constructed, though uniqueness is lost.  For $d = 3$, Chiodaroli and Krieger \cite{ChiodaroliKrieger} proved the existence of global, forward-in-time, smooth solutions for large (even in the Besov sense) initial data following the approach of \cite{KriegerSchlag}. 
Concerning blowup, the nonlinear asymptotic stability of the ground state solution \eqref{eq:WaveMapsSSS} has been established in a series of works in various functional analytic frameworks, see \cite{DonSch12, Don11, CosDonXia16, CosDonGlo16, ChaDonGlo17, DonningerWallauch, GlogicWaveMaps, DonWal25}. In $d \geqslant 7$, type II blowup bubbles have been constructed by Ghoul, Ibrahim and Nguyen \cite{GhoIbrNgu18}.\\ 
Finally, we mention recent work by Bonk and Donninger \cite{BonDon26}, who consider for odd $d\geqslant 3$ the problem of forward-in-time stability of the self-similar ground state \eqref{eq:WaveMapsSSS}. Although similar in spirit, their approach is very different from ours and based on a hyperboloidal foliation of the future light cone.  Taking this point of view, the linearised evolution  in self-similar coordinates  decays exponentially in Sobolev spaces of sufficiently high regularity. As a consequence, they prove that the ground state \eqref{eq:WaveMapsSSS} is \textit{asymptotically} stable.

\subsection{Notation and function spaces}
By  $\B^d_a$, we denote the open ball of radius $a$ centered at $0$ on $\R^d$ and by $\B_a^{c}$, its complement on $\R^d$, where we omit the dimension in the notation brevity.
For $u,v\in C_c^{\infty}(\R^d)$ and $s \geq 0$, we define
\begin{equation}
	\langle u,v\rangle_{\dot{H}^s} = \langle |\cdot|^s \hat{u}, |\cdot|^s\hat{v}\rangle_{L^2},
\end{equation}
where $\hat{u}$ denotes the (spatial) Fourier transform of $u$:
\begin{equation*}
	\hat{u}(\xi):=\frac{1}{(2\pi)^{\frac{d}{2}}} \int_{\R^d} e^{-ix\cdot \xi} u(x) dx.
\end{equation*}
Bold face letters denote tuples or operators acting on tuples. For $\mb{u}=(u_1,u_2), \mb{v}=(v_1,v_2) 
\in C_c^{\infty}(\R^d) \times C_c^{\infty}(\R^d)$, we define
\begin{equation*}
	\langle \mb{u}, \mb{v} \rangle_{\mathcal{H}^s} = \langle u_1, v_1 \rangle_{\dot{H}^s} + \langle u_2,v_2\rangle_{\dot{H}^{s-1}}.
\end{equation*}
 We denote by $C_{c,rad}^{\infty}(\R^d)$ radial representatives of radially symmetric smooth compactly supported functions on $\R^d$. 
The homogenous Sobolev spaces $\dot H^s(\R^d)$ is defined as the completion of $C_c^{\infty}(\R^d)$ with respect to the norm $\| \cdot \|_{\dot{H}^s}$ induced by the above inner product. Likewise, $\dot H^s_{\mathrm{rad}}(\R^d)$ denotes the completion of   $C_{c,rad}^{\infty}(\R^d)$ with respect to the homogeneous Sobolev norm. 
In the context of ODE theory, we define the Wronskian of two functions by
\begin{equation*}
	W(f_1,f_2)(x) = f_1(x)f_2'(x) - f_1'(x)f_2(x).
\end{equation*}
We shall drop arguments as needed to simplify the notation.

\subsection{Plan of the paper} We briefly outline the steps involved in proving our main results. In the first step, we transform the original equation(s) via forward self-similar coordinates. Then, we linearise around any of the admissible self-similar solutions, and obtain a nonlinear wave equation with potential terms. We fix the dimension to $d=6$ and treat the obtained equation as an operator equation
\begin{equation*}
    \partial_\tau \bm{\phi}= \mb{L}\bm{\phi} + \mb{N}(\bm{\phi}).
\end{equation*}

In Section~\ref{section:LinearisedEvolution}, we show that $\Lf$
generates a semigroup $\Sf$ on the scaling-critical space
$\mathcal H^2:=\dot H^2_{\mathrm{rad}}(\R^6)\times
\dot H^1_{\mathrm{rad}}(\R^6)$ and that the spectrum of $\Lf$ in the open
right half-plane consists of finitely many positive real eigenvalues of
finite algebraic multiplicity. Furthermore, all elements in the corresponding generalized eigenspace are functions which are supported in $\overline{\B^6}$. Let $\mb P$ denote the corresponding
Riesz projection, with $\mb P=0$ if there are no unstable eigenvalues.
Then, for every $\varepsilon>0$,
\[
    \|\Sf(\tau)(I-\mb P)\mb f\|_{\mathcal H^2}
    \lesssim_\varepsilon e^{\varepsilon\tau}\|\mb f\|_{\mathcal H^2}.
\]
To obtain sharper estimates, we use inverse Laplace transformation.
For smooth, compactly supported radial data and $\tau>0$, we have
\[
    \Sf(\tau)\mb f
    =\frac{1}{2\pi i}\lim_{N\to\infty}
      \int_{\gamma-iN}^{\gamma+iN}
      e^{\lambda\tau}\Rf_{\Lf}(\lambda)\mb f\,d\lambda,
\]
where $\gamma$ lies to the right of the growth bound of $\Sf$.

The resolvent equation, namely, $\mb{R}_{\mb{L}}(\lambda)\mb{f} = (\lambda - \mb{L})^{-1}\mb{f}=\mb{u}$ reduces to a second-order ordinary differential equation, which is analysed in Section \ref{section:ODEAnalysis}. A large part of this section follows \cite[Section 4]{DonningerWallauch} and \cite[Section 4]{Wallauch2024}. At this point, it is also worth mentioning how the blowup stability and forward-in-time stability analyses differ. When dealing with the forward-in-time stability of solutions, we are required to tackle three singular points, namely $\rho=0,1,\infty$ of the equation, in contrast to working on the backward light cone, hence considering the points $\rho=0,1$ only. We note that the forward-in-time stability cannot be restricted to any forward light cone using standard coordinates, owing to the domain of dependence/region of influence property of solutions to wave equations, in contrast to blowup analysis on the backward light cone. For more details, we refer the reader to Section \ref{section:InteriorEstimates}.

The construction of the resolvent is first accomplished locally, in the two regimes $(0,1)$ and $(1,\infty)$ and then, by gluing the obtained expressions together at the light cone boundary, which in similarity variables corresponds to $\rho=1$, we obtain the resolvent on the full space in Section \ref{section:ODEAnalysis}.

We separate off the free semigroup $\Sf_0$ and move the contour to
$\Re\lambda=\varepsilon$, where $\varepsilon>0$ is sufficiently small
that all unstable eigenvalues lie to its right. The residues crossed
in this shift contribute $\Sf(\tau)\mb P\mb f$. We thus obtain 
\begin{equation}\label{Eq:outline int}
\begin{aligned}
    [\Sf(\tau)(I-\mb P)\mb f]_1(\rho)
    ={}&[\Sf_0(\tau)\mb f]_1(\rho)\\
    &+\frac{1}{2\pi i}\lim_{N\to\infty}
      \int_{\varepsilon-iN}^{\varepsilon+iN}e^{\lambda\tau}
      \bigl[(\Rf_{\Lf}(\lambda)-\Rf_{\Lf_0}(\lambda))\mb f\bigr]_1
      (\rho)\,d\lambda .
\end{aligned}
\end{equation}
The resolvent difference exhibits improved decay properties with respect to the spectral parameter  for smooth
data. In Sections~\ref{section:InteriorEstimates}
and~\ref{section:StrichartzEstimatesExterior}, we derive Strichartz
estimates in the interior and exterior of the light cone based on the above representation. Since elements in the
range of $\mb P$ are supported in $\overline{\B^6}$ and the range is invariant
under $\Sf(\tau)$, the projected and full evolutions agree for $\rho>1$.
Under the assumption that zero is not a resonance, we obtain the
whole-space estimates, including
\[
    \|\Sf(\tau)(I-\mb P)\mb f\|_{\mathcal H^2}
    \lesssim \|\mb f\|_{\mathcal H^2}.
\]
In Section~\ref{sec:NonlinearityControl}, we use these estimates in
a fixed-point argument. For the whole-space result, a finite-dimensional
correction to the initial data accounts for the unstable modes.


\subsubsection{The gluing problem} An important part of the analysis is to mitigate the following problem which shows up when we glue the resolvent together at $\rho=1$: This is possible as there is a degree of freedom which stems from the fact that for $\Re \lambda>0$ there exists a function $u_0(\cdot;\lambda)\in H^2_{loc}(\R^6)\times H^1_{loc}(\R^6)$ that formally solves $(\lambda-\Lf)u_0(\cdot;\lambda)=0.$ Near $\rho=1$, $u_0$ takes the form
\begin{equation*}
    u_0(\rho;\lambda) = c_{0,1}(\lambda) u_1(\rho;\lambda) + c_{0,2}(\lambda) u_2(\rho;\lambda),
\end{equation*}
with $u_2(1;\lambda)=0$. Thus, if the coefficient $c_{0,1}(\lambda)$ vanishes at a certain spectral point $\lambda\in \C,$ then we cannot match the local resolvents at $\rho=1$. For $\Re \lambda>0$, this is exactly the case when $\lambda$ is an eigenvalue of $\Lf$. Similarly, roots of $c_{0,1}$ on the imaginary axis, correspond to (threshold) resonances. Consequently, a resonance on the imaginary axis prevent us from pushing the contour integral \eqref{Eq:outline int} onto that line. From Lemma \ref{lemma:SpectralCharacterisation}, we can show that the only possible threshold resonance is $0$. This leads the necessity of the assumption that $0$ is not a resonance in Theorems \ref{Thm: NLW2} and \ref{thm:MainTheoremWM2}.


\subsection{Declaration of AI use.}
In preparing this revised manuscript, the authors used ChatGPT
(GPT-6.2 Astra), which identified an omission in Proposition~3.4
of the previous version (arXiv:2607.02493).
This finding required a revision of the proposition and corresponding
adjustments to the subsequent arguments, including the statements
of the main theorems. The authors also used ChatGPT
(GPT-6.2 Astra and GPT-6.2 Sol) to improve the exposition and writing
and to identify minor errors in the first version of the paper.

\section{Reformulation of the problem in forward self-similar coordinates}
\subsection{Cubic wave equation}
We introduce forward self-similar coordinates:
\begin{equation}
	\tau:=\tau(t) = \log(t), \quad \xi :=\xi(t,x)=\frac{x}{t},
\end{equation}
which map $\R_{\geqslant 1} \times \R^d$ to $\R_{\geqslant 0} \times \R^d$. 
The derivatives with respect to the new variables $\tau,\xi$ become:
\begin{equation}
	\partial_t = e^{-\tau}(\partial_{\tau}-\Lambda), \quad \partial_{x_i} = e^{-\tau}\partial_{\xi_i},
\end{equation}
Furthermore, we define
\begin{equation*}
	v_1(\tau,\xi):=tu(t,x) = e^{\tau}u(e^{\tau},e^{\tau}\xi), \quad v_2(\tau,\xi):= t^2 \partial_t u(t,x) = (\partial_{\tau} -\Lambda-1)v_1(\tau,\xi),
\end{equation*}
where
\begin{equation*}
	\Lambda f(x) = x\cdot \nabla f(x)=	\sum_{i=1}^d x_i \partial_{x_i}f(x).
\end{equation*}

Consequently, \eqref{eq:cNLW} transforms into
\begin{equation}
	\label{eq:cNLWSimilarity}
	\left\{ \begin{array}{cl}
		\partial_{\tau}v_1 &= v_1 +v_2 +\Lambda v_1\\
		\partial_{\tau}v_2 &= 2v_2 + \Delta_{\xi}v_1 +\Lambda v_2 + v_1^3.
	\end{array} \right.
\end{equation}
For $\mb{v} = (v_1,v_2)$, we write the above as:
\begin{equation}
	\label{eq:AbstractEquation}
	\partial_{\tau}\mathbf{v}(\tau) = \tilde{\Lf}_0 \mathbf{v}(\tau) + \mathcal{N}(\mathbf{v}(\tau)),
\end{equation}
where
\begin{equation}
	\label{eq:L_0}
	\tilde{\Lf}_0\begin{bmatrix}
		v_1\\
		v_2
	\end{bmatrix} = \begin{bmatrix}
		\Lambda+1 &&1\\
		\Delta && \Lambda+2
	\end{bmatrix}
	\begin{bmatrix}
		v_1\\
		v_2
	\end{bmatrix},
\end{equation}
is the linear wave equation in forward self-similar coordinates, and the nonlinear term is defined as
\begin{equation*}
	\mathcal{N}(\mb{f}):=\begin{bmatrix}
		0\\
		f_1^3
	\end{bmatrix},
\end{equation*}
for $\mb{f}=(f_1,f_2)$. The initial data become
\begin{equation*}
	\mb{v}(0,\cdot):=\mathbf{v}_0(\cdot) =
	\begin{bmatrix}
		u(1,\cdot)\\
		\partial_t u(1,\cdot)
	\end{bmatrix}.
\end{equation*}

Let $u_1^*$ be an admissible self-similar solution to \eqref{eq:cNLW} and set $\mb{u}^* =(u_1^*, u_2^*) := (u_1^*, 
\partial_t u_1^*)$.   In the new variables, the self-similar solution becomes
\begin{equation*}
	\mb{u}^*(
    \xi)= \begin{bmatrix}
		U(|\xi|)\\
		-U(|\xi|)- |\xi| U'(|\xi|)
	\end{bmatrix},
\end{equation*}
where $U$ denotes the corresponding self-similar profile. 
Note that $\mb{u^*}$ is a static solution to \eqref{eq:cNLWSimilarity}. To analyse the stability of $\mathbf{u}^*$, for $\bm{\phi} = (\phi_1,\phi_2)$, we make the ansatz
\begin{equation*}
	\mb{v}(\tau) = \bm{\phi}(\tau)+\mb{u}^*.
\end{equation*}
Using the expansion
\begin{equation*}
	(u_1^*+\phi_1)^3 = (u_1^*)^3 + \phi_1^3 + 3u_1^*\phi_1^2 + 3(u_1^*)^2\phi_1,
\end{equation*}
we obtain the following equation for the perturbation
\begin{equation}
	\label{eq:EquationForPerturbation}
	\partial_{\tau}\bm{\phi}(\tau) = \tilde{\Lf}\bm{\phi}(\tau) + \mb{N}(\bm{\phi}(\tau))
\end{equation}
where for $V:= 3(u_1^*)^2$
\begin{equation}
	\label{eq:Linearised}
	\tilde{\Lf} = \mathbf{\tilde{\Lf}_0} + \Lf ',\quad  \Lf'\bm{\phi}=\begin{bmatrix}
		0\\
		V \phi_1
	\end{bmatrix},\quad \mb{N}(\bm{\phi}) = \begin{bmatrix}
		0\\
		\phi_1^3+ 3u_1^*\phi_1^2.
	\end{bmatrix}
\end{equation}
Equation \eqref{eq:EquationForPerturbation} is the central subject of our analysis, and we would like to analyse the above evolution on the space
\begin{equation}
	\label{eq:CriticalSpace}
	\mathcal{H}^{\frac{d}{2}-1}:= \{ \mb{f} \in \dot{H}^{\frac{d}{2}-1}(\R^d) \times \dot{H}^{\frac{d}{2}-2}(\R^d): \mb{f} \text{ radial}\}
\end{equation}
which is the scaling critical space for the cubic nonlinear wave equation. In later parts, when we fix $d=6$, we shall denote the corresponding space $\mathcal{H}^2$  by only $\mathcal{H}$.

\subsection{The wave maps equation}
We make the same transformation and scaling as for the wave equation, namely:
\begin{equation*}
	\tau=\tau(t):=\log(t), \quad \rho=\rho(r,t):=\frac{r}{t}.
\end{equation*}
Define $w(\tau,\rho):=tu(t,r)$ to obtain that \eqref{eq:WaveMapsCR} transforms to:
\begin{equation*}
	(\rho^2-1) \partial_{\rho}^2w+\partial_{\tau}^2 w -2\rho \partial_{\tau}\partial_{\rho} w+\Big(4\rho-\frac{d-1}{\rho}\Big) \partial_{\rho}w -3\partial_{\tau}w +2w     + \frac{d-3}{2\rho^3}\big(\sin(2\rho w)-2\rho w\big) =0.
\end{equation*}
Now, we set, as before $h_1(\tau,
\xi):=w(\tau,|\xi|)$ and $h_2(\tau,\xi):=(\partial_{\tau}-
\Lambda -1)h_1(\tau,\xi)$. Then, the resulting system of equations for $\mb{h}:=(h_1,h_2)$ is given by:
\begin{equation}
    \label{eq:WaveMapsCP}
	\partial_{\tau}\mb{h}(\tau)=\tilde{\Lf}_0 \mb{h}(\tau) + \mb{\tilde{N}}(\mb{h}(\tau)),
\end{equation}
where $\tilde{\Lf}_0$ is the same operator as defined in  \eqref{eq:L_0} and the nonlinear term is given by
\begin{equation*}
	\mb{\tilde{N}}(\mb{h})(\xi):=	\begin{bmatrix}
		0\\
		-\frac{d-3}{2|\xi|^3}\big(\sin(2 |\xi| h_1)-2 |\xi| h_1\big)
	\end{bmatrix}.
\end{equation*}
Let $\mathbf{u}_{\mathrm{WM}}^* = (u_{\mathrm{WM}_1}^*, u_{\mathrm{WM}_2}^*)$ where $u_{\mathrm{WM}_1}^*$ is an admissible self-similar solution to \eqref{eq:WaveMapsCR}, and $u_{\mathrm{WM}_2}^*$ the corresponding second compontent. We perturb $\mathbf{u}_{\mathrm{WM}}^* $ by $\bm{\phi}:=(\phi_1,\phi_2)$: 
\begin{equation*}
\mb{h}(\tau)=\mathbf{u}_{\mathrm{WM}}^*+\bm{\phi}(\tau)
\end{equation*}
and are led to the following equation for $\bm{\phi}$:
\begin{equation*}
    \partial_{\tau} \bm{\phi}(\tau) = (\mb{\tilde{L}_0} + \mb{L'})\bm{\phi}(\tau) + \mb{N}_{\mathrm{WM}}(\bm{\phi}(\tau)),
\end{equation*}
where $\mb{\tilde{L}_0}$ is as defined in \eqref{eq:L_0}, and
\begin{equation*}
    \mb{L'}\bm{\phi}(\xi) = \begin{bmatrix}
        0\\
      V(|\xi|) \phi_1(\xi)
    \end{bmatrix},  
    \quad    V(\rho) :=  -\frac{d-3}{\rho^2}\big[\cos(2\rho u_{\mathrm{WM}_1}^*)-1\big]
    \end{equation*}

    \begin{equation*}
        \mb{N}_{\mathrm{WM}}(\bm{\phi})=\begin{bmatrix}
        0\\
        -\frac{(d-3)}{2\rho^3} \big[\sin\big(2\rho (u_{\mathrm{WM}_1}^* + \phi_1)\big) - \sin(2\rho u_{\mathrm{WM}_1}^*) \big] +\frac{d-3}{\rho^2}\cos(2\rho u_{\mathrm{WM}_1}^*)\phi_1
    \end{bmatrix}.
\end{equation*}

The above transformation of the equation enables us to handle the linear parts, namely the operators $\tilde{\Lf}_0$ and $\Lf'$ simultaneously for the cubic wave equation and the wave maps equations. The nonlinear terms will be separately handled in Section \ref{sec:NonlinearityControl}.

\begin{remark}
With the above transformation, we observe that the potential $V$ is smooth for any admissible self-similar solution and decays like $\langle \rho\rangle^{-2}$. Moreover, the explicit self-similar solution from \eqref{eq:WaveMapsSSS}  in $n=d-2=4$ becomes
\begin{equation*}
	\mathbf{u}_{\mathrm{WM}}^*(\xi)=\begin{bmatrix}
		\frac{2}{|\xi|}
        \arctan\Big(\frac{|\xi|}{\sqrt{2}}\Big)\\
		-\frac{2\sqrt{2}}{2+|\xi|^2}
	\end{bmatrix}.
\end{equation*}

\end{remark}

\subsection{The free wave evolution in forward similarity coordinates}
We equip $\tilde{\Lf}_0$ with a domain
\begin{equation*}
	\mathcal{D}(\tilde{\Lf}_0) =  C^{\infty}_{c, \mathrm{rad}}(\R^d) \times C^{\infty}_{c, \mathrm{rad}}(\R^d).
\end{equation*}


\begin{lemma}
	\label{lemma:Dissipativity}
	Let $d\geqslant 5$ and $\mb{u}=(u_1,u_2) \in \mathcal{D}(\tilde{\Lf}_0)$. Then, for every $s  \geq 0$, 
	\begin{equation}
		\label{eq:Dissipativity}
		\Re\langle \tilde{\Lf}_0 \mb{u},\mb{u}\rangle_{\mathcal{H}^s} \leqslant \big(s-\frac{d}{2}+1\big)\langle \mb{u},\mb{u}\rangle_{\mathcal{H}^s}.
	\end{equation}
	\begin{proof}
		Using the definition of the inner product on $\mathcal{H}^{s}$ and the operator $\tilde{\Lf}_0$, we obtain
		\begin{equation}
			\label{eq:dissipative}
			\begin{split}
				\Re(\langle \tilde{\Lf}_0 \mb{u}, \mb{u}\rangle_{\mathcal{H}^s}) &= \Re(\langle \Lambda u_1 + u_1 +u_2, u_1 \rangle_{\dot{H}^s} + \langle \Delta u_1 + \Lambda u_2 + 2u_2, u_2\rangle_{\dot{H}^{s-1}})\\
				&=\Re(\langle \Lambda u_1,u_1 \rangle_{\dot{H}^{s}} + \langle u_1,u_1\rangle_{\dot{H}^s} + \langle u_2,u_1\rangle_{\dot{H}^{s}} + \langle \Delta u_1 ,u_2\rangle_{\dot{H}^{s-1}} \\
				&\quad+\langle \Lambda u_2 ,u_2\rangle_{\dot{H}^{s-1}} + 2\langle u_2,u_2\rangle_{\dot{H}^{s-1}})\\
				&=\big(s-\frac{d}{2}+1\big) \big(\langle u_1,u_1\rangle_{\dot{H}^s} + \langle u_2, u_2\rangle_{\dot{H}^{s-1}}\big)\\
				&=\big(s-\frac{d}{2}+1\big) \langle \mb{u}, \mb{u}\rangle_{\mathcal{H}^s},
			\end{split}
		\end{equation}
		where we used that
		\begin{equation*}
			\Re\langle \Lambda f,f \rangle_{\dot{H}^s} = \Big(s-\frac{d}{2}\Big)\langle f,f\rangle_{\dot{H}^s}.
		\end{equation*}
	\end{proof}
    \end{lemma}
    Since every densely defined dissipative operator is closable, we conclude that $\tilde{\Lf}_0: \mathcal{D}(\tilde{\Lf}_0) \subseteq \mathcal{H}^{\frac{d}{2}-1} 
    \to \mathcal{H}^{\frac{d}{2}-1}$ is closable and we denote its closure by $(\mb{L}_0, \mathcal{D}(\mb{L}_0))$. Then, $(\mb{L}_0, \mathcal{D}(\mb{L}_0))$ is also dissipative and satisfies
\begin{equation}
	\label{eq:LDissipativityOfClosedOperator}
	\Re\langle \mb{L}_0 \mb{u},\mb{u}\rangle_{\mathcal{H}^{\frac{d}{2}-1}} \leqslant 0.
\end{equation}

Now we link $\mb{L}_0$ to the linear wave evolution on $\R^d$. For initial data  prescribed at time $t=1$, namely $(f(x), g(x))=(u(1,x), \partial_t u(1,x))$ the solution to the corresponding initial value problem is given by
\begin{equation}
	u(t,x) = [\cos(|\nabla|(t-1))f](x) + \Big[\frac{\sin(|\nabla|(t-1))}{|\nabla|}g\Big](x).
\end{equation}

With this, we define
	\begin{equation*}
		\begin{split}
			(C(\tau)f)(\xi)&:=e^{\tau}[\cos(|\nabla|(e^{\tau}-1))f](e^{\tau}\xi),\\
			(S(\tau)f)(\xi)&:=e^{\tau} \Big[\frac{\sin(|\nabla|(e^{\tau}-1))}{|\nabla|}f\Big](e^{\tau} \xi).
		\end{split}
	\end{equation*}
	
	We have the following result:
	\begin{lemma}
		\label{lemma:DecayForPropagators}
		Let $f \in \mathcal{S}(\R^d)$, $s\geqslant 0$. Then, the following bounds hold for all $\tau \geqslant 0$:
		\begin{equation}
			\begin{split}
				\|C(\tau)f\|_{\dot{H}^s(\R^d)} &\lesssim e^{\tau(s-\frac{d}{2}+1)}\|f\|_{\dot{H}^s(\R^d)},\\
				\|S(\tau)f\|_{\dot{H}^s(\R^d)} &\lesssim e^{\tau(s-\frac{d}{2}+1)}\|f\|_{\dot{H}^{s-1}(\R^d)}
			\end{split}
		\end{equation}
		\begin{proof}
			The proof follows immediately by taking the Fourier transform and using the definition of the $\dot{H}^s$ norms.
		\end{proof}
	\end{lemma}
	
	With the above results, we can conclude the following:

\begin{proposition}
\label{prop:ExplicitSolution}
For $\mb{f}=(f_1,f_2) \in \mc S_{\mathrm{rad}}(\R^d) \times \mc S_{ \mathrm{rad}}(\R^d)$, $\tau \geqslant 0$ and $\xi \in \R^d$ let 
\begin{equation*}
	\mb{S}_0(\tau)(f_1,f_2)(\xi) := 
	\begin{bmatrix}
		e^{\tau} [\cos(|\nabla|(e^{\tau}-1))f_1](e^{\tau}\xi) + e^{\tau} \Big[\frac{\sin(|\nabla| (e^{\tau}-1))}{|\nabla|} f_2\Big](e^{\tau}\xi)\\
		-e^{2\tau}[\sin(|\nabla|(e^{\tau}-1))|\nabla|f_1](e^{\tau}\xi) + e^{2\tau} [\cos(|\nabla|(e^{\tau}-1))f_2](e^{\tau}\xi)
	\end{bmatrix}.
		\end{equation*}
		Then $(\mb{S}_0(\tau))_{\tau \geqslant 0}$ extends to a strongly continuous one parameter semigroup on $\mathcal{H}^{\frac{d}{2}-1}$   which is generated by  $(\mb{L}_0, \mathcal{D}(\mb{L}_0))$ and satisfies
        \begin{equation*}
		\|\mb{S}_0(\tau)\mb{f}\|_{\mathcal{H}^{\frac{d}{2}-1}} \lesssim \|\mb{f}\|_{\mathcal{H}^{\frac{d}{2}-1}}
	\end{equation*}
    for all $\mb{f} \in \mathcal{H}^{\frac{d}{2}-1}$.
    \begin{proof}
By using the given explicit expression, the semigroup property as well as strong continuity can be checked on $\mc S_{\mathrm{rad}}(\R^d) \times \mc S_{ \mathrm{rad}}(\R^d)$. Furthermore, one can also verify that
\begin{equation*}
    \big(\partial_\tau -\Lambda -1)[\mb{S}_0(\tau)\mb{f}]_1 = [\mb{S}_0(\tau)\mb{f}]_2,
\end{equation*}
and the growth estimate follows from Lemma \ref{lemma:DecayForPropagators}. Moreover, $C^{\infty}_{c, \mathrm{rad}}(\R^d) \times C^{\infty}_{c, \mathrm{rad}}(\R^d)$ is a dense subspace which is invariant under the free evolution and by differentiation, one obtains $\tilde{\Lf}_0$ there. 

\end{proof}
\end{proposition}



\subsection{Strichartz estimates for the free semigroup}

In the following, we will use Strichartz estimates for the free wave evolution in $d=6$ in similarity coordinates. 

\begin{proposition} \label{prop: free strichartz}
	Let $q \in [2,\infty]$ and $r\in [6,12]$ be such that
	\begin{equation}
    \label{eq:AdmissibleExponents}
		\frac{1}{q} + \frac{6}{r} = 1.
	\end{equation}
	Then, the estimates
	\begin{equation}
		\label{eq:FreeStrichartzEstimates}
		\begin{split}
			\| [\mb{S}_0(\tau)\mb{f}]_1 \|_{L_{\tau}^q(\R_+)L_{\rho}^r(\R^6)} &\lesssim \|\mb{f}\|_{\mathcal{H}^2},
            \\
			\Big \| \int_0^{\tau} [\Sf_0(\tau-\sigma)\mb{h}(\sigma)]_1d\sigma \Big\|_{L_{\tau}^q(\R_+)L_{\rho}^r(\R^6)} &\lesssim \|\mb{h}\|_{L_{\tau}^1(\R_+)\mathcal{H}^2},
            \\
			\| [\mb{S}_0(\tau)\mb{f}]_1 \|_{L_{\tau}^2(\R_+)\dot{W}_{\rho}^{1,4}(\R^6)} &\lesssim \|\mb{f}\|_{\mathcal{H}^2},\\
			\Big \| \int_0^{\tau} [\Sf_0(\tau-\sigma)\mb{h}(\sigma)]_1d\sigma \Big\|_{L_{\tau}^2(\R_+)\dot{W}_{\rho}^{1,4}(\R^6)} &\lesssim \|\mb{h}\|_{L_{\tau}^1(\R_+)\mathcal{H}^2}
            \\
			\| [\mb{S}_0(\tau)\mb{f}]_1 \|_{L_{\tau}^\infty(\R_+)\dot{H}_{\rho}^2(\R^6)} &\lesssim \|\mb{f}\|_{\mathcal{H}^2},
            \\
			\Big \| \int_0^{\tau} [\Sf_0(\tau-\sigma)\mb{h}(\sigma)]_1d\sigma \Big\|_{L_{\tau}^\infty(\R_+)\dot{H}_{\rho}^2(\R^6)} &\lesssim \|\mb{h}\|_{L_{\tau}^1(\R_+)\mathcal{H}^2},
		\end{split}
	\end{equation}
    hold for all $\mb{f} \in \mathcal{H}^2$ and $ \mb{h} \in L_{\tau}^1(\R_+)\mathcal{H}^2$.
	\begin{proof}
		The proof follows by reversing the coordinate transform previously introduced and standard Strichartz estimates for wave equations.
	\end{proof}
\end{proposition}

\section{The linearised evolution}
\label{section:LinearisedEvolution}
\subsection{The perturbed problem }
In this section, we deal with the perturbed problem for the cubic NLW and the wave maps equation. We note that for both problems and every admissible self-similar solution, the corresponding potential entering the perturbation operator $\Lf'$ is a smooth radial function, which satisfies 
\begin{equation}
	\label{eq:DerivativeDecayPotential}
	|\partial^{\alpha} V(|x|)| \lesssim_{\alpha} \langle x \rangle^{-|\alpha|-2}
\end{equation}
for all $\alpha \in \mathbb{N}_0^d$. For any such potential, we define on $\mathcal{D}(\Lf_V'):=\mathcal{D}(\tilde{\Lf}_0)$, 
\begin{equation}
   \Lf_V' \bm{\phi} = \begin{bmatrix}
        0\\
      V\phi_1
    \end{bmatrix}, \quad 
\end{equation}

\begin{lemma}

	\label{lemma:Boundedness}
	Let $d\geqslant 5$ be even. Then the operator $\Lf_V'$ extends to a bounded operator on $\mathcal{H}^{\frac{d}{2}-1}$.
    \end{lemma}
	\begin{proof}

The result follows immediately from the equivalence of norms 
		\begin{equation*}
			\|f\|_{\dot{H}^s(\R^d)} \simeq \sum_{|\alpha|=s} \|\partial^{\alpha}f\|_{L^2(\R^d)},\quad s \in \N,
		\end{equation*}
and the decay of the potential together with Hardy's inequality for $s<\frac{d}{2}$.
	\end{proof}

In the following, by showing that the operator $\Lf'$ is in fact compact, we conclude that the essential spectrum of $\mb{L}_0+\Lf_V'$ is the same as that of $\mb{L}_0$.

\begin{lemma}\label{lemma:CompactPerturbH2}
	Let $d\geqslant 5$ be even. Then, the operator $\Lf_V':\mathcal{D}(\Lf_V')\subseteq \mathcal{H}^{\frac{d}{2}-1} \to \mathcal{H}^{\frac{d}{2}-1}$ is compact. As a consequence, the bounded extension  $\Lf_V': \mathcal{H}^{\frac{d}{2}-1} \to \mathcal{H}^{\frac{d}{2}-1}$ is compact. 
\end{lemma}
	\begin{proof}
		To show the compactness of the mapping defined by $\Lf_V'$, we show that the image $(0,V u_1)^T$ is totally bounded in $\mathcal{H}^{\frac{d}{2}-1}$ for $\mb{u}$ bounded in $\mathcal{H}^{\frac{d}{2}-1}$. More precisely, we shall show that the set
		\begin{equation*}
			K_{\alpha} = \{ \partial^{\alpha} (V u_1): u_1 \in C_{c,rad}^{\infty}, \|u_1\|_{\dot{H}^{\frac{d}{2}-1}} \leqslant C \}
		\end{equation*}
	is totally bounded in $L^2(\R^d)$ for $|\alpha| =\frac{d}{2}-2$. To this end, we employ the Rellich--Kondrachov theorem. First, we first show that $K_{\alpha}$ is bounded in $H^1(\R^d)$. It suffices to control the terms: $\| \partial^{\alpha}(Vu_1)\|_{L^2(\R^d)}$ and $\| \partial^{\alpha+1}(Vu_1)\|_{L^2(\R^d)}$. The former term can be controlled exactly as in Lemma \ref{lemma:Boundedness}. For the latter, we observe using \eqref{eq:DerivativeDecayPotential}
	\begin{equation*}
		\| u_1\partial^{\alpha+1}V \|_{L^2(\R^d)} \lesssim \| | \cdot |^{-(\frac{d}{2}-1)}u_1\|_{L^2(\R^d)} \Big \| \frac{|\cdot|^{\frac{d}{2}-1}} {\langle \cdot \rangle^{\frac{d}{2}+1}} \Big \|_{L^{\infty}(\R^d)} \lesssim \|u_1\|_{\dot{H}^{\frac{d}{2}-1}(\R^d)}.
	\end{equation*}
Furthermore, 
\begin{equation*}
	\| V  \partial^{\alpha+1}u_1\|_{L^2(\R^d)} \lesssim  \| V \|_{L^{\infty}(\R^d)} \|\partial^{\alpha+1}u_1\|_{L^2(\R^d)} \lesssim\| u_1\|_{\dot{H}^{\frac{d}{2}-1}(\R^d)}.
\end{equation*}
More generally, we control for $|\beta| + |\gamma| = |\alpha|+1=\frac{d}{2}-1$, the following term:
\begin{equation*}
	\|\partial^{\beta}V  \partial^{\gamma}u_1 \|_{L^2(\R^d)} \lesssim \| |\cdot|^{-(\frac{d}{2}-1)+|\gamma|} \partial^{\gamma}u_1 \|_{L^2(\R^d)} \Big \| \frac{|\cdot|^{\frac{d}{2}-1-|\gamma|}} {\langle \cdot \rangle^{|\beta|+2}} \Big\|_{L^{\infty}(\R^d)} \lesssim \|u_1\|_{\dot{H}^{\frac{d}{2}-1}(\R^d)}.
\end{equation*}
 Exploiting the decay of $V$ and its derivatives, we use H\"older's inequality and Sobolev embedding to infer tightness as follows:  for $|\beta| + |\gamma|=|\alpha|=\frac{d}{2}-2$, we have
\begin{equation*}
	\begin{split}
	\| \partial^{\alpha}(Vu_1)\|_{L^2(B_R^c)} = \| \partial^{\beta} V \partial^{\gamma}u_1 \|_{L^2(B_R^c)} &\lesssim \| \langle \cdot \rangle^{-|\beta|-2} \partial^{\gamma}u_1\|_{L^2(B_R^c)}\\
	&\lesssim \| \langle \cdot \rangle^{-|\beta|-2} \|_{L^\frac{2d}{d-2-2|\gamma|}(B_R^c)} \| \partial^{\gamma}u_1 \|_{L^{\frac{d}{|\gamma|+1}}(B_R^c)}\\
	&\lesssim R^{-\varepsilon}\|u_1\|_{\dot{H}^{\frac{d}{2}-1}(\R^d)},
	\end{split}
\end{equation*}
for $\varepsilon >0$. The above can be imitated to bound $\| \partial^{\alpha+1}(Vu_1)\|_{L^2(B_R^c)}$. Indeed, we have for $|\beta|+|\gamma|=|\alpha|+1 = \frac{d}{2}-1$:
\begin{equation*}
	\begin{split}
	\| \partial^{\alpha+1}(V u_1)\|_{L^2(B_R^c)} = \|\partial^{\beta}V \partial^{\gamma}u_1\|_{L^2(B_R^c)} &\lesssim \| \langle \cdot \rangle^{-|\beta|-2}\partial^{\gamma}u_1\|_{L^2(B_R^c)} \\
	&\lesssim \| \langle \cdot \rangle^{-|\beta|-2}\|_{L^{\frac{2d}{d-2-2|\gamma|}}(B_R^c)} \| \partial^{\gamma}u_1\|_{L^{\frac{d}{|\gamma|+1}}(B_R^c)}\\
	&\lesssim R^{-\varepsilon'}\|u_1\|_{\dot{H}^{\frac{d}{2}-1}(\R^d)}.
	\end{split}
\end{equation*}
Compactness of the extended operator follows by density.
	\end{proof}

By standard semigroup theory we conclude that the operator $\tilde{\Lf}$ is closable with its closure being
\begin{equation*}
	\mb{L}  :=\mb{L}_0 + \Lf_V' , \quad \mathcal{D}(\mb{L}) = \mathcal{D}(\mb{L}_0).
\end{equation*}
Also, we get the following result. 

\begin{proposition}
Let $d\geqslant 5$ be even. Then, the operator $\mb{L}: \mathcal{D}(\mb{L})  
\subseteq \mathcal{H}^{\frac{d}{2}-1} \to \mathcal{H}^{\frac{d}{2}-1}$ generates a strongly continuous semigroup $(\mb{S}(\tau))_{\tau \geqslant 0}$ on $\mathcal{H}^{\frac{d}{2}-1}$.
\end{proposition}

\subsection{Spectral analysis}
In this section, we analyse the spectral properties of the linearisation and we restrict ourselves to $d = 6$ from now on. The first result concerns operators which arise from any admissible self-similar solution.

\begin{proposition}\label{Prop:SpectralCharacterization}
Let $d = 6$. There exists an $N = N(V) \in \mathbb{N}_0$ such that 
\begin{align*}
    \sigma(\Lf  )\cap \{z\in \C: \Re z > 0\} = 
    \{\lambda_1,\dots \lambda_N \}
\end{align*}
with distinct real eigenvalues $\lambda_1,\dots \lambda_N > 0$ of finite algebraic multiplicity. Moreover, the elements of the generalized eigenspaces are supported in the ball $\overline{\B^6}$. Finally, there are no eigenvalues on the imaginary axis. 
\end{proposition}

\begin{proof}
The fact that the unstable spectrum of $\Lf $ consists of isolated eigenvalues with finite algebraic multiplicity is an immediate consequence of the free semigroup bound stated in Proposition \ref{prop:ExplicitSolution} and the compactness of the perturbation $\Lf_V'$. We analyse the eigenvalue equation for $\lambda \in \C$ under the assumption $\mathrm{Re} \lambda \geq 0$. Then the operator eigenvalue equation leads to the following equations for the radial representatives of the components $(u_1,u_2)$,
	\begin{equation*}
		\begin{split}
			u_1''(\rho) + \tfrac{5}{\rho}u_1'(\rho) + \rho u_2'(\rho) +2u_2(\rho) +V(\rho)u_1(\rho) &= \lambda u_2(\rho),\\
			u_2(\rho) &= (\lambda-1)u_1(\rho) -\rho u_1'(\rho),
		\end{split}
	\end{equation*}
which implies that $u_1$ satisfies 
	\begin{equation}
		\label{eq:FullODE}
		(1-\rho^2)u_1''(\rho) + \Big( \tfrac{5}{\rho} + 2(\lambda-2)\rho \Big)u_1'(\rho) -\big( (\lambda-1)(\lambda-2)-V(\rho) \big )u_1(\rho) = 0.
	\end{equation}
First, by smoothness of the coefficients, solutions are smooth on the subintervals $(0,1)$ and $(1,\infty)$. We first analyse the behaviour at infinity, which is a regular singular point of the equation. By analyticity of self-similar profiles at infinity, see Lemma \ref{lem: wmadmissible} and \ref{lem: cubicadmissible} we can apply Frobenius' method. Explicitly, by setting $s=\frac{1}{\rho}$ one finds that the corresponding Frobenius indices are $\{ 1- 
\lambda, 2 - \lambda \}$. This implies that there is a local fundamental system $\{ v_1, v_2 \}$ for Eq.~\eqref{eq:FullODE} given by 
\begin{equation*}
	v_1(\rho) =	\rho^{\lambda-2}h_1\Big(\tfrac{1}{\rho}\Big), \quad v_2(\rho) = C\rho^{\lambda-2}h_1\Big(\tfrac{1}{\rho}\Big)\ln\Big(\tfrac{1}{\rho}\Big) + \rho^{\lambda-1}h_2\Big(\tfrac{1}{\rho}\Big)
\end{equation*}
for a constant $C$ and analytic functions $h_1$ and $h_2$ in a neighbourhood of zero with $h_1(0) = h_2(0) = 1$.

Consequently, any solution can locally be represented as
\begin{equation*}
	u_1 = \alpha_1 v_1 + \alpha_2 v_2.
\end{equation*}
If $(u_1,u_2)$ is a non-trivial eigenfunction, we have
$u_1 \in \dot H^2(\R^6)\hookrightarrow L^6(\R^6)$. From the form of $v_1$
and $v_2$ one sees that $\alpha_1 v_1 + \alpha_2 v_2 \notin L^6(\R^6\setminus\B_R^6)$ for $\alpha_2 \neq 0$,
$\mathrm{Re}\lambda\geq0$ and $R$ sufficiently large, which forces $\alpha_2 = 0$, i.e., $u_1$ is
a multiple of $v_1$. The second
component $u_2$ is given by
\begin{equation}
\label{Eq:eig secondcomp}
	u_2(\rho) = (\lambda-1)u_1(\rho) - \rho u_1'(\rho) =
    \alpha_1 \rho^{\lambda-2}
\left [ h_1\Big(\tfrac{1}{\rho}\Big) + \rho^{-1}h_1'\Big(\tfrac{1}{\rho}\Big) \right ].
\end{equation}
Since $h_1(0)=1$, the term in the bracket tends to $1$ as $\rho\to\infty$.
On the other hand, $u_2 \in \dot H^1(\R^6)\hookrightarrow L^3(\R^6)$, which
for $\operatorname{Re}\lambda\geq0$ forces $\alpha_1 = 0$. Hence,
$u_1 = u_2 = 0$ near infinity, and since Eq.~\eqref{eq:FullODE} is regular
on $(1,\infty)$, $u_1 = u_2 = 0$ on $(1,\infty)$.

This leaves the possibility that the eigenfunction has compact support.
If the potential is known to be analytic, the Frobenius indices around zero are given by $\{0,-4\}$, hence $u_1$ must be analytic there. At the light cone, the indices are $\{0,\lambda + \frac32 \}$. In order for the solution to be in $\dot H^2$ locally around the light cone, $u_1$ has to satisfy $u_1(1) = u_1'(1)=0$. Hence,
\[
  u_1(\rho) =  c(1-\rho)^{\lambda+\frac32}h(1-\rho)
\]
for an analytic function $h$ which satisfies $h(0)=1$. For merely smooth
$V$, the same behaviour of the fundamental systems near $\rho=0$ and
$\rho=1$, with $h$ smooth instead of analytic, follows from Volterra
iterations as in Section~\ref{section:ODEAnalysis} (cf.~Lemmas~\ref{Le:FundSysinsidecone}
and~\ref{lem:solutions interior regime}). For fixed $\lambda$ (no uniformity
in $\mathrm{Im}\lambda$ is needed) this argument applies to all
$\mathrm{Re}\lambda\geq0$ and not only to $\lambda\in S$. In
particular, our conclusions so far hold for any potential that is induced
by an admissible self-similar profile.

By inspection one easily checks that for $\mathrm{Re} \lambda = 0$, $u_1 \notin \dot H^2(1-\varepsilon,1)$ for small enough $\varepsilon > 0$. Hence, in this case $c=0$ and we conclude that there are no eigenvalues on the imaginary axis. For $\mathrm{Re} \lambda > 0$, we set $\kappa = \lambda + \frac32$ and transform the eigenvalue equation by introducing
\[ \rho = \tanh x,  \]
see e.g. \cite{DonAic2009} together with the variable transformation
\[ y(x(\rho)) = \rho^{5/2} (1 - \rho^2)^{-\kappa/2} u_1(\rho).\]
Equivalently,
\[ u_1(\tanh x) = (\sinh x)^{-5/2} (\cosh x)^{1-\lambda} y(x). \]

The spectral ODE then transforms into $A_V y = - \kappa^2 y$, where 
\[ (A_V y)(x) := - y''(x) + \frac{15}{4 \sinh^2 x} y(x) - \frac{V(\tanh x)}{\cosh^2 x} y(x). \]

The behaviour of $u_1$ inside the light cone implies 
\[ y(x) = \tilde c\, e^{-\kappa x} ( 1 + O(e^{-2x}))\]
as $x$ tends to infinity, where $\tilde c$ is a nonzero multiple of $c$.
Around zero $y(x) = O(x^{\frac{5}{2}})$ and thus $y \in L^2(0,\infty)$. As both endpoints of $A_V$ are in the limit point case, the corresponding maximal operator $(A_V,\mc D(A_V))$ is self-adjoint. Since $y$ is smooth on $(0,\infty)$ and satisfies the eigenvalue equation there, we conclude that $y \in \mc D(A_V)$ is an eigenfunction of $A_V$ with eigenvalue $\mu := - (\lambda + \frac32)^2$, which is therefore real. The assumption $\mathrm{Re} \lambda > 0$ implies $\lambda \in \R$, $\lambda >0$ and $\mu < -\frac{9}{4}$. By using the decomposition principle for the essential spectrum, see e.g. Theorem 9.11 in \cite{Teschl}, p. 229, together with the form of $A_V$ we infer that $\sigma_{\mathrm{ess}}(A_V) = [0,\infty)$. Since $A_V$ is bounded below by
\[ \langle A_V y, y \rangle_{L^2(0,\infty)} \geq - \| V \|_{L^{\infty}(0,1)} \| y \|^2_{L^2(0,\infty)},\]
$\mu$ must be contained within a finite set of negative eigenvalues. This proves the statement on the structure of the spectrum of $\Lf$. 
Since eigenfunctions corresponding to any of the eigenvalues $\lambda_1, \dots, \lambda_N$ are supported in $\overline{\B^6}$, the statement about the support of functions belonging to generalized eigenspaces follows by induction and the fact that the problem can be reduced to the homogeneous spectral ODE on $(1,\infty)$.
\end{proof}

The following example shows that the case $N =0$  in Proposition \ref{Prop:SpectralCharacterization} occurs for the explicit self-similar profile $\mathbf{u}_{\mathrm{WM}}^*$ for the wave maps equation. 

\begin{proposition}  
\label{lemma:UnstableSpectrumOfLWM}
Let $d = 6$ and consider for
\[	V_{\mathrm{WM}}(x):=\frac{48}{(|x|^2+2)^2}, \]
the operator $\Lf_{\mathrm{WM}} :=\mb{L}_0 + \Lf_{V_{\mathrm{WM}}}' : \mathcal{D}(\mb{L}) \subseteq \mathcal{H}^{2} \to \mathcal{H}^{2}.$
Then 

\begin{align*}
    \sigma(  \Lf_{\mathrm{WM}} )\subset \{z\in \C: \Re z\leqslant 0\}.
\end{align*}

Moreover, $0$ is not a resonance of $V_{\mathrm{WM}}$.
\end{proposition}

\begin{proof}
We use Proposition \ref{Prop:SpectralCharacterization} and argue by contradiction. Assume that there is an eigenvalue $
\lambda > 0$. By the exact same argument as in the proof of Proposition \ref{Prop:SpectralCharacterization} this implies the existence of an eigenvalue  $\mu := - (\lambda + \frac32)^2 \in \R$, $\mu < -\frac{9}{4}$ of the self-adjoint operator 
\[ (A_{V_{\mathrm{WM}}} y)(x) := - y''(x) + Q_{\mathrm{WM}}(x) y(x), 
\quad Q_{\mathrm{WM}}(x) = \frac{15}{4 \sinh^2 x}  - \frac{48 \cosh^2 x}{(2 + 3 \sinh^2 x)^2} \]
with $\mc D(A_{V_{\mathrm{WM}}}) = \{ y \in L^2(0,\infty) : y, y' \in AC_{\mathrm{loc}}(0,\infty),\
A_{V_{\mathrm{WM}}} y \in L^2(0,\infty) \}.$  We claim that 
\begin{align}\label{Eq:GGMTWM}
  \frac{3}{2} \int_0^{\infty} x^3 Q_{\mathrm{WM}}^{-}(x)^2 dx < 1 
\end{align}
where $Q_{\mathrm{WM}}^{-} \geq 0$ denotes the negative part of  the potential. Given \eqref{Eq:GGMTWM}, an application of the GGMT criterion, see e.g. \cite{GloSch2020} with $\alpha = 0$ and $p =2$, implies that $\sigma(A_{V_{\mathrm{WM}}}) = [0,\infty)$, a contradiction. To prove Eq.~\eqref{Eq:GGMTWM} we use the fact that 
$Q_{\mathrm{WM}}  \geq 0$ on $(0, a)$ with $a = \frac12 \log 5$ together with the estimate
\[ Q_{\mathrm{WM}}^{-}(x) \leq \frac{19}{12 \cosh^2 x} \leq \frac{19}{3} e^{-2x},\]
which implies that
\[  \frac{3}{2} \int_0^{\infty} x^3 Q_{\mathrm{WM}}^{-}(x)^2 dx  \leq \frac{3}{2} \left (\frac{19}{3} \right )^2 \int_a^{\infty} x^3  e^{-4x} dx  < 1 \]
by elementary integration. 

Finally, assume that $0$ is a resonance of $V_{\mathrm{WM}}$ with resonance
function $f_0 \not\equiv 0$. Then $u_1 := f_0$ solves Eq.~\eqref{eq:FullODE}
with $\lambda = 0$ on $(0,1)$ and is smooth at $\rho = 0$. Since
$f_0$ is continuously differentiable with $\mathrm{supp} f_0 \subset \overline{\B^6}$, it also
satisfies $u_1(1) = u_1'(1) = 0$. As in the proof of
Proposition~\ref{Prop:SpectralCharacterization}, $u_1(\rho) = c(1-\rho)^{\frac32}h(1-\rho)$ near
$\rho = 1$, and the transformation used there with $\kappa = \frac32$ yields
an eigenfunction of $A_{V_{\mathrm{WM}}}$ with eigenvalue $-\frac94 < 0$.
This contradicts $\sigma(A_{V_{\mathrm{WM}}}) = [0,\infty)$.
\end{proof}

Next, we prove the following:
\begin{lemma}\label{Le:Resolventbounds_Prüss}
	Let $d=6$. For every potential $V$ satisfying \eqref{eq:DerivativeDecayPotential} and every $\varepsilon>0$, there exist constants $C_{\varepsilon}, K_{\varepsilon}>0$ such that
	\begin{equation}
		\label{eq:BoundedResolvent}
		\| \Rf_\mb{L}(\lambda)\| \leqslant C_{\varepsilon},
	\end{equation}
for all $\lambda \in \C$ which satisfy $|\lambda| \geqslant K_{\varepsilon}$ and $\Re \lambda \geqslant \varepsilon$.
\begin{proof}
	Since the free semigroup $\mb{S}_0(\tau)$ satisfies
	\begin{equation*}
		\| \mb{S}_0(\tau)\| \leqslant 1,
	\end{equation*}
we have
\begin{equation}
	\|\Rf_{\mb{L}_0}(\lambda)\| \leqslant \frac{1}{\Re \lambda}, \quad \text{for all } \Re \lambda >0.
\end{equation}
We have
\begin{equation*}
	\lambda -\mb{L} = \lambda - \mb{L}_0-\mb{L}' = (\If-\mb{L}'\Rf_{\mb{L}_0}(\lambda)) (\lambda -\mb{L}_0)
\end{equation*}
which implies that $(\lambda -\mb{L})^{-1}$ exists if and only if $(\If - \mb{L}'\Rf_{\mb{L}_0}(\lambda))^{-1}$ exists. Let $\Rf_{\mb{L}_0}(\lambda)\mb{f} = \mb{u}$, then
\begin{equation}
	\label{eq:equations}
	\begin{split}
	u_2(\rho) &= (\lambda-1)u_1(\rho) - \rho u_1'(\rho) - f_1(\rho),\\
	\end{split}
\end{equation}
Now, using the above expression, the decay of the potential $V$ and Hardy's inequality, we obtain
\begin{equation*}
	\begin{split}
		\| \mb{L}'\Rf_{\mb{L}_0}(\lambda)\mb{f}\|_{\mathcal{H}^2}  = \| \mb{L}'\mb{u}\|_{\mathcal{H}^2}  = \|Vu_1\|_{\dot{H}^1} &\lesssim \frac{1}{|\lambda-1|} \Big( \|Vu_2\|_{\dot{H}^1} + \| |\cdot|Vu_1'\|_{\dot{H}^1} + \|Vf_1\|_{\dot{H}^1} \Big)\\
		&\lesssim \frac{1}{|\lambda-1|} \Big( \|\mb{u}\|_{\mathcal{H}^2} + \|\mb{f}\|_{\mathcal{H}^2} \Big)\\
		&=\frac{1}{|\lambda-1|} \Big( \|\Rf_{\mb{L}_0}(\lambda)\mb{f}\|_{\mathcal{H}^2} + \|\mb{f}\|_{\mathcal{H}^2} \Big)
	\end{split}
\end{equation*}
and we conclude 
\begin{equation*}
	\|\mb{L}'\Rf_{\mb{L}_0}(\lambda)\| \lesssim \frac{1}{|\lambda-1|}.
\end{equation*}
Thus, for $|\lambda|$ large enough, we have
\begin{equation}
	\| \mb{L}' \Rf_{\mb{L}_0}(\lambda) \| <1.
\end{equation}
With this we conclude that the Neumann series for $(\If - \mb{L}'\Rf_{\mb{L}_0}(\lambda))^{-1}$ converges and
\begin{equation*}
	\| \Rf_\mb{L}(\lambda)\| = \| \Rf_{\mb{L}_0}(\lambda) (\If-\mb{L}'\Rf_{\mb{L}_0}(\lambda))^{-1} \| \leqslant \| \Rf_{\mb{L}_0}(\lambda)\| \sum_{k=0}^{\infty}\| \mb{L}'\Rf_{\mb{L}_0}(\lambda)\|^k \leqslant C_{\varepsilon}.
\end{equation*}
This finishes the proof.
\end{proof}
\end{lemma}

The following, we choose $\varepsilon_V > 0$ such that for a given potential $V$ the corresponding eigenvalues $\lambda_1, \dots, \lambda_N$ of Proposition \ref{Prop:SpectralCharacterization} satisfy $\lambda_i > \varepsilon_V$. Furthermore, we denote by $\mb P$ the finite rank spectral projection onto $\sigma(\Lf) \cap \{z\in \C:\Re z\geq \varepsilon_V\}$.  Then the next result readily follows from Gearhart--Pr\"uss theorem and Lemma \ref{Le:Resolventbounds_Prüss}.

\begin{lemma}
Let $V$ be a potential associated to an admissible self-similar solution. Then for every $\varepsilon \in (0,\varepsilon_V)$ there exists $C_{\varepsilon} >0$ such that
	\begin{equation}
		\|(1- \mb P) \Sf(\tau)\mb{f}\|_{\mathcal{H}^2} \leqslant C_{\varepsilon}e^{\varepsilon\tau}\| \mb{f}\|_{\mathcal{H}^2}
	\end{equation}
for all $\mb{f} \in \mathcal{H}^2$ and all $\tau \geqslant 0$. 
\end{lemma}

However, the above estimate on the semigroup is not sufficient for our purposes, as there is no decay. Hence, in order to prove space-time estimates we will recast $\Sf$  via its inverse Laplace transform and bound the resulting oscillatory integrals in Strichartz spaces. Decreasing $\varepsilon_V$ if necessary, we may assume that
\begin{align}\label{Cond:SpecInt}
\sigma(\Lf)\cap\{z\in\C:0<\Re z<\varepsilon_V\}
=\emptyset.
\end{align}

By \cite[Corollary III.5.15]{EngelNagel}, for
$\ff\in\mathcal D(\Lf)$, $\tau>0$, and
$0<\varepsilon<\varepsilon_V$, we have
\begin{align*}
	\Sf(\tau)\tilde{\mathbf{f}} = \frac{1}{2\pi i}\lim_{N \to \infty} \int_{\varepsilon-iN}^{\varepsilon+iN} e^{\lambda\tau}\Rf_{\Lf}(\lambda)\tilde{\mathbf{f}}d\lambda, \quad \text{ for } \varepsilon >0,
\end{align*}
where $\tilde{\mathbf{f}} = (1- \mb P) \mb f$. For the derivation of the spacetime estimates, we initially restrict ourselves to $\ff\in C^\infty_{c,\mathrm{rad}}(\R^6) \times C^\infty_{c,\mathrm{rad}}(\R^6)$. The resulting estimates will subsequently be extended to $\mathcal H^2$ by density.

It is convenient to separate the free evolution from the
contribution of the potential. Choose
$0<\varepsilon<\varepsilon_V$, with $\varepsilon_V$ as in \eqref{Cond:SpecInt}, the inverse Laplace integral of $\Rf_{\Lf}(\lambda)\Pf\ff$ on this line vanishes for
$\tau>0$, since all poles of this finite-dimensional resolvent lie to the right of the contour. Subtracting the free inversion formula therefore yields
\begin{equation}\label{eq:abstract-perturbation-representation}
\Sf(\tau)(\If-\Pf)\ff = \Sf_0(\tau)\ff  +\frac{1}{2\pi i}\lim_{N\to\infty}
\int_{\varepsilon-iN}^{\varepsilon+iN}
e^{\lambda\tau}
\bigl[
\Rf_{\Lf}(\lambda)-\Rf_{\Lf_0}(\lambda)
\bigr]\ff\,d\lambda.
\end{equation}
At this stage, the representation involves only the
abstract resolvents on the critical space. In the following section, we provide an explicit (perturbative) construction of the first component of the resolvent, which serves as the basis for further estimates.

\section{ODE Analysis and resolvent construction}
\label{section:ODEAnalysis}
\subsection{Fundamental systems}
We fix $d=6$. From previous considerations, we already know that constructing the resolvent boils down to appropriately solving the ODE
\begin{equation}
	\label{eq:OriginalODE}
	(1-\rho^2)u_1''(\rho) + \Big(\tfrac{5}{\rho}+2(\lambda-2)\rho\Big)u_1'(\rho) - ((\lambda-1)(\lambda-2)-V(\rho))u_1(\rho)=F_{\lambda}(\rho),
\end{equation}
where $F_{\lambda}(\rho) :=(2-\lambda)f_1(\rho) + \rho f_1'(\rho)-f_2(\rho)$. For $0<\rho\leqslant 1$, this has already been accomplished in the work \cite{DonningerWallauch} with $\lambda$ replaced by $-\lambda$. However, the construction remains unaffected by the change of sign. Therefore, we recall the corresponding results from said work. For this, we let $V$ be any smooth radial potential with sufficient decay, which in our case means that $V$ decays like $\rho^{-2}$.
We also make the following definitions.
\begin{definition}
We define the strip $S$ as 
$$S:=\{\lambda\in\mathbb C : \Re\lambda\in(-\tfrac14,\tfrac14)\}.$$
In addition, we define an affine function $a:S\to \mathbb{C}$ as 
$$a(\lambda)=i\left(\frac{3}{2}+\lambda\right).$$
\end{definition}
\begin{definition}
Let $I\subset \R$, $\rho_0\in \R \setminus I$, and $\alpha \in \R$. We say that a smooth function $f:I \to \C$ is of symbol type and write $f(\rho)=\O((\rho_0-\rho)^{\alpha})$ if
\begin{align*}
|\partial_\rho^n f(\rho)|\lesssim_n |\rho_0-\rho|^{\alpha-n}
\end{align*}
 for all $\rho \in I$ and all $n\in \mathbb{N}_0$. Likewise, let $g:U \subset \C\to \R$, then $g=\O(\langle\omega\rangle^{\alpha})$, provided that that
 \begin{align*}
|\partial_\omega^n g(\varepsilon+i\omega)|\lesssim_n \langle\omega\rangle^{\alpha-n}
\end{align*}
where $\langle\omega\rangle$ denotes the Japanese bracket $\sqrt{1+|.|^2}$.
Analogously, we define 
$$
h(\rho,\lambda):=\O((\rho-\rho_0)^{\alpha} \langle\omega\rangle^{\beta}) \quad \text{ if } \quad |\partial_\rho^n\partial_\omega^k h(\rho,\lambda)|\lesssim_{n,k} |\rho_0-\rho|^{\alpha-n}\langle\omega\rangle^{\beta-k}
$$
for all $n,k\in \mathbb N_0$ and some fixed  $\alpha,\beta \in \R$.
\end{definition}
Ultimately, our goal is to construct the resolvent for $\rho\in (0,1)$ and for $\rho\in (1,\infty)$ separately and then glue them together using a degree of freedom that will be present on the interval (0,1) for almost all values of $\lambda$.

\subsection{The region $0<\rho<1$}
Most of the necessary ODE analysis on this region has already been carried out in \cite{Wallauch2024}. We start by transforming variables according to 

\begin{equation*}
	v(\rho) = \rho^{\frac{5}{2}} (1-\rho^2)^{-\frac{1}{4}-\frac{\lambda}{2}}u(\rho),
\end{equation*}
which kills the first order derivative and turns \eqref{eq:OriginalODE} with $F_\lambda=0$ into
\begin{equation}
	\label{eq:NoFirstDerivativeODEnearo0}
	v''(\rho) + \frac{-15+(10-12\lambda-4\lambda^2)\rho^2}{4\rho^2(1-\rho^2)^2} v(\rho) + \frac{V(\rho)}{1-\rho^2}v(\rho)=0.
\end{equation}
To state our results in a concise manner, we denote by  $J_2$ and $Y_2$, the ``good'' and the ``bad'' Bessel functions of order $2$ (i.e. $\displaystyle\lim_{\rho\to 0} J_2(\rho)=0$ while $Y_2$ is unbounded near $0$),
and by $\varphi$ the diffeomorphism $\varphi:(0,1)\to (0,\infty),\,\varphi(\rho):=\frac{1}{2}(\log(1+\rho)-\log(1-\rho))$. In addition to that, for $r\in (0,\infty) $ and $\rho_0\in (0,1)$ fixed, we define $\rho_\lambda$ as the smooth version of the function
$\min\{\frac{r}{|a(\lambda)|},\rho_0\}$ and, similarly, $\widehat{\rho}_\lambda$ as a smooth version of the function 
$\min\{\frac{2r}{|a(\lambda)|},\tfrac{1+\rho_0}{2}\}$.
We begin with the following result, which is a direct consequence of Lemmas 3.2 and 3.3 in \cite{Wallauch2024}.

\begin{lemma}
\label{Le:FundSysinsidecone}
There exist $r>1$ and $\rho_0\in [0,1)$ such that for $\lambda \in S$ the following holds. 
For $\rho\in (0,\widehat{\rho}_\lambda)$
equation \eqref{eq:NoFirstDerivativeODEnearo0}
has a fundamental system of solutions given by
\begin{align*}
\psi_1(\rho;\lambda)&=b_1(\rho;\lambda)[1+\rho^2 e_3(\rho;\lambda)] 
\\
&=\sqrt{(1-\rho^2)\varphi(\rho)}J_{2}( a(\lambda)\varphi(\rho))[1+\rho^2 e_3(\rho;\lambda)]
\\
\psi_2(\rho;\lambda)&= b_2(\rho;\lambda)[1+\rho^2 e_3(\rho;\lambda)]+ \psi_1(\rho;\lambda)e_4(\rho;\lambda).
\\
&=\sqrt{(1-\rho^2)\varphi(\rho)}Y_2( a(\lambda)\varphi(\rho))[1+\rho^2 e_3(\rho;\lambda)]+\O(\rho^{\frac{1}{2}}\langle\omega\rangle^{-2}).
\end{align*}
where $e_3$ satisfies
\begin{align*}
e_3(\rho;\lambda)=\widehat{e}_3(\rho;\lambda)+\O(\rho^{\frac{7}{2}}\langle\omega\rangle^{\frac{7}{2}})
\end{align*}
with
\begin{align}\label{esti:e12}
\partial_\rho^m\partial_\omega^n \widehat{e}_3(\rho;\lambda)\lesssim_{m,n} \langle\omega\rangle^{m-n}.
\end{align}
Similarly, for $\rho \in (\rho_\lambda,1)$, Eq.~\eqref{eq:NoFirstDerivativeODEnearo0} has a fundamental system of solutions given by
\begin{align*}
\psi_3(\rho;\lambda)&:=\frac{\sqrt{1-\rho^2}}{\sqrt{3+2\lambda}}\left(\frac{1+\rho}{1-\rho}\right)^{\frac{3}{4}+\frac{\lambda}{2}}[1+e_1(\rho;\lambda)][1+r_1(\rho;\lambda)]
\\
\psi_4(\rho;\lambda)&:=\frac{\sqrt{1-\rho^2}}{\sqrt{3+2\lambda}}\left(\frac{1-\rho}{1+\rho}\right)^{\frac{3}{4}+\frac{\lambda}{2}}[1+e_2(\rho;\lambda)][1+r_2(\rho;\lambda)]
\end{align*}
with 
$$
r_j(\rho;\lambda)=(1-\rho)\O(\rho^{0}\langle\omega\rangle^{-1}), \text{ and } e_j(\rho;\lambda)=(1-\rho)\O(\rho^{-1}\langle\omega\rangle^{-1})
$$
for $j=1,2$.
Moreover, in the case $V=0$, we have $r_1=r_2=0$.
\end{lemma}
\begin{proof}
    This is just a collection of special cases of Lemmas 3.1-3.3 in \cite{Wallauch2024}.
\end{proof}

In the special case $V=0$, we endow all functions appearing in the previous lemma with an additional subscript $\mathrm{f}$, e.g.~ 
$\psi_{\mathrm{f}_j}(\rho;\lambda),r_{\mathrm{f}_j}(\rho;\lambda) $.
Next, we want to glue these solutions together. To that end, we compute the following Wronskian:
\begin{align*}
    W(\psi_3(\cdot;\lambda),\psi_4(\cdot;\lambda))=-1
\end{align*}
\begin{lemma}
For $\rho\in (\rho_\lambda,\widehat{\rho}_\lambda)$ one has that
\begin{align*}
    \psi_1(\rho;\lambda)&=c_{0,1}(\lambda)\psi_3(\rho;\lambda)+c_{0,2}(\lambda)\psi_4(\rho;\lambda)
    \\
    \psi_2(\rho;\lambda)&=c_{\widetilde 0,1}(\lambda)\psi_3(\rho;\lambda)+c_{\widetilde 0,2}(\lambda)\psi_4(\rho;\lambda)
\end{align*}
where \begin{align*}
    c_{i,j}(\lambda)=\O(\langle\omega\rangle^0).
\end{align*}
and 
\begin{align*}
       c_{i,j}(\lambda)-    c_{\mathrm{f}_{i,j}}(\lambda) =\O(\langle\omega\rangle^{-1}).   
\end{align*}
for $i=0,\widetilde 0, j=1,2$, where $ c_{\mathrm{f}_{i,j}}$ are the corresponding connection coefficients for $V=0$.
\end{lemma}
\begin{proof}
Note that $$c_{0,1}(\lambda)= -  W(\psi_1(\cdot;\lambda),\psi_4(\cdot;\lambda))(\rho_\lambda)=-W\left(b_1(\cdot;\lambda),\frac{\sqrt{1-(\cdot)^2}}{\sqrt{3+2\lambda}}\left(\frac{1-(\cdot)}{1+(\cdot)}\right)^{\frac{3}{4}+\frac{\lambda}{2}}\right)(\rho_\lambda)+\O(\langle\omega\rangle^{-1}).$$
Thus, as 
$$
W\left(b_1(\cdot;\lambda,)\frac{\sqrt{1-(\cdot)^2}}{\sqrt{3+2\lambda}}\left(\frac{1-(\cdot)}{1+(\cdot)}\right)^{\frac{3}{4}+\frac{\lambda}{2}}\right)(\rho_\lambda)=\O(\langle\omega\rangle^0).
$$
The claimed identities on $c_{0,1}$ follow. Likewise, one obtains the remaining identities for the connection coefficients.
\end{proof}
To continue, we define a smooth cutoff function 
$\chi: [0,1]\times S\to
[0,1]$, $\chi_\lambda(\rho):=\chi(\rho;\lambda)$, that satisfies
$\chi_\lambda(\rho)=1$ for $\rho \in [0,\rho_\lambda]$, $
\chi_\lambda(\rho)=0$ for $\rho \in [\widehat{\rho}_\lambda,1]$, and
$|\partial_\rho^k\partial_\omega^\ell \chi_\lambda(\rho)|\leqslant
C_{k,\ell}\langle\omega\rangle^{k-\ell}$ for $k,\ell\in\mathbb N_0$. This cutoff allows us to glue our local solutions together. 
\begin{lemma}\label{lem:solutions interior regime}
 For $\lambda\in S$ and $\rho\in (0,1)$ the equation
\begin{equation}
	\label{eq:OriginalODE1}
	(1-\rho^2)u_1''(\rho) + \Big(\frac{5}{\rho}+2(\lambda-2)\rho\Big)u_1'(\rho) - ((\lambda-1)(\lambda-2)-V(\rho))u_1(\rho)=0,
\end{equation}
has the linearly independent solutions $u_0$, $\widetilde u_0$ which satisfy the following properties. 
On the support of $\chi_\lambda$, one has that
\begin{align*}
u_0(\rho;\lambda)&= \rho^{-\frac{5}{2}}(1-\rho^2)^{\frac{1}{4}+\frac{\lambda}{2}}\psi_1(\rho;\lambda)
\\
&= \rho^{-\frac{5}{2}}(1-\rho^2)^{\frac{1}{4}+\frac{\lambda}{2}}\sqrt{(1-\rho^2)\varphi(\rho)}J_{2}( a(\lambda)\varphi(\rho))[1+\rho^2 e_3(\rho;\lambda)]
\\
&=(1-\rho^2)^{\frac{1}{4}+\frac{\lambda}{2}}\O(\rho^0\langle\omega\rangle^2)[1+\O(\rho^{2}\langle\omega\rangle^0)]
\\
\widetilde u_0(\rho;\lambda)&= \rho^{-\frac{5}{2}}(1-\rho^2)^{\frac{1}{4}+\frac{\lambda}{2}}\psi_2(\rho;\lambda)\\
&=(1-\rho^2)^{\frac{1}{4}+\frac{\lambda}{2}}\left[\O(\rho^{-4}\langle\omega\rangle^{-2})+\O(\rho^{-2}\langle\omega\rangle^{-2})\right].
\end{align*}
Moreover, on the support of $(1-\chi_\lambda)$ one has that
\begin{align*}
    u_0(\rho;\lambda)&=c_{0,1}(\lambda)u_1(\rho;\lambda)+c_{0,2}(\lambda)u_2(\rho;\lambda)
    \\
    \widetilde u_0(\rho;\lambda)&=c_{\widetilde 0,1}(\lambda)u_1(\rho;\lambda)+c_{\widetilde 0,2}(\lambda)u_2(\rho;\lambda)
\end{align*}
with
\begin{equation*}
	\begin{split}
u_1(\rho;\lambda) &= \rho^{-\frac{5}{2}}(1-\rho^2)^{\frac{1}{4}+\frac{\lambda}{2}}\psi_3(\rho;\lambda)=
 \frac{\rho^{-\frac{5}{2}}(1+\rho)^{\frac{3}{2}+\lambda}}{\sqrt{3+2\lambda}}[1+e_1(\rho;\lambda)][1+r_1(\rho;\lambda)]\\
u_2(\rho;\lambda)&=\rho^{-\frac{5}{2}}(1-\rho^2)^{\frac{1}{4}+\frac{\lambda}{2}}\psi_4(\rho;\lambda)=
\frac{ \rho^{-\frac{5}{2}}(1-\rho)^{\frac32+\lambda}}{\sqrt{3+2\lambda}}[1+e_2(\rho;\lambda)][1+r_2(\rho;\lambda)]
\end{split}
\end{equation*}
with 
$$
r_j(\rho;\lambda)=(1-\rho)\O(\rho^{0}\langle\omega\rangle^{-1}), \text{ and } e_j(\rho;\lambda)=(1-\rho)\O(\rho^{-1}\langle\omega\rangle^{-1})
$$
for $j=1,2$.
Moreover, in case $V=0$, one has that $r_1=r_2=0$. 

\end{lemma}

Moving on, we denote by $\widehat S$ the set
$$
\widehat{S}:=S\setminus \{\lambda\in S  : c_{0,1}(\lambda)=0\}.
$$
Then, on $\widehat{S}$ we can divide by $c_{0,1}$ to conclude
\begin{align*}
 \left[c_{\widetilde 0,2}(\lambda)-c_{\widetilde 0,1}(\lambda)\frac{c_{0,2}(\lambda)}{c_{0,1}(\lambda)} \right] u_2(\rho;\lambda)=\widetilde u_0(\rho;\lambda)-\frac{c_{\widetilde 0,1}(\lambda)}{c_{0,1}(\lambda)}u_0(\rho;\lambda).
\end{align*}
This, in turn, implies that
$ \left[c_{\widetilde 0,2}(\lambda)-c_{\widetilde 0,1}(\lambda)\frac{c_{0,2}(\lambda)}{c_{0,1}(\lambda)} \right] $ is non-vanishing on $\widehat{S}$, as $u_0$ and $\widetilde u_0$ are linearly independent. Thus, we set
$\widehat{c}(\lambda)= \left[c_{\widetilde 0,2}(\lambda)-c_{\widetilde 0,1}(\lambda)\frac{c_{0,2}(\lambda)}{c_{0,1}(\lambda)} \right]^{-1}$ to conclude that
\begin{align*}
    u_2(\rho;\lambda)= \widehat{c}(\lambda)\left[\widetilde u_0(\rho;\lambda)-\frac{c_{\widetilde 0,1}(\lambda)}{c_{0,1}(\lambda)}u_0(\rho;\lambda)\right].
\end{align*}
With this at hand, we make the following definition. 

\begin{definition}
    Let $\lambda\in \widehat{S}$ and $f\in C^\infty_{rad}(\overline{\B^6})$. Then, for $c\in \mathbb{C}$, we define the function 
    \begin{align*}
            \mathcal{R}_c(f)(\rho;\lambda):(0,1)\times S \to \mathbb{C} 
    \end{align*}
as
\begin{align*}
      \mathcal{R}_c(f)(\rho;\lambda):&=u_2(\rho;\lambda)\int_0^\rho \frac{u_0(s;\lambda)f(s) }{(1-s^2)W(u_0(\cdot;\lambda),u_2(\cdot;\lambda))} ds+
      u_0(\rho;\lambda)\int_\rho^1 \frac{u_2(s;\lambda)f(s) }{(1-s^2)W(u_0(\cdot;\lambda),u_2(\cdot;\lambda))} ds 
      \\
      &\quad+cu_0(\rho;\lambda)
\end{align*}
where \begin{align*}
    {W(u_0(\cdot;\lambda),u_2(\cdot;\lambda))}(\rho)=-\frac{\pi}{2} c_{0,1}(\lambda)W(u_0(\cdot;\lambda),\widetilde u_0(\cdot;\lambda))(\rho)= -c_{0,1}(\lambda) \rho^{-5} (1-\rho^2)^{\frac{1}{2}+\lambda},
\end{align*}
and we used $\widehat{c}(\lambda)=-\frac{\pi}{2}c_{0,1}(\lambda)$.
\end{definition}

One directly concludes the following.
\begin{lemma}
\label{lemma:ResolventInsideCone}
 Let $\lambda\in \widehat{S}$ and $f\in C^\infty_{rad}(\overline{\B^6})$ and $c\in \mathbb{C}. $ Then, the function $\mathcal{R}_c (f) \in H^1(\B^6)$. Furthermore, for $\Re \lambda>0$ we have $\mathcal{R}_c (f) \in H^2(\B^6)$. In addition to that, it solves the equation
\begin{equation}
	\label{eq:OriginalODEagain}
	(1-\rho^2)\mathcal{R}_c (f)''(\rho) + \Big(\frac{5}{\rho}+2(\lambda-2)\rho\Big)\mathcal{R}_c (f)'(\rho) - ((\lambda-1)(\lambda-2)-V(\rho))\mathcal{R}_c (f)(\rho)=f(\rho)
\end{equation}
and is the unique one-parameter family of solutions to that equation in $ H^1\times L^2(\B^6)$.
\end{lemma}
\begin{proof}
    By construction $\mathcal{R}_c(f)$ solves \eqref{eq:OriginalODEagain} and clearly it is unique since $\widetilde u_0(\cdot;\lambda)\notin H^1(\B^d_1)$. Thus, we only need to verify the claimed regularity properties. Recall that
    \begin{align*}
        u_0(\rho;\lambda)=\chi_\lambda(\rho)(1-\rho^2)^{\frac{1}{4}+\frac{\lambda}{2}}\O(\rho^0\langle\omega\rangle^2)&+c_{0,1}(\lambda)(1-\chi_\lambda(\rho))\frac{\rho^{-\frac{5}{2}}(1+\rho)^{\frac{3}{2}+\lambda}}{\sqrt{3+2\lambda}}[1+e_1(\rho;\lambda)][1+r_1(\rho;\lambda)]
        \\
        &\quad+c_{0,2}(\lambda)(1-\chi_\lambda(\rho))\frac{\rho^{-\frac{5}{2}}(1-\rho)^{\frac{3}{2}+\lambda}}{\sqrt{3+2\lambda}}[1+e_2(\rho;\lambda)][1+r_2(\rho;\lambda)]
    \end{align*}
    from which we readily infer that
    $$u_0(\cdot;\lambda)\in H^1(\B^6), \text{ for all }\lambda\in S$$ as well as 
     $$u_0(\cdot;\lambda)\in H^2(\B^6), \text{ for all }\lambda\in S  \text{ with } \Re \lambda>0.$$ To the deal with integral terms, we let $\lambda\in S$ be fixed and plug in the explicit forms of the solutions that we derived in Lemma \ref{lem:solutions interior regime} to obtain that

     \begin{align*}
&\quad u_2(\rho;\lambda)\int_0^\rho\frac{u_0(s;\lambda)f(s)}{(1-s^2)W(u_0(\cdot;\lambda),u_2(\cdot;\lambda))(s)}\,ds+u_0(\rho;\lambda)\int_\rho^1\frac{u_2(s;\lambda)f(s)}{(1-s^2)W(u_0(\cdot;\lambda),u_2(\cdot;\lambda))(s)} ds
\\
&=\chi_\lambda(\rho)\O(\rho^{-4})\int_0^\rho \O(s^5)f(s)\,ds+\chi_\lambda(\rho)\O(\rho^0)\int_\rho^1 \chi_\lambda(s)\O(s)f(s) ds
\\
&\quad+\chi_\lambda(\rho)\O(\rho^0)\int_\rho^1 (1-\chi_\lambda(s))\frac{s^{\frac52}g_2(s;\lambda)}{(1+s)^{\frac32+\lambda}}f(s) ds
\\
&\quad+(1-\chi_\lambda(\rho))\frac{\rho^{-\frac52}(1-\rho)^{\frac32+\lambda}}{\sqrt{3+2\lambda}}g_2(\rho;\lambda)\int_0^\rho \chi_\lambda(s)\O(s^5)f(s) ds
\\
&\quad-(1-\chi_\lambda(\rho))\frac{\rho^{-\frac52}(1-\rho)^{\frac32+\lambda}}{3+2\lambda}g_2(\rho;\lambda)\int_0^\rho (1-\chi_\lambda(s))\frac{s^{\frac52}g_1(s;\lambda)}{(1-s)^{\frac32+\lambda}}f(s) ds
\\
&\quad-(1-\chi_\lambda(\rho))\frac{\rho^{-\frac52}(1+\rho)^{\frac32+\lambda}}{3+2\lambda}g_1(\rho;\lambda)\int_\rho^1 (1-\chi_\lambda(s))\frac{s^{\frac52}g_2(s;\lambda)}{(1+s)^{\frac32+\lambda}}f(s) ds
\\
&\quad-\frac{c_{0,2}(\lambda)}{c_{0,1}(\lambda)}(1-\chi_\lambda(\rho))\frac{\rho^{-\frac52}(1-\rho)^{\frac32+\lambda}}{3+2\lambda}g_2(\rho;\lambda)\int_0^1 (1-\chi_\lambda(s))
\frac{s^{\frac52}g_2(s;\lambda)}
{(1+s)^{\frac32+\lambda}}f(s) ds.
\end{align*}
     where $g_j:= [1+e_j(\rho;\lambda)][1+r_j(\rho;\lambda)]\in C^\infty(\supp (1-\chi_\lambda(\cdot))$.
Using this decomposition, one readily computes that $\mathcal{R}_c(f)$ satisfies the claimed regularity. 
\end{proof}

\subsection{On $[1,\infty)$} 
With this, we move to the exterior regime, $\rho\geqslant 1$. In this case, we use the transformation
\begin{equation}\label{Eq:transform}
	v(\rho) = \rho^{\frac{5}{2}} (\rho^2-1)^{-\frac{1}{4}-\frac{\lambda}{2}}u_1(\rho),
\end{equation}
which kills the first order derivative and turns \eqref{eq:OriginalODE} into
\begin{equation}
	\label{eq:NoFirstDerivativeODE}
	v''(\rho) + \frac{-15+(10-12\lambda-4\lambda^2)\rho^2}{4\rho^2(\rho^2-1)^2} v(\rho) - \frac{V(\rho)}{\rho^2-1}v(\rho)=0.
\end{equation}
Now, to construct solutions to this equation on the interval $(1,\infty)$, we rewrite it as
	\begin{equation}\label{eq:NoFirstDerivativeODE2}
	v''(\rho) + \frac{-5-12\lambda-4\lambda^2}{4(\rho^2-1)^2} v(\rho) =-\left[-\frac{V(\rho)}{\rho^2-1}+\frac{15}{4\rho^2(\rho^2-1)}\right]v(\rho)
\end{equation}
and note, that the equation
\begin{align*}
h''(\rho) + \frac{-5-12\lambda-4\lambda^2}{4(1-\rho^2)^2} h(\rho) =0
\end{align*}
has a fundamental system of solutions given by
\begin{align*}
h_1(\rho;\lambda)&=\frac{\sqrt{\rho^2-1}}{\sqrt{3+2\lambda}} \left(\frac{1+\rho}{\rho-1}\right)^{\frac{\lambda}{2}+\frac{3}{4}}
\\
h_2(\rho;\lambda)&=\frac{\sqrt{\rho^2-1}}{\sqrt{3+2\lambda}} \left(\frac{\rho-1}{1+\rho}\right)^{\frac{\lambda}{2}+\frac{3}{4}}
\end{align*}
where the square root of $3+2\lambda$ denotes the principal branch. 
\begin{lemma}\label{lem:solsnear1}

There exists an $r> 1$, such that Eq.~\eqref{eq:NoFirstDerivativeODE} has a fundamental system of the form 
\begin{align*}
v_1(\rho;\lambda)&=\frac{\sqrt{\rho^2-1}}{\sqrt{3+2\lambda}} \left(\frac{1+\rho}{\rho-1}\right)^{\frac{\lambda}{2}+\frac{3}{4}}g_1(\rho;\lambda)
\\
v_2(\rho;\lambda)&=\frac{\sqrt{\rho^2-1}}{\sqrt{3+2\lambda}} \left(\frac{\rho-1}{\rho+1}\right)^{\frac{\lambda}{2}+\frac{3}{4}}g_2(\rho;\lambda)
\end{align*}
for all $1\leqslant \rho<r$ and alll $\lambda \in S$,
where $$g_j(\rho;\lambda)=[1+(\rho-1)\O_o(\langle \omega\rangle^{-1})+(\rho-1)\widehat g_j(\rho;\lambda)] \text{ with }|\partial_\rho^n\partial_\omega^\ell \widehat g_j(\rho;\lambda)|\lesssim_{n,\ell}\langle\omega\rangle^{-2-\ell}$$
for all $n,\ell \in \mathbb{N}_0$
\end{lemma}
\begin{proof}
We prove this Lemma by means of Volterra iterations. For this, we set
$$
\widetilde V(\rho)=\frac{15}{4\rho^2}-V(\rho)
$$ and note that the Wronskian of $h_1$ and $h_2$ is given by
$$
W(h_1(\cdot;\lambda),h_2(\cdot;\lambda))(\rho)=1.
$$
Thus, we look for solutions of the integral equation
\begin{align}\label{eq:int1}
f(\rho;\lambda)=h_2(\rho;\lambda)-h_2(\rho;\lambda)\int_{\rho_0}^\rho \frac{h_1(s;\lambda)f(s;\lambda)\widetilde V(s)}{s^2-1}ds+h_1(\rho;\lambda)\int_{\rho_0}^\rho \frac{h_2(s;\lambda)f(s;\lambda)\widetilde V(s)}{s^2-1}ds
\end{align}
with $\rho_0\geqslant 1$.
Given that $h_2$ is non-vanishing, we can divide \eqref{eq:int1} by it and set $g=\frac{f}{h_2}$  to arrive at the equation
\begin{align}\label{eq:intk}
g(\rho;\lambda)=1+\int_{\rho_0}^\rho K(\rho,s;\lambda) g(s;\lambda) ds
\end{align}
where
\begin{align*}
K(\rho,s;\lambda)=\frac{\left[\left(\frac{1+\rho}{\rho-1}\frac{s-1}{s+1}\right)^{\lambda+\frac{3}{2}} -1\right]\widetilde{V}(s)}{3+2\lambda}
\end{align*}
Note that
$$
\int_{1}^r \sup_{\rho \in (s,r)} |K(\rho,s;\lambda)|ds\lesssim \frac{1}{|3+2\lambda|}.
$$

Hence, we can set $\rho_0=1$ and infer that there exists a unique solution $g_2$ to \eqref{eq:intk}. 
By performing minor modifications of the proof of Lemma 4.1 in \cite{DonningerWallauch}, one obtains that
$$
g_2(\rho;\lambda)=1+(\rho-1)\O(\langle\omega\rangle^{-1})+\O((\rho-1)^2\langle\omega\rangle^{-1}).
$$
To show that $g_2$ is smooth up to $\rho=1$, we undo the transformation \eqref{Eq:transform} to conclude that $u_2(\rho;\lambda):=\rho^{-\frac{5}{2}} (\rho^2-1)^{\frac{1}{4}+\frac{\lambda}{2}}v_2$ is an element of $C^1([1,r)),$ which is locally of the form $u_2(\rho;\lambda)=(\rho-1)^{\frac{3}{2}+\lambda}g_2(\rho;\lambda)$ and solves \eqref{eq:OriginalODE} on $(1,r)$. Hence, as the Frobenius indices at $\rho=1$ are given by $\{0,\frac{3}{2}+\lambda\}$ it must be a multiple of the vanishing Frobenius solution and the claim follows.  To show that the leading order behavior in $\omega$ is indeed given by $\O_o(\langle\omega\rangle^{-1})$, we expand 
\begin{align*}
\int_{\rho_0}^\rho K(\rho,s;\lambda)ds=-\frac{1}{3+2\lambda}\int_1^\rho\widetilde{V}(s)ds +\int_1^\rho \frac{\left(\frac{1+\rho}{\rho-1}\frac{s-1}{s+1}\right)^{\lambda+\frac{3}{2}}\widetilde{V}(s)}{3+2\lambda} ds.
\end{align*}
One readily computes that 
\begin{align*}
-\frac{1}{3+2\lambda}\int_1^\rho\widetilde{V}(s) ds=(\rho-1)e(\rho)\left[\frac{2\omega i}{|3+2\lambda|^2}- \frac{-3+2\varepsilon}{|3+2\lambda|^2}\right]= (\rho-1)e(\rho)[\O_o(\langle\omega\rangle^{-1})+\O(\langle\omega\rangle^{-2})]
\end{align*}
where $e=\O(\rho^{-1})$.
Furthermore, one has that
\begin{align*}
\partial_s \left(\frac{s-1}{s+1}\right)^{\lambda+\frac{3}{2}}=(3+2\lambda)\left(\frac{s-1}{s+1}\right)^{\lambda+\frac{3}{2}}(s^2-1)^{-1}
\end{align*}
which we use repeatedly to derive that
\begin{align*}
\int_1^\rho \frac{\left(\frac{1+\rho}{\rho-1}\frac{s-1}{s+1}\right)^{\lambda+\frac{3}{2}}\widetilde{V}(s)}{3+2\lambda} ds&=\int_1^\rho \widetilde{V}(s)(s^2-1)\frac{\partial_s\left(\frac{1+\rho}{\rho-1}\frac{s-1}{s+1}\right)^{\lambda+\frac{3}{2}}}{(3+2\lambda)^2} ds
\\
&= \frac{\widetilde V(\rho)(\rho^2-1)}{(3+2\lambda)^2}-\int_1^\rho \partial_s[\widetilde{V}(s)(s^2-1)]\frac{\left(\frac{1+\rho}{\rho-1}\frac{s-1}{s+1}\right)^{\lambda+\frac{3}{2}}}{(3+2\lambda)^2} ds
\\
&=(1-\rho^2)\sum_{j=2}^n r_n(\rho)(3+2\lambda)^{-j}+\int_1^\rho \widehat{V}_n(s)\frac{\left(\frac{1+\rho}{\rho-1}\frac{s-1}{s+1}\right)^{\lambda+\frac{3}{2}}}{(3+2\lambda)^n} ds
\end{align*}
for a smooth function $\widehat{V}_n(s)$.
Furthermore,  by scaling
\begin{align*}
    &\quad  +\int_1^\rho \widehat{V}(s)\frac{\left(\frac{1+\rho}{\rho-1}\frac{s-1}{s+1}\right)^{\lambda+\frac{3}{2}}}{(3+2\lambda)^n} ds
    \\
    &= \int_0^{\rho-1}\widehat{V}(s+1)\frac{\left(\frac{1+\rho}{\rho-1}\frac{s}{s+2}\right)^{\lambda+\frac{3}{2}}}{(3+2\lambda)^n} ds
    \\
    &=(\rho-1)\int_0^{1} \widehat{V}(1+t(\rho-1))\frac{\left(\frac{(1+\rho)t}{t(\rho-1)+2}\right)^{\lambda+\frac{3}{2}}}{(3+2\lambda)^n}dt.
\end{align*}
Consequently, since $n$ was arbitrary, we see that for $\ell,k\in \mathbb{N}_0$, the estimates
\begin{align*}
   \left| \partial_\rho^\ell \partial_\omega^k \int_1^\rho K(\rho,s;\lambda)ds\right|&\lesssim_{\ell,k} |\partial_\omega^k (3+2\lambda)^{-1} \partial_\rho^{\ell-1}\widetilde V(\rho)|+   
   \langle \omega \rangle^{-2-k}.
\end{align*}
Hence, one readily obtains the desired form of
$
g_2(\rho;\lambda).
$
To construct the second solution, we use the variation of constants formula as in the proof of Lemma 4.1 in \cite{DonningerWallauch}.
\end{proof}
In a similar fashion, one constructs the solutions near $\infty$.

\begin{lemma}
Let $r$ be as in Lemma \ref{lem:solsnear1} and for For $\sigma\in\{+1,-1\}$, set
\[
    h_\sigma(\rho;\lambda)
    :=\frac{\sqrt{\rho^2-1}}{\sqrt{3+2\lambda}}
    \left(\frac{\rho+1}{\rho-1}\right)^{
        \sigma(\frac34+\frac{\lambda}{2})}.
\] Then, there exists a fundamental system to Eq.~\eqref{eq:NoFirstDerivativeODE} of the form 
\begin{align*}
v_3(\rho;\lambda)&=h_+(\rho;\lambda)g_3(\rho;\lambda)
\\
v_4(\rho;\lambda)&=h_-(\rho;\lambda)g_4(\rho;\lambda)
\end{align*}
for all $\frac{1+r}{2}\leqslant \rho<\infty$ and $\lambda \in S$,
where the $g_j$ are smooth and $g_j(\rho;\lambda)=1+O(\rho^{-1}\langle\omega\rangle^{-1})$.
Furthermore, we can rewrite these solutions as 
\begin{align*}
    v_3(\rho;\lambda)
    &=h_+(\rho;\lambda)
    \left[
        1-\frac{1}{3+2\lambda}
        \int_\rho^\infty \widetilde V(s)\,ds
        +a_+(\rho;\lambda)
    \right]
    +h_-(\rho;\lambda)c_+(\rho;\lambda)
    \\
     v_4(\rho;\lambda)
    &=h_-(\rho;\lambda)
    \left[
        1+\frac{1}{3+2\lambda}
        \int_\rho^\infty \widetilde V(s)\,ds
        +a_-(\rho;\lambda)
    \right]
    +h_+(\rho;\lambda)c_-(\rho;\lambda)
\end{align*}
where $\widetilde V(\rho)=\frac{15}{4\rho^2}-V(\rho)$ and
\[
    \bigl|\partial_\rho^j\partial_\omega^n
        a_\sigma(\rho;\lambda)\bigr|
    +\bigl|\partial_\rho^j\partial_\omega^n
        c_\sigma(\rho;\lambda)\bigr|
    \lesssim_{j,n}
    \rho^{-j}\langle\omega\rangle^{-2-n}
\]
\end{lemma}
\begin{proof}
By performing two Volterra iterations, one for each solution, one constructs the two claimed solutions. Furthermore, to show that the constructed solutions are linearly independent, we note that for $\lambda\in S$ fixed we have, by construction,
\begin{align*}
    g_3(\rho;\lambda)&=1+\int_{\rho_0}^\rho \frac{\left[\left(\frac{\rho-1}{\rho+1}\frac{s+1}{s-1}\right)^{\lambda+\frac{3}{2}} -1\right]\widetilde{V}(s)}{3+2\lambda} g(s;\lambda) ds
    \\
    &=1+\int_{\rho}^\infty \frac{\left[\left(\frac{\rho-1}{\rho+1}\frac{s+1}{s-1}\right)^{\lambda+\frac{3}{2}} -1\right]\widetilde{V}(s)}{3+2\lambda} ds+\O(\rho^{-3}).
\end{align*}
and \begin{align*}
    g_4(\rho;\lambda)  =1-\int_{\rho}^\infty \frac{\left[\left(\frac{\rho+1}{\rho-1}\frac{s-1}{s+1}\right)^{\lambda+\frac{3}{2}} -1\right]\widetilde{V}(s)}{3+2\lambda} ds+\O(\rho^{-3}).
\end{align*}
In particular, 
$$
|\partial_\rho g_j(\rho;\lambda)|\lesssim_\lambda \rho^{-3}, \qquad \text{ for } j=3,4.
$$
Therefore,
\begin{align*}
    W(v_3(\cdot;\lambda),v_4(\cdot;\lambda))(\rho)= 1+\frac{\rho^2-1}{3+2\lambda}\O(\rho^{-3}).
\end{align*}
Thus, as the Wronskian of $v_3$ and $v_4$ is constant, we can let $\rho\to \infty$, to conclude that 
\begin{equation*}
W(v_3(\cdot;\lambda),v_4(\cdot;\lambda))(\rho)= +1.
\end{equation*}
Next, to establish the more detailed decomposition, 
we give the argument for $v_3$. Set
\[
    x:=\frac12\log\frac{\rho+1}{\rho-1},\qquad
    y:=\frac12\log\frac{s+1}{s-1}.
\]
Thus $\rho=\coth x$, $s=\coth y$, and $s=\infty$ corresponds to
$y=0$. For notational convenience we also define 
\[
    G(x;\lambda):=g_3(\coth x;\lambda),
    \qquad
    q(y):=(\coth^2 y-1)\widetilde V(\coth y).
\]
By applying these changes of variables to the Volterra equation above we arrive at
\begin{equation}\label{eq:far-field-volterra}
    G(x;\lambda)
    =1+\frac1{(3+2\lambda)}\int_0^x
    \bigl(e^{-(3+2\lambda)(x-y)}-1\bigr)q(y)G(y;\lambda)\,dy.
\end{equation}
Since $s^2\widetilde V(s)$ is smooth in $s^{-1}$ near infinity
and $s^{-1}=\tanh y$, the function $q$ extends smoothly to $y=0$.
Moreover, $x$ ranges over a fixed compact interval $[0,x_*]$.

To identify the leading terms, consider the first Volterra iterate.
One integration by parts gives
\begin{align}
    \frac1{3+2\lambda}\int_0^x
    \bigl(e^{-(3+2\lambda)(x-y)}-1\bigr)q(y)\,dy
    ={}&-\frac1{3+2\lambda}\int_0^x q(y)\,dy
        +\frac{q(x)}{(3+2\lambda)^2}
        -e^{-(3+2\lambda)x}\frac{q(0)}{(3+2\lambda)^2} \notag\\
       &-\frac1{(3+2\lambda)^2}\int_0^x e^{-(3+2\lambda)(x-y)}q'(y)\,dy.
    \label{eq:far-field-first-iterate}
\end{align}
In particular, the leading correction has order $(3+2\lambda)^{-1}$, and
the endpoint $y=0$ produces the phase $e^{-(3+2\lambda)x}$ at order $(3+2\lambda)^{-2}$.
The same two phases occur at every subsequent Volterra iteration.
Indeed, applying the kernel in \eqref{eq:far-field-volterra} to a
term of the form $A(y)+e^{-(3+2\lambda)y}C(y)$, we use
\[
    \bigl(e^{-(3+2\lambda)(x-y)}-1\bigr)e^{-(3+2\lambda)y}
    =e^{-(3+2\lambda)x}-e^{-(3+2\lambda)y}.
\]
Thus the oscillatory integrals involve only $e^{-(3+2\lambda)(x-y)}$ or
$e^{-(3+2\lambda)y}$. Integration by parts in $y$ gives a factor $(3+2\lambda)^{-1}$,
and evaluation at either endpoint produces either $1$ or the oscilatory factor $e^{-(3+2\lambda)x}$. Iterating this argument gives, to
arbitrary order, an expansion in powers of $(3+2\lambda)^{-1}$ with these two
phases. 
\end{proof}

We also readily compute the connection coefficients of these solutions.
\begin{lemma}\label{lem:concoef}
There exist functions $a_{i,j}(\lambda)$  with $i=1,2$, $j=3,4$ such that
\begin{align*}
v_1(\rho;\lambda)= a_{1,3}(\lambda) v_3(\rho;\lambda)+a_{1,4} v_4(\rho;\lambda)
\\
v_2(\rho;\lambda)= a_{2,3}(\lambda) v_3(\rho;\lambda)+a_{2,4} v_4(\rho;\lambda)
\end{align*}
with
\begin{align*} 
	a_{1,3}(\lambda)&=1+ \O_o(\langle\omega\rangle^{-1})+\O(\langle\omega\rangle^{-2})\quad &&a_{1,4}(\lambda)=(\tfrac{1+r}{1-r})^{\lambda}(\O\langle\omega\rangle^{-2})
	\\
	a_{2,3}(\lambda)&= (\tfrac{1+r}{1-r})^{-\lambda}(\O(\langle\omega\rangle^{-2}) \quad &&a_{2,4}(\lambda)=1+ \O_o(\langle\omega\rangle^{-1})+\O(\langle\omega\rangle^{-2}).
\end{align*}
for all $\rho \in (\frac{1+r}{2}, r)$.
\end{lemma}
\begin{proof}
Given that 
$$
W(v_3,v_4)=1
$$
one readily computes that
\begin{align*}
a_{1,3}=W(v_1(\cdot;\lambda)v_4(\cdot;\lambda))(\rho)=W(v_1(\cdot;\lambda)v_4(\cdot;\lambda))(r)=1+ \O_o(\langle\omega\rangle^{-1})+\O(\langle\omega\rangle^{-2})
\end{align*}
In a similar fashion, the remaining coefficients are computed.
\end{proof}
Hence, by reversing the transformation \eqref{Eq:transform}, that eliminated the first order terms in equation \eqref{eq:OriginalODE}
 and picking a smooth cut-off, which we call $\kappa$, that is supported on $[1,r]$ and equals $1$ on $[1,\frac{1}{2}(1+r)]$, one arrives at the following.
 \begin{lemma}
 	\label{lemma:FundamentalSystem1Infty}
For all $\rho \geqslant 1$ and all $\lambda\in S$, the equation \begin{equation}
	\label{eq:OriginalODEhom}
	(1-\rho^2)u_1''(\rho) + \Big(\frac{5}{\rho}+2(\lambda-2)\rho\Big)u_1'(\rho) - ((\lambda-1)(\lambda-2)-V(\rho))u_1(\rho)=0
\end{equation}
has a fundamental system of solutions given by
\begin{equation}
	\label{eq:FundamentalSystemOn1Infty}
	\begin{split}
		\tilde{u}_{1}(\rho;\lambda)&=
		\kappa(\rho)\rho^{-\frac{5}{2}}\frac{(1+\rho)^{\frac32+\lambda}}{\sqrt{3+2\lambda}}g_1(\rho;\lambda)\nonumber
		\\
		&\quad +(1-\kappa(\rho))\rho^{-\frac{5}{2}}\left[a_{1,3}(\lambda)\frac{(1+\rho)^{\frac32+\lambda}}{\sqrt{3+2\lambda}}g_3(\rho;\lambda)
		+
		a_{1,4}(\lambda)\frac{(\rho-1)^{\frac32+\lambda}}{\sqrt{3+2\lambda}}g_4(\rho;\lambda)\right]
		\\
		\tilde u_2(\rho;\lambda)&= \kappa(\rho)\rho^{-\frac{5}{2}}\frac{(\rho-1)^{\frac32+\lambda}}{\sqrt{3+2\lambda}}g_2(\rho;\lambda)
		\\
		&\quad
		+(1-\kappa(\rho))\rho^{-\frac{5}{2}}\left[a_{2,3}(\lambda)\frac{(1+\rho)^{\frac32+\lambda}}{\sqrt{3+2\lambda}}g_3(\rho;\lambda)
		+ a_{2,4}(\lambda)\frac{(\rho-1)^{\frac32+\lambda}}{\sqrt{3+2\lambda}}g_4(\rho;\lambda)\right]
	\end{split}
\end{equation}
for
\begin{align*} 
	a_{1,3}(\lambda)&=1+ \O_o(\langle\omega\rangle^{-1})+\O(\langle\omega\rangle^{-2})\quad &&a_{1,4}(\lambda)=(\tfrac{r+1}{r-1})^{\lambda}\O(\langle\omega\rangle^{-2})+\O\langle\omega\rangle^{-2})
	\\
	a_{2,3}(\lambda)&= (\tfrac{r+1}{r-1})^{-\lambda}\O(\langle\omega\rangle^{-2})+\O(\langle\omega\rangle^{-2}) \quad &&a_{2,4}(\lambda)=1+ \O_o(\langle\omega\rangle^{-1})+\O(\langle\omega\rangle^{-2})
\end{align*}
and where the $g_j$ are jointly smooth functions that satisfy
\begin{align*}
g_j(\rho;\lambda)=[1+(\rho-1)\widehat g_j(\rho;\lambda)] \text{ where }|\partial_\rho^n\partial_\omega^\ell \widehat g_j(\rho;\lambda)|\lesssim_{n,\ell}\langle\omega\rangle^{-1-\ell}
\end{align*}
for $j=1,2$, $n,\ell\in \mathbb{N}_0$, and all $1<\rho<r$ as well as $\lambda \in S$
 and
\begin{equation*}
g_j(\rho;\lambda)=1+O(\rho^{-1} \langle\omega\rangle^{-1})
\end{equation*}
for $j=3,4$ and $\rho \geqslant \frac{1+r}{2}$.
Furthermore, their Wronskian is given by
\begin{equation*}
W(\rho;\lambda):=W(\tilde{u}_1(\cdot;\lambda),\tilde{u}_2(\cdot;\lambda))(\rho)=\rho^{-5} (\rho^2-1)^{\frac{1}{2}+\lambda}.
\end{equation*}
\end{lemma}


\begin{remark}
    For later purposes, we define
\begin{equation}
    \zeta(\lambda)=a_{1,3}(\lambda)a_{2,4}(\lambda)-a_{1,4}(\lambda)a_{2,3}(\lambda),
\end{equation}
and 
\begin{equation}
    1=W(v_1,v_2)=\bigl(a_{1,3}a_{2,4}-a_{1,4}a_{2,3}\bigr)W(v_3,v_4)
=\zeta(\lambda)W(v_3,v_4)=\zeta(\lambda).
\end{equation}
\end{remark}

\subsubsection{The free counterparts} We also consider equation \eqref{eq:OriginalODE} with $V=0$:
\begin{equation}
	\label{eq:ODEWithV=0}
	(1-\rho^2)u_1''(\rho) + \Big(\frac{5}{\rho}+2(\lambda-2)\rho\Big)u_1'(\rho) - ((\lambda-1)(\lambda-2))u_1(\rho)=F_{\lambda}(\rho).
\end{equation}

We have the following result for \eqref{eq:ODEWithV=0}.
\begin{lemma}
	For all $\rho > 1$ and all $\lambda \in S$, the equation \eqref{eq:ODEWithV=0} has a fundamental system of solutions given by
\begin{equation}
	\begin{split}
		\tilde{u}_{\mathrm{f}_1}(\rho;\lambda)&=
		\kappa(\rho)\rho^{-\frac{5}{2}}\frac{(1+\rho)^{\frac32+\lambda}}{\sqrt{3+2\lambda}}g_{\mathrm{f}_1}(\rho;\lambda)\nonumber
		\\
		&\quad +(1-\kappa(\rho))\rho^{-\frac{5}{2}}\left[a_{\mathrm{f}_{1,3}}(\lambda)\frac{(1+\rho)^{\frac32+\lambda}}{\sqrt{3+2\lambda}}g_{\mathrm{f}_3}(\rho;\lambda)
		+
		a_{\mathrm{f}_{1,4}}(\lambda)\frac{(\rho-1)^{\frac32+\lambda}}{\sqrt{3+2\lambda}}g_{\mathrm{f}_4}(\rho;\lambda)\right]
		\\
		\tilde {u}_{\mathrm{f}_2}(\rho;\lambda)&= \kappa(\rho)\rho^{-\frac{5}{2}}\frac{(\rho-1)^{\frac32+\lambda}}{\sqrt{3+2\lambda}}g_{\mathrm{f}_2}(\rho;\lambda)
		\\
		&\quad
		+(1-\kappa(\rho))\rho^{-\frac{5}{2}}\left[a_{\mathrm{f}_{2,3}}(\lambda)\frac{(1+\rho)^{\frac32+\lambda}}{\sqrt{3+2\lambda}}g_{\mathrm{f}_3}(\rho;\lambda)
		+ a_{\mathrm{f}_{2,4}}(\lambda)\frac{(\rho-1)^{\frac32+\lambda}}{\sqrt{3+2\lambda}}g_{\mathrm{f}_4}(\rho;\lambda)\right]
	\end{split}
\end{equation}
\begin{align*} 
	a_{\mathrm{f}_{1,3}}(\lambda)&=1+ \O_o(\langle\omega\rangle^{-1})+\O(\langle\omega\rangle^{-2})\quad &&a_{\mathrm{f}_{1,4}}(\lambda)=(\tfrac{r+1}{r-1})^{\lambda}(\O\langle\omega\rangle^{-2})
	\\
	a_{\mathrm{f}_{2,3}}(\lambda)&= (\tfrac{r+1}{r-1})^{-\lambda}\O(\langle\omega\rangle^{-2}) \quad &&a_{\mathrm{f}_{2,4}}(\lambda)=1+ \O_o(\langle\omega\rangle^{-1})+\O(\langle\omega\rangle^{-2})
\end{align*}
and where the $g_{\mathrm{f}_j}$ are jointly smooth functions that satisfy
\begin{align*}
	g_{\mathrm{f}_j}(\rho;\lambda)=[1+(\rho-1)\widehat g_{\mathrm{f}_j}(\rho;\lambda)] \text{ where }|\partial_\rho^n\partial_\omega^\ell \widehat g_{\mathrm{f}_j}(\rho;\lambda)|\lesssim_{n,\ell}\langle\omega\rangle^{-1-\ell}
\end{align*}
for $j=1,2$, $n,\ell\in \mathbb{N}$, and all $1<\rho<r$ as well as $\lambda \in S$
and
\begin{align*}
	g_{\mathrm{f}_j}(\rho;\lambda)=1+\O(\rho^{-1} \langle\omega\rangle^{-1})
\end{align*}
for $j=3,4$. In parallel, we have
\begin{equation}
    \label{ZetaF}
    \zeta_{\mathrm{f}}(\lambda) :=a_{\mathrm{f}_{1,3}}(\lambda) a_{\mathrm{f}_{2,4}}(\lambda) - a_{\mathrm{f}_{2,3}}(\lambda)a_{\mathrm{f}_{1,4}}(\lambda) = 1.
\end{equation}
\end{lemma}




\subsection{Constructing and gluing up the resolvents}  In this section, we construct the resolvent on $(0,1)$ and $(1,\infty)$ separately and then glue them up to  construct the full solution to \eqref{eq:OriginalODE}. 




We prove the following result:
\begin{lemma}
		\label{lemma:ResolventFull}
		Let $f \in C^2_c(\R^6)$ and $\lambda \in \widehat{S}, 0 < \Re(\lambda)\leqslant \min\big(\varepsilon_V,\frac{1}{300}\big)$. Then, the function $\mathcal{R}(f)$ given by
			\begin{equation*}
			\mathcal{R}(f)(\rho;\lambda):=
			\begin{cases}
				&u_2(\rho;\lambda) \displaystyle\int_0^{\rho} \frac{u_0(s;\lambda)f(s)}{(1-s^2)W(u_0,u_2)(s;\lambda)}ds + u_0(\rho;\lambda) \displaystyle \int_{\rho}^1 \frac{u_2(s;\lambda) f(s)}{(1-s^2) W(u_0,u_2)(s;\lambda)}ds\\
                &\quad +c u_0(\rho;\lambda), \hspace{4.5cm} \rho\in[0,1], \\
				&\tilde{u}_1(\rho;\lambda)\displaystyle \int_{\rho}^{\infty} \frac{\tilde{u}_2(s;\lambda)f(s)}{W(\tilde{u}_1, \tilde{u}_2)(s;\lambda)(1-s^2)}ds\\
                &\quad - \tilde{u}_2(\rho;\lambda) \displaystyle\int_{\rho}^{\infty} \frac{\tilde{u}_1(s;\lambda)f(s)}{W(\tilde{u}_1, \tilde{u}_2)(s;\lambda)(1-s^2)} ds, \quad~ \rho \in [1,\infty).
			\end{cases}
		\end{equation*}


with 
\begin{equation*}
 c:=\mu(f;\lambda)=-\frac{1}{c_{0,1}(\lambda)} \int_1^{\infty} \frac{\tilde{u}_2(s;\lambda)f(s)}{(s^2-1)W(\tilde{u}_1,\tilde{u}_2)(s;\lambda)}ds.
\end{equation*}
		is a radial function $\mathfrak{u} \in \dot{H}^{2}(\R^6)$ solving the equation: 
		\begin{equation}
			\label{eq:Full}
			(1-\rho^2)\mathfrak{u}''(\rho) + \Big( \frac{5}{\rho} + 2(\lambda-2)\rho \Big)\mathfrak{u}'(\rho) -\big( (\lambda-1)(\lambda-2)-V(\rho))\mathfrak{u}(\rho) = f(\rho).
		\end{equation}
        Furthermore, if $f(\rho;\lambda)=(2-\lambda)f_1(\rho) +\rho f_1'(\rho)-f_2(\rho)$ for $f=(f_1,f_2) \in C_{c,rad}^{\infty}(\R^6) \times C_{c,rad}^{\infty}(\R^6)$, then $\mathcal{R}(f)$ is the unique solution of \eqref{eq:Full} such that  $(\mathfrak{u}, (\lambda-1)\mathfrak{u} - \rho \mathfrak{u}'-f_1) \in \mathcal{H}^2$.
        \begin{proof}
\textbf{On $\rho \in (0,1)$}: Using the fundamental system $\{u_0,u_2\}$ on $(0,1)$, and variation of constants, a solution is given by
\begin{align*}
      u_2(\rho;\lambda)\int_0^\rho \frac{u_0(s;\lambda)f(s) }{(1-s^2)W(u_0(\cdot;\lambda),u_2(\cdot;\lambda))} ds+
      u_0(\rho;\lambda)\int_\rho^1 \frac{u_2(s;\lambda)f(s) }{(1-s^2)W(u_0(\cdot;\lambda),u_2(\cdot;\lambda))} ds.
\end{align*}
Since $u_0 \in H^2(\B^6)$, we make use of the degree of freedom to write
\begin{equation*}
\begin{split}
      \mathcal{R}_{int}(f)(\rho;\lambda):&=u_2(\rho;\lambda)\int_0^\rho \frac{u_0(s;\lambda)f(s) }{(1-s^2)W(u_0(\cdot;\lambda),u_2(\cdot;\lambda))} ds+
      u_0(\rho;\lambda)\int_\rho^1 \frac{u_2(s;\lambda)f(s) }{(1-s^2)W(u_0(\cdot;\lambda),u_2(\cdot;\lambda))} ds\\
      &\quad +c u_0(\rho;\lambda).
      \end{split}
\end{equation*}
The claimed regularity has already been discussed in Lemma \ref{lemma:ResolventInsideCone}.\\
\textbf{On $(1,\infty)$}: We denote $W(\tilde{u}_1(\cdot;\lambda),\tilde{u}_2(\cdot;\lambda))=:W(\cdot;\lambda)$. Using the fundamental system from Lemma \ref{lemma:FundamentalSystem1Infty}, a solution to \eqref{eq:OriginalODE} on $[1,\infty)$ 
\begin{equation*}
	\begin{split}
	u(\rho;\lambda) &= \tilde{u}_1(\rho;\lambda)\int_{\rho}^{\infty} \frac{\tilde{u}_2(s;\lambda)f(s)}{W(s;\lambda)(1-s^2)}ds - \tilde{u}_2(\rho;\lambda) \int_{\rho}^{\infty} \frac{\tilde{u}_1(s;\lambda)f(s)}{W(s;\lambda)(1-s^2)} ds\\
    &=-\tilde{u}_1(\rho;\lambda) \int_{\rho}^{\infty} \frac{s^5 \tilde{u}_2(s;\lambda)f(s)}{(s^2-1)^{\frac{3}{2}+\lambda}} ds +\tilde{u}_2(\rho;\lambda) \int_{\rho}^{\infty} \frac{s^5 \tilde{u}_1(s;\lambda)f(s)}{(s^2-1)^{\frac{3}{2}+\lambda}} ds. 
	\end{split}
\end{equation*}
Spelling out the functions $\tilde{u}_i, i=1,2$, using integration by parts and observing that $1-\kappa(\cdot)$ is supported away from $1$, we can easily conclude that the resolvent is well-defined on $[1,\infty)$. Next, we glue the solutions at the point $\rho=1$. From the solution constructed on $(0,1)$, we have 
\begin{equation*}
\begin{split}
     u(\rho;\lambda)    &=u_2(\rho;\lambda)U_0(\rho;\lambda)f(\rho)-u_2(\rho;\lambda)\int_0^\rho U_0(s;\lambda) f'(s) ds
      \\
      &\quad -u_0(\rho;\lambda)U_2(\rho;\lambda)f(\rho)-u_0(\rho;\lambda)\int_\rho^1 U_2(s;\lambda) f'(s) ds
      \\
      &\quad+cu_0(\rho;\lambda)
      \end{split}
\end{equation*}
where we denote
\begin{equation*}
    U_j(\rho;\lambda):=\int_0^{\rho} \frac{u_j(s;\lambda)}{(1-s^2)W(u_0,u_2)(s;\lambda)}ds, \quad \text{ for } j=0,1,2.
\end{equation*}
We note that at the matching point, namely at $\rho=1$, the above evaluates to $c u_0(1;\lambda)$ for some constant $c$. Moreover, from the previous computation, we observe that as $\rho \to 1^+$, the non-trivial contribution is

\begin{equation}
    -\tilde{u}_1(1;\lambda)\int_1^{\infty} \frac{\tilde{u}_2(s;\lambda)f(s)}{(s^2-1)W(\tilde{u}_1,\tilde{u}_2)(s;\lambda)}ds
\end{equation}
To ensure that the solution is continuous at $\rho=1$, we choose
\begin{equation*}
    \label{eq:RawC}
    c:=\mu(f;\lambda)=-\frac{1}{c_{0,1}(\lambda)}\int_1^{\infty} \frac{\tilde{u}_2(s;\lambda)f(s)}{(s^2-1)W(\tilde{u}_1,\tilde{u}_2)(s;\lambda)}ds.
\end{equation*}
Since the coefficient of $u''$ vanishes at $\rho=1$, and the coefficient of $u'$ is non-zero for $Re (\lambda)>0$, one observes that the value of $u'(1)$ is determined by $u(1)$, hence $C^1$ and $u \in \dot{H}^2$ locally for $Re(\lambda) >0$.\\

\textbf{Regularity of the resolvent}: From the construction of the resolvent via variation of parameters and using the explicit expression for the fundamental systems, namely $u_0, \tilde{u}_0, \tilde{u}_1$ and $\tilde{u}_2$, we can deduce that $\mathcal{R}(f)(\rho;\lambda) \in \dot{H}^2(\R^6)$.
This completes the proof.
\end{proof}
\end{lemma}

For later uses, we write down the elaborate form of the constant $c(f;\lambda)$ in the following: 

\begin{lemma}
    \label{lemma:FullDefinitionOfC}
    The matching constant $c$ is given by
    \begin{equation}
        \begin{split}
            c=\mu(f;\lambda)&= -\frac{1}{c_{0,1}(\lambda)\sqrt{3+2\lambda}} \Big[\int_1^{\infty} \frac{\kappa(s)s^{\frac{5}{2}}g_2(s;\lambda)f(s)}{(s+1)^{\frac{3}{2}+\lambda}}ds + a_{2,3}(\lambda) \int_1^{\infty} \frac{(1-\kappa(s)) s^{\frac{5}{2}} g_3(s;\lambda) f(s)}{(s-1)^{\frac{3}{2}+\lambda}}ds\\
            &\quad + a_{2,4}(\lambda) \int_1^{\infty} \frac{(1-\kappa(s)) s^{\frac{5}{2}} g_4(s;\lambda) f(s)}{(s+1)^{\frac{3}{2}+\lambda}}ds\Big].
        \end{split}
    \end{equation}
    \begin{proof}
        The form follows by spelling out the Wronskian $W(\tilde{u}_1, \tilde{u}_2)$ and the function $\tilde{u}_2$.
    \end{proof}
\end{lemma}

	\subsection{Resolvent in the case $V \equiv 0$}
	Using the fundamental system corresponding to the case $V\equiv 0$, we follow the same procedure to obtain that the free resolvent $\mathcal{R}_{\mathrm{f}}(F_{\lambda})$ on $[1,\infty)$  is given by
	\begin{equation*}
		\begin{split}
			\mathcal{R}_{\mathrm{f}}(F_{\lambda})(\rho;\lambda) &= \tilde{u}_{\mathrm{f}_1}(\rho;\lambda) \int_{\rho}^{\infty} \frac{\tilde{u}_{\mathrm{f}_2}(s;\lambda)F_{\lambda}(s)}{W(s;\lambda)(1-s^2)} ds -\tilde{u}_{\mathrm{f}_2}(\rho;\lambda) \int_{\rho}^{\infty} \frac{\tilde{u}_{\mathrm{f}_1}(s;\lambda)F_{\lambda}(s)}{W(s;\lambda)(1-s^2)} ds.
		\end{split}
	\end{equation*}
	 Thus, we define the free resolvent $\mathcal{R}_{\mathrm{f}}$ as follows:

    \begin{equation*}
			\mathcal{R}_{\mathrm{f}}(f)(\rho;\lambda):=
			\begin{cases}
				&u_{\mathrm{f}_2}(\rho;\lambda) \displaystyle\int_0^{\rho} \frac{u_{\mathrm{f}_0}(s;\lambda)f(s)}{(1-s^2)W(u_{\mathrm{f}_0},u_{\mathrm{f}_2})(s;\lambda)}ds + u_{\mathrm{f}_0}(\rho;\lambda) \displaystyle \int_{\rho}^1 \frac{u_{\mathrm{f}_2}(s;\lambda) f(s)}{(1-s^2) W(u_{\mathrm{f}_0},u_{\mathrm{f}_2})(s;\lambda)}ds\\
                &\quad +c_{\mathrm{f}}(f;\lambda) u_{\mathrm{f}_0}(\rho;\lambda), \hspace{4.5cm} \rho\in[0,1], \\
				&\tilde{u}_{\mathrm{f}_1}(\rho;\lambda)\displaystyle \int_{\rho}^{\infty} \frac{\tilde{u}_{\mathrm{f}_2}(s;\lambda)f(s)}{W(\tilde{u}_1, \tilde{u}_2)(s;\lambda)(1-s^2)}ds\\
                &\quad - \tilde{u}_{\mathrm{f}_2}(\rho;\lambda) \displaystyle\int_{\rho}^{\infty} \frac{\tilde{u}_{\mathrm{f}_1}(s;\lambda)f(s)}{W(\tilde{u}_1, \tilde{u}_2)(s;\lambda)(1-s^2)} ds, \quad~ \rho \in [1,\infty),
			\end{cases}
		\end{equation*}
        with $c_{\mathrm{f}}$ defined accordingly, and we use the fact that $W(\tilde{u}_1,\tilde{u}_2) = W(\tilde{u}_{\mathrm{f}_1},\tilde{u}_{\mathrm{f}_2})$.
        

	\subsection{Integral representation of the semigroup} 

    By uniqueness in the preceding scalar construction, we have
\[
[\Rf_{\Lf}(\lambda)\ff]_1
=\mathcal R(\lambda)F_\lambda,
\qquad
[\Rf_{\Lf_0}(\lambda)\ff]_1
=\mathcal R_{\mathrm f}(\lambda)F_\lambda,
\]
where
\[
F_\lambda(\rho)
=(2-\lambda)f_1(\rho)+\rho f_1'(\rho)-f_2(\rho).
\]
Inserting these identities into 
\eqref{eq:abstract-perturbation-representation} for $0 < \varepsilon < \min\{\varepsilon_V, \frac{1}{300} \}$
yields the following scalar representation 
	\begin{equation}
		\label{eq:IntegralRepresentation}
		\begin{split}
			[(1 - \mb P)\mb{S}(\tau)\mb{f}]_1(\rho) &= [\mb{S}_0(\tau)\mb{f}]_1(\rho) + [\tilde{\mb{S}}(\tau)\mb{f}](\rho)\\
			&=: [\mb{S}_0(\tau)\mb{f}]_1(\rho) + \frac{1}{2\pi i} \lim_{N \to \infty} \int_{\varepsilon -iN}^{\varepsilon+iN} e^{\lambda \tau} [\mathcal{R}(F_{\lambda})(\rho;\lambda) - \mathcal{R}_{\mathrm{f}}(F_{\lambda})(\rho;\lambda)]d \lambda
		\end{split}
	\end{equation}
where $F_{\lambda}(\rho;\lambda):=(2-\lambda)f_1(\rho) + \rho f_1'(\rho)-f_2(\rho)$. \\


\section{Strichartz estimates in the interior region}
\label{section:InteriorEstimates}
The goal of this section is to derive estimates of the form
\begin{align*}
    \|[\Sf(\tau)(\If-\Pf)\ff]_1\|_{L^q(\R_+)L^r(\B^6)}\lesssim \|\ff\|_{\dot{H}^2(\R^6)\times \dot H^1(\R^6)},
\end{align*}
and related higher order estimates, where $\Pf$ is the projection onto the unstable spectrum. We will accomplish this by first taking a look at a different realisation of $\Lf$ in the energy space
\[ \mathcal E:=
\dot H^1_{\mathrm{rad}}(\R^6)\times L^2_{\mathrm{rad}}(\R^6). \]
\begin{lemma}
        The operator  $(\widetilde \Lf_0,C^\infty_{c,rad}(\R^6)\times C^\infty_{c,rad}(\R^6))$ is closable in $\mathcal E$. Furthermore, its closure denoted by $\Lf_{0,1}$ generates a $C_0$-semigroup $\Sf_{0,1}$ that satisfies the estimate
        \begin{align*}
            \|\Sf_{0,1}(\tau)\ff\|_{\dot{H}^1(\R^6)\times L^2(\R^6)}&\lesssim e^{-\tau}\|\ff\|_{\dot{H}^1(\R^6)\times L^2(\R^6)}
        \end{align*}
        for all $\tau\geqslant 0$ and all $\ff\in \mathcal E$.
    \end{lemma}

    \begin{proof}
    The result follows as in  Proposition \ref{prop:ExplicitSolution} from the explicit representation of the semigroup together with the bounds established in Lemma \ref{lemma:DecayForPropagators}.
    \end{proof}
    
    \begin{lemma}
    The operator $\Lf_V': \mc D(\widetilde \Lf_0) \subset \mathcal E \to  \mathcal E$ extends to a compact operator on $\mathcal E$. 
    \end{lemma}

    \begin{proof}
    The proof is along the lines of Lemma $\ref{lemma:CompactPerturbH2}$.
    \end{proof}
    
    To move on, we set $\Lf_1:=\Lf_{0,1}+\Lf_V'$ and denote the semigroup generated by this operator by $\Sf_1$. In the same  fashion as the result on $\Lf$ and $\Sf$, one proves the following.
\begin{lemma}
\label{lemma:SpectralCharacterisation}
There exists an $N = N(V) \in \mathbb{N}_0$  such that the spectrum of the operator $\Lf_1$ satisfies
\begin{align*}
    \sigma(\Lf_1  ) \cap \{z\in \C: \Re z >  -1\} = 
    \{\lambda_1,\dots \lambda_N \}
\end{align*}
with distinct eigenvalues $\lambda_1,\dots \lambda_N > -1$ of finite algebraic multiplicity. Moreover, the elements of the generalized eigenspaces are supported in the ball $\overline{\B^6}$. 
Additionally, $\Lf_1$ generates a $C_0$-semigroup of bounded linear operators on $\mathcal E$ such that if $\Pf_{-\frac14}$ is the finite rank spectral projection onto $\sigma(\Lf_1) \cap \{z\in \C:\Re z\geq -\frac{1}{4}\}$ then 
	\begin{equation}\label{Est:S1bound}
		\| \Sf_1(\tau)(\If-\Pf_{-\frac14})\mb{f}\|_{\dot{H}^1(\R^6)\times L^2(\R^6)} \lesssim e^{-\frac{1}{4}\tau}\| \mb{f}\|_{\dot{H}^1(\R^6)\times L^2(\R^6)}
	\end{equation}
for all $\mb{f} \in \mathcal E$ and all $\tau \geqslant 0$. Lastly, for $\lambda\in \widehat{S}$,  and $f\in C^\infty_{c,rad}(\R^6)$, the function $\mathcal{R}(f)$ is the unique solution in $\dot H^1(\R^6)$ to the equation 
		\begin{equation}
			\label{eq:FUll}
			(1-\rho^2)u_1''(\rho) + \Big( \frac{5}{\rho} + 2(\lambda-2)\rho \Big)u_1'(\rho) -\big( (\lambda-1)(\lambda-2)-V(\rho))u_1(\rho) = f(\rho).
		\end{equation}
\end{lemma}

\begin{proof}
First, the decay of the semigroup and the compactness of the perturbation imply that  $\sigma(\Lf_1  ) \cap \{z\in \C: \Re z >  -1\}$ consists only of eigenvalues of finite algebraic multiplicity. Now, assuming that $\lambda \in \C$, $\Re \lambda >  -1$ is an eigenvalue of $\Lf_1$ with eigenfunction $\mb u = (u_1,u_2)$, then $u_1$ solves \eqref{eq:FullODE} (which is Eq.~\eqref{eq:FUll} above). Using the ODE analysis near singular points of the proof of Proposition \ref{Prop:SpectralCharacterization} implies that 
$u_1$ and $u_2$ vanish on $(1,\infty)$. As is Proposition \ref{Prop:SpectralCharacterization}, to exclude a jump discontinuity at the light cone, i.e., to get $u_1 \in \dot H^1$ locally around the light cone, we must have $u_1(1) = 0$ and thus $u_1(\rho) = c (1 - \rho)^{\lambda+\frac{3}{2}}h(1-\rho)$ for $h$ an analytic function with $h(0) = 1$ and some constant $c \in C$. Note that for $\mathrm{Re} \lambda > -\frac12$, this implies that also $u'_1(1) = 0$.
Transforming to the corresponding one-dimensional Schrödinger operator $(A_V, \mc D(A_V))$ as in the proof of Proposition \ref{Prop:SpectralCharacterization}, we infer the existence of an eigenvalue $\mu = - (\lambda + \frac{3}{2})^2 < -\frac14$ of $A_V$, where the fact that $\mu$ must be real implies $\lambda \in \R$. The fact that $A_V$ is bounded below implies that $\lambda$ belongs to a finite set of eigenvalues.  

For $\Lf_1$, one obtains resolvent estimate as in Lemma \ref{Le:Resolventbounds_Prüss}: There are constants $C, K>0$ depending on $V$ such that
	\begin{equation}
		\label{eq:BoundedResolventL1}
		\| \Rf_{\Lf_1}(\lambda)\| \leqslant C,
	\end{equation}
for all $\lambda \in \C$ which satisfy $|\lambda| \geqslant K$ and $\Re \lambda \geqslant -\frac12$, where $\|\cdot \|$ here denotes the operator norm on $\dot{H}^1_{rad}(\R^6)\times L^2_{rad}(\R^6)$. In view of the structure of the spectrum of $\Lf_1$, we can define a finite rank spectral projection $\Pf_{-\frac14}$ onto $\sigma(\Lf_1) \cap \{z\in \C:\Re z\geq -\frac{1}{4}\}$ such that the estimate \eqref{Est:S1bound} follows from \eqref{eq:BoundedResolventL1} and an application of the Gearhart-Prüss theorem. 
\end{proof}

Now, given that any eigenvalue of $\Lf$ with $\Re \lambda>0$ is automatically an eigenvalue of $\Lf_1$ we obtain the following. 
\begin{lemma}
We have 
\begin{align}\label{Semigroup_Consistency}
\Sf(\tau)\ff=\Sf_1(\tau)\ff,
\qquad \forall \ff\in\mathcal E\cap\mathcal H^2,\quad \tau\geq0.
\end{align}
Moreover, 
\[ \Sigma:= \sigma (\Lf)\cap \{z\in \C:\Re z>0\}=\sigma(\Lf_1)\cap \{z\in \C:\Re z>0\}\]
and for every $\lambda  \in \{z\in \C:\Re z>0\}\setminus\Sigma$ and  we have $\Rf_{\Lf}(\lambda)  = \Rf_{\Lf_1}(\lambda)$ as well as  $\Pf =\Pf_1$ in $\mathcal E\cap\mathcal H^2$, where $\Pf$ and $\Pf_1$ denote the corresponding  Riesz projections associated
with $\Sigma$. 

Suppose additionally that $0 \notin\sigma_p(\Lf_1)$. Then for every $\ff\in \mathcal E\cap\mathcal H^2$,
the function
\[
\lambda \mapsto\Rf_{\Lf}(\lambda) \ff,
\qquad \lambda \in\{z\in \C:\Re z>0\}\setminus\Sigma,
\]
admits an analytic continuation across the imaginary axis,
with values in $H^1_{\mathrm{loc}}(\R^6)\times L^2_{\mathrm{loc}}(\R^6)$.
This continuation is given by $\Rf_{\Lf_1}(\lambda)\ff$.
\end{lemma}

\begin{proof}
We consider $\mathcal E\cap\mathcal H^2$ equipped with the sum of the two norms.
For sufficiently large real $\lambda >0$, both resolvents exist and agree
on $\mathcal E\cap\mathcal H^2$. Indeed, the free evolutions agree on the intersection and thus also their resolvents. Since the perturbation $\Lf'_V$ is bounded on both spaces the assertion for the full resolvent follows from a  Neumann series argument. Thus, the Hille--Yosida approximations agree on $\mathcal E\cap\mathcal H^2$.
Taking their limits and comparing them in
$H^1_{\mathrm{loc}}(\R^6)\times L^2_{\mathrm{loc}}(\R^6)$, we obtain \ref{Semigroup_Consistency}. In particular, this identity holds for
$\ff\in C^\infty_{c,\mathrm{rad}}(\R^6) \times C^\infty_{c,\mathrm{rad}}(\R^6)$.
Let
$\Lf_1\ff=\lambda\ff$, where $\ff=(f_1,f_2)\neq0$ and
$\Re\lambda>0$. The analysis near singular points shows that 
$f_1\in\dot H^2_{\mathrm{rad}}(\R^6)$, $f_2=(\lambda-1)f_1-\rho f_1'
\in\dot H^1_{\mathrm{rad}}(\R^6)$. 
Hence $\ff\in \mathcal E\cap\mathcal H^2$, and the identity \eqref{Semigroup_Consistency} yields $\Sf(\tau)\ff=\Sf_1(\tau)\ff=e^{\lambda\tau}\ff.$ Consequently,
\[
\frac{\Sf(\tau)\ff-\ff}{\tau}
=
\frac{e^{\lambda\tau}-1}{\tau}\ff
\longrightarrow\lambda\ff
\quad\text{in }\mathcal H^2
\quad\text{as }\tau\downarrow0.
\]
By the definition of the generator, $\ff\in\mathcal D(\Lf)$
and $\Lf\ff=\lambda\ff$. For the reverse inclusion, we note that every eigenvector of $\Lf$ corresponding to an eigenvalue with
positive real part is supported in $\overline{\B}^6$. Therefore it 
belongs to $\mathcal E$ and by an analogous argument we have 
\[
\sigma(\Lf)\cap\{z\in\C:\Re z>0\}
\subset
\sigma(\Lf_1)\cap\{z\in\C:\Re z>0\}.
\]
This proves the asserted equality of the unstable spectra.

Set  $\Omega:=\{z\in \C:\Re z>0\}\setminus\Sigma$.
For every $\ff\in  \mathcal E\cap\mathcal H^2$ and $R>0$, $\Rf_{\Lf}(\lambda)\ff$ and $\Rf_{\Lf_1}(\lambda)\ff$ are analytic on $\Omega$ with values in
$H^1(\B_R^6)\times L^2(\B_R^6)$. They agree for all sufficiently
large positive real $\lambda$. Since $\Sigma$ is finite and   
$\Omega$ is connected the two resolvent orbit agree throughout $\Omega$ by the identity theorem. Since $R > 0$ is arbitrary we conclude by that 
\[
\Rf_{\Lf}(\lambda)\ff=\Rf_{\Lf_1}(\lambda)\ff
\qquad \lambda\in\Omega.
\]
Integrating this identity over a common positively oriented
contour implies the statement on the projections. If $\Sigma$ is empty, both projections are zero. Finally, by the previous result, $0\notin\sigma_p(\Lf_1)$ implies $i\R\subset\varrho(\Lf_1) := \C \setminus \sigma(\Lf_1) $. Thus, for every $\ff \in \mc E$, $ \Rf_{\Lf_1}(\lambda) \ff$ is analytic in a neighbourhood $U$ of the imaginary axis with values in $\mathcal E$, hence also in
$H^1_{\mathrm{loc}}(\R^6)\times L^2_{\mathrm{loc}}(\R^6)$.
By the resolvent identity proved above, it agrees with
$\Rf_{\Lf}(\lambda)\ff$ on $U\cap\Omega$ and consequently
provides the claimed analytic continuation.
\end{proof}


The representation \eqref{eq:IntegralRepresentation} expresses
the stable evolution through scalar resolvents on a line
$\Re\lambda=\varepsilon>0$. For the interior estimates, we
will use the corresponding representation on the imaginary
axis. The compatibility of the critical and energy-space
realizations, together with the energy-space resolvent
bounds, allows us to pass to this boundary. Since
$\varepsilon$ is chosen below the unstable spectrum, this
passage crosses no poles. 
Next, we relate the eigenvalues of $\Lf_1$ to the zeros of the
connection coefficient $c_{0,1}$ associated with the solutions
of Eq.~\eqref{eq:FullODE} that are smooth at the origin and
at $\rho=1$, respectively.
\begin{lemma}\label{lem:zeros}
Assume that $i\R\subset \varrho(\Lf_1):=
\mathbb{C}\setminus\sigma(\Lf_1)$.
Then, for $\Re\lambda\in[0,\frac14)$, we have
$c_{0,1}(\lambda)=0$ if and only if
$\lambda\in\sigma(\Lf_1)$.
\end{lemma}
\begin{proof}
From the preceding considerations, it follows that, for
$\Re\lambda \in [0,\frac14)$, $\lambda$ is an eigenvalue of
$\Lf_1$ if and only if there exists a nontrivial solution
to Eq.~\eqref{eq:FullODE} that is smooth at the origin
and vanishes at $\rho=1$.
By the definition of the connection coefficient $c_{0,1}$,
this is equivalent to $c_{0,1}(\lambda)=0$.
\end{proof}
Thus, we see that, as one would expect, our resolvent construction in Section 4 works out precisely if and only if $\lambda$ is not an eigenvalue of $\Lf_1$.
\begin{lemma}\label{lem:interior-semigroup-representation}
Assume that $i\R\subset\varrho(\Lf_1)$, and let $\Pf$ denote the
finite-rank Riesz projection associated with
\[
\sigma(\Lf)\cap\{z\in\C:\Re z>0\}.
\]
Then, for every
\[
\ff=(f_1,f_2)\in
C^\infty_{c,\mathrm{rad}}(\R^6)
\times C^\infty_{c,\mathrm{rad}}(\R^6),
\]
every $\tau>0$, and every $\rho\in(0,1)$, we have
\begin{align*}
[\Sf(\tau)(\If-\Pf)\ff]_1(\rho)
={}&[\Sf_0(\tau)\ff]_1(\rho)\\
&+\lim_{N\to\infty}\frac{1}{2\pi}
\int_{-N}^{N}e^{i\omega\tau}
\Bigl(
[\mathcal R(i\omega)-\mathcal R_{\mathrm f}(i\omega)]
F_{i\omega}
\Bigr)(\rho)\,d\omega,
\end{align*}
where
\[
F_\lambda(\rho)
=(2-\lambda)f_1(\rho)+\rho f_1'(\rho)-f_2(\rho).
\]
The scalar resolvents on the imaginary axis are understood
through the energy-space resolvents. The limit is locally
uniform for $(\tau,\rho)\in(0,\infty)\times(0,1)$.
\end{lemma}

\begin{proof}
Choose $ 0<\varepsilon_0<\min\{\varepsilon_V,1/300\}$ sufficiently small that
$$\{\lambda\in\C:0\leq\Re\lambda\leq\varepsilon_0\}
\subset\varrho(\Lf_1).$$
By the energy-space version of
Lemma \ref{Le:Resolventbounds_Prüss}, both
$\Rf_{\Lf_1}(\lambda)$ and $\Rf_{\Lf_{0,1}}(\lambda)$
are uniformly bounded on $\mathcal E$ for
$\Re\lambda\geq0$ and $|\lambda|$ sufficiently large. For smooth compactly supported radial $\ff$, we have $\ff\in\mathcal D(\Lf_1^2)\cap\mathcal D(\Lf_{0,1}^2)$.
Applying the identity
\[
\Rf_{\Lf_1}(\lambda)\ff
=\lambda^{-1}\ff
+\lambda^{-1}\Rf_{\Lf_1}(\lambda)\Lf_1\ff
\]
twice yields
\[
\Rf_{\Lf_1}(\lambda)\ff
=\lambda^{-1}\ff
+\lambda^{-2}\Lf_1\ff
+\lambda^{-2}\Rf_{\Lf_1}(\lambda)\Lf_1^2\ff.
\]
Subtracting the corresponding identity for $\Lf_{0,1}$
and using $\Lf_1-\Lf_{0,1}=\Lf'_V$, we obtain
\[
\bigl[\Rf_{\Lf_1}(\lambda)
-\Rf_{\Lf_{0,1}}(\lambda)\bigr]\ff
=
\lambda^{-2}
\Bigl[
\Lf'_V\ff
+\Rf_{\Lf_1}(\lambda)\Lf_1^2\ff
-\Rf_{\Lf_{0,1}}(\lambda)\Lf_{0,1}^2\ff
\Bigr].
\]
The uniform resolvent bounds therefore imply
\[
\left\|
\bigl[\Rf_{\Lf_1}(\lambda)
-\Rf_{\Lf_{0,1}}(\lambda)\bigr]\ff
\right\|_{\mathcal E}
\lesssim_{\ff}|\lambda|^{-2},
\]
uniformly for $\Re\lambda\geq0$ and $|\lambda|$
sufficiently large. Consequently, there is a constant $C_{\ff} > 0$ such that 
\begin{equation}\label{eq:uniform-boundary-resolvent-bound}
\sup_{0\leq\varepsilon\leq\varepsilon_0}
\left\|
\bigl[
\Rf_{\Lf_1}(\varepsilon+i\omega)
-\Rf_{\Lf_{0,1}}(\varepsilon+i\omega)
\bigr]\ff
\right\|_{\mathcal E}
\leq C_{\ff}(1+\omega^2)^{-1}.
\end{equation}

For $0<\varepsilon\leq\varepsilon_0$, we have 
\[
\begin{split}
&\left[
\bigl(
\Rf_{\Lf_1}(\varepsilon+i\omega)
-\Rf_{\Lf_{0,1}}(\varepsilon+i\omega)
\bigr)\ff
\right]_1(\rho)\\
&\qquad=
\Bigl(
[\mathcal R(\varepsilon+i\omega)
-\mathcal R_{\mathrm f}(\varepsilon+i\omega)]
F_{\varepsilon+i\omega}
\Bigr)(\rho).
\end{split}
\]
At $\varepsilon=0$, this identity specifies the boundary
values through the energy-space resolvents. By Sobolev embedding, \eqref{eq:uniform-boundary-resolvent-bound} provides an integrable bound for the scalar expressions locally on every compact subinterval of $(0,1)$, uniformly on $0\leq\varepsilon\leq\varepsilon_0$. This, together with the continuity of the integrand locally uniformly in $\tau$ and $\rho$ allows us to pass to the limit and to obtain the claimed representation. 
\end{proof}

  To establish estimates on $\Sf$, we recall the definition of the restriction of $\mathcal{R}$ to the interior $\mathcal{R}_{int}$:
  \begin{align*}
      \mathcal{R}_{int}(f)(\rho;\lambda):&=u_2(\rho;\lambda)\int_0^\rho \frac{u_0(s;\lambda)f(s) }{(1-s^2)W(u_0(\cdot;\lambda),u_2(\cdot;\lambda))} ds+
      u_0(\rho;\lambda)\int_\rho^1 \frac{u_2(s;\lambda)f(s) }{(1-s^2)W(u_0(\cdot;\lambda),u_2(\cdot;\lambda))} ds 
      \\
      &\quad+\mu(f) u_0(\rho;\lambda)
\end{align*}
where 
\begin{align*}
  \mu(f)=  c_{0,1}(\lambda)^{-1}\int_1^\infty \frac{\tilde{u}_2(s;\lambda)f(s)}{W(s;\lambda)(1-s^2)}ds
\end{align*}
and
\begin{align*}
    {W(u_0(\cdot;\lambda),u_2(\cdot;\lambda))}(\rho)=\widehat{c}(\lambda)W(u_0(\cdot;\lambda),\widetilde u_0(\cdot;\lambda))(\rho)= -c_{0,1}(\lambda) \frac{(1-\rho^2)^{\frac{1}{2}+\lambda}}{\rho^{5}}.
\end{align*}
Next, we manipulate this expression as follows. We perform an integration by parts in each of the two integrals to arrive at 
\begin{align*}
      \mathcal{R}_{int}(f)(\rho;\lambda):    &=u_2(\rho;\lambda)U_0(\rho;\lambda)f(\rho)-u_2(\rho;\lambda)\int_0^\rho U_0(s;\lambda) f'(s) ds
      \\
      &\quad -u_0(\rho;\lambda)U_2(\rho;\lambda)f(\rho)-u_0(\rho;\lambda)\int_\rho^1 U_2(s;\lambda) f'(s) ds
      \\
      &\quad+u_0(\rho;\lambda)[\mu(f)+U_2(1)f(1)]
\end{align*}
where $$U_j(\rho;\lambda):=\int_0^\rho \frac{ u_j(s;\lambda)}{(1-s^2)W(u_0(\cdot;\lambda),u_2(\cdot;\lambda))}ds= \int_0^\rho \frac{\pi u_j(s;\lambda)s^{5}}{2(1-s^2)^{\frac{3}{2}+\lambda}\widehat{c}(\lambda)} ds.$$

In the next subsection, we will infer that the summands in the integral of the difference of the resolvents are of $\langle \omega \rangle^{-2}$, which allows for a liberal interchange of the limit and the integral. However, in a later section, when we prove Strichartz estimates for the derivatives of the semigroup, this will not be straightforward: taking a derivative improves the decay in $\rho$ but worsens it in $\lambda$. Nevertheless, we can argue by using integration by parts that one requires $k$ integration by parts to interchange $k$ derivatives with the integral via the dominated convergence theorem. For example, see \cite[Lemma 5.1]{Wallauch2024}.

\subsection{A first set of Strichartz estimates on $\Sf$}
Our method of deriving estimates on $\Sf-\Sf_0$, hence, on  $\Sf$, will be to split it into smaller pieces and derive estimates on each of the terms. We start this procedure with the following terms:
\begin{align*}
    D_1(f)(\rho;\lambda):=[u_2(\rho;\lambda)U_0(\rho;\lambda)-u_0(\rho;\lambda)U_2(\rho;\lambda)-u_{\mathrm{f}_2}(\rho;\lambda)U_{\mathrm{f}_0}(\rho;\lambda)+u_{\mathrm{f}_0}(\rho;\lambda)U_{\mathrm{f}_2}(\rho;\lambda)]f(\rho).
\end{align*}
\begin{lemma}\label{lem:decompD1}
    We can decompose $D_1(f)$ as
    \begin{align*}
        D_1(f)(\rho;\lambda)=f(\rho)\sum_{j=1}^5 G_j(\rho;\lambda)
    \end{align*}
    where
    \begin{align*}
        G_1(\rho;\lambda)&=\chi_\lambda(\rho)(1-\rho^2)^{\frac{\lambda}{2}}\int_0^\rho (1-s^2)^{-\frac{\lambda}{2}}\O(\rho^0s\langle\omega\rangle^{-1})ds,
        \\
        G_2(\rho;\lambda)&=(1-\chi_\lambda(\rho))\rho^{-\frac{5}{2}}(1-\rho)^{\frac{3}{2}+\lambda}\int_0^\rho \chi_\lambda(s)(1-s^2)^{-\frac{\lambda}{2}}\O(\rho^0s\langle\omega\rangle^{-\frac{7}{2}})ds
        \\
        G_3(\rho;\lambda)&=(1-\chi_\lambda(\rho))\rho^{-\frac{5}{2}}(1+\rho)^{\frac{3}{2}+\lambda}\int_0^\rho \chi_\lambda(s)(1-s^2)^{-\frac{\lambda}{2}}\O(\rho^0s\langle\omega\rangle^{-\frac{7}{2}})ds
        \\
        G_4(\rho;\lambda)&= (1-\chi_\lambda(\rho))\rho^{-\frac{5}{2}}(1-\rho)^{\frac{3}{2}+\lambda}\int_0^\rho (1-\chi_\lambda(s))s^{\frac{5}{2}}(1-s)^{-\frac{3}{2}-\lambda}\O(\rho^0s^0\langle\omega\rangle^{-2})ds
        \\
        G_5(\rho;\lambda)&=(1-\chi_\lambda(\rho))\rho^{-\frac{5}{2}}(1+\rho)^{\frac{3}{2}+\lambda}\int_0^\rho (1-\chi_\lambda(s))s^{\frac{5}{2}}(1+s)^{-\frac{3}{2}-\lambda}\O(\rho^0s^0\langle\omega\rangle^{-2})ds.
    \end{align*}
\end{lemma}
\begin{proof}
    Note that on the support of $\chi_\lambda$ we have the leading order behavior
    \begin{align*}
u_2(\rho;\lambda)&=(1-\rho^2)^{\frac{\lambda}{2}}\O(\rho^0\langle\omega\rangle^2)+(1-\rho^2)^{\frac{\lambda}{2}}\O(\rho^{-4}\langle\omega\rangle^{-2})=(1-\rho^2)^{\frac{\lambda}{2}}\O(\rho^{-4}\langle\omega\rangle^{-2})
        \\
      u_0(\rho;\lambda)&=(1-\rho^2)^{\frac{\lambda}{2}}\O(\rho^0\langle\omega\rangle^2),
    \end{align*}
    where the first identity follows from the fact that $\chi_\lambda(\rho)\O(\rho^0\langle\omega\rangle^2)=\chi_\lambda(\rho)\O(\rho^{-4}\langle\omega\rangle^{-2})$.
    Furthermore, as $u_j$ and $u_{\mathrm{f}_j}$ agree at leading order, one readily computes that on the support of $\chi_\lambda$ we have the identities
        \begin{align*}
u_2(\rho;\lambda)-u_{\mathrm{f}_2}(\rho;\lambda)&=\O(\rho^{-4}\langle\omega\rangle^{-3})
        \\
u_0(\rho;\lambda)-u_{\mathrm{f}_0}(\rho;\lambda)&=\O(\rho^0\langle\omega\rangle^1).
    \end{align*}
    Similarly, on the support of $(1-\chi_\lambda(\rho))$ we know that
        \begin{align*}
 u_2(\rho;\lambda)&=\rho^{-\frac{5}{2}}\frac{(1-\rho)^{\frac{3}{2}+\lambda}}{\sqrt{3+2\lambda}}[1+e_2(\rho;\lambda)][1+r_2(\rho;\lambda)]
 \\
u_2(\rho;\lambda)-u_{\mathrm{f}_2}(\rho;\lambda)&=\rho^{-\frac{5}{2}}\frac{(1-\rho)^{\frac{3}{2}+\lambda}}{\sqrt{3+2\lambda}}[1+e_2(\rho;\lambda)]r_2(\rho;\lambda).
        \\
         u_0(\rho;\lambda)&=c_{0,2}(\lambda)\rho^{-\frac{5}{2}}\frac{(1-\rho)^{\frac{3}{2}+\lambda}}{\sqrt{3+2\lambda}}[1+e_2(\rho;\lambda)][1+r_2(\rho;\lambda)]
         \\
         &\quad +c_{0,1}(\lambda)\rho^{-\frac{5}{2}}\frac{(1+\rho)^{\frac{3}{2}+\lambda}}{\sqrt{3+2\lambda}}[1+e_1(\rho;\lambda)][1+r_1(\rho;\lambda)]
         \\
u_0(\rho;\lambda)-u_{\mathrm{f}_0}(\rho;\lambda)
&=\rho^{-\frac52}
  \frac{(1-\rho)^{\frac32+\lambda}}{\sqrt{3+2\lambda}}
  [1+e_2(\rho;\lambda)]
  \Big[c_{0,2}(\lambda)[1+r_2(\rho;\lambda)]
       -c_{\mathrm{f}_{0,2}}(\lambda)\Big]
\\
&\quad+\rho^{-\frac52}
  \frac{(1+\rho)^{\frac32+\lambda}}{\sqrt{3+2\lambda}}
  [1+e_1(\rho;\lambda)]
  \Big[c_{0,1}(\lambda)[1+r_1(\rho;\lambda)]
       -c_{\mathrm{f}_{0,1}}(\lambda)\Big]
\\
&=\rho^{-\frac52}
  \frac{(1-\rho)^{\frac32+\lambda}}{\sqrt{3+2\lambda}}
  [1+e_2(\rho;\lambda)]
  \Big[c_{0,2}(\lambda)r_2(\rho;\lambda)
       +c_{0,2}(\lambda)-c_{\mathrm{f}_{0,2}}(\lambda)\Big]
\\
&\quad+\rho^{-\frac52}
  \frac{(1+\rho)^{\frac32+\lambda}}{\sqrt{3+2\lambda}}
  [1+e_1(\rho;\lambda)]
  \Big[c_{0,1}(\lambda)r_1(\rho;\lambda)
       +c_{0,1}(\lambda)-c_{\mathrm{f}_{0,1}}(\lambda)\Big].
    \end{align*}
    By using these identities and the fact that 
    $$
    (1-\chi_\lambda(\rho))e_j(\rho;\lambda)= (1-\chi_\lambda(\rho))\O(\rho^0\langle\omega\rangle^0),
    $$
    one readily computes that
    \begin{align*}
   &\quad     u_2(\rho;\lambda)U_0(\rho;\lambda)-u_{\mathrm{f}_2}(\rho;\lambda)U_{\mathrm{f}_0}(\rho;\lambda)
   \\
   &=\chi_\lambda(\rho)(1-\rho^2)^{\frac{\lambda}{2}}\int_0^\rho (1-s^2)^{-\frac{\lambda}{2}}\O(\rho^0s^1\langle\omega\rangle^{-1})
        \\
  &\quad + (1-\chi_\lambda(\rho))\rho^{-\frac{5}{2}}(1-\rho)^{\frac{3}{2}+\lambda}\int_0^\rho \chi_\lambda(s)(1-s^2)^{-\frac{\lambda}{2}}\O(\rho^0 s^5\langle\omega\rangle^{\frac{1}{2}})
  \\
  &\quad +   (1-\chi_\lambda(\rho))\rho^{-\frac{5}{2}}(1-\rho)^{\frac{3}{2}+\lambda}\int_0^\rho (1-\chi_\lambda(s))s^{\frac{5}{2}}(1-s)^{-\frac{3}{2}-\lambda}\O(\rho^0s^0\langle\omega\rangle^{-2})
   \\
  &\quad + (1-\chi_\lambda(\rho))\rho^{-\frac{5}{2}}(1-\rho)^{\frac{3}{2}+\lambda}\int_0^\rho (1-\chi_\lambda(s))s^{\frac{5}{2}}(1+s)^{-\frac{3}{2}-\lambda}\O(\rho^0s^0\langle\omega\rangle^{-2}).
    \end{align*}
   By decomposing $ u_0(\rho;\lambda)U_2(\rho;\lambda)-u_{\mathrm{f}_0}(\rho;\lambda)U_{\mathrm{f}_2}(\rho;\lambda)$
   in the same way and grouping terms together, one arrives at the desired result. 
\end{proof}
Motivated by this decomposition, we define operators $T_{j}$ and $\dot{T}_{j}$ as
\begin{align*}
    T_{j}(\tau)f(\rho):=\lim_{N\to \infty}\int_{-N}^N e^{i \omega \tau } G_j(\rho;i\omega) f(\rho)d\omega
\end{align*}
and, motivated by the summand $\lambda f_1$ in the definition of $F_\lambda$, as
\begin{align*}
    \dot T_{j}(\tau)f(\rho):=\lim_{N\to \infty}\int_{-N}^N e^{i \omega \tau }\omega  G_j(\rho;i\omega) f(\rho)d\omega
\end{align*}
for $\tau \in \R_+, f\in C^\infty_{c,rad}(\R^6)$ and $j=1,\dots,5$.
\begin{lemma}
The estimates 
\begin{align*}
    \|T_{j}(\cdot)f\|_{L^q(\R_+)L^r(\B^6)}\lesssim \|f\|_{\dot{H}^1(\R^6)}
    \\
     \|\dot T_{j}(\cdot)f\|_{L^q(\R_+)L^r(\B^6)}\lesssim \|f\|_{\dot{H}^2(\R^6)}
\end{align*}
hold for $j=1,\dots,5$, all $q\in [2,\infty],r\in[6,12]$ satisfy \eqref{eq:AdmissibleExponents} and all $f\in C^\infty_{c,rad}(\R^6)$.

\end{lemma}
\begin{proof}
    By arguing as in the proof of Lemma 5.14 in \cite{DonningerWallauch} one concludes the estimates
    \begin{align*}
    \|T_{j}(\cdot)f\|_{L^q(\R_+)L^r(\B^6)}\lesssim \|f\|_{H^1(\B^6)}
    \\
     \|\dot T_{j}(\cdot)f\|_{L^q(\R_+)L^r(\B^6)}\lesssim \|f\|_{H^2(\B^6)}
\end{align*}
which immediately imply the desired bounds. 
\end{proof}


Next, we look at 
\begin{align*}
    D_2(f)(\rho;\lambda):&=-u_2(\rho;\lambda)\int_0^\rho U_0(s;\lambda) f'(s) ds
-u_0(\rho;\lambda)\int_\rho^1 U_2(s;\lambda) f'(s) ds
\\
&\quad+u_{\mathrm{f}_2}(\rho;\lambda)\int_0^\rho U_{\mathrm{f}_0}(s;\lambda) f'(s) ds+u_{\mathrm{f}_0}(\rho;\lambda)\int_\rho^1 U_{\mathrm{f}_2}(s;\lambda) f'(s) ds.
\end{align*}

\begin{lemma}
    We can decompose $D_2(f)$ as
    \begin{align*}
D_2(f)(\rho;\lambda)=\sum_{j=6}^{15} G_j(f)(\rho;\lambda)
    \end{align*}
    where
    \begin{align*}
        G_6(f)(\rho;\lambda)&:=\chi_\lambda(\rho) (1-\rho^2)^{\frac{\lambda}{2}}\int_0^\rho \int_0^s (1-t^2)^{-\frac{\lambda}{2}}\O(\rho^{-4}t^5\langle\omega\rangle^{-1}) dt f'(s) ds
        \\
         G_7(f)(\rho;\lambda)&:=(1-\chi_\lambda(\rho)) \rho^{-\frac{5}{2}}(1-\rho)^{\frac{3}{2}+\lambda}\int_0^\rho \int_0^s \chi_\lambda(t) (1-t^2)^{-\frac{\lambda}{2}}\O(\rho^0t^{\frac{5}{2}}\langle\omega\rangle^{-2}) dt f'(s) ds
         \\
         G_8(f)(\rho;\lambda)&:=(1-\chi_\lambda(\rho)) \rho^{-\frac{5}{2}}(1-\rho)^{\frac{3}{2}+\lambda}\int_0^\rho \int_0^s (1-\chi_\lambda(t))\frac{\O(\rho^0t^{\frac{5}{2}}\langle\omega\rangle^{-2})}{(1-t)^{\frac{3}{2}+\lambda }}dt f'(s) ds
         \\
          G_{9}(f)(\rho;\lambda)&:=(1-\chi_\lambda(\rho)) \rho^{-\frac{5}{2}}(1-\rho)^{\frac{3}{2}+\lambda}\int_0^\rho \int_0^s (1-\chi_\lambda(t))\frac{\O(\rho^0t^{\frac{5}{2}}\langle\omega\rangle^{-2})}{(1+t)^{\frac{3}{2}+\lambda}}dt f'(s) ds
          \\
          G_{10}(f)(\rho;\lambda)&:=\chi_\lambda(\rho) (1-\rho^2)^{\frac{\lambda}{2}}\int_\rho^1 \int_0^s \chi_\lambda(t)  (1-t^2)^{-\frac{\lambda}{2}}\O(\rho^0t\langle\omega\rangle^{-1}) dt f'(s) ds
            \\
          G_{11}(f)(\rho;\lambda)&:=\chi_\lambda(\rho) (1-\rho^2)^{\frac{\lambda}{2}}\int_\rho^1 \int_0^s (1-\chi_\lambda(t))\frac{\O(\rho^0t\langle\omega\rangle^{-1}) }{ (1+t)^{\frac{3}{2}+\lambda}}dt f'(s) ds
              \\
          G_{12}(f)(\rho;\lambda)&:=(1-\chi_\lambda(\rho)) \rho^{-\frac{5}{2}}(1-\rho)^{\frac{3}{2}+\lambda}\int_\rho^1 \int_0^s \chi_\lambda(t)(1-t^2)^{-\frac{\lambda}2}\O(\rho^0 t \langle\omega\rangle^{-\frac{7}{2}})dt f'(s) ds
             \\
          G_{13}(f)(\rho;\lambda)&:=(1-\chi_\lambda(\rho)) \rho^{-\frac{5}{2}}(1+\rho)^{\frac{3}{2}+\lambda}\int_\rho^1 \int_0^s \chi_\lambda(t)(1-t^2)^{-\frac{\lambda}2}\O(\rho^0 t \langle\omega\rangle^{-\frac{7}{2}})dt f'(s) ds
                        \\
          G_{14}(f)(\rho;\lambda)&:=(1-\chi_\lambda(\rho)) \rho^{-\frac{5}{2}}(1-\rho)^{\frac{3}{2}+\lambda}\int_\rho^1 \int_0^s (1-\chi_\lambda(t))\frac{\O(\rho^0t^{\frac{5}{2}}\langle\omega\rangle^{-2}) }{ (1+t)^{\frac{3}{2}+\lambda}}dt f'(s) ds
             \\
          G_{15}(f)(\rho;\lambda)&:=(1-\chi_\lambda(\rho)) \rho^{-\frac{5}{2}}(1+\rho)^{\frac{3}{2}+\lambda}\int_\rho^1 \int_0^s (1-\chi_\lambda(t))\frac{\O(\rho^0t^{\frac{5}{2}}\langle\omega\rangle^{-2}) }{ (1+t)^{\frac{3}{2}+\lambda}}dt f'(s) ds.
    \end{align*}
\end{lemma}
\begin{proof}
    This follows from the same considerations as Lemma \ref{lem:decompD1}.
\end{proof}
As above, we define associated operators
$T_{j}$ and $\dot{T}_{j}$ as
\begin{align*}
    T_{j}(\tau)f(\rho):=\lim_{N\to \infty}\int_{-N}^N e^{i \omega \tau } G_j(f)(\rho;i\omega) d\omega
\end{align*}
and, motivated by the summand $\lambda f_1$ in the definition of $F_\lambda$, as
\begin{align*}
    \dot T_{j}(\tau) f(\rho):=\lim_{N\to \infty}\int_{-N}^N e^{i \omega \tau }\omega  G_j(f)(\rho;i\omega) d\omega
\end{align*}
for $f\in C^\infty_{c,rad}(\R^6).$
\begin{lemma}
The estimates 
\begin{align*}
    \|T_{j}(\cdot)f\|_{L^q(\R_+)L^r(\B^6)}\lesssim \|f\|_{\dot{H}^1(\R^6)}
    \\
     \|\dot T_{j}(\cdot)f\|_{L^q(\R_+)L^r(\B^6)}\lesssim \|f\|_{\dot{H}^2(\R^6)}
\end{align*}
hold for $j=6,\dots,15,$ all $q\in [2,\infty], ~r\in[6,12]$ satisfy \eqref{eq:AdmissibleExponents} and all $f\in C^\infty_{c,rad}(\R^6)$.
\end{lemma}
\begin{proof}
Once again, this follows from arguments exhibited in \cite{DonningerWallauch}. This time, one readily employs the exact arguments employed in the proof of Lemma 5.17 in said work.
\end{proof}
With this, we come to $D_3(f)(\rho;\lambda)$ which we define as
\begin{align*}
 D_3(f)(\rho;\lambda):=  u_0(\rho;\lambda)[ \mu(f)(\lambda)+U_2(1;\lambda)f(1)]-u_{\mathrm{f}_0}(\rho;\lambda)[\mu_{\mathrm{f}}(f)(\lambda)+U_{\mathrm{f}_2}(1;\lambda)f(1)].
   \end{align*}

We will also need to manipulate these terms a bit in order to prove estimates. To that end, we compute that
\begin{align*}
    \mu(f)(\lambda)&=c_{0,1}(\lambda)^{-1}\int_1^\infty \kappa(s)\frac{s^{\frac{5}{2}}g_2(s;\lambda)}{\sqrt{3+2\lambda}(1+s)^{\frac{3}{2}+\lambda}}f(s) ds
    \\
    &\quad +c_{0,1}(\lambda)^{-1}\int_1^\infty \frac{(1-\kappa(s))}{\sqrt{3+2\lambda}}\left[\frac{s^{\frac{5}{2}}a_{2,3}(\lambda)g_3(s;\lambda)}{(s-1)^{\frac{3}{2}+\lambda}}+\frac{a_{2,4}(\lambda)s^{\frac{5}{2}}g_4(s;\lambda)}{(1+s)^{\frac{3}{2}+\lambda}}\right]f(s) ds
    \\
    &= c_{0,1}(\lambda)^{-1}\int_1^\infty \frac{2\partial_s[\kappa(s)s^{\frac{5}{2}}g_2(s;\lambda)f(s)]}{(1+2\lambda)\sqrt{3+2\lambda}(1+s)^{\frac{1}{2}+\lambda}} ds+\frac{c_{0,1}(\lambda)^{-1} f(1)}{(1+2\lambda)\sqrt{3+2\lambda}2^{\frac{1}{2}+\lambda}}
    \\
    &\quad +\frac{2c_{0,1}(\lambda)^{-1}}{\sqrt{3+2\lambda}(1+2\lambda)}\int_1^\infty \frac{a_{2,3}(\lambda)\partial_s[ (1-\kappa(s)) s^{\frac{5}{2}}g_3(s;\lambda)f(s)]}{(s-1)^{\frac{1}{2}+\lambda}} ds
    \\
    &\quad +\frac{2c_{0,1}(\lambda)^{-1}}{\sqrt{3+2\lambda}(1+2\lambda)}\int_1^\infty\frac{a_{2,4}(\lambda)\partial_s[ (1-\kappa(s)) s^{\frac{5}{2}}g_4(s;\lambda)f(s)]}{(1+s)^{\frac{1}{2}+\lambda}} ds.
\end{align*}
Furthermore, we compute that
\begin{align*}
    W(u_0(\cdot;\lambda)),u_2(\cdot;\lambda))(\rho)= c_{0,1}(\lambda) W(u_1(\cdot;\lambda)),u_2(\cdot;\lambda))(\rho)=c_{0,1}(\lambda)\frac{(1-\rho^2)^{\frac{1}{2}+\lambda}}{\rho^{5}},
\end{align*}
which implies that
 \begin{align*}
    U_2(1;\lambda)&=\int_0^1 \frac{\chi_\lambda(s)\O(s\langle\omega\rangle^{-1})}{(1-s^2)^{\frac{\lambda}{2}}}ds+ c_{0,1}(\lambda)^{-1}\int_0^1\frac{(1-\chi_\lambda(s))s^{\frac{5}{2}}[1+e_2(s;\lambda)][1+r_2(s;\lambda)]}{\sqrt{3+2\lambda}(1+s)^{\frac{3}{2}+\lambda}}ds
    \\
    &= \int_0^1 \frac{\chi_\lambda(s)\O(s\langle\omega\rangle^{-1})}{(1-s^2)^{\frac{\lambda}{2}}}ds+ c_{0,1}(\lambda)^{-1}\int_0^1\frac{2\partial_s[(1-\chi_\lambda(s))s^{\frac{5}{2}}[1+e_2(s;\lambda)][1+r_2(s;\lambda)]]}{(1+2\lambda)\sqrt{3+2\lambda}(1+s)^{\frac{1}{2}+\lambda}}ds
    \\
    &\quad -\frac{2c_{0,1}(\lambda)^{-1}}{(1+2\lambda)\sqrt{3+2\lambda}2^{\frac{1}{2}+\lambda}}.
\end{align*}
Therefore, we conclude that
\begin{align*}
    &\quad \mu(f)(\lambda)+U_2(1;\lambda)f(1)
    \\
    &=c_{0,1}(\lambda)^{-1}\int_1^\infty \frac{2\partial_s[\kappa(s)s^{\frac{5}{2}}g_2(s;\lambda)f(s)]}{(1+2\lambda)\sqrt{3+2\lambda}(1+s)^{\frac{1}{2}+\lambda}} ds
    \\
     &\quad +\frac{2c_{0,1}(\lambda)^{-1}}{\sqrt{3+2\lambda}(1+2\lambda)2^{\frac{1}{2}+\lambda}}\int_1^\infty \frac{a_{2,3}(\lambda)\partial_s[ (1-\kappa(s)) s^{\frac{5}{2}}g_3(s;\lambda)f(s)]}{(s-1)^{\frac{1}{2}+\lambda}} ds
    \\
    &\quad +\frac{2c_{0,1}(\lambda)^{-1}}{\sqrt{3+2\lambda}(1+2\lambda)2^{\frac{1}{2}+\lambda}}\int_1^\infty\frac{a_{2,4}(\lambda)\partial_s[ (1-\kappa(s)) s^{\frac{5}{2}}g_4(s;\lambda)f(s)]}{(1+s)^{\frac{1}{2}+\lambda}} ds
    \\
    &\quad +f(1)\left[\int_0^1 \frac{\chi_\lambda(s)\O(s\langle\omega\rangle^{-1})}{(1-s^2)^{\frac{\lambda}{2}}}ds+ c_{0,1}(\lambda)^{-1}\int_0^1\frac{2\partial_s[(1-\chi_\lambda(s))s^{\frac{5}{2}}[1+e_2(s;\lambda)][1+r_2(s;\lambda)]]}{(1+2\lambda)\sqrt{3+2\lambda}(1+s)^{\frac{1}{2}+\lambda}}ds\right].
\end{align*}

\begin{lemma}
       We can decompose $ D_3(f)$ as 
       \begin{align*}
           D_3(f)(\rho;\lambda)=\sum_{j=16}^{27} G_j(f)(\rho;\lambda)
       \end{align*}
       with
       \begin{align*}
           G_{16}(f)(\rho;\lambda)&=\chi_\lambda(\rho)(1-\rho^2)^{\frac{\lambda}{2}}\O(\rho^0\langle\omega\rangle^{-\frac{1}{2}})\int_1^\infty \frac{\partial_s[\kappa(s)\O(s^{\frac{5}{2}}\langle\omega\rangle^0)f(s)]}{(1+s)^{\frac{1}{2}+\lambda}}ds 
           \\
              G_{17}(f)(\rho;\lambda)&=(1-\chi_\lambda(\rho))(1-\rho)^{\frac{3}{2}+\lambda}\O(\rho^{-\frac{5}{2}}\langle\omega\rangle^{-3})\int_1^\infty \frac{\partial_s[\kappa(s)\O(s^{\frac{5}{2}}\langle\omega\rangle^0)f(s)]}{(1+s)^{\frac{1}{2}+\lambda}}ds 
              \\
              G_{18}(f)(\rho;\lambda)&=(1-\chi_\lambda(\rho))(1+\rho)^{\frac{3}{2}+\lambda}\O(\rho^{-\frac{5}{2}}\langle\omega\rangle^{-3})\int_1^\infty \frac{\partial_s[\kappa(s)\O(s^{\frac{5}{2}}\langle\omega\rangle^0)f(s)]}{(1+s)^{\frac{1}{2}+\lambda}}ds 
              \\
               G_{19}(f)(\rho;\lambda)&=\chi_\lambda(\rho)(1-\rho^2)^{\frac{\lambda}{2}}\O(\rho^0\langle\omega\rangle)f(1)\int_0^1 \frac{\chi_\lambda(s)\O(s\langle\omega\rangle^{-1})}{(1-s^2)^{\frac{\lambda}{2}}}ds
           \\
              G_{20}(f)(\rho;\lambda)&=(1-\chi_\lambda(\rho))(1-\rho)^{\frac{3}{2}+\lambda}\O(\rho^{-\frac{5}{2}}\langle\omega\rangle^{-\frac{7}{2}})f(1)\int_0^1 \frac{\chi_\lambda(s)\O(s\langle\omega\rangle^{-1})}{(1-s^2)^{\frac{\lambda}{2}}}ds
              \\
              G_{21}(f)(\rho;\lambda)&=(1-\chi_\lambda(\rho))(1+\rho)^{\frac{3}{2}+\lambda}\O(\rho^{-\frac{5}{2}}\langle\omega\rangle^{-\frac{7}{2}})f(1)\int_0^1 \frac{\chi_\lambda(s)\O(s\langle\omega\rangle^{-1})}{(1-s^2)^{\frac{\lambda}{2}}}ds
               \\
               G_{22}(f)(\rho;\lambda)&=\chi_\lambda(\rho)(1-\rho^2)^{\frac{\lambda}{2}}\O(\rho^0\langle\omega\rangle^{-\frac12})f(1)\int_0^1\frac{\partial_s[(1-\chi_\lambda(s))\O(s^{\frac{5}{2}}\langle\omega\rangle^0)]}{(1+s)^{\frac{1}{2}+\lambda}}ds
           \\
              G_{23}(f)(\rho;\lambda)&=(1-\chi_\lambda(\rho))(1-\rho)^{\frac{3}{2}+\lambda}\O(\rho^{-\frac{5}{2}}\langle\omega\rangle^{-3})f(1)\int_0^1\frac{\partial_s[(1-\chi_\lambda(s))\O(s^{\frac{5}{2}}\langle\omega\rangle^0)]}{(1+s)^{\frac{1}{2}+\lambda}}ds
              \\
              G_{24}(f)(\rho;\lambda)&=(1-\chi_\lambda(\rho))(1+\rho)^{\frac{3}{2}+\lambda}\O(\rho^{-\frac{5}{2}}\langle\omega\rangle^{-3})f(1)\int_0^1\frac{\partial_s[(1-\chi_\lambda(s))\O(s^{\frac{5}{2}}\langle\omega\rangle^0)]}{(1+s)^{\frac{1}{2}+\lambda}}ds
              \\
              G_{25}(f)(\rho;\lambda)&=\chi_\lambda(\rho)(1-\rho^2)^{\frac{\lambda}{2}}\O(\rho^0\langle\omega\rangle^{-\frac{1}{2}})
              \\
              &\quad \times \int_1^\infty \left[\frac{\partial_s[ (1-\kappa(s)) \O(s^{\frac{5}{2}}\langle\omega\rangle^0)f(s)]}{(s-1)^{\frac{1}{2}+\lambda}}-\frac{\partial_s[ (1-\kappa(s)) \O(s^{\frac{5}{2}}\langle\omega\rangle^0)f(s)]}{(1+s)^{\frac{1}{2}+\lambda}}\right] ds
              \\
              G_{26}(f)(\rho;\lambda)&=(1-\chi_\lambda(\rho))(1-\rho)^{\frac{3}{2}+\lambda}\O(\rho^{-\frac{5}{2}}\langle\omega\rangle^{-3})\\
              &\quad \times \int_1^\infty \left[\frac{\partial_s[ (1-\kappa(s)) \O(s^{\frac{5}{2}}\langle\omega\rangle^0)f(s)]}{(s-1)^{\frac{1}{2}+\lambda}}-\frac{\partial_s[ (1-\kappa(s)) \O(s^{\frac{5}{2}}\langle\omega\rangle^0)f(s)]}{(1+s)^{\frac{1}{2}+\lambda}}\right] ds
              \\
              G_{27}(f)(\rho;\lambda)&=(1-\chi_\lambda(\rho))(1+\rho)^{\frac{3}{2}+\lambda}\O(\rho^{-\frac{5}{2}}\langle\omega\rangle^{-3})\\
              &\quad \times \int_1^\infty \left[\frac{\partial_s[ (1-\kappa(s)) \O(s^{\frac{5}{2}}\langle\omega\rangle^0)f(s)]}{(s-1)^{\frac{1}{2}+\lambda}}-\frac{\partial_s[ (1-\kappa(s)) \O(s^{\frac{5}{2}}\langle\omega\rangle^0)f(s)]}{(1+s)^{\frac{1}{2}+\lambda}}\right] ds.
       \end{align*}
   \end{lemma}
   Note, that some of the above kernels are influenced by the behavior of $f$ at $\infty$. Thus, we will need to argue differently for them. Consequently, we define operators $T_{j},\dot T_{j}$
   as
   \begin{align*}
     T_{j}(\tau)f(\rho)&:=\lim_{N\to \infty} \int_{-N}^N e^{i\omega \tau}G_j(f)(\rho;i\omega) d\omega 
     \\
   \dot T_{j}(\tau)f(\rho)&:=\lim_{N\to \infty} \int_{-N}^N \omega e^{i\omega \tau}G_j( f)(\rho;i\omega) d\omega 
   \end{align*}
   for $j=16,\dots,24$, $\tau \in \R_+$ and $f\in C^\infty_{c,rad}(\R^6).$
    \begin{lemma}\label{lem: strichartz D_3 k_2}
       The estimates
       \begin{align*}
       \| T_{j}(\cdot)f\|_{L^q(\R_+)L^r(\B^6)}&\lesssim \|f\|_{\dot{H}^1(\R^6)}
       \\
        \| T_{j}(\cdot)f\|_{L^q(\R_+)L^r(\B^6)}&\lesssim \|f\|_{\dot{H}^2(\R^6)}
        \\
         \|\dot T_{j}(\cdot)f\|_{L^q(\R_+)L^r(\B^6)}&\lesssim \|f\|_{\dot{H}^2(\R^6)}
       \end{align*}
       hold for $j=16,\dots,24,$ all $q\in [2,\infty],~r\in[6,12]$ satisfy \eqref{eq:AdmissibleExponents} and all $f\in C^\infty_{c,rad}(\R^6)$.
       
    \end{lemma}

   \begin{proof}
       We start with the operators $T_{16}$. Since we can trade decay powers of $\rho$ for decay in $\omega$ on the support of the cutoff $\chi_\lambda$, we conclude that 
       \begin{align*}
           T_{16}(\tau)f(\rho)&= \int_{-\infty}^\infty e^{i\omega \tau}G_{16}( f)(\rho; i\omega) d\omega 
       \end{align*}
       for all $\rho\in (0,1).$ Then, from the identity $\chi_\lambda(\rho)\O(\rho^0\langle\omega\rangle^{-\frac{1}{2}})=\chi_\lambda(\rho)\O(\rho^{-\frac{1}{2}-\frac{1}{24}}\langle\omega\rangle^{-1-\frac{1}{24}})$ as well as the fact that the support of $\chi_\lambda$ is contained in $ [0,\rho_0) $ for all $\lambda \in S$ and some $\rho_0<1$, 
       an application of Lemma 5.1 in \cite{DonningerWallauch} yields the estimate
       \begin{align*}
           |T_{16}(\tau)f(\rho)|&\lesssim 1_{[0,\rho_0)}(\rho)\rho^{-\frac{1}{2}-\frac{1}{24}}\int_1^\infty\frac{|\partial_s[\kappa(s)s^{\frac{5}{2}}f(s)]|}{(1+s)^{\frac{1}{2}}}\langle\tau+\frac{1}2\log(1-\rho^2)-\log(s+1)\rangle^{-2} ds
           \\
           &\lesssim 1_{[0,\rho_0)}(\rho)\rho^{-\frac{1}{2}-\frac{1}{24}}\int_1^\infty\frac{|\partial_s[\kappa(s)s^{\frac{5}{2}}f(s)]|}{(1+s)^{\frac{1}{2}}}\langle\tau \rangle^{-2} ds
       \end{align*}
       where the inequality
       $$
       \langle\tau+\frac{1}2\log(1-\rho^2)-\log(s+1)\rangle^{-2} \lesssim \langle\tau\rangle^{-2} 
       $$
       follows as $\frac{1}2\log(1-\rho^2)-\log(s+1)$ is uniformly bounded on the support of $ 1_{[0,\rho_0)}(\rho)\kappa(s)$.
       Thus, we conclude that 
\begin{align*}
   \|T_{16}(\cdot)f\|_{L^\infty(\R_+)L^6(\B^6)}\lesssim \int_1^\infty\frac{|\partial_s[\kappa(s)s^{\frac{5}{2}} f(s)]|}{(1+s)^{\frac{1}{2}}} ds
   \end{align*}
   which readily implies the two bounds: 
   \begin{align*}
    \| T_{16}(\cdot)f\|_{L^\infty(\R_+)L^6(\B^6)}\lesssim\|f\|_{\dot{H}^1(\R^6)}
    \\
    \| T_{16}(\cdot)f\|_{L^\infty(\R_+)L^6(\B^6)}\lesssim\|f\|_{\dot{H}^2(\R^6)}
   \end{align*}
   due to the compact support of $\kappa$. 
   For the other endpoint estimate, we use Lemma 5.4 in \cite{DonningerWallauch} to obtain 
   \begin{align*}
           |T_{16}(\tau)f(\rho)|&\lesssim 1_{[0,\rho_0)}(\rho)\rho^{-\frac{1}{2}+\frac{1}{24}}\int_1^\infty\frac{|\partial_s[\kappa(s)s^{\frac{5}{2}} f(s)]|}{(1+s)^{\frac{1}{2}}}g(\tau+\frac{1}2\log(1-\rho^2)-\log(s+1))ds
   \end{align*}
      with $g(x):= \langle x\rangle^{-2} |x|^{-\frac{1}{24}}$. Thus, from Minkowski's inequality, we conclude that
      \begin{align*}
            \|T_{16}(\tau)f\|_{L^{12}(\B^6)}&\lesssim \int_1^\infty\frac{|\partial_s[\kappa(s)s^{\frac{5}{2}}]|}{(1+s)^{\frac{1}{2}}}\|\rho^{-\frac{1}{2}+\frac{1}{24}}g(\tau+\frac{1}2\log(1-\rho^2)-\log(s+1))\|_{L^{12}_\rho(\B_{\rho_0})}.
      \end{align*}
      Now, we have that
      \begin{align*}
        &\quad \|\rho^{-\frac{1}{2}+\frac{1}{24}}g(\tau+\frac{1}2\log(1-\rho^2)-\log(s+1))\|_{L^{12}_\rho(\B_{\rho_0})}
        \\&\lesssim \langle \tau-\log(s+1)\rangle^{-2}  \|\rho^{-\frac{1}{2}+\frac{1}{24}}|\tau+\frac{1}2\log(1-\rho^2)-\log(s+1)|^{-\frac{1}{24}}\|_{L^{12}_\rho(\B_{\rho_0})}.
      \end{align*}
      Furthermore, we compute that
      \begin{align*}
          \|\rho^{-\frac{1}{2}+\frac{1}{24}}|\tau+\frac{1}2\log(1-\rho^2)-\log(s+1)|^{-\frac{1}{24}}\|_{L^{12}_\rho(\B_{\rho_0})}^{12}&= \int_0^{\rho_0} \rho^{-\frac{1}{2}}|\tau+\frac{1}2\log(1-\rho^2)-\log(s+1)|^{-\frac{1}{2}} d\rho
          \\
         & \lesssim |\tau-\log(s+1)|^{-\frac{1}{2}} 
      \end{align*}
      where the inequality is thanks to Lemma 5.7 in \cite{DonningerWallauch}.
   Hence, 
\begin{align*}
    \|T_{16}(\tau)f\|_{L^{12}(\B^6)}&\lesssim \int_1^\infty\frac{|\partial_s[\kappa(s)s^{\frac{5}{2}} f(s)]|}{(1+s)^{\frac{1}{2}}}g(\tau-\log(s+1)) ds.
\end{align*}
Next, we set $\widehat{f}(s)=|\partial_s[\kappa(s)s^{\frac{5}{2}} f(s)]|$, change variables to $s=e^y-1$, and use the fact that $\supp(\kappa)\subset (1,r')$ for some $r'\in (1,\infty)$ to obtain 
   \begin{align*}
        \|T_{16}(\cdot)f\|_{L^2(\R_+)L^{12}(\B^6)}&\lesssim \left\|\int_{\log 2}^\infty 
       \widehat{f}(e^y-1)e^{\frac{y}{2}}g(\tau-y) dy \right\|_{L^2_\tau(\R_+)}
       \\
       &= \left\|\left(1_{(\log 2,\infty )}
       \widehat{f}(e^{(\cdot)}-1)e^{\frac{1}{2}(\cdot) }\right) *g \right\|_{L^2_\tau(\R_+)}
       \\
       &\lesssim \left\|1_{(\log 2,\infty )}
       \widehat{f}(e^{(\cdot)}-1)e^{\frac{1}{2}(\cdot) } \right\|_{L^2(\R)}\|g\|_{L^1}
       \\
       &\lesssim 
       \left\|
       \widehat{f}(s)  \right\|_{L^2_s((1,\infty))}
       \\
       &\lesssim \|s^{\frac{5}{2}}f(s)\|_{H^1((1,r'))}\lesssim \|f\|_{\dot{H}^1(\R^6)}
   \end{align*}
   To prove the desired estimates on $\dot T_{16}$, we integrate by parts once more to obtain
   \begin{align*}
 G_{16}(f)(\rho;\lambda)&=\chi_\lambda(\rho)(1-\rho^2)^{\frac{\lambda}{2}}\O(\rho^0\langle\omega\rangle^{-\frac{1}{2}})  \int_1^\infty \frac{\partial_s[\kappa(s)\O(s^{\frac{5}{2}}\langle\omega\rangle^0)f(s)]}{(1+s)^{\frac{1}{2}+\lambda}}ds
 \\
 &= \chi_\lambda(\rho)(1-\rho^2)^{\frac{\lambda}{2}}\O(\rho^0\langle\omega\rangle^{-\frac{1}{2}})\int_1^\infty \frac{\partial_s^2[\kappa(s)\O(s^{\frac{5}{2}}\langle\omega\rangle^{-1})f(s)]}{(1+s)^{-\frac{1}{2}+\lambda}}ds
  \\
  &\quad+\chi_\lambda(\rho)(1-\rho^2)^{\frac{\lambda}{2}}\O(\rho^0\langle\omega\rangle^{-\frac{1}{2}})\left[\frac{\partial_s[\kappa(s)\O(s^{\frac{5}{2}}\langle\omega\rangle^{-1})f(s)]}{(1+s)^{-\frac{1}{2}+\lambda}}\right]\Bigg|_{s=1}^{s=\infty}
  \\
&=:  \dot G_{16,1}(f)(\rho;\lambda)+  \dot G_{16,2}(f)(\rho;\lambda)
   \end{align*}
   This resulting extra decay in $\lambda$ allows us to treat the term 
   \begin{align*}
\lim_{N\to \infty} \int_{-N}^N e^{i\omega \tau}\omega \dot G_{16,1}(f)(\rho;i\omega) d\omega
   \end{align*}
   in the same fashion as $T_{16}$. Furthermore, these arguments also yield the estimate
\begin{align*}
\left\|\lim_{N\to \infty} \int_{-N}^N e^{i\omega \tau}\omega \dot G_{16,2}(f)(\rho;i\omega) d\omega\right\|_{L^q(\R_+)L^r(\B^6)}\lesssim |f(1)|+|f'(1)|\lesssim \|f\|_{\dot{H}^2(\R^6)}
\end{align*}
   where the last inequality follows as
   \begin{align*}
   |f(1)|\lesssim \left|\int_0^1 \partial_\rho \left[\rho^{\frac{7}{2}} f(\rho)\right] d\rho\right|\lesssim \|f\|_{H^1(\B^6)}&\lesssim \|f\|_{\dot{H}^2(\R^6)}.
   \end{align*}
   Thus, the estimate for general pairs $(q,r)$ follows from elementary interpolation and we move on to $T_{17}(\tau)f(\rho)$.
Applying Lemma 5.3 of \cite{DonningerWallauch} with $n=2$ yields  
\begin{align*}
|T_{17}(\tau)f(\rho)|&\leqslant \rho^{-\frac{1}{2}} (1-\rho) \int_1^\infty\frac{|\partial_s[\kappa(s)s^{\frac{5}{2}}f(s)]|}{(1+s)^{\frac{1}{2}}}\langle\tau+\log(1-\rho)-\log(s+1)\rangle^{-2} ds.
\end{align*}
Hence, one readily concludes that
   \begin{align*}
    \| T_{17}f\|_{L^\infty(\R_+)L^6(\B^6)}\lesssim\|f\|_{\dot{H}^1(\R^6)}
    \\
    \| T_{17}f\|_{L^\infty(\R_+)L^6(\B^6)}\lesssim\|f\|_{\dot{H}^2(\R^6)}.
   \end{align*}
   For the other endpoint, we apply Lemma  5.4 of \cite{DonningerWallauch} to infer that 
   \begin{align*}
|T_{17}(\tau)f(\rho)|&\leqslant\rho^{-\frac{1}{2}+\frac{1}{24}} (1-\rho) \int_1^\infty\frac{|\partial_s[\kappa(s)s^{\frac{5}{2}} f(s)]|}{(1+s)^{\frac{1}{2}}}g(\tau+\log(1-\rho)-\log(s+1)) ds
\end{align*}
with $g(x)=\langle x\rangle^{-2}| x|^{-\frac{1}{24}}$ as above. Furthermore, as $s$ is compactly supported, we have the elementary estimate 
\begin{align*}
    (1-\rho)\langle\tau+\log(1-\rho)-\log(s+1)\rangle^{-2}\lesssim \langle\tau\rangle^{-2}
\end{align*}
which allows us to obtain the bound 
\begin{align*}
      \|T_{17}(\cdot)f\|_{L^2(\R_+)L^{12}(\B^6)}&\lesssim \|f\|_{\dot H^1(\R^d)}
\end{align*}
in the same fashion as we did for $T_{16}$. In order to bound $ \dot T_{17}$, one performs an integration by parts as before and argues in the same fashion. One readily modifies these arguments to estimate the remaining operators and we conclude this proof.
      \end{proof}
      To prove estimates on the remaining kernels, we recall that we are not actually interested in estimating $\mathcal{R}_{int}(f)$ but actually $\mathcal{R}_{int}(F_\lambda)$ with 
      $$
      F_\lambda(\rho):=(2-\lambda)f_1(\rho) + \rho f_1'(\rho)-f_2(\rho)=-\Big(\frac{1}{2} +\lambda \Big) f_1(\rho) +\rho f_1'(\rho)+\frac{5}{2}f_1(\rho)-f_2(\rho).
      $$
Motivated by this, we first look at the leading order terms in the integrals which are given by 
\begin{align*}
   \int_{1}^{\infty} \frac{(1-\kappa)s^{\frac{5}{2}} ((2-\lambda)f_1(s) +sf_1'(s))}{(s\pm 1)^{\frac{3}{2}+\lambda}} ds&= \int_{1}^{\infty} \frac{(1-\kappa)s^{\frac{5}{2}} [-(\frac{1}{2}+\lambda)f_1(s) +\frac{5}{2}f_1(s)+s f_1'(s))]}{(s\pm 1)^{\frac{3}{2}+\lambda}} ds
\end{align*}
Now, an integration by parts yields
\begin{align*}
    -\int_{1}^{\infty} \frac{(1-\kappa(s))s^{\frac{5}{2}} (\frac{1}{2}+\lambda)}{(s\pm 1)^{\frac{3}{2}+\lambda}}f_1(s) ds
    &=-\int_{1}^{\infty} \frac{(s\pm1)\partial_s[(1-\kappa(s))s^{\frac{5}{2}}f_1(s)]}{(s\pm 1)^{\frac{3}{2}+\lambda}} ds
\end{align*}
which implies that
\begin{align*}
    		& \int_{1}^{\infty} \frac{(1-\kappa)s^{\frac{5}{2}} ((2-\lambda)f_1(s) +sf_1'(s))}{(s\pm 1)^{\frac{3}{2}+\lambda}} ds\\
		&\quad= -\int_{1}^{\infty} \frac{(s\pm1)\kappa'(s)s^{\frac{5}{2}}f_1(s)}{(s\pm 1)^{\frac{3}{2}+\lambda}} ds-\frac{5}{2} \int_{1}^{\infty} \frac{(1-\kappa(s))s^{\frac{3}{2}}f_1(s)}{(s\pm1)^{\frac{3}{2}+\lambda}} ds -\int_{1}^{\infty} \frac{(1-\kappa(s))s^{\frac{5}{2}}f_1'(s)}{(s\pm 1)^{\frac{3}{2}+\lambda}}ds
        \\
        &=-\frac{1}{1+2\lambda}\int_1^\infty\frac{\partial_s \left[2(s\pm1)\kappa'(s)s^{\frac{5}{2}}f_1(s)\mp5(1-\kappa(s))s^{\frac{3}{2}}f_1(s)\mp2(1-\kappa(s))s^{\frac{5}{2}}f_1'(s)\right]}{(s\pm 1)^{\frac{1}{2}+\lambda}}ds.
\end{align*}	
\begin{lemma}
We can rewrite
\[
\sum_{j=25}^{27}G_j(F_\lambda)(\rho;\lambda)
=
\sum_{j=1}^{3}
\left[
\widetilde G_{j,+}(F_+)(\rho;\lambda)
+\widetilde G_{j,-}(F_-)(\rho;\lambda)
\right]
+R(f_1,f_2)(\rho;\lambda),
\]
where
\begin{align*}
F_\pm(s)
&=2(s\pm1)\kappa'(s)s^{\frac52}f_1(s)
\mp5(1-\kappa(s))s^{\frac32}f_1(s)\\
&\quad-2(1-\kappa(s))s^{\frac52}
\bigl(f_2(s)\pm f_1'(s)\bigr),
\end{align*}
and
\begin{align*}
\widetilde G_{1,\pm}(F)(\rho;\lambda)
&=\chi_\lambda(\rho)(1-\rho^2)^{\frac{\lambda}{2}}
\O(\rho^0\langle\omega\rangle^{-\frac12})
\int_1^\infty
\frac{\partial_s[\O(s^0\langle\omega\rangle^0)F(s)]}
{(s\pm1)^{\frac12+\lambda}}\,ds,
\\
\widetilde G_{2,\pm}(F)(\rho;\lambda)
&=(1-\chi_\lambda(\rho))(1-\rho)^{\frac32+\lambda}
\O(\rho^{-\frac52}\langle\omega\rangle^{-3})
\int_1^\infty
\frac{\partial_s[\O(s^0\langle\omega\rangle^0)F(s)]}
{(s\pm1)^{\frac12+\lambda}}\,ds,
\\
\widetilde G_{3,\pm}(F)(\rho;\lambda)
&=(1-\chi_\lambda(\rho))(1+\rho)^{\frac32+\lambda}
\O(\rho^{-\frac52}\langle\omega\rangle^{-3})
\int_1^\infty
\frac{\partial_s[\O(s^0\langle\omega\rangle^0)F(s)]}
{(s\pm1)^{\frac12+\lambda}}\,ds
\end{align*}
and where $R(f_1,f_2)(\rho;\lambda)$ is a collection of lower order terms that all possess the same kernel structures as one of the $\widetilde G_j.$
\end{lemma}
We also record the following lemma, which follows immediately from Hardy's inequality.
\begin{lemma}\label{lem: norm estimate delta}
Let
    $$F_\pm(s)= \left[2(s\pm1)\kappa'(s)s^{\frac{5}{2}}f_1(s)\mp5(1-\kappa(s))s^{\frac{3}{2}}f_1(s)-2(1-\kappa(s))s^{\frac{5}{2}}(f_2(s)\pm f_1'(s))\right].$$
    Then,
    \begin{align*}
        \|F_\pm'\|_{L^{2}((0,\infty))}&\lesssim \|f_1\|_{\dot{H}^{2}(\R^6)}+\|f_2\|_{\dot{H}^{1}(\R^6)}
    \end{align*}
    and
       \begin{align*}
        \||.|^{-1}F_\pm\|_{L^{2}((0,\infty))}&\lesssim \|f_1\|_{\dot{H}^{2}(\R^6)}+\|f_2\|_{\dot{H}^{1}(\R^6)}
    \end{align*}
    for all $f_1,f_2\in C_{c,rad}^\infty(\R^6).$
\end{lemma}

For $j=1,2,3$ we define another set of integral operators $\widetilde{T}_{j}\ff$ as
\begin{align*}
   \widetilde{T}_j(\tau)\ff(\rho):=\lim_{N\to \infty}\int_{-N}^Ne^{i\omega \tau }\left[\widetilde G_{j,+}(F_+)(\rho;i\omega )+\widetilde G_{j,-}(F_-)(\rho;i\omega )\right] d\omega.
\end{align*}
      \begin{lemma}\label{lem: strichartz1-k2 lemma}
          The estimates
      \begin{align*}
    \|   \widetilde{T}_{j}(\cdot)\ff\|_{L^q(\R_+)L^r(\B^6)}&\lesssim \|\ff\|_{\dot{H}^2(\R^6)\times \dot{H}^1(\R^6)}
       \end{align*}
       hold for $j=1,2,3$, all $q\in [2,\infty],~r\in[6,12]$ satisfy \eqref{eq:AdmissibleExponents} and all $\ff\in C^\infty_{c,rad}(\R^6)\times C^\infty_{c,rad}(\R^6)$.
            \end{lemma}
      \begin{proof}
         By arguing as above, we conclude that for $j=1,2,3$ and $r\in [6,12]$
\begin{align*}
   \| \widetilde T_{j}(\tau)\ff\|_{L^{r}(\B^6)}&\lesssim \int_1^\infty\left| \frac{\partial_s[\O(s^{0})F_-(s)]}{(s-1)^{\frac{1}{2}}} \right|g(\tau -\log(s-1))ds
   \\
   &\quad+\int_1^\infty\left| \frac{\partial_s[\O(s^{0})F_+(s)]}{(s+1)^{\frac{1}{2}}} \right|g(\tau -\log(s+1))ds
   \\
   &=:I_-+I_+
\end{align*}
with $g(x)=\langle x\rangle^{-2}| x|^{-\frac{1}{24}}$.
To estimate $I_-$, we change variables according to $s=1+e^y$, to obtain that
\begin{align*}
    I_-\lesssim\int_{-\infty}^\infty \left[|(1+e^{y})^{-1} F_-(1+e^{y})|+ |F_-'(1+e^y)|\right]e^{\frac{y}{2}}g( \tau -y) dy
\end{align*}
Now, for $r\in [2,\infty]$, an application of Young's inequality yields 
\begin{align*}
    \|I_-\|_{L^q((0,\infty))}&\lesssim \left[\|(1+e^{y})^{-1}F_-(1+e^y)e^{\frac{y}{2}}\|_{L^{2}(\R)}+\|F_-'(1+e^y)e^{\frac{y}{2}}\|_{L^{2}(\R)}\right]\|g\|_{L^{\frac{1}{1+\frac{1}{q}-\frac{1}{2}}}}.
    \end{align*}
    Thus, given that $q\in [2,\infty]$, which, in turn, implies that $$1 \leqslant \frac{1}{1+\frac{1}{q}-\frac{1}{2}}\leqslant 2,$$
    we have that 
    \begin{align*}
       \|g\|_{L^{\frac{1}{1+\frac{1}{q}-\frac{1}{2}}}} \lesssim 1
    \end{align*}
    and therefore
    \begin{align*}
      \|I_-\|_{L^p((0,\infty))}&\lesssim\left( \int_\mathbb{R} |(1+e^{y})^{-1}F_-(1+e^y)|^{2} e^{y} dy \right)^{\frac{1}{2}}+\left( \int_\mathbb{R} |F_-'(1+e^y)|^{2} e^{y} dy \right)^{\frac{1}{2}}
        \\
        &=\left( \int_1^\infty |s^{-1}F_-(s)|^{2} ds \right)^{\frac{1}{2}}+\left( \int_1^\infty |F_-'(s)|^{2} ds \right)^{\frac{1}{2}}
        \\
        &\lesssim||.|^{-1}F_-\|_{L^{2}((0,\infty))}+ \|F_-'\|_{L^{2}((0,\infty))}
        \\
        &\lesssim \|f_1\|_{\dot{H}^{2}(\R^6)}+\|f_2\|_{\dot{H}^{1}(\R^6)}
\end{align*}
thanks to Lemma \ref{lem: norm estimate delta}. The term $I_+$ is estimated analogously, which concludes the proof.
      \end{proof}
By combining all these kernel estimates with the free Strichartz estimates \ref{prop: free strichartz} we arrive at a first set of Strichartz estimates.
\begin{lemma}
\label{lem: Strichartz interior}
          The estimates
          \begin{align}
              \|[\Sf(\If-\Pf)\ff]_1\|_{L^q(\R_+)L^r(\B^6)}&\lesssim \|\ff\|_{\mathcal{H}^2}
          \end{align}
          hold for all $\ff\in C^\infty_{c,rad}(\R^6)\times C^\infty_{c,rad}(\R^6)$ and all $q\in [2,\infty],~r\in[6,12]$ that satisfy \eqref{eq:AdmissibleExponents}.
                \end{lemma}
                
\subsection{Derivative estimates on $\Sf$}
The goal of this part is to arrive at the estimates
\begin{align}\label{eq: first higher norm}
    \|\Sf (\If-\Pf)\ff\|_{L^\infty(\R_+) H^2(\B^6)\times  H^1(\B^6)}\lesssim \|\ff\|_{\mathcal{H}^2}.
\end{align}
and 
\begin{align}\label{eq: first Strichrtz dif }
    \|[\Sf (\If-\Pf) \ff]_1\|_{L^2(\R_+) W^{1,4}(\B^6)}\lesssim \|\ff\|_{\mathcal{H}^2}.
\end{align}
Since \eqref{eq: first Strichrtz dif } is technically quite less involved than \eqref{eq: first higher norm}, and follows from considerations already exhibited in previous proofs or \cite{DonningerWallauch}, we only prove \eqref{eq: first higher norm}. For this, we recall that $\mathcal{R}_{int}$ is given by
\begin{align*}
      \mathcal{R}_{int}(f)(\rho;\lambda):&=u_2(\rho;\lambda)U_0(\rho;\lambda)f(\rho)-u_2(\rho;\lambda)\int_0^\rho U_0(s;\lambda) f'(s) ds
      \\
      &\quad -u_0(\rho;\lambda)U_2(\rho;\lambda)f(\rho)-u_0(\rho;\lambda)\int_\rho^1 U_2(s;\lambda) f'(s) ds
      \\
      &\quad+u_0(\rho;\lambda)[\mu(f)+U_2(1)f(1)]
\end{align*}
where  $$U_j(\rho;\lambda):=\int_0^\rho \frac{u_j(s;\lambda)}{(1-s^2)W(u_0(\cdot;\lambda),u_2(\cdot;\lambda))(s)}ds= \int_0^\rho \frac{u_j(s;\lambda)s^{5}}{(1-s^2)^{\frac{3}{2}+\lambda}\widehat{c}(\lambda)} ds.$$
 A direct computation shows that 
\begin{align*}
    \partial_\rho \mathcal{R}_{int}(f)(\rho;\lambda)
     &=u_2'(\rho;\lambda)U_0(\rho;\lambda)f(\rho)-u_2'(\rho;\lambda)\int_0^\rho U_0(s;\lambda) f'(s) ds
      \\
      &\quad -u_0'(\rho;\lambda)U_2(\rho;\lambda)f(\rho)-u_0'(\rho;\lambda)\int_\rho^1 U_2(s;\lambda) f'(s) ds
      \\
      &\quad+u_0'(\rho;\lambda)[\mu(f)+U_2(1)f(1)],
\end{align*}
from which we conclude that \begin{align*}
    \partial_\rho^2 \mathcal{R}_{int}(f)(\rho;\lambda)
     &=\partial_\rho[u_2'(\rho;\lambda)U_0(\rho;\lambda)-u_0'(\rho;\lambda)U_2(\rho;\lambda)]f(\rho)-u_2''(\rho;\lambda)\int_0^\rho U_0(s;\lambda) f'(s) ds
      \\
      &\quad -u_0''(\rho;\lambda)\int_\rho^1 U_2(s;\lambda) f'(s) ds+u_0''(\rho;\lambda)[\mu(f)+U_2(1)f(1)].
\end{align*}

We manipulate this expression by decomposing it into the pieces supported near $\rho=0$ and near $\rho=1$:
\begin{align*}
 &\quad    u_2'(\rho;\lambda)U_0(\rho;\lambda)-u_0'(\rho;\lambda)U_2(\rho;\lambda)
 \\
 &=\chi_\lambda(\rho) \partial_\rho\left[ \rho^{-\frac{5}{2}}(1-\rho^2)^{\frac{1}{4}+\frac{\lambda}{2}}[c_{2,0}(\lambda)\psi_1(\rho;\lambda) +c_{2,\widetilde 0}(\lambda)\psi_2(\rho;\lambda)]\right]U_0(\rho;\lambda)
 \\
 &\quad -\chi_\lambda(\rho) \partial_\rho\left[ \rho^{-\frac{5}{2}}(1-\rho^2)^{\frac{1}{4}+\frac{\lambda}{2}}\psi_1(\rho;\lambda) \right]U_2(\rho;\lambda)
 \\
 &\quad +\frac{(1-\chi_\lambda(\rho))}{\sqrt{3+2\lambda}} \partial_\rho\left[ \rho^{-\frac{5}{2}}(1-\rho)^{\frac{3}{2}+\lambda}[1+e_2(\rho;\lambda)][1+r_2(\rho;\lambda)]\right]\int_0^\rho\chi_\lambda(s) \frac{s^5\psi_1(s;\lambda)}{\widehat{c}(\lambda)(1-s^2)^{\frac{5}{4}+\frac{\lambda}{2}}} ds 
 \\
 &\quad  -\frac{(1-\chi_\lambda(\rho))}{\sqrt{3+2\lambda}}c_{0,1}(\lambda) \partial_\rho\left[ \rho^{-\frac{5}{2}}(1+\rho)^{\frac{3}{2}+\lambda}[1+e_1(\rho;\lambda)][1+r_1(\rho;\lambda)]\right]
 \\
 &\quad \times \int_0^\rho\chi_\lambda(s) \frac{s^5[c_{2,0}(\lambda)\psi_1(s;\lambda) +c_{2,\widetilde 0}(\lambda)\psi_2(s;\lambda)]}{\widehat{c}(\lambda)(1-s^2)^{\frac{5}{4}+\frac{\lambda}{2}}} ds
 \\
&\quad -\frac{(1-\chi_\lambda(\rho))}{\sqrt{3+2\lambda}} c_{0,2}(\lambda)\partial_\rho\left[ \rho^{-\frac{5}{2}}(1-\rho)^{\frac{3}{2}+\lambda}[1+e_2(\rho;\lambda)][1+r_2(\rho;\lambda)]\right]
\\
&\quad \times \int_0^\rho\chi_\lambda(s) \frac{s^5[c_{2,0}(\lambda)\psi_1(s;\lambda) +c_{2,\widetilde 0}(\lambda)\psi_2(s;\lambda)]}{\widehat{c}(\lambda)(1-s^2)^{\frac{5}{4}+\frac{\lambda}{2}}} ds
\\
&\quad +\frac{(1-\chi_\lambda(\rho))}{3+2\lambda}c_{0,1}(\lambda) \partial_\rho\left[ \rho^{-\frac{5}{2}}(1-\rho)^{\frac{3}{2}+\lambda}[1+e_2(\rho;\lambda)][1+r_2(\rho;\lambda)]\right]
\\
&\quad \times \int_0^\rho(1-\chi_\lambda(s)) \frac{s^{\frac{5}{2}}[1+e_1(s;\lambda)][1+r_1(s;\lambda)]}{\widehat{c}(\lambda)(1-s)^{\frac{3}{2}+\lambda}} ds
\\
&\quad -\frac{(1-\chi_\lambda(\rho))}{3+2\lambda}c_{0,1}(\lambda) \partial_\rho\left[ \rho^{-\frac{5}{2}}(1+\rho)^{\frac{3}{2}+\lambda}[1+e_1(\rho;\lambda)][1+r_1(\rho;\lambda)]\right]
\\
&\quad \times \int_0^\rho(1-\chi_\lambda(s)) \frac{s^{\frac{5}{2}}[1+e_1(s;\lambda)][1+r_1(s;\lambda)]}{\widehat{c}(\lambda)(1+s)^{\frac{3}{2}+\lambda}} ds.
\end{align*}

We are now concerned with the last two terms, for the simple reason that hitting these terms with another derivative produces the singular factor $(1-\rho)^{-\frac{1}{2}+\lambda}$. Thus, we set
\[
g_j(\rho;\lambda)
:=[1+e_j(\rho;\lambda)][1+r_j(\rho;\lambda)],
\]
and
\begin{align*}
E(\rho;\lambda)
&:=\frac{c_{0,1}(\lambda)(1-\chi_\lambda(\rho))}
{(3+2\lambda)\widehat c(\lambda)}
\partial_\rho\left[
\rho^{-\frac52}(1-\rho)^{\frac32+\lambda}g_2(\rho;\lambda)
\right]
\int_0^\rho
\frac{h_\lambda(s)s^{\frac52}g_1(s;\lambda)}
{(1-s)^{\frac32+\lambda}}\,ds
\\
&\quad-
\frac{c_{0,1}(\lambda)(1-\chi_\lambda(\rho))}
{(3+2\lambda)\widehat c(\lambda)}
\partial_\rho\left[
\rho^{-\frac52}(1+\rho)^{\frac32+\lambda}g_1(\rho;\lambda)
\right]
\int_0^\rho
\frac{(1-\chi_\lambda(s))s^{\frac52}g_2(s;\lambda)}
{(1+s)^{\frac32+\lambda}}\,ds.
\end{align*}
Integrating by parts yields
\begin{align*}
E(\rho;\lambda)
&=
\frac{c_{0,1}(\lambda)(1-\chi_\lambda(\rho))^2}
{(3+2\lambda)\widehat c(\lambda)(\frac12+\lambda)}
\Bigg[
(1-\rho)
\partial_\rho\left[\rho^{-\frac52}g_2(\rho;\lambda)\right]
\rho^{\frac52}g_1(\rho;\lambda)
\\
&\hspace{5cm}
+(1+\rho)
\partial_\rho\left[\rho^{-\frac52}g_1(\rho;\lambda)\right]
\rho^{\frac52}g_2(\rho;\lambda)
\Bigg]
\\
&\quad-
\frac{c_{0,1}(\lambda)(1-\chi_\lambda(\rho))}
{(3+2\lambda)\widehat c(\lambda)(\frac12+\lambda)}
\partial_\rho\left[
\rho^{-\frac52}(1-\rho)^{\frac32+\lambda}g_2(\rho;\lambda)
\right]
\int_0^\rho
\frac{\partial_s[(1-\chi_\lambda(s))s^{\frac52}g_1(s;\lambda)]}
{(1-s)^{\frac12+\lambda}}\,ds
\\
&\quad-
\frac{c_{0,1}(\lambda)(1-\chi_\lambda(\rho))}
{(3+2\lambda)\widehat c(\lambda)(\frac12+\lambda)}
\partial_\rho\left[
\rho^{-\frac52}(1+\rho)^{\frac32+\lambda}g_1(\rho;\lambda)
\right]
\int_0^\rho
\frac{\partial_s[(1-\chi_\lambda(s))s^{\frac52}g_2(s;\lambda)]}
{(1+s)^{\frac12+\lambda}}\,ds.
\end{align*}
Here the boundary contributions obtained by differentiating
$(1-\rho)^{\frac32+\lambda}$ and
$(1+\rho)^{\frac32+\lambda}$ cancel.

Integrating by parts once more in both remaining integrals gives
\begin{align*}
E(\rho;\lambda)
&=
\frac{c_{0,1}(\lambda)(1-\chi_\lambda(\rho))^2}
{(3+2\lambda)\widehat c(\lambda)(\frac12+\lambda)}
\Bigg[
(1-\rho)
\partial_\rho\left[\rho^{-\frac52}g_2(\rho;\lambda)\right]
\rho^{\frac52}g_1(\rho;\lambda)
\\
&\hspace{5cm}
+(1+\rho)
\partial_\rho\left[\rho^{-\frac52}g_1(\rho;\lambda)\right]
\rho^{\frac52}g_2(\rho;\lambda)
\Bigg]
\\
&\quad-
\frac{c_{0,1}(\lambda)(1-\chi_\lambda(\rho))
      (1-\rho)^{\frac12-\lambda}}
{(3+2\lambda)\widehat c(\lambda)
 (\frac12+\lambda)(-\frac12+\lambda)}
\partial_\rho\left[
\rho^{-\frac52}(1-\rho)^{\frac32+\lambda}g_2(\rho;\lambda)
\right]
\partial_\rho\left[
(1-\chi_\lambda(\rho))\rho^{\frac52}g_1(\rho;\lambda)
\right]
\\
&\quad+
\frac{c_{0,1}(\lambda)(1-\chi_\lambda(\rho))
      (1+\rho)^{\frac12-\lambda}}
{(3+2\lambda)\widehat c(\lambda)
 (\frac12+\lambda)(-\frac12+\lambda)}
\partial_\rho\left[
\rho^{-\frac52}(1+\rho)^{\frac32+\lambda}g_1(\rho;\lambda)
\right]
\partial_\rho\left[
(1-\chi_\lambda(\rho))\rho^{\frac52}g_2(\rho;\lambda)
\right]
\\
&\quad+
\frac{c_{0,1}(\lambda)(1-\chi_\lambda(\rho))}
{(3+2\lambda)\widehat c(\lambda)
 (\frac12+\lambda)(-\frac12+\lambda)}
\partial_\rho\left[
\rho^{-\frac52}(1-\rho)^{\frac32+\lambda}g_2(\rho;\lambda)
\right]
\int_0^\rho
\frac{\partial_s^2[(1-\chi_\lambda(s))s^{\frac52}g_1(s;\lambda)]}
{(1-s)^{-\frac12+\lambda}}\,ds
\\
&\quad-
\frac{c_{0,1}(\lambda)(1-\chi_\lambda(\rho))}
{(3+2\lambda)\widehat c(\lambda)
 (\frac12+\lambda)(-\frac12+\lambda)}
\partial_\rho\left[
\rho^{-\frac52}(1+\rho)^{\frac32+\lambda}g_1(\rho;\lambda)
\right]
\int_0^\rho
\frac{\partial_s^2[(1-\chi_\lambda(s))s^{\frac52}g_2(s;\lambda)]}
{(1+s)^{-\frac12+\lambda}}\,ds.
\end{align*}
Using this, we arrive at the following:
\begin{lemma}
    We can decompose 
    \begin{align*}
  D''_1f(\rho;\lambda):=      \left[\partial_\rho[u_2'(\rho;\lambda)U_0(\rho;\lambda)-u_0'(\rho;\lambda)U_2(\rho;\lambda)]-\partial_\rho[u_{\mathrm{f}_2}'(\rho;\lambda)U_{\mathrm{f}_0}(\rho;\lambda)-u_{\mathrm{f}_0}'(\rho;\lambda)U_{\mathrm{f}_2}(\rho;\lambda)]\right]f(\rho)
    \end{align*}
    as
    \begin{align*}
         D''_1f(\rho;\lambda)=f(\rho)\sum_{j=1}^{6}G_j''(\rho;\lambda)
    \end{align*}
    where
\begin{align*}
  G_1''(\rho;\lambda)&:=  \chi_\lambda(\rho)(1-\rho^2)^{\frac{\lambda}{2}}\left[(1-\rho^2)^{-\frac{\lambda}{2}}\O(\rho^0\langle\omega\rangle^{-1})+\int_0^\rho(1-s^2)^{-\frac{\lambda}{2}}\O(\rho^{-2}s\langle\omega\rangle^{-1}) ds\right]\\
  G_2''(\rho;\lambda)&:=(1-\chi_\lambda(\rho))(1-\rho)^{-\frac{1}{2}+\lambda}\left[\int_0^\rho\chi_\lambda(s)\O(\rho^{-\frac{5}{2}}s\langle\omega\rangle^{-\frac{3}{2}})(1-s^2)^{-\frac{\lambda}{2}}ds+\chi_\lambda(\rho)\O(\rho^{-\frac{3}{2}}\langle\omega\rangle^{-\frac{5}{2}})(1-\rho^2)^{-\frac{\lambda}{2}} \right]
  \\
  G_3''(\rho;\lambda)&:=(1-\chi_\lambda(\rho))(1+\rho)^{-\frac{1}{2}+\lambda}\left[\int_0^\rho\chi_\lambda(s)\O(\rho^{-\frac{5}{2}}s\langle\omega\rangle^{-\frac{3}{2}})(1-s^2)^{-\frac{\lambda}{2}}ds+\chi_\lambda(\rho)\O(\rho^{-\frac{3}{2}}\langle\omega\rangle^{-\frac{5}{2}})(1-\rho^2)^{-\frac{\lambda}{2}}\right]
  \\
  G_4''(\rho;\lambda)&:=(1-\chi_\lambda(\rho))^2\O(\rho^{-1}\langle\omega\rangle^{-3})+(1-\chi_\lambda(\rho))\chi_\lambda'(\rho)\O(\rho^{-1}\langle\omega\rangle^{-3})+\chi_\lambda'(\rho)^2\O(\rho^{-1}\langle\omega\rangle^{-3})
  \\
  G_5''(\rho;\lambda)&:=(1-\chi_\lambda(\rho))(1-\rho)^{-\frac{1}{2}+\lambda}
  \\
  &\quad \times \int_0^\rho\frac{(1-\chi_\lambda(s))\O(\rho^{-\frac{5}{2}}s^{\frac{1}{2}}\langle\omega\rangle^{-2})+\chi_\lambda'(s)\O(\rho^{-\frac{5}{2}}s^{\frac{1}{2}}\langle\omega\rangle^{-2})\chi_\lambda''(s)\O(\rho^{-\frac{5}{2}}s^{\frac{1}{2}}\langle\omega\rangle^{-2})}{(1-s)^{-\frac{1}{2}+\lambda}} ds
   \\
  G_6''(\rho;\lambda)&:=(1-\chi_\lambda(\rho))(1+\rho)^{-\frac{1}{2}+\lambda}
  \\
  &\quad \times \int_0^\rho\frac{(1-\chi_\lambda(s))\O(\rho^{-\frac{5}{2}}s^{\frac{1}{2}}\langle\omega\rangle^{-2})+\chi_\lambda'(s)\O(\rho^{-\frac{5}{2}}s^{\frac{1}+{2}}\langle\omega\rangle^{-2})\chi_\lambda'(s)\O(\rho^{-\frac{5}{2}}s^{\frac{1}{2}}\langle\omega\rangle^{-2})}{(1+s)^{-\frac{1}{2}+\lambda}}ds.
\end{align*}
\end{lemma}
As above, we define corresponding integral operators as
\begin{align*}
    T_{j}''(\tau)f(\rho)&:=\lim_{N\to \infty}\int_{-N}^N e^{i \omega \tau } G_j''(\rho;i \omega)f(\rho) d\omega
    \\
  \dot T_{j}''(\tau)f(\rho)&:=\lim_{N\to \infty}\int_{-N}^N e^{i \omega \tau } \omega G_j''(\rho; i\omega)f(\rho) d\omega
\end{align*}
for $j=1,\dots, 6$ $\tau \in \R_+$, and $f\in C^\infty_c(\R^6)$
\begin{lemma}
    The estimates
\begin{align*}
        \|T_{j}''(\cdot)f\|_{L^\infty(\R_+) L^{2}(\B^6)}&\lesssim \|f\|_{\dot H^{1}(\R^6)}
        \\
        \|T_{j}''(\cdot)f\|_{L^\infty(\R_+) L^{2}(\B^6)}&\lesssim \|f\|_{\dot H^{2}(\R^6)}
        \\
        \|\dot T_{j}''(\cdot)f\|_{L^\infty(\R_+) L^{2}(\B^6)}&\lesssim \|f\|_{\dot H^{2}(\R^6)}
\end{align*}
hold for $j=1,\dots, 6$ and all $f\in C^\infty_{c,rad}(\R^6).$
\end{lemma}
\begin{proof}
    An application of Lemma 5.1 in \cite{DonningerWallauch} and some small computations yield the estimates 
    \begin{align*}
        |T_{1}''(\tau)f(\rho)|&\lesssim |f(\rho)|\rho^{-\frac{1}{4}}
        \\
        |\dot T_{1}''(\tau)f(\rho)|&\lesssim |f(\rho)|\rho^{-\frac{5}{4}}
    \end{align*}
    which implies that
 \begin{align*}
      \|T_{1}''(\cdot)f\|_{L^\infty (\R_+)L^{2}(\B^6)}&\lesssim \||.|^{-1} f\|_{ L^{2}(\B^6)}
      \\
       \|\dot T_{1}''(\cdot)f\|_{L^\infty(\R_+)L^{2}(\B^6)}&\lesssim \||.|^{-2} f\|_{L^{2}(\B^6)}
 \end{align*}
 and the desired estimates for $j=1$ follow. For $j=2$ one rewrites $$\chi_{\lambda}(s)(1-\chi_\lambda(\rho))\O(\rho^{-\frac{5}{2}}s\langle\omega\rangle^{-\frac{3}{2}})=(1-\chi_\lambda(\rho))\O(\rho^{-\frac{5}{2}}s^0\langle\omega\rangle^{-\frac{5}{2}})$$
 and applies Lemma 5.3 of \cite{DonningerWallauch} to obtain
 \begin{align*}
      |T_{2}''(\tau)f(\rho)|&\lesssim (1-\rho)^{-\frac{1}{2}}|f(\rho)|\langle\tau+\log(1-\rho)\rangle^{-2},
 \end{align*}
 hence, changing variables according to $\rho=1-e^{-x}$ and applying Young's inequality shows that
 \begin{align*}
     \|T_{2}''(\tau)f\|_{L^\infty L^{2}(\B^6)}&\lesssim \left\|\int_0^1\rho^{5} (1-\rho)^{-1}|f(\rho)|^{2}\langle\tau+\log(1-\rho)\rangle^{-2}d\rho\right\|_{L^\infty_\tau(\R_+)}^{\frac{1}{2}}
     \\
     &= \left\|\int_0^\infty (1-e^{-x})^{5} |f(1-e^{-x})|^{2}\langle\tau-x\rangle^{-2}dx\right\|_{L^\infty_\tau(\R_+)}^{\frac{1}{2}}
     \\
 & \lesssim   \| (1-e^{-x})^{5} |f(1-e^{-x})|^{2}\|_{L^\infty_y(0,\infty)}^{\frac12}
 \\
 &=\| \rho^{5}  |f(\rho)|^{2}\|_{L^\infty_\rho(0,1)}^{\frac12}
 \\
 &=\| \rho^{\frac52} f(\rho)\|_{L^\infty_\rho(0,1)}
 \\
 &\lesssim \||.|^{-1} f\|_{L^{2}(\B^6)}+ \|f\|_{H^{1}(\B^6)} \lesssim  \|f\|_{H^{1}(\B^6)}
 \end{align*}
 and the bounds on $T_{2}''$ follow.
Similar arguments can be applied to conclude that \begin{align*}
       \|\dot T_{2}''(\cdot)f\|_{L^\infty(\R_+) L^{2}(\B^6)}&\lesssim\|f\|_{H^{2}(\B^6)}\lesssim \|f\|_{\dot H^{2}(\R^6)}
 \end{align*}
and, likewise, one obtains the claimed estimates for $j=3,4$ and we move to the operator
 $T_{5}''$. Thanks to Lemma 5.3 in \cite{DonningerWallauch} one has 
 \begin{align*}
     |T_{5}''(\tau)f(\rho)|&\lesssim  |f(\rho)|\rho^{-1}(1-\rho)^{-\frac{1}{2}}\int_0^\rho  \frac{1}{(1-s)^{\frac{1}{2}}}\langle\tau+\log(1-\rho)-\log(1-s)\rangle^{-2} ds
     \\
     &= 1_{(0,\frac{1}{2})}(\rho) |f(\rho)|\rho^{-1}(1-\rho)^{-\frac{1}{2}}\int_0^\rho  \frac{1}{(1-s)^{\frac{1}{2}}}\langle\tau+\log(1-\rho)-\log(1-s)\rangle^{-2} ds
     \\
     &\quad+ 1_{(\frac{1}{2},1)}(\rho)|f(\rho)|\rho^{-1}(1-\rho)^{-\frac{1}{2}}\int_0^\rho  \frac{1}{(1-s)^{\frac{1}{2}}}\langle\tau+\log(1-\rho)-\log(1-s)\rangle^{-2} ds
     \\
     &\lesssim  1_{(0,\frac{1}{2})}(\rho) |f(\rho)|\rho^{-1}\int_0^\rho 1 ds
     \\
     &\quad+ 1_{(\frac{1}{2},1)}(\rho)|f(\rho)|(1-\rho)^{-\frac{1}{2}}\int_0^\rho  \frac{1}{(1-s)^{\frac{1}{2} }}\langle\tau+\log(1-\rho)-\log(1-s)\rangle^{-2} ds.
 \end{align*}
 Now clearly
 \begin{align*}
     \left\|1_{(0,\frac{1}{2})}(\rho) |f(\rho)|\rho^{-1}\int_0^\rho 1 ds\right\|_{L^{2}_\rho(\B^6)}&\lesssim \|f\|_{H^1(\B^6)}.
 \end{align*}
 Hence, we only need to focus on 
 \begin{align*}
     I:=1_{(\frac{1}{2},1)}(\rho)|f(\rho)|(1-\rho)^{-\frac{1}{2}}\int_0^\rho  \frac{1}{(1-s)^{\frac{1}{2}}}\langle\tau+\log(1-\rho)-\log(1-s)\rangle^{-2} ds.
 \end{align*}
By changing variables according to $\rho=1-e^{-x}$ and $s=1-e^y$, we obtain
\begin{align*}
\|I(\tau,\cdot)\|_{L^2(\B^6)}^2
&\lesssim
\int_{\log 2}^\infty
|f(1-e^{-x})|^2(1-e^{-x})^5
\left(
\int_{-x}^0
e^{\frac y2}\langle\tau-x-y\rangle^{-2}\,dy
\right)^2\,dx\\
&\le
\int_{\log 2}^\infty
|f(1-e^{-x})|^2(1-e^{-x})^5
\left(
\int_{-\infty}^0
e^{\frac{y}{2}}\langle\tau-x-y\rangle^{-2}\,dy
\right)^2\,dx\\
&\lesssim
\int_{\log 2}^\infty
|f(1-e^{-x})|^2(1-e^{-x})^5
\langle\tau-x\rangle^{-4}\,dx\\
&\lesssim
\|\rho^{5/2}f(\rho)\|_{L^\infty_\rho(0,1)}^2.
\end{align*}
    Thus, by arguing as in the case $j=2$ one obtains 
    \begin{align*}
         \|I\|_{L^\infty (\R_+)L^2(\B^6)}&\lesssim \|f\|_{H^1(\B^6)}. 
    \end{align*}
    As one derives the remaining bounds in the same fashion, combined with one more integration by parts to recover enough decay in $\omega$, we conclude this proof.
 \end{proof}
 We move on to 
  \begin{align*}
  D''_2(f)(\rho;\lambda)&:= - u_2''(\rho;\lambda)\int_0^\rho U_0(s;\lambda) f'(s) ds -u_0''(\rho;\lambda)\int_\rho^1 U_2(s;\lambda) f'(s) ds
  \\
&\quad + u_{\mathrm{f}_2}''(\rho;\lambda)\int_0^\rho U_{\mathrm{f}_0}(s;\lambda) f'(s) ds +u_{\mathrm{f}_0}''(\rho;\lambda)\int_\rho^1 U_{\mathrm{f}_2}(s;\lambda) f'(s) ds.
 \end{align*}

    \begin{lemma}
        We can decompose $D_2''(f)$
        as
        \begin{align*}
        D_2''(f)(\rho;\lambda)=\sum_{j=7}^{16} G_j''(f)(\rho;\lambda) 
        \end{align*}

          where
    \begin{align*}
        G_7''(f)(\rho;\lambda)&:=\chi_\lambda(\rho) (1-\rho^2)^{\frac{\lambda}{2}}\int_0^\rho \int_0^s (1-t^2)^{-\frac{\lambda}{2}}\O(\rho^{-6}t^5\langle\omega\rangle^{-1}) dt f'(s) ds
        \\
         G_8''(f)(\rho;\lambda)&:=(1-\chi_\lambda(\rho)) \rho^{-\frac{5}{2}}(1-\rho)^{-\frac{1}{2}+\lambda}\int_0^\rho \int_0^s \chi_\lambda(t) (1-t^2)^{-\frac{\lambda}{2}}\O(\rho^0t^{\frac{5}{2}}\langle\omega\rangle^{0}) dt f'(s) ds
         \\
         G_9''(f)(\rho;\lambda)&:=(1-\chi_\lambda(\rho)) \rho^{-\frac{5}{2}}(1-\rho)^{-\frac{1}{2}+\lambda}\int_0^\rho \int_0^s (1-\chi_\lambda(t))\frac{\O(\rho^0t^{\frac{5}{2}}\langle\omega\rangle^{0})}{(1-t)^{\frac{3}{2}+\lambda }}dt f'(s) ds
         \\
          G_{10}''(f)(\rho;\lambda)&:=(1-\chi_\lambda(\rho)) \rho^{-\frac{5}{2}}(1-\rho)^{-\frac{1}{2}+\lambda}\int_0^\rho \int_0^s (1-\chi_\lambda(t))\frac{\O(\rho^0t^{\frac{5}{2}}\langle\omega\rangle^{0})}{(1+t)^{\frac{3}{2}+\lambda}}dt f'(s) ds
          \\
          G_{11}''(f)(\rho;\lambda)&:=\chi_\lambda(\rho) (1-\rho^2)^{\frac{\lambda}{2}}\int_\rho^1 \int_0^s \chi_\lambda(t)  (1-t^2)^{-\frac{\lambda}{2}}\O(\rho^{-2}t\langle\omega\rangle^{-1}) dt f'(s) ds
            \\
          G_{12}''(f)(\rho;\lambda)&:=\chi_\lambda(\rho) (1-\rho^2)^{\frac{\lambda}{2}}\int_\rho^1 \int_0^s (1-\chi_\lambda(t))\frac{\O(\rho^{-2}t\langle\omega\rangle^{-1}) }{ (1+t)^{\frac{3}{2}+\lambda}}dt f'(s) ds
              \\
          G_{13}''(f)(\rho;\lambda)&:=(1-\chi_\lambda(\rho)) \rho^{-\frac{5}{2}}(1-\rho)^{-\frac{1}{2}+\lambda}\int_\rho^1 \int_0^s \chi_\lambda(t)(1-t^2)^{-\frac{\lambda}2}\O(\rho^0 t \langle\omega\rangle^{-\frac{3}{2}})dt f'(s) ds
             \\
          G_{14}''(f)(\rho;\lambda)&:=(1-\chi_\lambda(\rho)) \rho^{-\frac{5}{2}}(1+\rho)^{-\frac{1}{2}+\lambda}\int_\rho^1 \int_0^s \chi_\lambda(t)(1-t^2)^{-\frac{\lambda}2}\O(\rho^0 t \langle\omega\rangle^{-\frac{3}{2}})dt f'(s) ds
                        \\
          G_{15}''(f)(\rho;\lambda)&:=(1-\chi_\lambda(\rho)) \rho^{-\frac{5}{2}}(1-\rho)^{-\frac{1}{2}+\lambda}\int_\rho^1 \int_0^s (1-\chi_\lambda(t))\frac{\O(\rho^0t^{\frac{5}{2}}\langle\omega\rangle^{0}) }{ (1+t)^{\frac{3}{2}+\lambda}}dt f'(s) ds
             \\
          G_{16}''(f)(\rho;\lambda)&:=(1-\chi_\lambda(\rho)) \rho^{-\frac{5}{2}}(1+\rho)^{-\frac{1}{2}+\lambda}\int_\rho^1 \int_0^s (1-\chi_\lambda(t))\frac{\O(\rho^0t^{\frac{5}{2}}\langle\omega\rangle^{0}) }{ (1+t)^{\frac{3}{2}+\lambda}}dt f'(s) ds.
    \end{align*}
\end{lemma}
We will need a second decomposition of $D_2''$. In particular, we need to manipulate the parts of 
$$
- u_2''(\rho;\lambda)\int_0^\rho U_0(s;\lambda) f'(s) ds -u_0''(\rho;\lambda)\int_\rho^1 U_2(s;\lambda) f'(s) ds$$ that are supported near $\rho=1$ and where both derivative hit an oscillatory factor that causes loss of decay in $\omega.$ In other words, we study
\begin{align*}
&\quad    (1-\chi_\lambda(\rho))(\frac{1}{2}+\lambda)\rho^{-\frac{5}{2}}(1-\rho)^{-\frac{1}{2}+\lambda}g_2(\rho;\lambda)
   c_{0,1}(\lambda) \int_0^\rho \int_0^s (1-\chi_\lambda(t))\frac{t^{\frac{5}{2}}g_1(t;\lambda)}{(1-t)^{\frac{3}{2}+\lambda }}dt f'(s) ds
   \\
   &\quad +(1-\chi_\lambda(\rho))(\frac{1}{2}+\lambda)\rho^{-\frac{5}{2}}(1-\rho)^{-\frac{1}{2}+\lambda}g_2(\rho;\lambda)
   c_{0,2}(\lambda) \int_0^\rho \int_0^s (1-\chi_\lambda(t))\frac{t^{\frac{5}{2}}g_2(t;\lambda)}{(1+t)^{\frac{3}{2}+\lambda }}dt f'(s) ds
   \\
   &\quad    +(1-\chi_\lambda(\rho))(\frac{1}{2}+\lambda)\rho^{-\frac{5}{2}}(1+\rho)^{-\frac{1}{2}+\lambda}g_1(\rho;\lambda)
   c_{0,1}(\lambda) \int_\rho^1 \int_0^s (1-\chi_\lambda(t))\frac{t^{\frac{5}{2}}g_2(t;\lambda)}{(1+t)^{\frac{3}{2}+\lambda }}dt f'(s) ds
   \\
   &\quad +(1-\chi_\lambda(\rho))(\frac{1}{2}+\lambda)\rho^{-\frac{5}{2}}(1-\rho)^{-\frac{1}{2}+\lambda}g_2(\rho;\lambda)
   c_{0,2}(\lambda) \int_\rho^1 \int_0^s (1-\chi_\lambda(t))\frac{t^{\frac{5}{2}}g_2(t;\lambda)}{(1+t)^{\frac{3}{2}+\lambda }}dt f'(s) ds.
\end{align*}
After integrating by parts twice in each term, one obtains
\begin{align*}
&\quad
(1-\chi_\lambda(\rho))(\tfrac12+\lambda)
\rho^{-\frac52}(1-\rho)^{-\frac12+\lambda}
g_2(\rho;\lambda)c_{0,1}(\lambda)
\int_0^\rho\int_0^s
\frac{(1-\chi_\lambda(t))t^{\frac52}g_1(t;\lambda)}
{(1-t)^{\frac32+\lambda}}\,dt\,f'(s)\,ds
\\
&\quad+
(1-\chi_\lambda(\rho))(\tfrac12+\lambda)
\rho^{-\frac52}(1+\rho)^{-\frac12+\lambda}
g_1(\rho;\lambda)c_{0,1}(\lambda)
\int_\rho^1\int_0^s
\frac{(1-\chi_\lambda(t))t^{\frac52}g_2(t;\lambda)}
{(1+t)^{\frac32+\lambda}}\,dt\,f'(s)\,ds
\\
&=
-\frac{1-\chi_\lambda(\rho)}{\lambda-\frac12}
\rho^{-\frac52}(1-\rho)^{-\frac12+\lambda}
g_2(\rho;\lambda)c_{0,1}(\lambda)
\int_0^\rho
\frac{\partial_s[
(1-\chi_\lambda(s))s^{\frac52}g_1(s;\lambda)f'(s)]}
{(1-s)^{-\frac12+\lambda}}\,ds
\\
&\quad-
\frac{1-\chi_\lambda(\rho)}{\lambda-\frac12}
\rho^{-\frac52}(1+\rho)^{-\frac12+\lambda}
g_1(\rho;\lambda)c_{0,1}(\lambda)
\int_\rho^1
\frac{\partial_s[
(1-\chi_\lambda(s))s^{\frac52}g_2(s;\lambda)f'(s)]}
{(1+s)^{-\frac12+\lambda}}\,ds
\\
&\quad-
\frac{1-\chi_\lambda(\rho)}{\lambda-\frac12}
\rho^{-\frac52}(1-\rho)^{-\frac12+\lambda}
g_2(\rho;\lambda)c_{0,1}(\lambda)
\int_0^\rho
\frac{\partial_s[
(1-\chi_\lambda(s))s^{\frac52}g_1(s;\lambda)]}
{(1-s)^{-\frac12+\lambda}}\,f'(s)\,ds
\\
&\quad-
\frac{1-\chi_\lambda(\rho)}{\lambda-\frac12}
\rho^{-\frac52}(1+\rho)^{-\frac12+\lambda}
g_1(\rho;\lambda)c_{0,1}(\lambda)
\int_\rho^1
\frac{\partial_s[
(1-\chi_\lambda(s))s^{\frac52}g_2(s;\lambda)]}
{(1+s)^{-\frac12+\lambda}}\,f'(s)\,ds
\\
&\quad+
\frac{1-\chi_\lambda(\rho)}{\lambda-\frac12}
\rho^{-\frac52}(1-\rho)^{-\frac12+\lambda}
g_2(\rho;\lambda)c_{0,1}(\lambda)
\int_0^\rho\int_0^s
\frac{\partial_t^2[
(1-\chi_\lambda(t))t^{\frac52}g_1(t;\lambda)]}
{(1-t)^{-\frac12+\lambda}}\,dt\,f'(s)\,ds
\\
&\quad+
\frac{1-\chi_\lambda(\rho)}{\lambda-\frac12}
\rho^{-\frac52}(1+\rho)^{-\frac12+\lambda}
g_1(\rho;\lambda)c_{0,1}(\lambda)
\int_\rho^1\int_0^s
\frac{\partial_t^2[
(1-\chi_\lambda(t))t^{\frac52}g_2(t;\lambda)]}
{(1+t)^{-\frac12+\lambda}}\,dt\,f'(s)\,ds
\\
&\quad+
\frac{(1-\chi_\lambda(\rho))c_{0,1}(\lambda)}
{(\lambda-\frac12)2^{\lambda-\frac12}}
\rho^{-\frac52}(1+\rho)^{-\frac12+\lambda}
g_1(\rho;\lambda)
(1-\chi_\lambda(1))g_2(1;\lambda)f'(1)
\end{align*}
as well as
\begin{align*}
&\quad
(1-\chi_\lambda(\rho))(\tfrac12+\lambda)
\rho^{-\frac52}(1-\rho)^{-\frac12+\lambda}
g_2(\rho;\lambda)c_{0,2}(\lambda)
\int_0^\rho\int_0^s
\frac{(1-\chi_\lambda(t))t^{\frac52}g_2(t;\lambda)}
{(1+t)^{\frac32+\lambda}}\,dt\,f'(s)\,ds
\\
&\quad+
(1-\chi_\lambda(\rho))(\tfrac12+\lambda)
\rho^{-\frac52}(1-\rho)^{-\frac12+\lambda}
g_2(\rho;\lambda)c_{0,2}(\lambda)
\int_\rho^1\int_0^s
\frac{(1-\chi_\lambda(t))t^{\frac52}g_2(t;\lambda)}
{(1+t)^{\frac32+\lambda}}\,dt\,f'(s)\,ds
\\
&=
\frac{1-\chi_\lambda(\rho)}{\lambda-\frac12}
\rho^{-\frac52}(1-\rho)^{-\frac12+\lambda}
g_2(\rho;\lambda)c_{0,2}(\lambda)
\int_0^1\int_0^s
\frac{\partial_t^2[
(1-\chi_\lambda(t))t^{\frac52}g_2(t;\lambda)]}
{(1+t)^{-\frac12+\lambda}}\,dt\,f'(s)\,ds
\\
&\quad-
\frac{1-\chi_\lambda(\rho)}{\lambda-\frac12}
\rho^{-\frac52}(1-\rho)^{-\frac12+\lambda}
g_2(\rho;\lambda)c_{0,2}(\lambda)
\int_0^1
\frac{\partial_s[
(1-\chi_\lambda(s))s^{\frac52}g_2(s;\lambda)]}
{(1+s)^{-\frac12+\lambda}}\,f'(s)\,ds
\\
&\quad-
\frac{1-\chi_\lambda(\rho)}{\lambda-\frac12}
\rho^{-\frac52}(1-\rho)^{-\frac12+\lambda}
g_2(\rho;\lambda)c_{0,2}(\lambda)
\int_0^1
\frac{\partial_s[
(1-\chi_\lambda(s))s^{\frac52}g_2(s;\lambda)f'(s)]}
{(1+s)^{-\frac12+\lambda}}\,ds
\\
&\quad+
\frac{(1-\chi_\lambda(\rho))c_{0,2}(\lambda)}
{(\lambda-\frac12)2^{\lambda-\frac12}}
\rho^{-\frac52}(1-\rho)^{-\frac12+\lambda}
g_2(\rho;\lambda)
(1-\chi_\lambda(1))g_2(1;\lambda)f'(1).
\end{align*}
By performing the same integration by parts for the remaining terms that are supported near $\rho=1$, we obtain the following second decomposition of $D_2''$.
    \begin{lemma}
        We can decompose $D_2''(f)$ as
        \begin{align*}
        D_2''(f)(\rho;\lambda)=\sum_{j=7}^{16} \dot G_j''(f)(\rho;\lambda) 
        \end{align*}

          where
    \begin{align*}
       \dot G_7''(f)(\rho;\lambda)&:=G_7''(f)(\rho;\lambda)
        \\
         \dot G_8''(f)(\rho;\lambda)&:= G_8''(f)(\rho;\lambda)
         \\
         \dot G_9''(f)(\rho;\lambda)&:=(1-\chi_\lambda(\rho)) \rho^{-\frac{5}{2}}(1-\rho)^{-\frac{1}{2}+\lambda}\int_0^\rho \int_0^s \frac{\partial_t^2[(1-\chi_\lambda(t))\O(\rho^0t^{\frac{5}{2}}\langle\omega\rangle^{-2})]}{(1-t)^{-\frac{1}{2}+\lambda }}dt f'(s) ds
         \\
         &\quad + (1-\chi_\lambda(\rho)) \rho^{-\frac{5}{2}}(1-\rho)^{-\frac{1}{2}+\lambda}\int_0^\rho \frac{\partial_s[(1-\chi_\lambda(s))\O(\rho^0s^{\frac{5}{2}}\langle\omega\rangle^{-2})] f'(s)}{(1-s)^{-\frac{1}{2}+\lambda }} ds
         \\
          &\quad + (1-\chi_\lambda(\rho)) \rho^{-\frac{5}{2}}(1-\rho)^{-\frac{1}{2}+\lambda}\int_0^\rho \frac{\partial_s[(1-\chi_\lambda(s))\O(\rho^0s^{\frac{5}{2}}\langle\omega\rangle^{-2})f'(s)] }{(1-s)^{-\frac{1}{2}+\lambda }} ds
         \\
          \dot G_{10}''(f)(\rho;\lambda)&:=(1-\chi_\lambda(\rho)) \rho^{-\frac{5}{2}}(1-\rho)^{-\frac{1}{2}+\lambda} \int_0^1 \int_0^s \frac{\partial_t^2[(1-\chi_\lambda(t))\O(\rho^0t^{\frac{5}{2}}\langle\omega\rangle^{-2})]}{(1+t)^{-\frac{1}{2}+\lambda }}dt f'(s) ds
     \\
     &\quad +  (1-\chi_\lambda(\rho)) \rho^{-\frac{5}{2}}(1-\rho)^{-\frac{1}{2}+\lambda}\int_0^1  \frac{\partial_s[(1-\chi_\lambda(s))\O(\rho^0s^{\frac{5}{2}}\langle\omega\rangle^{-2})] f'(s)}{(1+s)^{-\frac{1}{2}+\lambda }} ds
   \\
  &\quad  +(1-\chi_\lambda(\rho)) \rho^{-\frac{5}{2}}(1-\rho)^{-\frac{1}{2}+\lambda} \int_0^1  \frac{\partial_s[(1-\chi_\lambda(s))\O(\rho^0s^{\frac{5}{2}}\langle\omega\rangle^{-2})f'(s)] }{(1+s)^{-\frac{1}{2}+\lambda }}ds
          \\
          \dot G_{11}''(f)(\rho;\lambda)&:=G_{11}''(f)(\rho;\lambda)
            \\
         \dot G_{12}''(f)(\rho;\lambda)&:= G_{12}''(f)(\rho;\lambda)
              \\
          \dot G_{13}''(f)(\rho;\lambda)&:= G_{13}''(f)(\rho;\lambda)
             \\
         \dot G_{14}''(f)(\rho;\lambda)&:= G_{14}''(f)(\rho;\lambda)
                        \\
          \dot G_{15}''(f)(\rho;\lambda)&:=(1-\chi_\lambda(\rho)) \rho^{-\frac{5}{2}}(1+\rho)^{-\frac{1}{2}+\lambda}\int_\rho^1 \int_0^s \frac{\partial_t^2[(1-\chi_\lambda(t))\O(\rho^0t^{\frac{5}{2}}\langle\omega\rangle^{-2})]}{(1+t)^{-\frac{1}{2}+\lambda }}dt f'(s) ds
         \\
         &\quad + (1-\chi_\lambda(\rho)) \rho^{-\frac{5}{2}}(1+\rho)^{-\frac{1}{2}+\lambda}\int_\rho^1 \frac{\partial_s[(1-\chi_\lambda(s))\O(\rho^0s^{\frac{5}{2}}\langle\omega\rangle^{-2})] f'(s)}{(1+s)^{-\frac{1}{2}+\lambda }} ds
         \\
          &\quad + (1-\chi_\lambda(\rho)) \rho^{-\frac{5}{2}}(1+\rho)^{-\frac{1}{2}+\lambda}\int_\rho^1 \frac{\partial_s[(1-\chi_\lambda(s))\O(\rho^0s^{\frac{5}{2}}\langle\omega\rangle^{-2})f'(s)] }{(1+s)^{-\frac{1}{2}+\lambda }} ds
          \\
     \dot G_{16}''(f)(\rho;\lambda)&:=   (1-\chi_\lambda(\rho))2^{-\lambda+\frac12}
\rho^{-\frac52}[(1-\rho)^{-\frac12+\lambda}\O(\rho^0 \langle\omega\rangle^{-2})+(1+\rho)^{-\frac12+\lambda}\O(\rho^0 \langle\omega\rangle^{-2})].
f'(1)
    \end{align*}
    
\end{lemma}
To move on, we set 
\begin{align*}
    T_{j}''(\tau)f(\rho)&:=\lim_{N\to \infty}\int_{-N}^N e^{i \omega \tau } G_j''(f)(\rho; i \omega)d\omega
    \\
  \dot T_{j}''(\tau)f(\rho)&:=\lim_{N\to \infty}\int_{-N}^N e^{i \omega \tau } \omega \dot G_j''(f)(\rho;i \omega) d\omega
\end{align*}
for $j=7,\dots, 16$,  $\tau \in \R_+$, and $f\in C^\infty_{c,rad}(\R^6)$.
\begin{lemma}
    The estimates
\begin{align*}
        \|T_{j}''(\cdot)f\|_{L^\infty(\R_+) L^2(\B^6)}&\lesssim \|f\|_{\dot H^1(\R^6)}
        \\
        \|T_{j}''(\cdot)f\|_{L^\infty(\R_+) L^2(\B^6)}&\lesssim \|f\|_{\dot H^2(\R^6)}
        \\
        \|\dot T_{j}''(\cdot)f\|_{L^\infty(\R_+) L^2(\B^6)}&\lesssim \|f\|_{\dot H^2(\R^6)}
\end{align*}
hold for $j=7,\dots, 16$ and all $f\in C^\infty_{c,rad}(\R^6).$
\end{lemma}
\begin{proof}
Since $t\leqslant \rho$ we can rewrite $\chi_\lambda(\rho)\O(\rho^{-2}t^5\langle\omega\rangle^{3})=\chi_\lambda(\rho)\O(\rho^{-\frac{5}{2}}t^{\frac{3}{2}-\frac{1}{10}}\langle\omega\rangle^{-1-\frac{1}{10}})$
and again apply Lemma 5.1 of \cite{DonningerWallauch} to obtain 
    \begin{align*}
        |T_{7}''(\tau)f(\rho)|&\lesssim \rho^{-\frac{5}{2}}\int_0^1 s^{\frac{5}{2}-\frac{1}{10}} |f'(s)| ds \lesssim \rho^{-\frac{5}{2}}\|s^{\frac{5}{2}} f'(s)\|_{L_s^{2}(0,1)}\|s^{-\frac{1}{10}}\|_{L^2(0,1)}\lesssim \rho^{-\frac{5}{2}}\|f\|_{L^{2}(\B^6)}
    \end{align*}
    Hence
    \begin{align*}
         \|T_{7}''(\cdot)f\|_{L^\infty(\R_+) L^{2}(\B^6)}&\lesssim \|f\|_{ H^1(\B^6)}
    \end{align*}
    and, by trading an additional power of $t$ for decay in $\omega$, one also obtains 
    \begin{align*}
         \|\dot T_{7}''(\cdot)f\|_{L^\infty(\R_+) L^{2}(\B^6)}&\lesssim \|f\|_{ H^2(\B^6)}.
    \end{align*}
    A similar computation yields the estimate
    \begin{align*}
         |T_{8}''(\tau)f(\rho)|&\lesssim \|f\|_{H^1(\B^6)}\rho^{-\frac{5}{2}}(1-\rho)^{-\frac{1}2} \langle\tau+\log(1-\rho) \rangle^{-2} 
    \end{align*}
    and one concludes 
    \begin{align*}
          \|T_{8}''(\cdot)f\|_{L^\infty(\R_+) L^2(\B^6)}&\lesssim \|f\|_{ H^1(\B^6)}
    \end{align*}
    by employing a minor variation of the arguments used to estimate $T''_{5}$. Likewise, one bounds $\dot T_{8}''$ and we turn to $T_{9}''$. To estimate this operator, we integrate by parts twice to calculate that
    \begin{align*}
        \int_0^\rho \int_0^s (1-\chi_\lambda(t))\frac{\O(\rho^0t^{\frac{5}{2}}\langle\omega\rangle^{0})}{(1-t)^{\frac{3}{2}+\lambda }}dt f'(s) ds&=\int_0^\rho  (1-\chi_\lambda(s))\frac{\O(\rho^0 s^{\frac{5}{2}}\langle\omega\rangle^{-1})}{(1-s)^{\frac{1}{2}+\lambda }} f'(s) ds
        \\
        &\quad +\int_0^\rho \int_0^s \frac{\partial_t[(1-\chi_\lambda(t))\O(\rho^0t^{\frac{5}{2}}\langle\omega\rangle^{-1})]}{(1-t)^{\frac{1}{2}+\lambda }}dt f'(s) ds
        \\
        &=\int_0^\rho  (1-\chi_\lambda(s))\frac{\O(\rho^0 s^{\frac{5}{2}}\langle\omega\rangle^{-1})}{(1-s)^{\frac{1}{2}+\lambda }} f'(s) ds
        \\
        &\quad +\int_0^\rho \frac{\partial_s[(1-\chi_\lambda(s))\O(\rho^0s^{\frac{5}{2}}\langle\omega\rangle^{-2})]}{(1-s)^{-\frac{1}{2}+\lambda }} f'(s) ds
        \\
        &\quad+\int_0^\rho \int_0^s \frac{\partial_t^2[(1-\chi_\lambda(t))\O(\rho^0t^{\frac{5}{2}}\langle\omega\rangle^{-2})]}{(1-t)^{-\frac{1}{2}+\lambda }}dt f'(s) ds
        \\
        &=:I_1+I_2+I_3.
    \end{align*}
    Note that $I_2$ and $I_3$ decay quadratically in $\omega$, hence, we will focus on estimating the operator corresponding to $I_1$. In other words, we look at 
    \begin{align*}
      &\quad T_{9,1}''(\tau)f(\rho) 
      \\
      &=\lim_{N\to \infty} \int_{-N}^N e^{i\omega\tau}(1-\chi_{i \omega}(\rho)) \rho^{-\frac{5}{2}}(1-\rho)^{-\frac{1}{2}+i\omega}\int_0^\rho  (1-\chi_{i\omega}(s))\frac{\O(\rho^0 s^{\frac{5}{2}}\langle\omega\rangle^{-1})}{(1-s)^{\frac{1}{2}+i\omega  }} f'(s) ds d\omega
       \\
       &=\lim_{N\to \infty}\int_0^\rho f'(s) \int_{-N}^N e^{i\omega\tau}(1-\chi_{i \omega}(\rho)) \rho^{-\frac{5}{2}}(1-\rho)^{-\frac{1}{2}+i\omega}  (1-\chi_{i\omega}(s))\frac{\O(\rho^0 s^{\frac{5}{2}}\langle\omega\rangle^{-1})}{(1-s)^{\frac{1}{2}+i\omega  }} d\omega ds.
    \end{align*}
    Now, we can use the identity
    $$
    (1-\chi_\lambda)(s)\O(\rho^{-\frac{5}{2}}s^{\frac{5}{2}}\langle\omega\rangle^{-1})= (1-\chi_\lambda)(\rho) (1-\chi_\lambda)(s)\O(\rho^{-\frac{5}{2}+\frac{1}{10}}s^{\frac{5}{2}+\frac{1}{10}}\langle\omega\rangle^{-\frac{8}{10}})
    $$and a straightforward variation of Lemma 5.2 in \cite{DonningerWallauch} to obtain 
    \begin{align*}
      |  T_{9,1}''(\tau)f(\rho)|&\lesssim \rho^{-\frac{5}{2}}(1-\rho)^{-\frac{1}{2}} \int_0^\rho \frac{|f'(s)|s^{\frac{5}{2}}}{(1-s)^{\frac{1}{2}}} g(\tau+\log(1-\rho)-\log(1-s)) ds
    \end{align*}
    where $g(x)=\langle x\rangle^{-2} |x|^{-\frac{1}{5}}$.
    So, 
\begin{align*}
  &\quad\|  T_{9,1}''(\cdot)f\|_{L^2(\B^6)}
  \\
  &\lesssim\left( \int_0^1 (1-\rho)^{-1}\left(\int_{0}^1 s^{\frac{5}{2}} |f'(s)|\frac{1}{(1-s)^{\frac{1}{2}}}g(\tau+\log(1-\rho)-\log(1-s)) ds \right)^{2} d\rho\right)^{\frac{1}{2}}\\
    &\leqslant   \left( \int_0^\infty \left(\int_{-\infty}^0 (1-e^{y})^{\frac{5}{2}} |f'(1-e^{y})|e^{\frac{y}{2}}g(\tau-x-y) dy \right)^{2} dx\right)^{\frac{1}{2}}.
\end{align*}
By employing Young's inequality twice, we then obtain 
\begin{align*}
 \|  T_{9,1}''(\cdot)f\|_{L^\infty(\R_+)L^2(\B^6)}& \lesssim  \left( \int_\R \left(\int_{-\infty}^0 (1-e^{y})^{\frac{5}{2}} |f'(1-e^{y})|e^{\frac{y}{2}}g(\tau-y) dy \right)^{2} d\tau\right)^{\frac{1}2}
 \\
 &\lesssim \|(1-e^{y})^{\frac{5}{2}} |f'(1-e^{y})|e^{\frac{y}{2}}\|_{L^2((-\infty,0))}\|g\|_{L^1(\R)}
 \\
 &\lesssim  \|f\|_{ H^2(\B^6)}.
\end{align*}
To bound the remaining operators, one readily adapts the strategies used for the cases $j=7,8,9$ and in previously lemmas to conclude this proof.
\end{proof}
With this, we come to \begin{align*}
 D_3''(f)(\rho;\lambda):=  u_0''(\rho;\lambda)[ \mu(f)(\lambda)+U_2(1;\lambda)f(1)]-u_{\mathrm{f}_0}''(\rho;\lambda)[\mu_{\mathrm{f}}(f)(\lambda)+U_{\mathrm{f}_2}(1;\lambda)f(1)].
   \end{align*}
\begin{lemma}
       We can decompose $ D_3''(f)$ as 
       \begin{align*}
           D_3''(f)(\rho;\lambda)=\sum_{j=17}^{28} G_j''(f)(\rho;\lambda)
       \end{align*}
       with
       \begin{align*}
           G_{17}''(f)(\rho;\lambda)&=\chi_\lambda(\rho)(1-\rho^2)^{\frac{\lambda}{2}}\O(\rho^{-2}\langle\omega\rangle^{-\frac{1}{2}})\int_1^\infty \frac{\partial_s[\kappa(s)\O(s^{\frac{5}{2}}\langle\omega\rangle^0)f(s)]}{(1+s)^{\frac{1}{2}+\lambda}}ds 
           \\
              G_{18}''(f)(\rho;\lambda)&=(1-\chi_\lambda(\rho))(1-\rho)^{-\frac{1}{2}+\lambda}\O(\rho^{-\frac{5}{2}}\langle\omega\rangle^{-1})\int_1^\infty \frac{\partial_s[\kappa(s)\O(s^{\frac{5}{2}}\langle\omega\rangle^0)f(s)]}{(1+s)^{\frac{1}{2}+\lambda}}ds 
              \\
              G_{19}''(f)(\rho;\lambda)&=(1-\chi_\lambda(\rho))(1+\rho)^{-\frac{1}{2}+\lambda}\O(\rho^{-\frac{5}{2}}\langle\omega\rangle^{-1})\int_1^\infty \frac{\partial_s[\kappa(s)\O(s^{\frac{5}{2}}\langle\omega\rangle^0)f(s)]}{(1+s)^{\frac{1}{2}+\lambda}}ds 
                \\
               G_{20}''(f)(\rho;\lambda)&=\chi_\lambda(\rho)(1-\rho^2)^{\frac{\lambda}{2}}\O(\rho^{-2}\langle\omega\rangle)f(1)\int_0^1 \frac{\chi_\lambda(s)\O(s\langle\omega\rangle^{-1})}{(1-s^2)^{\frac{\lambda}{2}}}ds
           \\
              G_{21}''(f)(\rho;\lambda)&=(1-\chi_\lambda(\rho))(1-\rho)^{-\frac{1}{2}+\lambda}\O(\rho^{-\frac{5}{2}}\langle\omega\rangle^{-\frac{3}{2}})f(1)\int_0^1 \frac{\chi_\lambda(s)\O(s\langle\omega\rangle^{-1})}{(1-s^2)^{\frac{\lambda}{2}}}ds
              \\
              G_{22}''(f)(\rho;\lambda)&=(1-\chi_\lambda(\rho))(1+\rho)^{-\frac{1}{2}+\lambda}\O(\rho^{-\frac{5}{2}}\langle\omega\rangle^{-\frac{3}{2}})f(1)\int_0^1 \frac{\chi_\lambda(s)\O(s\langle\omega\rangle^{-1})}{(1-s^2)^{\frac{\lambda}{2}}}ds
               \\
               G_{23}''(f)(\rho;\lambda)&=\chi_\lambda(\rho)(1-\rho^2)^{\frac{\lambda}{2}}\O(\rho^{-2}\langle\omega\rangle^{-\frac12})f(1)\int_0^1\frac{\partial_s[(1-\chi_\lambda(s))\O(s^{\frac{5}{2}}\langle\omega\rangle^0)]}{(1+s)^{\frac{1}{2}+\lambda}}ds
           \\
              G_{24}''(f)(\rho;\lambda)&=(1-\chi_\lambda(\rho))(1-\rho)^{-\frac{1}{2}+\lambda}\O(\rho^{-\frac{5}{2}}\langle\omega\rangle^{-1})f(1)\int_0^1\frac{\partial_s[(1-\chi_\lambda(s))\O(s^{\frac{5}{2}}\langle\omega\rangle^0)]}{(1+s)^{\frac{1}{2}+\lambda}}ds
              \\
              G_{25}''(f)(\rho;\lambda)&=(1-\chi_\lambda(\rho))(1+\rho)^{-\frac{1}{2}+\lambda}\O(\rho^{-\frac{5}{2}}\langle\omega\rangle^{-1})f(1)\int_0^1\frac{\partial_s[(1-\chi_\lambda(s))\O(s^{\frac{5}{2}}\langle\omega\rangle^0)]}{(1+s)^{\frac{1}{2}+\lambda}}ds
              \\
              G_{26}''(f)(\rho;\lambda)&=\chi_\lambda(\rho)(1-\rho^2)^{\frac{\lambda}{2}}\O(\rho^{-2}\langle\omega\rangle^{-\frac{1}{2}})
              \\
              &\quad \times \int_1^\infty \left[\frac{\partial_s[ (1-\kappa(s)) \O(s^{\frac{5}{2}}\langle\omega\rangle^0)f(s)]}{(s-1)^{\frac{1}{2}+\lambda}}-\frac{\partial_s[ (1-\kappa(s)) \O(s^{\frac{5}{2}}\langle\omega\rangle^0)f(s)]}{(1+s)^{\frac{1}{2}+\lambda}}\right] ds
              \\
              G_{27}''(f)(\rho;\lambda)&=(1-\chi_\lambda(\rho))(1-\rho)^{-\frac{1}{2}+\lambda}\O(\rho^{-\frac{5}{2}}\langle\omega\rangle^{-1})\\
              &\quad \times \int_1^\infty \left[\frac{\partial_s[ (1-\kappa(s)) \O(s^{\frac{5}{2}}\langle\omega\rangle^0)f(s)]}{(s-1)^{\frac{1}{2}+\lambda}}-\frac{\partial_s[ (1-\kappa(s)) \O(s^{\frac{5}{2}}\langle\omega\rangle^0)f(s)]}{(1+s)^{\frac{1}{2}+\lambda}}\right] ds
              \\
              G_{28}''(f)(\rho;\lambda)&=(1-\chi_\lambda(\rho))(1+\rho)^{-\frac{1}{2}+\lambda}\O(\rho^{-\frac{5}{2}}\langle\omega\rangle^{-1})\\
              &\quad \times \int_1^\infty \left[\frac{\partial_s[ (1-\kappa(s)) \O(s^{\frac{5}{2}}\langle\omega\rangle^0)f(s)]}{(s-1)^{\frac{1}{2}+\lambda}}-\frac{\partial_s[ (1-\kappa(s)) \O(s^{\frac{5}{2}}\langle\omega\rangle^0)f(s)]}{(1+s)^{\frac{1}{2}+\lambda}}\right] ds.
       \end{align*}
   \end{lemma}
   We define associated operators $T_{j}$ as
   \begin{align*}
       T_{j}''(\tau)f(\rho)&:=\lim_{N \to \infty} \int_{-N}^N e^{i\omega \tau } G_{j}''( f)(\rho;i\omega)d \omega\\
         \dot T_{j}''(\tau)f(\rho)&:=\lim_{N \to \infty} \int_{-N}^N e^{i\omega \tau } \omega G_{j}''(f)(\rho;i\omega)d \omega
   \end{align*}
for $j=17,\dots 25$, $\tau \in \R_+$, and $f\in C^\infty_{c,rad}(\R^6).$
 \begin{lemma}
    The estimates
\begin{align*}
        \|T_{j}''(\cdot)f\|_{L^\infty(\R_+) L^{2}(\B^6)}&\lesssim \|f\|_{\dot H^1(\R^6)}
        \\
        \|T_{j}''(\cdot)f\|_{L^\infty(\R_+) L^2(\B^6)}&\lesssim \|f\|_{\dot H^2(\R^6)}
        \\
        \|\dot T_{j}''(\cdot)f\|_{L^\infty(\R_+) L^2(\B^6)}&\lesssim \|f\|_{\dot H^2(\R^6)}
\end{align*}
hold for $j=17,\dots, 25$, and all $f\in C^\infty_{c,rad}(\R^6).$
\end{lemma}
\begin{proof}
    An application of Lemma 5.1 of \cite{DonningerWallauch} yields
    \begin{align*}
        |T_{17}''(\tau)f|&\lesssim \rho^{-\frac{5}{2}-\delta}\int_1^\infty |\partial_s [\kappa(s) f(s)]| ds
        \\
        &\lesssim \rho^{-\frac{5}{2}-\delta}\|f\|_{H^1((1,r'))}
    \end{align*}
    for some $r' \in (1,\infty).$
    So, one readily obtains the desired bounds on $T_{17}''$. To control $\dot T_{17}''$, one performs one more integration by parts and then argues in the same fashion to obtain
    \begin{align*}
        |T_{17}''(\cdot)f|&\lesssim \rho^{-\frac{5}{2}-\delta}\left(|f(1)|+|f'(1)|+\|f\|_{H^2((1,r'))}\right)
    \end{align*}
    which implies the claimed estimates on $\dot T_{17}''$. To bound  $T_{18}''$, we infer the estimate
    \begin{align*}
        |T_{18}''(\tau)f(\rho)|&\lesssim \rho^{-\frac{5}{2}}(1-\rho)^{-\frac{1}{2}}\int_1^\infty |\partial_s [\kappa(s) f(s)]|g(\tau +\log(1-\rho)-\log(s+1))ds,
    \end{align*}
    with $g(x)=\langle x\rangle^{-2}|x|^{-\frac{1}{24}}$.
    Thus, by decomposing the $\rho$ integral into two integrals and changing variables in the well established fashion and setting $ \widetilde f(s)=|\partial_s [\kappa(s) f(s)]$, we obtain that
    \begin{align*}
         \|T_{18}''(\tau)f\|_{L^2(\B^6)}^{2}&\lesssim \int_1^\infty |\partial_s [\kappa(s) f(s)]|^2\int_0^\frac{1}{2} g(\tau +\log(1-\rho)-\log(s+1))^{2}d\rho ds 
         \\
         &\quad +\int_{-\infty}^\infty |\widetilde f(1+e^y)|^2e^y\int_{-\log \frac{1}{2}}^\infty g(\tau -x-y)^{2}dx dy
         \\
         &\lesssim \int_1^\infty |\partial_s [\kappa(s) f(s)]|^2\int_0^\frac{1}{2}\rho^{-\frac{1}{2}} |\tau +\log(1-\rho)-\log(s+1)|^{-\frac{1}{4}}d\rho ds 
         \\
         &\quad +\int_{-\infty}^\infty |\widetilde f(1+e^y)|^2e^y\int_{\R} g(x)^{2}dx dy
         \\
         &\lesssim \int_1^\infty |\partial_s [\kappa(s) f(s)]|^2ds \lesssim \|f\|_{H^1((1,r'))}^2
    \end{align*}
    and we obtain the desired estimates. Once more, one obtains the estimates for $ \dot T_{18}''f$ by integrating by parts another time and a straightforward modification of the arguments exhibited in this proof. One also bounds $T_{19}'',\dots, T_{22}''$ and $\dot T_{19}'',\dots, \dot T_{22}''$  with the same strategy. Lastly, one readily modifies the arguments employed to bound the operators $T_{22},\dots, T_{24}$ and $ \dot T_{22},\dots, \dot T_{24}$ in the proof of Lemma \ref{lem: strichartz D_3 k_2} to estimate the remaining operators and we conclude this proof.
\end{proof}

With this, we come to the terms supported near $\infty.$

\begin{lemma}
 We can rewrite $\sum_{j=26}^{28}G_j''(F_\lambda)(\rho;\lambda)$ as 
as 
\[\sum_{j=26}^{28}G_j''(F_\lambda)(\rho;\lambda)=\sum_{j=1}^{3}\widetilde G_j''(F_+)(\rho;\lambda)+\widetilde G_j''(F_-)(\rho;\lambda)+R''(f_1,f_2)(\rho;\lambda) \]
where 
$$F_\pm(s)= \left[2(s\pm1)\kappa'(s)s^{\frac{5}{2}}f_1(s)\mp5(1-\kappa(s))s^{\frac{3}{2}}f_1(s)-2(1-\kappa(s))s^{\frac{5}{2}}(f_2(s)\pm f_1'(s))\right]$$
and 
\begin{align*}
   \widetilde G_{1,\pm}''(F)(\rho;\lambda)&=\chi_\lambda(\rho)(1-\rho^2)^{\frac{\lambda}{2}}\O(\rho^{-2}\langle\omega\rangle^{-\frac{1}{2}})
             \int_1^\infty
\frac{\partial_s[\O(s^0\langle\omega\rangle^0)F(s)]}
{(s\pm1)^{\frac12+\lambda}} ds
              \\
           \widetilde   G_{2,\pm}''(F)(\rho;\lambda)&=(1-\chi_\lambda(\rho))(1-\rho)^{-\frac{1}{2}+\lambda}\O(\rho^{-\frac{5}{2}}\langle\omega\rangle^{-1})
              \int_1^\infty
\frac{\partial_s[\O(s^0\langle\omega\rangle^0)F(s)]}
{(s\pm1)^{\frac12+\lambda}} ds
              \\
              \widetilde G_{3,\pm}''(F)(\rho;\lambda)&=(1-\chi_\lambda(\rho))(1+\rho)^{-\frac{1}{2}+\lambda}\O(\rho^{-\frac{5}{2}}\langle\omega\rangle^{-1})
             \int_1^\infty
\frac{\partial_s[\O(s^0\langle\omega\rangle^0)F(s)]}
{(s\pm1)^{\frac12+\lambda}} ds,
\end{align*}
and where $R''(f_1,f_2)(\rho;\lambda)$ is a collection of lower order terms that all possess the same kernel structures as one of the $\widetilde G_{j,\pm}''.$
\end{lemma}
Proceeding as above, we define another set of operators 
\begin{align*}
 \widetilde    T_{j,\pm}''(\tau)\ff(\rho):=\lim_{N\to \infty}\int_{-N}^N e^{i \omega \tau } \left[\widetilde G_j''(F_-)(\rho;i \omega)+\widetilde G_j''(F_+)(\rho; i \omega)\right] d\omega
\end{align*}
for $j=1,2,3$, $\tau \in \R_+$, and $\ff \in C_{c,rad}^{\infty}(\R^6)\times C_{c,rad}^{\infty}(\R^6)$.
  \begin{lemma}
    The estimates 
    \begin{align*}
        \| \widetilde    T_{j}''\ff\|_{L^\infty(\R_+)L^2(\B^6)}&\lesssim \|\ff\|_{\dot H^2(\R^6)\times \dot H^1 (\R^6)}
    \end{align*}
    hold for $j=1,2,3$ and $\ff \in C_{c,rad}^{\infty}(\R^6)\times C_{c,rad}^{\infty}(\R^6)$.
\end{lemma}
\begin{proof}
One estimates $\widetilde T_{20}''$ by applying the same strategy used in the proof of Lemma \ref{lem: strichartz1-k2 lemma}. By 
applying Lemma 5.2 of \cite{DonningerWallauch} we conclude the estimate
\begin{align*}
   &\quad  \|\widetilde T_{21}''(\tau)\ff\|_{L^2(\B^6)}^{2}
   \\
    &\lesssim  \int_1^\infty \frac{|F_-'(s)|^2+s^{-1}|F_-(s)|^2}{(s-1)}\langle\tau-\log(s-1)\rangle^{-2}\int_0^\frac{1}{2} \rho^{-\frac{1}{10}}|\tau+\log(1-\rho)-\log(s-1)|^{-\frac{1}{10}} d\rho ds
    \\
    &\quad+ \left\|\int_1^\infty \frac{|F_-'(s)|+s^{-1}|F_-(s)|}{(s-1)^{\frac{1}{2}}} (1-\rho)^{-\frac12}g(\tau+\log(1-\rho)-\log(s-1)) ds\right\|_{L^2_\rho((\frac{1}{2},1))}^2
    \\
    &\quad+ \int_1^\infty \frac{|F'_+(s)|^2+s^{-1}|F_+(s)|^2}{(1+s)}\langle\tau-\log(1+s)\rangle^{-2}\int_0^\frac{1}{2} \rho^{-\frac{1}{10}}|\tau+\log(1-\rho)-\log(1+s)|^{-\frac{1}{10}} d\rho ds
    \\
    &\quad+ \left\|\int_1^\infty \frac{|F_+'(s)|+s^{-1}|F_+(s)|}{(s+1)^{\frac{1}{2}}} (1-\rho)^{-\frac12}g(\tau+\log(1-\rho)-\log(s+1)) ds\right\|_{L^2_\rho((\frac{1}{2},1))}
    \\
    &=:I_{1,-}(\tau)+I_{2,-}(\tau)+I_{1,+}(\tau)+I_{2,+}\tau)
\end{align*}
where $g(x)=\langle x\rangle^{-2}|x|^{-\frac{1}{10}}$. 
Given that \begin{align*}
    \int_0^\frac{1}{2} \rho^{-\frac{1}{10}}|\tau+\log(1-\rho)-\log( s\pm 1)|^{-\frac{1}{10}} d\rho\lesssim 1,
\end{align*}
one estimates $I_{1,-}$ and $I_{1,+}$ in the same way that we bounded $I_1$ in the proof of Lemma \ref{lem: strichartz1-k2 lemma}. To estimate $I_{2,-}$, we define the new integration variables $\rho=1-e^{-x}$ and $s=1+e^y$ to obtain that
\begin{align*}
    &\quad \|I_{2,-}\|_{L^\infty(\R_+)}\\
    &\lesssim \left\|\left(\int_{\log 2}^\infty \left( \int_{-\infty}^\infty \left(|F'(1+ e^y)|+(1+ e^y)^{-1}|F(1+ e^y)|\right)e^{\frac{y}{2}}g(\tau-x-y) dy \right)^{2}dx\right)^{\frac{1}{2}}
    \right\|_{L^\infty_\tau(\R_+)}
    \\
&= \left\|1_{(\log 2,\infty)}*\left( \int_{-\infty}^\infty \left(|F'(1+ e^y)|+(1+e^y)^{-1}|F(1+e^y)|\right)e^{\frac{y}{2}}g(\cdot-y) dy \right)^{\frac{1}{2}}
    \right\|_{L^\infty(\R_+)}^{\frac{1}{2}}.
\end{align*}
Thus, applying Young's inequality twice shows that 
\begin{align*}
        \|I_{2,-}\|_{L^\infty(\R_+)} &\lesssim \|1_{(-\log 2,\infty)}\|_{L^\infty(\R)}^{\frac{1}{2}}\left\|\int_{-\infty}^\infty \left(|F'(1+ e^y)|+(1+e^y)^{-1}|F(1+ e^y)|\right)e^{\frac{y}{2}}g(\tau-y) dy 
    \right\|_{L^{2}_\tau(\R)}
    \\
    &=\left\|\left[\left(|F'(1+ e^{(\cdot)})|+(1+ e^{(\cdot)})^{-1}|F(1+ e^{(\cdot)})|\right)e^{\frac{(\cdot)}{2}}\right]*g 
    \right\|_{L^{2}(\R)}
    \\
    &\lesssim \left\|\left(|F'(1+ e^{y})|+(1+ e^{y})^{-1}|F(1+ e^{y})|\right)e^{\frac{y}{2}} \right\|_{L^{2}_y(\R)}\|g\|_{L^1(\R)} 
    \\
    &\lesssim \|\ff\|_{\dot H^2(\R^6) \times \dot H^1(\R^6)},
    \end{align*}
thanks to Lemma \ref{lem: norm estimate delta}. In the same way, one estimates $I_{2,+}$ as well as $\widetilde T_{22}''$ which concludes this proof.
\end{proof}

Hence, by putting everything together, and by employing the same techniques to establish $L^2 W^{1,4}$-type Strichartz estimates, we finally arrive at the desired result of this section.
\begin{proposition}
\label{prop:InteriorSTrichartz}
    The semigroup $\Sf$ satisfies the estimates
    \begin{align*}
        \|[\Sf(\tau)(\If-\Pf)\ff]_1\|_{L^q_\tau(\R_+) L^r(\B^6)}&\lesssim  \|\ff\|_{\mathcal{H}^2}
    \end{align*}
    for all $q\in [2,\infty],~r\in[6,12]$ satisfy \eqref{eq:AdmissibleExponents}
    and all $\ff\in C_{c,rad}^{\infty}(\R^6)\times  C_{c,rad}^{\infty}(\R^6)$.
    Additionally, the estimates
    \begin{align*}
        \|\Sf(\tau)(\If-\Pf)\ff\|_{L^\infty_\tau(\R_+) H^2(\B^6)\times H^1(\B^6)}&\lesssim  \|\ff\|_{\mathcal{H}^2}
    \end{align*}
    and 
       \begin{align*}
        \|[\Sf(\tau)(\If-\Pf)\ff]_1\|_{L^2_\tau(\R_+) W^{1,4}(\B^6)}&\lesssim  \|\ff\|_{\mathcal{H}^2}
    \end{align*}
    hold for all $\ff\in C_{c,rad}^{\infty}(\R^6)\times  C_{c,rad}^{\infty}(\R^6).$
\end{proposition}
\begin{proof}
By employing the previously established kernel bounds we conclude that
\begin{align*}
    \|[\Sf(\tau)(\If-\Pf)\ff]_1\|_{L^\infty_\tau (\R_+)H^2(\B^6)}\lesssim \|\ff\|_{\mathcal{H}^2}.
\end{align*}
To prove the estimate
\begin{align*}
    \|[\Sf(\tau)(\If-\Pf)\ff]_2\|_{L^\infty_\tau (\R_+)H^1(\B^6)}\lesssim \|\ff\|_{\mathcal{H}^2}
\end{align*}
we note that according to \eqref{eq:cNLWSimilarity}
we have the identity
\begin{align*}
 [\Sf(\tau)(\If-\Pf)\ff]_2(\rho)=   \partial_\tau [\Sf(\tau)(\If-\Pf)\ff]_1(\rho)-[\Sf(\tau)(\If-\Pf)\ff]_1(\rho)-\rho\partial_\rho [\Sf(\tau)(\If-\Pf)\ff]_1(\rho).
\end{align*}
Note that the $\tau$-derivative produces a factor $\lambda$. Consequently, when estimating $[\Sf(\tau)(\If-\Pf)\ff]_2(\rho)$ in $H^1(\B^6)$, one arrives at operators that are of the same form as the previously bounded operators $T_{j}''$, $\dot T_{j}''$, and $\widetilde T_{j}''$, which one readily estimates in the well-established fashion. 
\end{proof}

	\section{Strichartz estimates in the exterior of the light cone}
    \label{section:StrichartzEstimatesExterior}
    In this section, our aim is to estimate the integral given in \eqref{eq:IntegralRepresentation} on the imaginary axis, i.e. when $\varepsilon\to 0$. It is to be noted that in the case $\rho\in [1,\infty)$, it is acceptable to let $\varepsilon
    \to 0$ since the matching problem is limited to the fundamental system on $(0,1)$. This is a consequence of the fact that wave equations are well-posed on the exterior of the light cone. Thus, we can leverage the representation \eqref{eq:IntegralRepresentation} to obtain the following estimates:
    \begin{equation*}
    \begin{split}
    \|[\Sf(\tau)\ff]_1\|_{L^q(\R_+)L^r(\B^c_1)}&\lesssim \|\ff\|_{\dot{H}^2(\R^6)\times \dot H^1(\R^6)},\\
    \|[\Sf(\tau)\ff]_1\|_{L^2(\R_+)\dot{W}^{1,4}(\B^c_1)}&\lesssim \|\ff\|_{\dot{H}^2(\R^6)\times \dot H^1(\R^6)},\\
     \|[\Sf(\tau)\ff]_1\|_{L^\infty(\R_+)\dot{H}^2(\B^c_1)}&\lesssim \|\ff\|_{\dot{H}^2(\R^6)\times \dot H^1(\R^6)},\\
     \|[\Sf(\tau)\ff]_2\|_{L^\infty(\R_+)\dot{H}^1(\B^c_1)}&\lesssim \|\ff\|_{\dot{H}^2(\R^6)\times \dot H^1(\R^6)}
     \end{split}
\end{equation*}
for $(q,r)$ satisfying \eqref{eq:AdmissibleExponents}. We shall prove the first and the third estimate above, noting that the proof of the second estimate follows along the same lines as of the first one. The last estimate is proved in Section \ref{section:SecondComponent}.

\subsection{Difference of the resolvents}
In the following, to estimate the difference 
 \begin{equation}
 \begin{split}
 \label{eq:DifferenceOfResolvents}
 \mathcal{R}(F_{\lambda})(\rho;\lambda) - \mathcal{R}_{\mathrm{f}}(F_{\lambda})(\rho;\lambda )&=   -\tilde{u}_1(\rho;\lambda) \int_{\rho}^{\infty} \frac{s^5\tilde{u}_2(s;\lambda)F_{\lambda}(s)}{(s^2-1)^{\frac{3}{2}+\lambda}}ds+\tilde{u}_2(\rho;\lambda) \int_{\rho}^{\infty} \frac{s^5\tilde{u}_1(s;\lambda)F_{\lambda}(s)}{(s^2-1)^{\frac{3}{2}+\lambda}}ds \\
	&\quad + \tilde{u}_{\mathrm{f}_1} (\rho;\lambda) \int_{\rho}^{\infty}  \frac{s^5\tilde{u}_{\mathrm{f}_2}(s;\lambda)F_{\lambda}(s)}{(s^2-1)^{\frac{3}{2}+\lambda}}ds - \tilde{u}_{\mathrm{f}_2} (\rho;\lambda) \int_{\rho}^{\infty} \frac{s^5\tilde{u}_{\mathrm{f}_1}(s;\lambda)F_{\lambda}(s)}{(s^2-1)^{\frac{3}{2}+\lambda}}ds.
    \end{split}
 \end{equation}
 
 We split the above into several terms so as to separate the singular point $1$ and $\infty$ and estimate the corresponding terms.  For this we state the following result.
 \begin{lemma}\label{lemB asymp}
Let $\lambda=\varepsilon+i\omega\in S$ and set
$W(r):=\int_r^\infty V(t)\,dt$. Then
\[
    (1-\kappa(\rho))
    \bigl[R(\lambda)-R_f(\lambda)\bigr]
    \bigl((1-\kappa)F_\lambda\bigr)(\rho)
    =B_+F_\lambda(\rho;\lambda)+B_-F_\lambda(\rho;\lambda),
\]
where, for $\sigma\in\{+1,-1\}$,
\begin{align} \label{eq:B-+Terms}
    B_\sigma F_\lambda(\rho;\lambda)
    :=\frac{(1-\kappa(\rho))\rho^{-5/2}}{(3+2\lambda)^2}
    \int_\rho^\infty
    \left(\frac{\rho+\sigma}{s+\sigma}\right)^{\frac32+\lambda}
    (1-\kappa(s))s^{5/2}
    \gamma_\sigma(\rho,s;\lambda)F_\lambda(s)\,ds
\end{align}
with
\[
    \gamma_\sigma(\rho,s;\lambda)
    =W(s)-W(\rho)
     +\mathcal O\!\left(\rho^0s^0\langle\omega\rangle^{-1}\right)
\]
uniformly for $\rho\leq s$ in $\operatorname{supp}(1-\kappa)$.
\end{lemma}

\begin{proof}
Apply the Volterra expansion from the proof of Lemma~4.8 with
initial point $s$ to the far-field Green kernel. Collecting the
terms with the same phase gives the two ratios displayed above.
The constant terms cancel upon subtraction of the free kernel,
and the first correction to either amplitude is
\[
    \int_\rho^s\bigl(\widetilde V(t)-\widetilde V_f(t)\bigr)\,dt
    =-\int_\rho^s V(t)\,dt=W(s)-W(\rho).
\]
The remaining terms have the stated symbol bound.
\end{proof}
With this,  decompose
 \begin{align*}
     \mathcal{R}(F_{\lambda})(\rho;\lambda) - \mathcal{R}_{\mathrm{f}}(F_{\lambda})(\rho;\lambda)=\big(A_-+ A_+ + B_-+B_+  +\sum_{i,j\in\{+,-\}} C_{ij}\big)F_{\lambda}(\rho;\lambda),
 \end{align*}
 where $B_-$ and $B_+$ are as in Lemma \ref{lemB asymp} and
 \begin{equation}
 \begin{split}
 \label{eq:A-+Terms}
  A_- F_{\lambda}(\rho;\lambda)&=-\frac{\kappa(\rho)\rho^{-\frac{5}{2}}(\rho-1)^{\frac{3}{2}+\lambda}}{3+2\lambda} \Big[g_{\mathrm{f}_2}(\rho;\lambda)\int_{\rho}^{\infty} \frac{\kappa(s)s^{\frac{5}{2}}g_{\mathrm{f}_1}(s;\lambda)F_{\lambda}(s)}{(s-1)^{\frac{3}{2}+\lambda}} ds\\
    &\quad -g_2(\rho;\lambda)\int_{\rho}^{\infty} \frac{\kappa(s)s^{\frac{5}{2}}g_1(s;\lambda)F_{\lambda}(s)}{(s-1)^{\frac{3}{2}+\lambda}}ds\Big],\\
   A_+ F_{\lambda}(\rho;\lambda)&=-\frac{\kappa(\rho)\rho^{-\frac{5}{2}}(\rho+1)^{\frac{3}{2}+\lambda}}{3+2\lambda} \Big[g_1(\rho;\lambda)\int_{\rho}^{\infty} \frac{\kappa(s)s^{\frac{5}{2}}g_2(s;\lambda)F_{\lambda}(s)}{(s+1)^{\frac{3}{2}+\lambda}}ds\\\
    &\quad -
    g_{\mathrm{f}_1}(\rho;\lambda)\int_{\rho}^{\infty} \frac{\kappa(s)s^{\frac{5}{2}}g_{\mathrm{f}_2}(s;\lambda)F_{\lambda}(s)}{(s+1)^{\frac{3}{2}+\lambda}} ds\Big],
     \end{split}
    \end{equation}

 \begin{equation}
     \begin{split}
     \label{eq:C+-Terms}
     C_{+-}F_{\lambda}(\rho;\lambda)&:=-\frac{\kappa(\rho)\rho^{-\frac{5}{2}}(\rho+1)^{\frac{3}{2}+\lambda}}{3+2\lambda} \Big[a_{2,3}(\lambda)g_1(\rho;\lambda) \int_{\rho}^{\infty} \frac{(1-\kappa(s))s^{\frac{5}{2}}g_3(s;\lambda) F_{\lambda}(s)}{(s-1)^{\frac{3}{2}+\lambda}}ds\\
     &\quad - a_{\mathrm{f}_{2,3}}(\lambda) g_{\mathrm{f}_1}(\rho;\lambda) \int_{\rho}^{\infty} \frac{(1-\kappa(s))s^{\frac{5}{2}}g_{\mathrm{f}_3}(s;\lambda) F_{\lambda}(s)}{(s-1)^{\frac{3}{2}+\lambda}}ds\Big]\\
     &-\frac{(1-\kappa(\rho))\rho^{-\frac{5}{2}}(\rho+1)^{\frac{3}{2}+\lambda}}{3+2\lambda} \Big[ a_{\mathrm{f}_{2,3}}(\lambda) g_{\mathrm{f}_3}(\rho;\lambda) \int_{\rho}^{\infty} \frac{\kappa(s)s^{\frac{5}{2}}g_{\mathrm{f}_1}(s;\lambda)F_{\lambda}(s)}{(s-1)^{\frac{3}{2}+\lambda}}ds\\
     &\quad -a_{2,3}(\lambda)g_3(\rho;\lambda) \int_{\rho}^{\infty} \frac{\kappa(s)s^{\frac{5}{2}}g_1(s;\lambda)F_{\lambda}(s)}{(s-1)^{\frac{3}{2}+\lambda}}ds\Big].
     \end{split}
 \end{equation}

 \begin{equation}
     \begin{split}
     \label{eq:C-+Terms}
         C_{-+}F_{\lambda}(\rho;\lambda)&:=\frac{\kappa(\rho)\rho^{-\frac{5}{2}}(\rho-1)^{\frac{3}{2}+\lambda}}{3+2\lambda} \Big[a_{1,4}(\lambda)g_2(\rho;\lambda) \int_{\rho}^{\infty} \frac{(1-\kappa(s))s^{\frac{5}{2}}g_4(s;\lambda) F_{\lambda}(s)}{(s+1)^{\frac{3}{2}+\lambda}}ds\\
     &\quad -a_{\mathrm{f}_{1,4}}(\lambda) g_{\mathrm{f}_2}(\rho;\lambda) \int_{\rho}^{\infty} \frac{(1-\kappa(s))s^{\frac{5}{2}}g_{\mathrm{f}_4}(s;\lambda) F_{\lambda}(s)}{(s+1)^{\frac{3}{2}+\lambda}}ds\Big]\\
     &- \frac{(1-\kappa(\rho))\rho^{-\frac{5}{2}} (\rho-1)^{\frac{3}{2}+\lambda}} {3+2\lambda} \Big[a_{1,4}(\lambda)g_4(\rho;\lambda) \int_{\rho}^{\infty} \frac{\kappa(s)s^{\frac{5}{2}} g_2(s;\lambda)F_{\lambda}(s)}{(s+1)^{\frac{3}{2}+\lambda}}ds\\
     &\quad -a_{\mathrm{f}_{1,4}}(\lambda) g_{\mathrm{f}_4}(\rho;\lambda) \int_{\rho}^{\infty} \frac{\kappa(s)s^{\frac{5}{2}}g_{\mathrm{f}_2}(s;\lambda)F_{\lambda}(s)}{(s+1)^{\frac{3}{2}+\lambda}}ds\Big].
     \end{split}
 \end{equation}

 \begin{equation}
     \begin{split}
        \label{eq:C--Terms}
         C_{--}F_{\lambda}(\rho;\lambda)&:=+\frac{\kappa(\rho)\rho^{-\frac{5}{2}}(\rho-1)^{\frac{3}{2}+\lambda}}{3+2\lambda} \Big[a_{1,3}(\lambda)g_2(\rho;\lambda) \int_{\rho}^{\infty} \frac{(1-\kappa(s))s^{\frac{5}{2}}g_3(s;\lambda) F_{\lambda}(s)}{(s-1)^{\frac{3}{2}+\lambda}}ds\\
     &\quad - a_{\mathrm{f}_{1,3}}(\lambda)g_{\mathrm{f}_2}(\rho;\lambda) \int_{\rho}^{\infty} \frac{(1-\kappa(s))s^{\frac{5}{2}}g_{\mathrm{f}_3}(s;\lambda) F_{\lambda}(s)}{(s-1)^{\frac{3}{2}+\lambda}}ds\Big]\\
     &+\frac{(1-\kappa(\rho))\rho^{-\frac{5}{2}} (\rho-1)^{\frac{3}{2}+\lambda}} {3+2\lambda}\Big[a_{2,4}(\lambda)g_4(\rho;\lambda) \int_{\rho}^{\infty} \frac{\kappa(s)s^{\frac{5}{2}}g_1(s;\lambda)F_{\lambda}(s)}{(s-1)^{\frac{3}{2}+\lambda}}ds\\
     &\quad -a_{\mathrm{f}_{2,4}}(\lambda) g_{\mathrm{f}_4}(\rho;\lambda) \int_{\rho}^{\infty} \frac{\kappa(s)s^{\frac{5}{2}} g_{\mathrm{f}_1}(s;\lambda)F_{\lambda}(s)}{(s-1)^{\frac{3}{2}+\lambda}}ds \Big].
     \end{split}
 \end{equation}

 \begin{equation}
    \begin{split}
    \label{eq:C++Terms}
     C_{++}F_{\lambda}(\rho;\lambda)&:=-\frac{\kappa(\rho)\rho^{-\frac{5}{2}}(\rho+1)^{\frac{3}{2}+\lambda}}{3+2\lambda} \Big[a_{2,4}(\lambda)g_1(\rho;\lambda) \int_{\rho}^{\infty} \frac{(1-\kappa(s))s^{\frac{5}{2}}g_4(s;\lambda) F_{\lambda}(s)}{(s+1)^{\frac{3}{2}+\lambda}}ds\\
     &\quad - a_{\mathrm{f}_{2,4}}(\lambda)g_{\mathrm{f}_1}(\rho;\lambda) \int_{\rho}^{\infty} \frac{(1-\kappa(s))s^{\frac{5}{2}}g_{\mathrm{f}_4}(s;\lambda) F_{\lambda}(s)}{(s+1)^{\frac{3}{2}+\lambda}}ds\Big]\\
     &-\frac{(1-\kappa(\rho)) \rho^{-\frac{5}{2}} (\rho+1)^{\frac{3}{2}+\lambda}}{3+2\lambda}\Big[a_{1,3}(\lambda)g_3(\rho;\lambda) \int_{\rho}^{\infty} \frac{\kappa(s)s^{\frac{5}{2}} g_2(s;\lambda)F_{\lambda}(s)}{(s+1)^{\frac{3}{2}+\lambda}}ds\\
     &\quad -a_{\mathrm{f}_{1,3}}(\lambda)g_{\mathrm{f}_3}(\rho;\lambda) \int_{\rho}^{\infty} \frac{\kappa(s)s^{\frac{5}{2}}g_{\mathrm{f}_2}(s;\lambda)F_{\lambda}(s)}{(s+1)^{\frac{3}{2}+\lambda}} ds\Big].
     \end{split}
 \end{equation}

We shall deal with the term $F_{\lambda}(s) = (2-\lambda)f_1(\rho) +\rho f_1'(\rho)-f_2(\rho)$ in two ways by means of the following operators:
\begin{equation}
    T_{G}(\tau)f_2(\rho):= \frac{1}{2\pi i} \lim_{\varepsilon \to 0}\lim_{N \to \infty}  \int_{\varepsilon-iN}^{\varepsilon+iN} e^{\lambda \tau} G(f_2)(\rho;\lambda) d\lambda,
\end{equation}
and 
\begin{equation}
    \dot{T}_{G}(\tau)f_1(\rho):= \frac{1}{2\pi i} \lim_{\varepsilon \to 0}  \lim_{N \to \infty} \int_{\varepsilon-iN}^{\varepsilon+iN} e^{\lambda \tau} G\big((2-\lambda)f_1(\cdot) + \cdot f_1'(\cdot)\big)(\rho;\lambda) d\lambda,
\end{equation}
 where $G$ is one of the terms in the expression \eqref{eq:DifferenceOfResolvents}.


\subsection{Strichartz estimates} For $G$ in \eqref{eq:DifferenceOfResolvents}, we aim to prove the following estimates:
\begin{equation*}
    \begin{split}
        \| T_{G}(\tau)f_2\|_{L^q_{\tau}(\R_+)L^r_{\rho}(\B^c)} &\lesssim \|f_2\|_{\dot{H}^1(\R^6)},\\
        \| \dot{T}_{G}(\tau)f_1\|_{L^q_{\tau}(\R_+)L^r_{\rho}(\B^c)} &\lesssim \|f_1\|_{\dot{H}^2(\R^6)}.
    \end{split}
\end{equation*}

\label{section:StrichartzEstimates}
\subsubsection{The operator $T$}
We shall repeatedly make use of the following naive identity:
\begin{equation}
	\label{eq:ab-cd}
	a_1b_1-a_2b_2 = a_1(b_1-b_2)+(a_1-a_2)b_2
\end{equation}

We deal with the leading order terms in \eqref{eq:DifferenceOfResolvents} in the following:\\

\textbf{$A$ terms}:  We start with a preliminary result for the terms supported near $\rho=1$.
\begin{lemma}\label{lem:A asymp}
We have that 
\begin{align*}
A_\pm F_\lambda(\rho;\lambda):=\mathcal{O}(\langle \omega \rangle^{-2}) \kappa(\rho) \rho^{-\frac{5}{2}} (\rho\pm1)^{\frac{3}{2}+\lambda} \int_{\rho}^{\infty} \frac{\kappa(s)s^{\frac{5}{2}}F_{\lambda}(s)}{(s\pm1)^{\frac{3}{2}+\lambda}}\gamma_{A_\pm}(\rho,s;\lambda) ds
\end{align*}
where the functions $\gamma_{A_-}$ and $\gamma_{A_+}$ are functions on $[1,\infty)^2\times S$ that satisfy the estimates
\begin{align*}
|\partial_\rho^k\partial_s^\ell\partial_\omega^n \gamma_{A_\pm}(\rho,s;\lambda) |\lesssim_{k,\ell,n}\langle\omega\rangle^{-n}
\end{align*} 
for all $k,\ell,n \in \mathbb{N}_0$ and all $(\rho,s;\lambda)\in [1,\infty)^2\times S$.
\end{lemma}
\begin{proof}
By employing \eqref{eq:ab-cd} we obtain that
\begin{align*}
A_- F_\lambda(\rho;\lambda)&=\frac{\kappa(\rho)\rho^{-\frac{5}{2}}(\rho-1)^{\frac{3}{2}+\lambda}}{3+2\lambda} \Big[\big(g_2(\rho;\lambda)-g_{\mathrm{f}_2}(\rho;\lambda)\big) \int_{\rho}^{\infty} \frac{\kappa(s)s^{\frac{5}{2}}g_{\mathrm{f}_1}(s;\lambda)F_{\lambda}(s)}{(s-1)^{\frac{3}{2}+\lambda}} ds
    \\
    &\quad + g_2(\rho;\lambda)\int_{\rho}^{\infty} \frac{\kappa(s)s^{\frac{5}{2}}\big( g_1(s;\lambda)-g_{\mathrm{f}_1}(s;\lambda)\big)F_{\lambda}(s)}{(s-1)^{\frac{3}{2}+\lambda}}ds\Big]
\end{align*}
and the claim follows by plugging in the asymptotic expansions
\begin{align*}
	g_{j}(\rho;\lambda)=[1+(\rho-1)\widehat g_{j}(\rho;\lambda)] \text{ where }|\partial_\rho^n\partial_\omega^\ell \widehat g_{j}(\rho;\lambda)|\lesssim_{n,\ell}\langle\omega\rangle^{-1-\ell}
	\\
		g_{\mathrm{f}_j}(\rho;\lambda)=[1+(\rho-1)\widehat g_{\mathrm{f}_j}(\rho;\lambda)] \text{ where }|\partial_\rho^n\partial_\omega^\ell \widehat g_{\mathrm{f}_j}(\rho;\lambda)|\lesssim_{n,\ell}\langle\omega\rangle^{-1-\ell}
\end{align*}
for $j=1,2$. Likewise, one argues for $A_+$.
\end{proof}
In view of Lemma \ref{lem:A asymp},
we observe that it suffices to control the following for $A_\pm$, which, for the ease of notation, we again denote by $A_\pm$:
\begin{equation}
\label{eq:A_-}
A_\pm(F_{\lambda})(\rho;\lambda) :=\mathcal{O}(\langle \omega \rangle^{-2}) \kappa(\rho) \rho^{-\frac{5}{2}} (\rho\pm1)^{\frac{3}{2}+\lambda} \int_{\rho}^{\infty} \frac{\kappa(s)s^{\frac{5}{2}}F_{\lambda}(s)}{(s\pm1)^{\frac{1}{2}+\lambda}} ds.
\end{equation}


As previously mentioned, we define associated oscillatory integral operators as
\begin{equation*}
	T_{A_\pm}(\tau)f_2(\rho):=\frac{1}{2\pi} \lim_{\varepsilon \to 0}\lim_{N\to \infty} \int_{-N}^{N}e^{\tau(\varepsilon+i \omega)}\mathcal{O}(\langle \omega \rangle^{-2}) \kappa(\rho) \rho^{-\frac{5}{2}} (\rho\pm1)^{\frac{3}{2}+\varepsilon+i\omega} \int_{\rho}^{\infty} \frac{\kappa(s)s^{\frac{5}{2}}f_2(s)}{(s\pm1)^{\frac{1}{2}+\varepsilon+i\omega}} dsd\omega
\end{equation*}
and note that the above are absolutely convergent and thus well-defined. 


\begin{lemma}
\label{lemma:TApm}
The operators $T_{A_\pm}$ satisfy the estimates
\begin{align*}
\|T_{A_\pm}(\cdot)f_2\|_{L^q(\R_+)L^r(\B^c)}&\lesssim \|f_2\|_{\dot{H}^1(\R^6)}
\end{align*}
for all $q\in [2,\infty],~r\in[6,12]$ that satisfy \eqref{eq:AdmissibleExponents} and $f_2\in C^\infty_{c,rad}(\R^6)$.
\end{lemma}
\begin{proof}
We start with the bounds on $A_-$.
Using Fubini's theorem and \cite[Lemma 5.1]{DonningerWallauch}, we have
\begin{equation*}
	\begin{split}
		|T_{A_-}(\tau)f_2(\rho)| &\lesssim \Big| \kappa(\rho)\rho^{-\frac{5}{2}}(\rho-1)^{\frac{3}{2}} \int_{\rho}^{\infty} \frac{\kappa(s)s^{\frac{5}{2}}f_2(s)}{(s-1)^{\frac{1}{2}}} \langle \tau +\log\big(\frac{\rho-1}{s-1}\big)\rangle^{-2}ds\Big|.
	\end{split}
\end{equation*}
Thus,
\begin{equation*}
	\begin{split}
		\|T_{A_-}(\tau)f_2\|_{L^6(\B^c)}^6 &\lesssim \int_{1}^{\infty} \kappa(\rho)\rho^{-10}(\rho-1)^9 \Big (\int_{\rho}^{\infty} \frac{\kappa(s)s^{\frac{5}{2}}f_2(s)}{(s-1)^{\frac{1}{2}}} \langle \tau +\log\big(\frac{\rho-1}{s-1}\big)\rangle^{-2}ds\Big)^6 d\rho\\
		&\lesssim \int_{1}^{\infty} \kappa(\rho) \Big (\int_{\rho}^{\infty} \frac{\kappa(s)s^{\frac{5}{2}}f_2(s)}{(s-1)^{\frac{1}{2}}} \langle \tau +\log\big(\frac{\rho-1}{s-1}\big)\rangle^{-2}ds\Big)^6 d\rho.
	\end{split}
\end{equation*}
We change variables as: $\rho-1=e^{-x}$ and $s-1=e^y$ to obtain
\begin{equation*}
	\begin{split}
			\|T_{A_-}(\tau)f_2\|_{L^6(\B^c)}^6 &\lesssim \int_{\R}\kappa(e^{-x}+1)e^{-x}\Big(\int_{\R} \frac{\kappa(e^y+1)(e^y+1)^{\frac{5}{2}} f_2(e^y+1)}{e^{\frac{y}{2}}} \langle \tau-x-y\rangle^{-2} e^ydy\Big)^6 dx\\
			&\lesssim \int_{\R} \mathbf{1}_{[0,\infty)}(x)g(\tau-x)dx,
	\end{split}
\end{equation*}
where
\begin{equation*}
	g(\tau-x):= \Big(\int_{\R} \frac{\kappa(e^y+1)(e^y+1)^{\frac{5}{2}} f_2(e^y+1)}{e^{\frac{y}{2}}} \langle \tau-x-y\rangle^{-2} e^ydy\Big)^6.
\end{equation*}
Thus, 
\begin{equation*}
	\|T_{A_-}(\tau)f_2\|_{L^6(\B^c)} \lesssim |(\mathbf{1}_{[0,\infty)}\ast g)(\tau)|^{\frac{1}{6}},
\end{equation*}
and, consequently, applying Young's convolution inequality twice yields
\begin{equation*}
	\begin{split}
	\|T_{A_-}(\tau)f_2\|_{L^{\infty}_{\tau}(\R_+)L^6(\B^c)}& \lesssim \|\mathbf{1}_{[0,\infty)}\|_{L^{\infty}}^{\frac{1}{6}} \|g\|_{L_{\tau}^1(\R_+)}^{\frac{1}{6}}
	\\
	&\lesssim \Big( \int_1^{\infty} \big(\int_{\R} \kappa(e^y+1)(e^y+1)^{\frac{5}{2}} f_2(e^y+1)e^y \langle \tau-y\rangle^{-2}dy\big)^6 d\tau\Big)^{\frac{1}{6}}
	\\
	&\lesssim \| \kappa(e^{(\cdot)}+1)(e^{(\cdot)}+1)^{\frac{5}{2}}f_2(e^{(\cdot)}+1)e^{\frac{(\cdot)}{2}} \ast \langle \cdot\rangle^{-2}\|_{L^6(\R_+)}
	\\
	&\lesssim \|\kappa(e^{(\cdot)}+1)(e^{(\cdot)}+1)^{\frac{5}{2}}f_2(e^{(\cdot)}+1)e^{\frac{(\cdot)}{2}}\|_{L^2(\R_+)} \| \langle \cdot \rangle^{-2}\|_{L^{\frac{3}{2}}(\R_+)}
	\\
	&\lesssim \Big(\int_{\R}\kappa(e^y+1)(e^y+1)^{5}|f_2(e^y+1)|^2 e^ydy\Big)^{\frac{1}{2}}.
	\end{split}
\end{equation*}
Now, we undo the change of variables: $e^y+1=s$ to obtain that the above can be estimated by
\begin{equation*}
	\Big( \int_{\R} \kappa(s)s^5 |f_2(s)|^2 ds\Big)^{\frac{1}{2}} \lesssim \|f_2\|_{\dot{H}^1(\B^c)},
\end{equation*}
where we utilise the support property of $\kappa$ and Hardy's inequality in the last step. Next, by using the elementary inequality $\rho^{-2}(\rho\pm 1) \langle a+\log(\rho\pm 1)\rangle^{-24} \lesssim \langle a\rangle^{-24}$, for $\rho\in(1,\infty)$ and $a\in \R$ we obtain
\begin{equation*}
	\begin{split}
\|T_{A_-}(\tau)f_2\|_{L^{12}(\B^c)}^{12}& \lesssim \int_1^{\infty} \kappa(\rho)\rho^{-25}(\rho-1)^{18} \Big(\int_{\rho}^{\infty} \frac{\kappa(s)s^{\frac{5}{2}}f_2(s)}{(s-1)^{\frac{1}{2}}} \langle \tau +\log\big(\frac{\rho-1}{s-1}\big)\rangle^{-2}ds\Big)^{12} d\rho\\
&\lesssim \int_1^{\infty} \rho^{-2} d\rho \left(\int_{1}^{\infty} \frac{\kappa(s)s^{\frac{5}{2}}f_2(s)}{(s-1)^{\frac{1}{2}}} \langle \tau -\log(s-1)\rangle^{-2}ds\right)^{12} 
\\
&\lesssim \left(\int_{1}^{\infty} \frac{\kappa(s)s^{\frac{5}{2}}f_2(s)}{(s-1)^{\frac{1}{2}}} \langle \tau -\log(s-1)\rangle^{-2}ds\right)^{12} .
	\end{split}
\end{equation*}
This gives
\begin{equation*}
    \begin{split}
       \|T_{A_-}(\tau)f_2\|_{L^{12}(\B^c)} &\lesssim  \int_1^{\infty} \rho^{-2} d\rho \left(\int_{1}^{\infty} \frac{\kappa(s)s^{\frac{5}{2}}f_2(s)}{(s-1)^{\frac{1}{2}}} \langle \tau -\log(s-1)\rangle^{-2}ds\right)\\
       &=\int_{\R} \kappa(e^y+1) (e^y+1)^{\frac{5}{2}}f_2(e^y+1) e^{\frac{y}{2}} \langle \tau - y\rangle^{-2}dy\\
       &=:(p_1\ast p_2)(\tau),
    \end{split}
\end{equation*}
where 
\begin{equation*}
p_1(y) = \kappa(e^y+1) (e^y+1)^{\frac{5}{2}}f_2(e^y+1) e^{\frac{y}{2}}, \quad p_2(y):=\langle y\rangle^{-2}.
\end{equation*}
Consequently,
\begin{equation*}
\begin{split}
    \|T_{A_-}(\tau)f_2\|_{L^2(\R_+)L^{12}(\B^c)} &\lesssim \|p_1\|_{L^2} \|p_2\|_{L^1} \lesssim \|p_1\|_{L^2} \lesssim \|f_2\|_{\dot{H}^1(\B^c)}.
\end{split}
\end{equation*}
 \end{proof}

\textbf{$B$ terms}: Next, we deal with the terms supported near $\rho=\infty$.
Thus, we once more define corresponding integral operators as
\begin{equation*}
	T_{B_\pm}(\tau)f(\rho):=\frac{1}{2\pi i} \lim_{\varepsilon \to 0}\lim_{N\to \infty} \int_{\varepsilon-iN}^{\varepsilon+iN}e^{\lambda \tau}\mathcal{O}(\langle \omega \rangle^{-2}) (1-\kappa(\rho)) \rho^{-\frac{5}{2}} (\rho\pm1)^{\frac{3}{2}+\lambda} \int_{\rho}^{\infty} \frac{(1-\kappa(s))s^{\frac{5}{2}}f(s)}{(s\pm1)^{\frac{3}{2}+\lambda}} dsd\lambda,
\end{equation*}
for $\lambda =\varepsilon +i\omega$.

\begin{lemma}
    \label{lemma:TBpm}
The operators $T_{B_\pm}$ satisfy the estimates
\begin{align*}
\|T_{B_\pm}(\cdot)f_2\|_{L^q(\R_+)L^r(\B^c_1)}&\lesssim \|f_2\|_{\dot{H}^1(\R^6)}
\end{align*}
for all $q\in [2,\infty],~r\in[6,12]$  that satisfy \eqref{eq:AdmissibleExponents} and  $f_2\in C^\infty_{c,rad}(\R^6)$.
\end{lemma}
\begin{proof}

Imitating the steps we did in order to bound $T_{A_-}$, we obtain
\begin{equation*}
	|T_{B_-}(\tau)f_2| \lesssim (1-\kappa(\rho)) \rho^{-\frac{5}{2}} (\rho-1)^{\frac{3}{2}} \int_{\rho}^{\infty} \frac{(1-\kappa(s))s^{\frac{5}{2}}f_2(s)}{(s-1)^{\frac{3}{2}}} \langle \tau +\log\Big(\frac{\rho-1}{s-1}\Big)\rangle^{-2}ds,
	\end{equation*}
	and the $L^6_{\rho}(\B^c)$ norm becomes
	\begin{equation*}
		\|T_{B_-}(\tau)f_2\|_{L_{\rho}^6(\B^c)}^6 \lesssim \int_1^{\infty} (1-\kappa(\rho)) \rho^{-10} (\rho-1)^9 \Big(\int_{\rho}^{\infty} \frac{(1-\kappa(s)) s^{\frac{5}{2}} f_2(s)}{(s-1)^{\frac{3}{2}}} \langle \tau +\log\Big(\frac{\rho-1}{s-1}\Big)\rangle^{-2}ds\Big)^6 d\rho,
 	\end{equation*}
 	and we change variables $\rho-1=e^{-x}, s-1=e^y$ to obtain that
 	\begin{equation*}
 		\begin{split}
 				\|T_{B_-}(\tau)f_2\|_{L_{\rho}^6(\B^c)}^6 \lesssim 	\int_{\R} \tilde{g}(\tau-x)dx,
 		\end{split}
 	\end{equation*}
 	where 
 	\begin{equation*}
 		\tilde{g}(\tau-x) :=\Big( \int_{\R} \frac{(1-\kappa(e^y+1))(e^y+1)^{\frac{5}{2}}f_2(e^y+1)}{e^{\frac{y}{2}}} \langle \tau-x-y\rangle^{-2}dy\Big)^6.
 	\end{equation*}
 	The computations remain as before with a slight change in $g$ (which becomes $\tilde{g}$ now) and $\kappa$ changing to $1-\kappa$. We conclude briefly:
 	\begin{equation*}
 		\begin{split}
 		\| T_{B_-}(\tau)f_2\|_{L_{\tau}^{\infty}(\R_+)L_{\rho}^6(\B^c)} \lesssim \| \mb{1} \ast \tilde{g}\|_{L^{\infty}(\R)}^{\frac{1}{6}} &\lesssim \|\tilde{g}\|_{L^1(\R)}^{\frac{1}{6}}\\& =\|\tilde{g}^{\frac{1}{6}}\|_{L^6(\R)}\\
 		 &\lesssim \Big\| \frac{(1-\kappa(e^{y}+1))(e^y+1)^{\frac{5}{2}}f_2(e^y+1)}{e^\frac{y}{2}} \Big\|_{L^2(\R)} \| \langle \cdot \rangle^{-2}\|_{L^{\frac{3}{2}}(\R)}\\ &\lesssim \|f_2\|_{\dot{H}^1(\R^6)},
 		\end{split}
 	\end{equation*}
  	where we use Hardy's inequality to obtain the last display. For the $(2,12)$ pair, we have that $\| T_{B_-}(\tau)f_2\|_{L^{12}(\B^c)}^{12}$ is bounded by
    \begin{equation*}
        \begin{split}
             &\quad \int_1^{\infty} (1-\kappa(\rho))\rho^{-25} (\rho-1)^{18} \Big(\int_{\rho}^{\infty} \frac{(1-\kappa(s))s^{\frac{5}{2}}f_2(s)}{(s-1)^{\frac{3}{2}}} \langle \tau +\log(\rho-1)-\log(s-1)\rangle^{-2}ds\Big)^{12}d\rho\\
            &\lesssim \int_1^{\infty} \rho^{-7}\langle \log(\rho-1)\rangle^{24}d\rho \Big(\int_{1}^{\infty} \frac{(1-\kappa(s))s^{\frac{5}{2}}f_2(s)}{(s-1)^{\frac{3}{2}}} \langle \tau -\log(s-1)\rangle^{-2}ds\Big)^{12}\\
            &\lesssim \Big(\int_{1}^{\infty} \frac{(1-\kappa(s))s^{\frac{5}{2}}f_2(s)}{(s-1)^{\frac{3}{2}}} \langle \tau -\log(s-1)\rangle^{-2}ds\Big)^{12}\\
            &\lesssim \big((q_1 \ast q_2)(\tau)\big)^{12},
        \end{split}
    \end{equation*}
    where 
    \begin{equation*}
        q_1(y) := (1-\kappa(e^y+1)) (e^y+1)^{\frac{5}{2}}f_2(e^y+1) e^{-\frac{y}{2}}, \quad q_2(y):=\langle y\rangle^{-2}.
    \end{equation*}
    Consequently, Young's and Hardy's inequalities give
    \begin{equation*}
    \begin{split}
        \|T_{B_-}(\tau)f_2\|_{L_{\tau}^2(\R_+)L^{12}(\B^c)} &\lesssim \|q_1\|_{L^2(\R)}\|q_2\|_{L^1(\R)} \lesssim \|f_2\|_{\dot{H}^1(\B^c)}.
        \end{split}
    \end{equation*}
    Minor modifications of this strategy yield the bounds on $B_+$, and we conclude this proof.
\end{proof}

\textbf{$C$ terms}: In this section, we deal with the mixed terms, namely $C_{\star \circ}, ~\star, \circ \in \{+,-\}$. The corresponding operators, in this case,  read
\begin{equation*}
    T_{C_{\star \circ}}(\tau)f(\rho):=\frac{1}{2\pi i} \lim_{\varepsilon \to 0} \lim_{N\to \infty} \int_{\varepsilon-iN}^{\varepsilon+iN} e^{\lambda \tau} C_{\star \circ}(\lambda) f(\rho)d\lambda,
\end{equation*}
and we prove the following:
\begin{lemma}
    \label{lemma:TCTerms}
    The operators $T_{C_{\star \circ}}$ satisfy the estimate:
    \begin{equation*}
        \|T_{C_{\star \circ}}(\cdot) f_2\|_{L^q(\R_+)L^r(\B^c)} \lesssim \|f_2\|_{\dot{H}^1(\B^c)},
    \end{equation*}
    for all $f_2\in C_{c,rad}^{\infty}(\R^6)$ and all $q\in [2,\infty],~r\in[6,12]$ that satisfy \eqref{eq:AdmissibleExponents}.
    \begin{proof}
        We deal with $C_{+-}$ term; the rest being similar or potentially even better-behaved, can be dealt with using similar arguments. In the first place, after employing \eqref{eq:ab-cd}, we single out the leading order term from \eqref{eq:C+-Terms}:
        \begin{equation*}
            \frac{\big(a_{2,3}(\lambda) - a_{\mathrm{f}_{2,3}}(\lambda)\big)}{3+2\lambda} \kappa(\rho)\rho^{-\frac{5}{2}} (\rho+1)^{\frac{3}{2}+\lambda} \int_{\rho}^{\infty} \frac{(1-\kappa(s))s^{\frac{5}{2}}F_{\lambda}(s)}{(s-1)^{\frac{3}{2}+\lambda}}ds
        \end{equation*}
        Next, we use an integration by parts on the above to arrive at the following leading order terms for $C_{+-}$:
        \begin{equation}
        \label{eq:CPlusMinusLeading}
            \begin{split}
&\frac{2\big(a_{2,3}(\lambda) - a_{\mathrm{f}_{2,3}}(\lambda)\big)}{(1+2\lambda)(3+2\lambda)} \kappa(\rho)\rho^{-\frac{5}{2}} (\rho+1)^{\frac{3}{2}+\lambda}\Big[\frac{(1-\kappa(\rho))\rho^{\frac{5}{2}} F_{\lambda}(\rho)}{(\rho-1)^{\frac{1}{2}+\lambda}} + \int_{\rho}^{\infty} \frac{((1-\kappa(s))s^{\frac{5}{2}}F_{\lambda}(s))'}{(s-1)^{\frac{1}{2}+\lambda}}ds\Big]. 
                \end{split}
        \end{equation}
        The operator $T$ corresponding to the first term can be bounded by:
        \begin{equation*}
            \begin{split}
                &|\kappa(\rho)(1-\kappa(\rho)) f_2(\rho)(\rho+1)^{\frac{3}{2}}(\rho-1)^{-\frac{1}{2}} \langle \tau +\log(\rho+1)-\log(\rho-1)\rangle^{-2}| \lesssim (1-\kappa(\rho))|f_2(\rho)|\langle \tau \rangle^{-2},
            \end{split}
        \end{equation*}
        where we use that the support of $\kappa (1-\kappa)$ is bounded and away from $1$. The $L^{\infty}_{\tau}(\R_+)L^6(\B^c)$ norm of the above can be bounded by $\|f_2\|_{\dot{H}^1(\B^c)}$ using $|f_2(\rho)| \lesssim \rho^{-2}\|f_2\|_{\dot{H}^1(\B^c)}$ on the compact support of $\kappa (1-\kappa)$. The same also enables to control the $(2,12)$ pair.
        
        The second term in \eqref{eq:CPlusMinusLeading} leads to the following bound for the operator $T$:
        \begin{equation*}
        \begin{split}
            &\Big|\kappa(\rho)\rho^{-\frac{5}{2}}(\rho+1)^{\frac{3}{2}} \int_{\rho}^{\infty} \frac{(1-\kappa(s))s^{\frac{5}{2}} f_2'(s)}{(s-1)^{\frac{1}{2}}} \langle \tau+\log(\rho+1)-\log(s-1)\rangle^{-2}ds\Big|\\
            &\lesssim \Big|\kappa(\rho)\rho^{-\frac{5}{2}}(\rho+1)^{\frac{3}{2}} \int_{\rho}^{\infty} \frac{(1-\kappa(s))s^{\frac{5}{2}} f_2'(s)}{(s-1)^{\frac{1}{2}}} \langle \tau-\log(s-1)\rangle^{-2}ds\Big|
        \end{split}
        \end{equation*}
        Owing to the support of $\kappa$ and the change of variable $s-1=e^y$, the $L^{6}(\B^c)$ norm of the above is estimated by
        \begin{equation*}
            \int_{\R} \frac{(1-\kappa(e^y+1))(e^y+1)^{\frac{5}{2}} f_2'(e^y+1)e^y}{e^\frac{y}{2}}\langle \tau-y\rangle^{-2}dy :=p_1\ast p_2(\tau),
        \end{equation*}
        where $p_1(y):=\frac{(1-\kappa(e^y+1))(e^y+1)^{\frac{5}{2}} f_2'(e^y+1)e^y}{e^{\frac{y}{2}}}$, and $p_2(y):=\langle y\rangle^{-2}$. Consequently, the $L^{\infty}(\R_+)L^6(\B^c)$ norm is bounded by
        \begin{equation*}
            \|p_1\ast p_2\|_{L^{\infty}(\R_+)} \lesssim \|p_1\|_{L^2} \|p_2\|_{L^2} \lesssim \|f_2\|_{\dot{H}^1(\B^c)}.
        \end{equation*}
        Furthermore, for the $(2,12)$ endpoint, we have that the $L^{12}(\B^c)$ norm is still given by $(p_1\ast p_2)(\tau)$, which further implies for the $L^2_{\tau}(\R_+)$ norm that:
        \begin{equation*}
              \|p_1 \ast p_2\|_{L^2_{\tau}(\R_+)} \lesssim \|p_1\|_{L^2} \|p_2\|_{L^1} \lesssim \|f_2\|_{\dot{H}^1(\B^c)}.  
        \end{equation*}
    \end{proof}
\end{lemma}

\subsubsection{The operator $\dot{T}$}
\label{subsection:TDot}
 Recall that
\begin{equation}
	\label{eq:TDot}
	\dot{T}_{G}(\tau)(f)(\rho) = \frac{1}{2\pi i} \lim_{\varepsilon \to 0} \lim_{N\to \infty} \int_{\varepsilon -iN}^{\varepsilon +iN} e^{\lambda \tau} G\big((2-\lambda)f_1(\cdot) +(\cdot) f'_1(\cdot)\big)(\rho;\lambda) d\lambda,
\end{equation}
where  $G$ is one of the summand in the difference of the resolvents \eqref{eq:DifferenceOfResolvents}. As in the previous section, we shall deal with the leading-order terms previously obtained. First of all, we observe that we can write 
\begin{equation*}
	f(s;\lambda):=(2-\lambda)f_1(s) + sf_1'(s) = -\Big(\frac{1}{2} +\lambda \Big) f_1(s) +sf_1'(s)+\frac{5}{2}f_1(s).
\end{equation*}
After performing an integration by parts and a simple manipulation, we obtain, for $\gamma =\gamma(\cdot) \in\{(1-\kappa(\cdot)), \kappa(\cdot)\}$:\footnote{We skip the terms with derivatives of the cutoff functions, since they are well-behaved.}
\begin{equation}
	\begin{split}
		\label{eq:IntegrationByPartsFinal}
		&\int_{\rho}^{\infty} \frac{\gamma(s)s^{\frac{5}{2}} ((2-\lambda)f_1(s) +sf_1'(s))}{(s\pm1)^{\frac{3}{2}+\lambda}} ds\\
		&\quad= -\frac{\gamma(\rho)\rho^{\frac{5}{2}} f_1(\rho)}{(\rho\pm1)^{\frac{1}{2}+\lambda}} \mp\frac{5}{2} \int_{\rho}^{\infty} \frac{\gamma(s)s^{\frac{3}{2}}f_1(s)}{(s\pm1)^{\frac{3}{2}+\lambda}} ds \mp\int_{\rho}^{\infty} \frac{\gamma(s)s^{\frac{5}{2}}f_1'(s)}{(s\pm1)^{\frac{3}{2}+\lambda}}ds.
	\end{split}
\end{equation}
\textbf{$A$ terms}: We prove the following lemma:
\begin{lemma}
    \label{lemma:TDotA}
    Let $q\in [2,\infty],r\in[6,12]$ satisfy \eqref{eq:AdmissibleExponents}, $i \in \{+,-\}$ and $f_1\in C_{c,rad}^{\infty}(\R^6)$.  Then, the following estimate holds
 	\begin{equation*}
 		\|\dot{T}_{A_{i}}(\cdot)f_1\|_{L^q(\R_+)L^r(\B^c)} \lesssim \|f_1\|_{\dot{H}^2(\B^c)}.
 	\end{equation*}
    \begin{proof}
        For brevity, we note that the proof follows along the lines of Lemma \ref{lemma:TDotB} (see below) by replacing $1-\kappa(\cdot)$ by $\kappa(\cdot)$, tweaking the kernel slightly and noting that the decay is sufficient.
    \end{proof}
\end{lemma}
\vspace{0.5cm}

\textbf{$B$ terms}:  Set $B=B_++B_-$We prove the following:
\begin{lemma}
 	\label{lemma:TDotB}
 	Let $q\in [2,\infty],r\in[6,12]$ satisfy \eqref{eq:AdmissibleExponents} and $f_1\in C_{c,rad}^{\infty}(\R^6)$.  Then, the following estimate holds
 	\begin{equation*}
 		\|\dot{T}_B(\cdot)f_1\|_{L^q(\R_+)L^r(\B^c)} \lesssim \|f_1\|_{\dot{H}^2(\B^c)}.
 	\end{equation*}
 \begin{proof}
 We apply the integration by parts Eq.~\eqref{eq:IntegrationByPartsFinal}
 to each term in $B=B_++B_-$. Since all Green kernels in the ODE construciton of the involved function vanish at the diagonal, 
$s=\rho$, we have
\[
    \gamma_+(\rho,\rho;\lambda)+\gamma_-(\rho,\rho;\lambda)=0.
\]
Hence the two boundary terms satisfy
\[
    -\frac{(1-\kappa(\rho))^2 f_1(\rho)}{(3+2\lambda)^2}
    \sum_{\sigma=\pm1}(\rho+\sigma)\gamma_\sigma(\rho,\rho;\lambda)
    =
    -2\frac{(1-\kappa(\rho))^2 f_1(\rho)}{(3+2\lambda)^2}
    \gamma_+(\rho,\rho;\lambda)=O(\rho^0\langle \omega \rangle^{-2}) (1-\kappa(\rho))^2 f_1(\rho).
\]
Now, for $B_+$ the leading order parts of the integral terms are then given by
	\begin{equation}\label{eq:LeadingOrderTermAtInfinityTDot}
	\begin{split}
		B_+((2-\lambda)f_1(\cdot)+(\cdot)f_1'(\cdot))(\rho;\lambda) &= -\O(\langle \omega \rangle^{-2}) (1-\kappa(\rho)) \rho^{-\frac{7}{2}} (\rho+1)^{\frac{3}{2}+\lambda} \\
		&\quad \times \Big[ \frac{5}{2}\int_{\rho}^{\infty} \frac{(1-\kappa(s))s^{\frac{3}{2}}f_1(s)}{(s+1)^{\frac{3}{2}+\lambda}} ds +\int_{\rho}^{\infty} \frac{(1-\kappa(s))s^{\frac{5}{2}} f_1'(s)}{(s+1)^{\frac{3}{2}+\lambda}}ds \Big].
	\end{split}
\end{equation}
Now the operator $\dot{T}_{B_+^1}$ satisfies
\begin{equation*}
	|\dot{T}_{B_+^1}(\tau)f_1(\rho)|:= \Big| \frac{1}{2\pi i}\lim_{\varepsilon \to 0}\lim_{N \to \infty} \int_{\varepsilon -iN}^{\varepsilon +iN} e^{\lambda \tau} \O(\rho^0\langle \omega \rangle^{-2}) (1-\kappa(\rho))^2 f_1(\rho) d\lambda \Big| \lesssim (1-\kappa(\rho))^2 |f_1(\rho)| \langle \tau \rangle^{-2}.
\end{equation*}
The $L^{\infty}L^6$ norm of the above can be controlled by the $\dot{H}^2(\R^6)$ norm of $f_1$ via Sobolev embedding $\dot{H}^2(\R^6) \hookrightarrow L^6(\R^6)$.  For the other endpoint, namely $L_{\tau}^{2}(\R_+)L^{12}(\B^c)$, we first note that \cite[Proposition B.1]{GlogicWaveMaps} gives for $g \in C_{c,rad}^{\infty}(\R^6)$
\begin{equation}
	\label{eq:LInfinityHardy}
	\| |(\cdot)| g\|_{L^{\infty}(\R^6)} \lesssim \|g\|_{\dot{H}^2(\R^6)}
\end{equation}
which on interpolating with the Sobolev estimate
\begin{equation}
	\| g \|_{L^6(\R^6)} \lesssim \|g\|_{\dot{H}^2(\R^6)}
\end{equation}
implies
\begin{equation}
	\label{eq:L12InterpolationResult}
	\| |(\cdot)|^{\frac{1}{2}} g\|_{L^{12}(\R^6)} \lesssim \|g\|_{\dot{H}^2(\R^6)}.
\end{equation}
The above suffices for us to conclude that
\begin{equation*}
	\| \dot{T}_{B_+^1}(\tau)f_1(\rho)\|_{L_{\tau}^{2}(\R_+)L^{12}(\B^c)} \lesssim \| f_1\|_{\dot{H}^2(\R^6)}.
\end{equation*}
For the second term in \eqref{eq:LeadingOrderTermAtInfinityTDot} the resulting operator can be bounded by
\begin{equation}
	\label{eq:TDotI2}
	|\dot{T}_{B_+^2}(\tau)f(\rho)| \lesssim  (1-\kappa(\rho))\rho^{-\frac{5}{2}} (\rho+1)^{\frac{3}{2}} \int_{\rho}^{\infty} \frac{(1-\kappa(s))s^{\frac{3}{2}} f_1(s)}{(s+1)^{\frac{3}{2}}} \langle \tau +\log(\rho+1)-\log(s+1)\rangle^{-2}ds. 
\end{equation}

This gives, by a change variables: $\rho+1 = e^{-x}, s+1 = e^y$
\begin{equation*}
	\begin{split}
		&\quad \|\dot{T}_{B_+^2}(\tau)f(\cdot)\|_{L^6(\B^c)}\\
        &\lesssim \Big( \int_{1}^{\infty} (1-\kappa(\rho))\rho^{-1} \big( \int_{\rho}^{\infty} \frac{(1-\kappa(s))s^{\frac{3}{2}}f_1(s)}{(s+1)^{\frac{3}{2}} }\langle \tau + \log(\rho+1)-\log(s+1)\rangle^{-2}ds \big)^6 d\rho \Big)^{\frac{1}{6}},\\
		& \lesssim \Big( \int_{\R} \mb{1}_{(-\infty,0)} (x) \big( \int_{\R} \frac{(1-\kappa(e^y-1))(e^y-1)^{\frac{3}{2}} f_1(e^y-1)\langle \tau -x-y\rangle^{-2}}{e^{\frac{y}{2}}} dy \big)^6 dx \Big)^{\frac{1}{6}}.
	\end{split}
\end{equation*}
As a result,
\begin{equation*}
	\begin{split}
		\|\dot{T}_{B_+^2}(\tau)f(\cdot)\|_{L_{\tau}^{\infty}(\R_+)L^6(\B^c)} &\lesssim \Big \| \mb{1}_{(-\infty,0)} \ast \Big(\int_{\R} \frac{(1-\kappa(e^y-1))(e^y-1)^{\frac{3}{2}} f_1(e^y-1)\langle \tau -y\rangle^{-2}}{e^{\frac{y}{2}}} dy \Big)^6 \Big\|_{L^{\infty}_{\tau}}^{\frac{1}{6}}\\
		&\lesssim \|\mb{1}_{(-\infty,0)} \|_{L^{\infty}}^{\frac{1}{6}} \Big \| \int_{\R} \frac{(1-\kappa(e^y-1))(e^y-1)^{\frac{3}{2}} f_1(e^y-1)\langle \tau-y\rangle^{-2}}{e^{\frac{y}{2}}} dy \Big \|_{L^6}\\
		&\lesssim \Big\| \frac{(1-\kappa(e^{(\cdot)}-1))(e^{(\cdot)}-1)^{\frac{3}{2}} f_1(e^{(\cdot)}-1)}{e^{\frac{(\cdot)}{2}}} \ast \langle (\cdot) \rangle^{-2} \Big\|_{L^6} \\
		&\lesssim \|f_1\|_{\dot{H}^2(\R^6)}.
	\end{split}
\end{equation*}
The other endpoint is controlled as follows:
\begin{equation*}
	\begin{split}
		&\|\dot{T}_{B_+^2}(\tau)f\|_{L_{\tau}^{2}(\R_+)L^{12}(\B^c)}\\
		 &\lesssim \Big \| \Big( \int_1^{\infty} (1-\kappa(\rho)) \rho^{-7} \big( \int_{\rho}^{\infty} \frac{(1-\kappa(s))s^{\frac{3}{2}} f_1(s)\langle \tau + \log(\rho+1)-\log(s+1)\rangle^{-2}}{(s+1)^{\frac{3}{2}}}  ds \big)^{12} d\rho \Big)^{\frac{1}{12}} \Big\|_{L^2_{\tau}} \\
		&\lesssim  \Big \| \Big( \int_1^{\infty} (1-\kappa(\rho)) \rho^{-6} \big( \int_{\rho}^{\infty} \frac{(1-\kappa(s))s^{\frac{3}{2}} f_1(s)\langle \tau -\log(s+1)\rangle^{-2}}{(s+1)^{\frac{3}{2}}}  ds \big)^{12} d\rho \Big)^{\frac{1}{12}} \Big\|_{L_{\tau}^2}\\
		&\lesssim \Big\|\int_{\R} \frac{(1-\kappa(s))s^{\frac{3}{2}} f_1(s)}{(s+1)^{\frac{3}{2}}} \langle \tau-\log(s+1) \rangle^{-2} ds \Big\|_{L_{\tau}^2}.
	\end{split}
\end{equation*}
As before, we perform a change of variables: $s+1=e^y$:
\begin{equation*}
	\|\dot{T}_{B_+^2}(\tau)f(\cdot)\|_{L_{\tau}^{2}(\R_+)L^{12}(\B^c)} \lesssim  \Big\| \frac{(1-\kappa(e^{(\cdot)}-1))(e^{(\cdot)}-1)^{\frac{3}{2}} f_1(e^{(\cdot)}-1)}{e^{\frac{(\cdot)}{2}}} \ast \langle (\cdot) \rangle^{-2} \Big\|_{L^2} 
	\lesssim \|f_1\|_{\dot{H}^2(\R^6)}.
\end{equation*}

The third term in \eqref{eq:LeadingOrderTermAtInfinityTDot} is given by:
\begin{equation}
	\label{eq:I3}
	B_+^3(f_1)(\rho;\lambda):=\O(\langle \omega \rangle^{-2} ) (1-\kappa(\rho)) \rho^{-\frac{5}{2}} (1+\rho)^{\frac{3}{2}+\lambda} \int_{\rho}^{\infty} \frac{(1-\kappa(s))s^{\frac{5}{2}} f_1'(s)}{(s+1)^{\frac{3}{2}+\lambda}} ds.
\end{equation}
This is a form very similar to what we have dealt with before, and we omit the details.\\


We recall the leading order term in $B_-$ is given by:
\begin{equation*}
	B_-(f)(\rho;\lambda):=\O(\langle \omega \rangle^{-2}) (1-\kappa(\rho))\rho^{-\frac{7}{2}} (\rho-1)^{\frac{3}{2}+\lambda} \int_{\rho}^{\infty} \frac{(1-\kappa(s))s^{\frac{5}{2}}f(s)}{(s-1)^{\frac{3}{2}+\lambda}}ds,
\end{equation*}
which using \eqref{eq:IntegrationByPartsFinal} boils down to
\begin{equation}
\label{eq:IByPBMinus}
	\begin{split}
		&\O(\langle \omega \rangle^{-2})(1-\kappa(\rho)) \rho^{-\frac{7}{2}}(\rho-1)^{\frac{3}{2}+\lambda}\\
		&\quad \times \Big[\frac{(1-\kappa(\rho))\rho^{\frac{5}{2}}f_1(\rho)}{(\rho-1)^{\frac{1}{2}+\lambda}} +\frac{5}{2}\int_{\rho}^{\infty} \frac{(1-\kappa(s))s^{\frac{3}{2}}f_1(s)}{(s-1)^{\frac{3}{2}+\lambda}} ds + \int_{\rho}^{\infty} \frac{(1-\kappa(s))s^{\frac{5}{2}}f_1'(s)}{(s-1)^{\frac{3}{2}+\lambda}}\Big].
			\end{split}
\end{equation}
The first term, which is similar to the first term of $B_+$ has been treated beforehand, and we consider the second term, the estimation of the third term being similar. For $\lambda =\varepsilon +i\omega$, the resulting operator is given by
\begin{equation*}
	\begin{split}
		|\dot{T}_{B_-^2}(\tau)f(\rho)|&\lesssim \Big|\frac{1}{2\pi i} \lim_{\varepsilon \to 0}\lim_{N \to \infty} \int_{\varepsilon-iN}^{\varepsilon+iN} e^{\lambda \tau} \O(\langle \omega \rangle^{-2})(1-\kappa(\rho)) \rho^{-\frac{5}{2}} (\rho-1)^{\frac{3}{2}+\lambda} \int_{\rho}^{\infty} \frac{(1-\kappa(s))s^{\frac{3}{2}}f_1(s)}{(s-1)^{\frac{3}{2}+\lambda}} dsd\lambda\Big|\\
		&\lesssim \Big| (1-\kappa(\rho)) \rho^{-\frac{5}{2}}(\rho-1)^{\frac{3}{2}} \int_{\rho}^{\infty} \frac{(1-\kappa(s))s^{\frac{3}{2}}f_1(s)}{(s-1)^{\frac{3}{2}}} \langle \tau+\log(\rho-1)-\log(s-1)\rangle^{-2}ds\Big|.
	\end{split}
\end{equation*}

Using the same change of variables as before, namely, $s-1=e^y, \rho-1=e^{-x}$, we obtain 
\begin{equation*}
	\begin{split}
&\quad 	\|\dot{T}_{B_-^2}(\tau)f(\rho)\|_{L^6_{\rho}(\B^c)}^6\\
	 &\lesssim \int_1^{\infty} (1-\kappa(\rho))\rho^{-10}(\rho-1)^9\Big(\int_{\rho}^{\infty} \frac{(1-\kappa(s))s^{\frac{3}{2}}f_1(s)}{(s-1)^{\frac{3}{2}}} \langle \tau+\log(\rho-1)-\log(s-1)\rangle^{-2}ds\Big)^6 d\rho\\
	&\lesssim \int_{\R} \mb{1}_{(-\infty,1]}(x)\Big(\int_{\R}\frac{(1-\kappa(e^y+1))(e^y+1)^{\frac{3}{2}} f_1(e^y+1)\langle \tau-x-y\rangle^{-2}}{e^{\frac{y}{2}}}dy\Big)^6 dx\\
	&:=h_1\ast h_2,
	\end{split}
\end{equation*}
where
\begin{equation*}
	\begin{split}
		h_1(x):=\mb{1}_{(-\infty,1]}(x), \quad h_2(\tau):=\Big(\int_{\R}\frac{(1-\kappa(e^y+1))(e^y+1)^{\frac{3}{2}} f_1(e^y+1)\langle \tau-y\rangle^{-2}}{e^{\frac{y}{2}}}dy\Big)^6.
		\end{split}
\end{equation*}
This gives
\begin{equation*}
	\begin{split}
			\|\dot{T}_{B_-^2}(\tau)f(\rho)\|_{L^{\infty}_{\tau}L^6_{\rho}(\B^c)} \lesssim \|h_1 \ast h_2\|_{L_{\tau}^{\infty}(\R_+)}^{\frac{1}{6}} \lesssim \|h_1\|_{L_{\tau}^{\infty}(\R)}^{\frac{1}{6}} \|h_2\|_{L_{\tau}^{1}(\R)}^{\frac{1}{6}} \lesssim \|h_2\|_{L_{\tau}^{1}(\R)}^{\frac{1}{6}} \lesssim \|\tilde{h}\|_{L_{\tau}^{6}(\R)},
	\end{split}
\end{equation*}
where $\tilde{h}:=h_2^{\frac{1}{6}} =\tilde{h}_1 \ast \tilde{h}_2$ and
\begin{equation*}
	\tilde{h}_1(y) = \frac{(1-\kappa(e^y+1))(e^y+1)^{\frac{3}{2}} f_1(e^y+1)}{e^{\frac{y}{2}}}, \quad \tilde{h}_2(y):= \langle y\rangle^{-2}.
\end{equation*}
Hence, using Young's convolution inequality and Hardy's inequality, we get
\begin{equation*}
	\|\dot{T}_{B_-^2}(\tau)f(\rho)\|_{L^{\infty}_{\tau}L^6_{\rho}(\B^c)} \lesssim \|\tilde{h}_1\|_{L^2} \|\tilde{h}_2\|_{L^{\frac{3}{2}}} \lesssim \|\tilde{h}_1\|_{L^2} \lesssim \|f_1\|_{\dot{H}^2(\B^c)}.
\end{equation*}
The norm for the endpoint $(2,12)$ is controlled using the same change of variables as before:
\begin{equation*}
	\begin{split}
		&\quad \|\dot{T}_{B_-^2}(\tau)f(\rho)\|_{L^{12}_{\rho}(\B^c)}^{12}\\
		&\lesssim \int_1^{\infty} (1-\kappa(\rho))^{12} \rho^{-25}(\rho-1)^{18} \Big(\int_{\rho}^{\infty} \frac{(1-\kappa(s))s^{\frac{3}{2}}f_1(s)}{(s-1)^{\frac{3}{2}}} \langle \tau+\log(\rho-1)-\log(s-1)\rangle^{-2}ds\Big)^{12}d\rho\\
		&\lesssim \big((p_1\ast p_2)(\tau)\big)^{12}
	\end{split}
\end{equation*}
where
\begin{equation*}
	\begin{split}
		p_1(y):=\tilde{h}_1(y)=(1-\kappa(e^y+1))(e^y+1)^{\frac{3}{2}} f_1(e^y+1) e^{-\frac{y}{2}}, \quad p_2(y):= \langle y \rangle^{-2}.
	\end{split}
	\end{equation*}
Once again, we conclude
\begin{equation*}
	\begin{split}
		\|\dot{T}_{B_-^2}(\tau)f(\rho)\|_{L^2_{\tau}(\R_+)L^{12}_{\rho}(\B^c)} \lesssim \|p_1\ast p_2\|_{L^2_{\tau}(\R_+)} \lesssim \|p_1\|_{L^{2}}  \|p_2\|_{L^1}\lesssim \|f_1\|_{\dot{H}^2(\R^6)}.
	\end{split}
\end{equation*}
    \end{proof}
    \end{lemma}

\textbf{$C$ terms}: The corresponding result for $C_{\star\circ}$ is as follows:
\begin{lemma}
    \label{lemma:TDotC}
    Let $q\in [2,\infty],~r\in[6,12]$ satisfy \eqref{eq:AdmissibleExponents}.  Then, the following estimate
 	\begin{equation*}
 		\|\dot{T}_{C_{\star\circ}}(\cdot)f_1\|_{L^q(\R_+)L^r(\B^c)} \lesssim \|f_1\|_{\dot{H}^2(\R^6)},
 	\end{equation*}
    holds for all $f_1 \in C_{c,rad}^{\infty}(\R^6)$.
    \begin{proof}
        As before, we treat the leading order term $C_{+-}$ as given in \eqref{eq:CPlusMinusLeading}, which in this case reads (when the derivative falls on $F_{\lambda} = f$):
                \begin{equation*}
            \O(\langle \omega \rangle^{-4}) \kappa(\rho)\rho^{-\frac{5}{2}} (\rho+1)^{\frac{3}{2}+\lambda} \int_{\rho}^{\infty} \frac{(1-\kappa(s))s^{\frac{5}{2}} \big((3-\lambda)f_1'(s)+sf_1''(s)\big)}{(s-1)^{\frac{1}{2}+\lambda}} ds.
        \end{equation*}
        We write
        \begin{equation*}
            (3-\lambda)f_1'(s)+sf_1''(s)=\Big(\frac{1}{2}-\lambda\Big) f_1'(s)+ \frac{5}{2}f_1'(s) +sf_1''(s),
        \end{equation*}
        and use integration by parts to obtain the following:
        \begin{equation}
            \begin{split}
                \label{eq:InytegrationByPartsForC}
                &\int_{\rho}^{\infty} \frac{(1-\kappa(s))s^{\frac{5}{2}} \big((3-\lambda)f_1'(s)+sf_1''(s)\big)}{(s-1)^{\frac{1}{2}+\lambda}} ds\\
                &\quad= -(1-\kappa(\rho)) \rho^{\frac{5}{2}} (\rho-1)^{\frac{1}{2}-\lambda} f_1'(\rho) + \frac{5}{2}\int_{\rho}^{\infty} \frac{(1-\kappa(s))s^{\frac{3}{2}}f_1'(s)}{(s-1)^{\frac{1}{2}+\lambda}}ds +\int_{\rho}^{\infty} \frac{(1-\kappa(s))s^{\frac{5}{2}}f_1''(s)}{(s-1)^{\frac{1}{2}+\lambda}}ds.
            \end{split}
    \end{equation}
    The operator resulting from the third term in the above display can be bounded by:
    \begin{equation*}
    \begin{split}
        &\kappa(\rho)\rho^{-\frac{5}{2}} (\rho+1)^{\frac{3}{2}} \int_{\rho}^{\infty}\frac{(1-\kappa(s))s^{\frac{5}{2}}f_1''(s)}{(s-1)^{\frac{1}{2}}} \langle \tau +\log(\rho+1)-\log(s-1)\rangle^{-2}ds\\
        &\lesssim \kappa(\rho)\rho^{-\frac{5}{2}} (\rho+1)^{\frac{3}{2}} \int_{\rho}^{\infty}\frac{(1-\kappa(s))s^{\frac{5}{2}}f_1''(s)}{(s-1)^{\frac{1}{2}}} \langle \tau -\log(s-1)\rangle^{-2}ds,
        \end{split}
    \end{equation*}
    and the $L_{\rho}^6(\B^c)$ norm can be estimated by
    \begin{equation*}
    \begin{split}
        & \quad \int_1^{\infty} \frac{(1-\kappa(s))s^{\frac{5}{2}}f_1''(s)}{(s-1)^{\frac{1}{2}}} \langle \tau -\log(s-1)\rangle^{-2}ds\\
        &\lesssim \int_{\R} (1-\kappa(e^y+1))(e^y+1)^{\frac{5}{2}} f_1''(e^y+1)\langle \tau-y\rangle^{-2} e^{\frac{y}{2}}dy,
        \end{split}
    \end{equation*}
    which consequently gives that the $L^{\infty}_{\tau}(\R_+)L^6_{\rho}(\B^c)$ norm can be bounded by $\| h_1 \ast h_2\|_{L^{\infty}}$, where 
    \begin{equation*}
        h_1(y) :=  (1-\kappa(e^y+1))(e^y+1)^{\frac{5}{2}} f_1''(e^y+1) e^{\frac{y}{2}}, \quad h_2(y):=\langle y\rangle^{-2}.
    \end{equation*}
    Clearly, by undoing the change of variables and Hardy's inequality, we obtain that the above can be bounded by $\|f_1\|_{\dot{H}^2(\R^6)}$. For the $(2,12)$ endpoint, we have that the $L^{12}(\B^c)$ norm can be estimated by
    \begin{equation*}
        \begin{split}
            &\Big[\int_{\R} \kappa(\rho) \rho^{-25} (\rho+1)^{18} \Big( \int_{\rho}^{\infty} \frac{(1-\kappa(s))s^{\frac{5}{2}} f_1''(s)}{(s-1)^{\frac{1}{2}}} \langle \tau -\log(s-1)\rangle^{-2}ds\Big)^{12} d\rho \Big]^{\frac{1}{12}}\\
            &\lesssim \int_1^{\infty} \frac{(1-\kappa(s))s^{\frac{5}{2}} f_1''(s)}{(s-1)^{\frac{1}{2}}} \langle \tau-\log(s-1)\rangle^{-2}ds\\
            &\lesssim \int_{\R} (1-\kappa(e^y+1)) (e^y+1)^{\frac{5}{2}} e^{\frac{y}{2}} f_1''(e^y+1) \langle \tau-y\rangle^{-2}dy\\
            &=:(w_1 \ast w_2)(\tau),
        \end{split}
    \end{equation*}
    where
    \begin{equation*}
        w_1(y):= (1-\kappa(e^y+1)) (e^y+1)^{\frac{5}{2}} e^{\frac{y}{2}} f_1''(e^y+1), \quad w_2(y):=\langle y \rangle^{-2}.
    \end{equation*}
    So, the $L^2(\R_+)$ norm of the above can be bounded by
    \begin{equation*}
        \lesssim \|w_1\ast w_2\|_{L^2(\R_+)} \lesssim \|w_1\|_{L^2} \|w_2\|_{L^1} \lesssim \|f_1\|_{\dot{H}^2(\R^6)}.
    \end{equation*}
    The rest of the terms employ similar techniques, and we skip the details.
    \end{proof}
\end{lemma}

\subsection{Strichartz estimates for second derivatives}
In this section, we obtain better control of the semigroup by estimating the semigroup in $L_{\tau}^{\infty}(\R_+)\dot{H}^2(\R^6)$. 
For $\rho \in [1,\infty)$, we have
\begin{equation}
	\begin{split}
		\label{eq:SecondDerivativeOfResolvents}
		&(\mathcal{R}(F_{\lambda}) -\mathcal{R}_{\mathrm{f}}(F_{\lambda}))''(\rho;\lambda) \\
        &\quad = \big(\tilde{u}_1''(\rho;\lambda) -\tilde{u}_{\mathrm{f}_1}''(\rho;\lambda)\big) \int_{\rho}^{\infty} \frac{s^5\tilde{u}_2(s;\lambda)F_{\lambda}(s)}{(s^2-1)^{\frac{3}{2}+\lambda}}ds +\tilde{u}_{\mathrm{f}_1}''(\rho;\lambda) \int_{\rho}^{\infty} \frac{s^5\big(\tilde{u}_2(s;\lambda) - \tilde{u}_{\mathrm{f}_2}(s;\lambda)\big)F_{\lambda}(s)}{(s^2-1)^{\frac{3}{2}+\lambda}}ds\\
        &\quad \quad +\big(\tilde{u}_{\mathrm{f}_2}''(\rho;\lambda) - \tilde{u}_2''(\rho;\lambda)\big) \int_{\rho}^{\infty} \frac{s^5\tilde{u}_{\mathrm{f}_1}(s;\lambda)F_{\lambda}(s)}{(s^2-1)^{\frac{3}{2}+\lambda}}ds + \tilde{u}_2''(\rho;\lambda) \int_{\rho}^{\infty} \frac{s^5\big(\tilde{u}_{\mathrm{f}_1}(s;\lambda)- \tilde{u}_1(s;\lambda)\big)F_{\lambda}(s)}{(s^2-1)^{\frac{3}{2}+\lambda}}ds
	\end{split}
\end{equation}

since $W(\tilde{u}_1(\rho;\lambda),\tilde{u}_2(\rho;\lambda)) = W(\tilde{u}_{\mathrm{f}_1}(\rho;\lambda), \tilde{u}_{\mathrm{f}_2}(\rho;\lambda)) =: W(\rho;\lambda) = \rho^{-5}(\rho^2-1)^{\frac{1}{2}+\lambda}$.

 
 \subsubsection{Identifying the leading order term-an example} Writing out the complete expression for the above is not just tedious but also obscures relevant information.  Alternatively, we choose to deal with the derivative of the leading order terms for the first derivative, which we outline below. Spelling out the functions $\tilde{u}_i, i=1,2$ reveals that we can have more than one leading order term. However, it is sufficient to estimate one of them since the treatment of the others will be similar. Here, we elucidate how and why we refer to the terms in the following section as leading-order terms. Consider (one of) the leading order terms in $C_{+-}(F_{\lambda})(\rho;\lambda)$:
\begin{equation*}
	\begin{split}
	C_{+-}(F_{\lambda})(\rho;\lambda)& :=\frac{a_{2,3}(\lambda) - a_{\mathrm{f}_{2,3}}(\lambda)}{3+2\lambda} \kappa(\rho)\rho^{-\frac{5}{2}} (\rho+1)^{\frac{3}{2}+\lambda} \int_{\rho}^{\infty} \frac{(1-\kappa(s))s^{\frac{5}{2}}F_{\lambda}(s)}{(s-1)^{\frac{3}{2}+\lambda}}ds\\
	&\quad + \frac{a_{\mathrm{f}_{2,3}}(\lambda) - a_{2,3}(\lambda)}{3+2\lambda} (1-\kappa(\rho))\rho^{-\frac{5}{2}} (\rho+1)^{\frac{3}{2}+\lambda} \int_{\rho}^{\infty} \frac{\kappa(s)s^{\frac{5}{2}}F_{\lambda}(s)}{(s-1)^{\frac{3}{2}+\lambda}} ds.
	\end{split}
\end{equation*}
On taking the derivative of the above terms, we obtain that the boundary terms annihilate each other, and we have
\begin{equation}
	\begin{split}
		C'_{+-}(F_{\lambda})(\rho;\lambda)&:=-\frac{5}{2}\frac{a_{2,3}(\lambda) - a_{\mathrm{f}_{2,3}}(\lambda)}{3+2\lambda} \kappa(\rho)\rho^{-\frac{7}{2}} (\rho+1)^{\frac{3}{2}+\lambda} \int_{\rho}^{\infty} \frac{(1-\kappa(s))s^{\frac{5}{2}}F_{\lambda}(s)}{(s-1)^{\frac{3}{2}+\lambda}} ds\\
		&\quad -\frac{5}{2}\frac{a_{\mathrm{f}_{2,3}}(\lambda) - a_{2,3}(\lambda)}{3+2\lambda} (1-\kappa(\rho))\rho^{-\frac{7}{2}} (\rho+1)^{\frac{3}{2}+\lambda} \int_{\rho}^{\infty} \frac{\kappa(s)s^{\frac{5}{2}}F_{\lambda}(s)}{(s-1)^{\frac{3}{2}+\lambda}} ds\\
		&\quad +\frac{a_{2,3}(\lambda) - a_{\mathrm{f}_{2,3}}(\lambda)} {2}\kappa(\rho)\rho^{-\frac{5}{2}} (\rho+1)^{\frac{1}{2}+\lambda} \int_{\rho}^{\infty} \frac{(1-\kappa(s))s^{\frac{5}{2}}F_{\lambda}(s)}{(s-1)^{\frac{3}{2}+\lambda}} ds\\
		&\quad +\frac{a_{\mathrm{f}_{2,3}}(\lambda) - a_{2,3}(\lambda)}{2} (1-\kappa(\rho)) \rho^{-\frac{5}{2}} (\rho+1)^{\frac{1}{2}+\lambda} \int_{\rho}^{\infty} \frac{\kappa(s)s^{\frac{5}{2}}F_{\lambda}(s)}{(s-1)^{\frac{3}{2}+\lambda}} ds.
	\end{split}
\end{equation}
From the above, it is evident that the first two terms are well-behaved in $\omega$. Thus, on taking the derivative of the remaining two terms, we arrive at the following expression for $C''_{+-}(F_{\lambda})(\rho;\lambda)$:
\begin{equation}
	\begin{split}
		C''_{+-}(F_{\lambda})(\rho;\lambda)&:=-\frac{5}{4}\big(a_{2,3}(\lambda) - a_{\mathrm{f}_{2,3}}(\lambda)\big) \kappa(\rho)\rho^{-\frac{7}{2}}(\rho+1)^{\frac{1}{2}+\lambda} \int_{\rho}^{\infty} \frac{(1-\kappa(s))s^{\frac{5}{2}}F_{\lambda}(s)}{(s-1)^{\frac{3}{2}+\lambda}}ds\\
		&\quad -\frac{5}{4} \big(a_{\mathrm{f}_{2,3}}(\lambda) - a_{2,3}(\lambda)\big) (1-\kappa(\rho))\rho^{-\frac{7}{2}} (\rho+1)^{\frac{1}{2}+\lambda} \int_{\rho}^{\infty} \frac{\kappa(s)s^{\frac{5}{2}}F_{\lambda}(s)}{(s-1)^{\frac{3}{2}+\lambda}} ds\\
		&\quad +\frac{a_{2,3}(\lambda) - a_{\mathrm{f}_{2,3}}(\lambda)}{4}(1+2\lambda) \kappa(\rho)\rho^{-\frac{5}{2}} (\rho+1)^{-\frac{1}{2}+\lambda} \int_{\rho}^{\infty} \frac{(1-\kappa(s))s^{\frac{5}{2}}F_{\lambda}(s)}{(s-1)^{\frac{3}{2}+\lambda}} ds \\
		&\quad +\frac{a_{\mathrm{f}_{2,3}}(\lambda) - a_{2,3}(\lambda)}{4}(1+2\lambda) (1-\kappa(\rho))\rho^{-\frac{5}{2}} (\rho+1)^{-\frac{1}{2}+\lambda} \int_{\rho}^{\infty} \frac{\kappa(s)s^{\frac{5}{2}}F_{\lambda}(s)}{(s-1)^{\frac{3}{2}+\lambda}} ds.
	\end{split}
\end{equation}
Again, we shall drop the well-behaved first two terms and use an integration by parts on the last two terms to conclude that the leading order term for the second order derivative of $C_{+-}(F_{\lambda})$ is given by
\begin{equation}
	\begin{split}
		C''_{+-}(F_{\lambda})(\rho;\lambda)&:=\frac{a_{2,3}(\lambda) - a_{\mathrm{f}_{2,3}}(\lambda)}{2} \kappa(\rho) \rho^{-\frac{5}{2}} (\rho+1)^{-\frac{1}{2}+\lambda} \int_{\rho}^{\infty} \frac{\big((1-\kappa(s))s^{\frac{5}{2}}F_{\lambda}(s)\big)'}{(s-1)^{\frac{1}{2}+\lambda}}ds\\
		&\quad + \frac{a_{\mathrm{f}_{2,3}}(\lambda) - a_{2,3}(\lambda)}{2} (1-\kappa(\rho))\rho^{-\frac{5}{2}} (\rho+1)^{-\frac{1}{2}+\lambda} \int_{\rho}^{\infty} \frac{\big(\kappa(s)s^{\frac{5}{2}}F_{\lambda}(s)\big)'}{(s-1)^{\frac{1}{2}+\lambda}} ds,
	\end{split}
\end{equation}
noting that the boundary terms that arise are annulled. To make the exposition readable, we list the leading order terms for the first derivatives, use the same steps as described above to write the leading order terms for the second derivatives.
 \begin{equation}
\begin{split}
\label{eq:DerivativeA-}
A_-'(F_{\lambda})(\rho;\lambda)&:=\frac{1}{2(3+2\lambda)} \kappa(\rho)\rho^{-\frac{5}{2}} (\rho-1)^{\frac{1}{2}+\lambda} \int_{\rho}^{\infty} \frac{\kappa(s)s^{\frac{5}{2}}F_{\lambda}(s)}{(s-1)^{\frac{3}{2}+\lambda}}ds, 
\end{split}
\end{equation}

\begin{equation}
\begin{split}
\label{eq:DerivativeA+}
A_+'(F_{\lambda})(\rho;\lambda)&:=-\frac{1}{2(3+2\lambda)}\kappa(\rho)\rho^{-\frac{5}{2}} (\rho+1)^{\frac{1}{2}+\lambda} \int_{\rho}^{\infty} \frac{\kappa(s)s^{\frac{5}{2}}F_{\lambda}(s)}{(s+1)^{\frac{3}{2}+\lambda}}ds,
\end{split}
\end{equation}

\begin{equation}
\begin{split}
\label{eq:DerivativeB-}
B_-'(F_{\lambda})(\rho;\lambda)&:=\frac{(1-\kappa(\rho))\rho^{-\frac{5}{2}} (\rho-1)^{\frac{1}{2}+\lambda}}{1+2\lambda}\mathcal{O}(\langle \omega \rangle^{-1})  \int_{\rho}^{\infty} \frac{\big((1-\kappa(s))s^{\frac{5}{2}}F_{\lambda}(s)\big)'}{(s-1)^{\frac{1}{2}+\lambda}}ds,
\end{split}
\end{equation}

\begin{equation}
\begin{split}
\label{eq:DerivativeB+}
B_+'(F_{\lambda})(\rho;\lambda)&:= \frac{(1-\kappa(\rho))\rho^{-\frac{5}{2}} (\rho+1)^{\frac{1}{2}+\lambda}}{1+2\lambda} \mathcal{O}(\langle \omega \rangle^{-1})\int_{\rho}^{\infty} \frac{\big((1-\kappa(s))s^{\frac{5}{2}}F_{\lambda}(s)\big)'}{(s+1)^{\frac{1}{2}+\lambda}}ds
\end{split}
\end{equation}


\begin{equation}
\begin{split}
\label{eq:DerivativeC+-}
C_{+-}'(F_{\lambda})(\rho;\lambda)&:=\frac{a_{2,3}(\lambda) - a_{\mathrm{f}_{2,3}}(\lambda)}{2} \kappa(\rho)\rho^{-\frac{5}{2}} (\rho+1)^{\frac{1}{2}+\lambda} \int_{\rho}^{\infty} \frac{(1-\kappa(s))s^{\frac{5}{2}}F_{\lambda}(s)}{(s-1)^{\frac{3}{2}+\lambda}} ds\\
&\quad - \frac{a_{2,3}(\lambda) - a_{\mathrm{f}_{2,3}}(\lambda)}{2}(1-\kappa(\rho)) \rho^{-\frac{5}{2}} (\rho+1)^{\frac{1}{2}+\lambda} \int_{\rho}^{\infty} \frac{\kappa(s)s^{\frac{5}{2}}F_{\lambda}(s)}{(s-1)^{\frac{3}{2}+\lambda}} ds
\end{split}
\end{equation}

\begin{equation}
    \begin{split}
        \label{eq:DerivativeC-+}
        C_{-+}'(F_{\lambda})(\rho;\lambda)&:=\frac{a_{\mathrm{f}_{1,4}}(\lambda) - a_{1,4}(\lambda)}{2} \kappa(\rho)\rho^{-\frac{5}{2}} (\rho-1)^{\frac{1}{2}+\lambda} \int_{\rho}^{\infty} \frac{(1-\kappa(s))s^{\frac{5}{2}}F_{\lambda}(s)}{(s+1)^{\frac{3}{2}+\lambda}}ds\\
        &\quad + \frac{a_{1,4}(\lambda) -a_{\mathrm{f}_{1,4}}(\lambda)}{2}(1-\kappa(\rho))\rho^{-\frac{5}{2}} (\rho-1)^{\frac{1}{2}+\lambda} \int_{\rho}^{\infty}\frac{ \kappa(s)s^{\frac{5}{2}}F_{\lambda}(s)}{(s+1)^{\frac{3}{2}+\lambda}}ds.
    \end{split}
\end{equation}

\begin{equation}
\begin{split}
\label{eq:DerivativeC--}
C_{--}'(F_{\lambda})(\rho;\lambda)&:=-\frac{a_{1,3}(\lambda) - a_{\mathrm{f}_{1,3}}(\lambda)}{2} \kappa(\rho)\rho^{-\frac{5}{2}} (\rho-1)^{\frac{1}{2}+\lambda} \int_{\rho}^{\infty} \frac{(1-\kappa(s))s^{\frac{5}{2}} F_{\lambda}(s)}{(s-1)^{\frac{3}{2}+\lambda}}ds\\
&\quad + \frac{a_{\mathrm{f}_{2,4}}(\lambda) - a_{2,4}(\lambda)}{2} (1-\kappa(\rho)) \rho^{-\frac{5}{2}} (\rho-1)^{\frac{1}{2}+\lambda} \int_{\rho}^{\infty} \frac{\kappa(s)s^{\frac{5}{2}}F_{\lambda}(s)}{(s-1)^{\frac{3}{2}+\lambda}}ds,
\end{split}
\end{equation}

\begin{equation}
    \begin{split}
    \label{eq:DerivativeC++}
  C_{++}'(F_{\lambda})(\rho;\lambda)&:= \frac{a_{2,4}(\lambda) - a_{\mathrm{f}_{2,4}}(\lambda)}{2} \kappa(\rho)\rho^{-\frac{5}{2}} (\rho+1)^{\frac{1}{2}+\lambda} \int_{\rho}^{\infty} \frac{(1-\kappa(s))s^{\frac{5}{2}} F_{\lambda}(s)}{(s+1)^{\frac{3}{2}+\lambda}}ds\\
  &\quad + \frac{a_{1,3}(\lambda) - a_{\mathrm{f}_{1,3}}(\lambda)}{2} (1-\kappa(\rho)) \rho^{-\frac{5}{2}} (\rho+1)^{\frac{1}{2}+\lambda} \int_{\rho}^{\infty} \frac{\kappa(s)s^{\frac{5}{2}}F_{\lambda}(s)}{(s+1)^{\frac{3}{2}+\lambda}}ds.
\end{split}
\end{equation}

 In the following, we list the leading order terms (in $\omega$) for the second derivatives. We note that when the derivative hits the integral term, we obtain a boundary term, which in each case cancels with its counterpart arising from one of the other terms. 
 \begin{equation}
 \begin{split}
 A''_{-}(F_{\lambda})(\rho;\lambda)&:=\frac{1}{2(3+2\lambda)}\kappa(\rho)\rho^{-\frac{5}{2}} (\rho-1)^{-\frac{1}{2}+\lambda} \Big[\frac{5}{2}\int_{\rho}^{\infty} \frac{\kappa(s)s^{\frac{3}{2}}F_{\lambda}(s)}{(s-1)^{\frac{1}{2}+\lambda}}ds + \int_{\rho}^{\infty} \frac{\kappa(s)s^{\frac{5}{2}}F_{\lambda}'(s)}{(s-1)^{\frac{1}{2}+\lambda}}ds\Big],\\
 A''_{+}(F_{\lambda})(\rho;\lambda)&:=-\frac{1}{2(3+2\lambda)}\kappa(\rho)\rho^{-\frac{5}{2}} (\rho+1)^{-\frac{1}{2}+\lambda} \Big[\frac{5}{2}\int_{\rho}^{\infty} \frac{\kappa(s)s^{\frac{3}{2}}F_{\lambda}(s)}{(s+1)^{\frac{1}{2}+\lambda}}ds + \int_{\rho}^{\infty} \frac{\kappa(s)s^{\frac{5}{2}}F_{\lambda}'(s)}{(s+1)^{\frac{1}{2}+\lambda}}ds\Big],
 \end{split}
 \end{equation}

 \begin{equation}
     \begin{split}
     B''_{-}(F_{\lambda})(\rho;\lambda)&:=\frac{(1-\kappa(\rho))\rho^{-\frac{5}{2}}(\rho-1)^{-\frac{1}{2}+\lambda}}{2}\mathcal{O}_o(\langle \omega \rangle^{-1})\int_{\rho}^{\infty} \frac{\big((1-\kappa(s))s^{\frac{5}{2}}F_{\lambda}(s)\big)'}{(s-1)^{\frac{1}{2}+\lambda}}ds\\
 B''_{+}(F_{\lambda})(\rho;\lambda)&:=\frac{(1-\kappa(\rho))\rho^{-\frac{5}{2}}(\rho+1)^{-\frac{1}{2}+\lambda}}{2}\mathcal{O}_o(\langle \omega \rangle^{-1})\int_{\rho}^{\infty} \frac{\big((1-\kappa(s))s^{\frac{5}{2}}F_{\lambda}(s)\big)'}{(s+1)^{\frac{1}{2}+\lambda}}ds,
 \end{split}
\end{equation}

\begin{equation}
\begin{split}
    C''_{+-}(F_{\lambda})(\rho;\lambda)&:=\frac{a_{2,3}(\lambda) - a_{\mathrm{f}_{2,3}}(\lambda)}{2} \kappa(\rho)\rho^{-\frac{5}{2}}(\rho+1)^{-\frac{1}{2}+\lambda} \int_{\rho}^{\infty} \frac{\big(1-\kappa(s))s^{\frac{5}{2}} F_{\lambda}(s)\big)'}{(s-1)^{\frac{1}{2}+\lambda}}ds\\
    &\quad + \frac{a_{\mathrm{f}_{2,3}}(\lambda) -a_{2,3}(\lambda)}{2} (1-\kappa(\rho))\rho^{-\frac{5}{2}} (\rho+1)^{-\frac{1}{2}+\lambda} \int_{\rho}^{\infty} \frac{\big(\kappa(s)s^{\frac{5}{2}} F_{\lambda}(s)\big)'}{(s-1)^{\frac{1}{2}+\lambda}}ds.
\end{split}
\end{equation}

\begin{equation}
\begin{split}
    C''_{-+}(F_{\lambda})(\rho;\lambda)&:=\frac{a_{\mathrm{f}_{1,4}}(\lambda) -a_{1,4}(\lambda)}{2} \kappa(\rho)\rho^{-\frac{5}{2}} (\rho-1)^{-\frac{1}{2}+\lambda} \int_{\rho}^{\infty} \frac{\big((1-\kappa(s))s^{\frac{5}{2}}F_{\lambda}(s)\big)'}{(s+1)^{\frac{1}{2}+\lambda}}ds\\
    &\quad + \frac{a_{1,4}(\lambda) - a_{\mathrm{f}_{1,4}}(\lambda)}{2}(1-\kappa(\rho))\rho^{-\frac{5}{2}} (\rho-1)^{-\frac{1}{2}+\lambda} \int_{\rho}^{\infty} \frac{\big(\kappa(s)s^{\frac{5}{2}}F_{\lambda}(s)\big)'}{(s+1)^{\frac{1}{2}+\lambda}}ds.
 \end{split}
\end{equation}

\begin{equation}
\begin{split}
 C''_{--}(F_{\lambda})(\rho;\lambda)&:=\frac{a_{\mathrm{f}_{1,3}}(\lambda) - a_{1,3}(\lambda)}{2} \kappa(\rho)\rho^{-\frac{5}{2}} (\rho-1)^{-\frac{1}{2}+\lambda} \int_{\rho}^{\infty} \frac{\big((1-\kappa(s))s^{\frac{5}{2}}F_{\lambda}(s)\big)'}{(s-1)^{\frac{1}{2}+\lambda}}ds\\
 &\quad-\frac{a_{2,4}(\lambda) - a_{\mathrm{f}_{2,4}}(\lambda)}{2}(1-\kappa(\rho))\rho^{-\frac{5}{2}} (\rho-1)^{-\frac{1}{2}+\lambda} \int_{\rho}^{\infty} \frac{\big(\kappa(s)s^{\frac{5}{2}}F_{\lambda}(s)\big)'}{(s-1)^{\frac{1}{2}+\lambda}}ds,
 \end{split}
 \end{equation}

 \begin{equation}
\begin{split}
C''_{++}(F_{\lambda})(\rho;\lambda)&:=\frac{a_{2,4}(\lambda) - a_{\mathrm{f}_{2,4}}(\lambda)}{2}\kappa(\rho)\rho^{-\frac{5}{2}} (\rho+1)^{-\frac{1}{2}+\lambda}\int_{\rho}^{\infty} \frac{\big((1-\kappa(s))s^{\frac{5}{2}}F_{\lambda}(s)\big)'}{(s+1)^{\frac{1}{2}+\lambda}}ds\\
&\quad +\frac{a_{1,3}(\lambda) - a_{\mathrm{f}_{1,3}}(\lambda)}{2} (1-\kappa(\rho))\rho^{-\frac{5}{2}} (\rho+1)^{-\frac{1}{2}+\lambda} \int_{\rho}^{\infty} \frac{\big(\kappa(s)s^{\frac{5}{2}}F_{\lambda}(s)\big)'}{(s+1)^{\frac{1}{2}+\lambda}}ds.
 \end{split}
\end{equation}

\subsubsection{The operator $T_{G''}$} We start with estimating the operator $T_{G''}$. Recall that
\begin{equation*}
    T_{G''}(\tau) f(\rho):=\frac{1}{2\pi i} \lim_{\varepsilon\to0} \lim_{N\to \infty} \int_{\varepsilon-iN}^{\varepsilon+iN} e^{\lambda \tau} G''(f)(\rho;\lambda)d\lambda.
\end{equation*}

\textbf{$A''$ terms}: We shall prove the following result in this section:
\begin{lemma}
    \label{lemma:TA''}
    The following estimate holds:
    \begin{equation}
        \|T_{A''_i}(\cdot)f_2\|_{L^{\infty}_{\tau}(\R_+)L^2(\B^c)}\lesssim \|f_2\|_{\dot{H}^1(\R^6)},
    \end{equation}
    for $i \in\{+,-\}$, and $f_2\in C_{c,rad}^{\infty}(\R^6)$.
    \begin{proof}
       Tracing back the terms and the decay which results from integrating by parts, we observe that the coefficient in $A''_-$ is explicitly given by
       \begin{equation*}
           \frac{1}{3+2\lambda} = \frac{3+2\varepsilon -2i\omega}{(3+2\varepsilon)^2 + 4\omega^2}= \O_o(\langle \omega \rangle^{-1}) + \O(\langle \omega \rangle^{-2}).
       \end{equation*}
       which is odd in $\omega$ in the leading order.
      
       We shall salvage the odd decay in $\omega$ via \cite[Lemma 4.2]{Donninger2010}. Observing that the first term in $A''_{-}$ is amenable to another integration by parts to make the decay in $\omega$ like $\O(\langle \omega \rangle^{-2})$, we treat the second term: the operator $T$ reads
       \begin{equation*}
       \begin{split}
           &\frac{1}{2\pi i} \lim_{\varepsilon\to 0}\lim_{N\to \infty} \int_{\varepsilon-iN}^{\varepsilon+iN} e^{\lambda\tau} \O_o(\langle \omega \rangle^{-1})\kappa(\rho)\rho^{-\frac{5}{2}} (\rho-1)^{-\frac{1}{2}+\lambda} \int_{\rho}^{\infty} \frac{\kappa(s)s^{\frac{5}{2}}f_2'(s)}{(s-1)^{\frac{1}{2}+\lambda}}ds\\
           &\lesssim \kappa(\rho)\rho^{-\frac{5}{2}} (\rho-1)^{-\frac{1}{2}} \int_{\rho}^{\infty} \frac{\kappa(s)s^{\frac{5}{2}}f_2'(s)}{(s-1)^{\frac{1}{2}}} \langle \tau +\log(\rho-1)-\log(s-1)\rangle^{-2}ds.
           \end{split}
 \end{equation*}
 Using the change of variables: $\rho-1=e^{-x}, s-1=e^y$, the squared $L^2(\B_1^c)$ norm of the above can be estimated by
 \begin{equation*}
     (\mb{1}_{[1,\infty)} \ast w)(\tau)
 \end{equation*}
 where 
 \begin{equation*}
 w(\tau-x):=\Big(\int_{\R} \frac{\kappa(e^y+1)(e^y+1)^{\frac{5}{2}}f_2'(e^y+1)e^y}{e^{\frac{y}{2}}}\langle \tau-x-y\rangle^{-2}dy\Big)^2,
 \end{equation*}
 implying that the $L^{\infty}_{\tau}(\R_+)L^{2}_{\rho}(\B^c)$ can be estimated by
 \begin{equation*}
     \begin{split}
     \|\mb{1}_{[1,\infty)}\|_{L^{\infty}}^{\frac{1}{2}} \|w\|_{L^1}^{\frac{1}{2}} \lesssim \|w^{\frac{1}{2}}\|_{L^2} \lesssim \|\tilde{w} \ast \langle \cdot\rangle^{-2}\|_{L^2} \lesssim \|\tilde{w}\|_{L^2} \lesssim \|f_2\|_{\dot{H}^1(\R^6)},
     \end{split}
 \end{equation*}
 for
 \begin{equation*}
 \tilde{w}(y) := \frac{\kappa(e^y+1)(e^y+1)^{\frac{5}{2}}f_2'(e^y+1)e^y}{e^{\frac{y}{2}}}.
 \end{equation*}
 The $A''_+$ terms can be handled in a similar fashion and we conclude the proof.
 \end{proof}
\end{lemma}

\textbf{$B''$ terms}:
By reThe corresponding result for $B''$ terms is as follows.
\begin{lemma}
    \label{lemma:TB''}
    The following estimate holds:
    \begin{equation}
        \|T_{B''_i}(\cdot)f_2\|_{L^{\infty}_{\tau}(\R_+)L^2(\B^c)}\lesssim \|f_2\|_{\dot{H}^1(\R^6)},
    \end{equation}
    for  $i \in\{+,-\}$, and $f_2\in C_{c,rad}^{\infty}(\R^6)$.
    \begin{proof}
    This follows in the same fashion as \ref{lemma:TA''}.
    \end{proof}
\end{lemma}

\textbf{$C''$ terms}: We prove the following:
\begin{lemma}
    \label{lemma:TC''}
    The following estimate holds
    \begin{equation}
        \|T_{C''_{ij}}(\cdot)f_2\|_{L^{\infty}_{\tau}(\R_+)L^2(\B^c)}\lesssim \|f_2\|_{\dot{H}^1(\R^6)}.
    \end{equation}
   for $f_2\in C_{c,rad}^{\infty}(\R^6)$ and $i,j \in\{+,-\}$.
    \begin{proof}
    We continue with analysing the term $C''_{--}$: the first term is amenable to another integration by parts to improve the decay in $\omega$ to be of $\langle \omega \rangle^{-2}$, thus we deal with the second term for which the operator $T$ is estimated by:
    \begin{equation*}
        (1-\kappa(\rho))\rho^{-\frac{5}{2}}(\rho-1)^{-\frac{1}{2}}\int_{\rho}^{\infty} \frac{\kappa(s)s^{\frac{5}{2}}f_2'(s)}{(s-1)^{\frac{1}{2}}} \langle \tau +\log(\rho-1)-\log(s-1)\rangle^{-2}ds.
    \end{equation*}
    Note that, to obtain the above, we have again employed the fact that \[a_{2,4}(\lambda) - a_{\mathrm{f}_{2,4}}(\lambda) = \O_o(\langle \omega \rangle^{-1}+\O(\langle \omega \rangle^{-2})\] from Lemma \ref{lemma:FundamentalSystem1Infty}. The rest of the proof now follows by adjusting the steps from Lemma \ref{lemma:TA''}.
    \end{proof}
\end{lemma}

\subsection{The operator $\dot{T}_{G''}$} For the sake of lucidity, we recall the definition of the operator $\dot{T}$:
\begin{equation*}
    \dot{T}_{G''}(\tau)f(\rho):=\frac{1}{2\pi i} \lim_{\varepsilon \to 0}\lim_{N \to \infty} \int_{\varepsilon-iN}^{\varepsilon+iN} e^{\lambda \tau} G''(f)(\rho;\lambda)d\lambda,
\end{equation*}
where $G''$ is one of the terms from \eqref{eq:SecondDerivativeOfResolvents}, and $f(\rho;\lambda):=(2-\lambda)f_1(\rho)+\rho f_1'(\rho)$. Furthermore, we also note the following for $\gamma =\gamma(\cdot) \in\{(1-\kappa(\cdot)), \kappa(\cdot)\}$:
\begin{equation}
	\begin{split}
		\label{eq:IntegrationByPartsDerivative}
		&\int_{\rho}^{\infty} \frac{\gamma(s)s^{\frac{5}{2}} ((3-\lambda)f_1'(s) +sf_1''(s))}{(s\pm 1)^{\frac{1}{2}+\lambda}} ds\\
		&\quad= -\gamma(\rho)\rho^{\frac{5}{2}} f_1'(\rho)(\rho\pm1)^{\frac{1}{2}-\lambda} \mp\frac{5}{2} \int_{\rho}^{\infty} \frac{\gamma(s)s^{\frac{3}{2}}f_1'(s)}{(s\pm1)^{\frac{1}{2}+\lambda}} ds \mp \int_{\rho}^{\infty} \frac{\gamma(s)s^{\frac{5}{2}}f_1''(s)}{(s\pm1)^{\frac{1}{2}+\lambda}}ds,
	\end{split}
\end{equation}
and
\begin{equation}
	\begin{split}
		\label{eq:IntegrationByPartsTwo}
		&\int_{\rho}^{\infty} \frac{\gamma(s)s^{\frac{3}{2}} ((2-\lambda)f_1(s) +sf_1'(s))}{(s\pm1)^{\frac{1}{2}+\lambda}} ds\\
		&\quad= -\gamma(\rho)\rho^{\frac{3}{2}} f_1(\rho)(\rho\pm1)^{\frac{1}{2}-\lambda} \mp \frac{3}{2} \int_{\rho}^{\infty} \frac{\gamma(s)s^{\frac{1}{2}}f_1(s)}{(s\pm1)^{\frac{1}{2}+\lambda}} ds \mp\int_{\rho}^{\infty} \frac{\gamma(s)s^{\frac{3}{2}}f_1'(s)}{(s\pm1)^{\frac{1}{2}+\lambda}}ds.
	\end{split}
\end{equation}

Recall that $f(\rho;\lambda):=(2-\lambda)f_1(\rho)+\rho f_1'(\rho)$.\\

\textbf{$A''$ terms}: We start with estimating the terms supported near $\rho=1$.
\begin{lemma}
    \label{lemma:TDotA''}
    The following estimate holds:
    \begin{equation}
        \|\dot{T}_{A''_{i}}(\cdot)f\|_{L^{\infty}_{\tau}(\R_+)L^2(\B^c)}\lesssim \|f_1\|_{\dot{H}^2(\R^6)}.
    \end{equation}
   for $f_1\in C_{c,rad}^{\infty}(\R^6)$, and  $i \in\{+,-\}$.
    \begin{proof}
We centre our attention to $A''_-$ and observe that the second term therein is not amenable to another straightforward integration by parts. However, as noted in Lemma \ref{lemma:TA''}, the coefficient has a decay of $\O_o(\langle \omega \rangle^{-1})$ in $\omega$ which will suffice for our purposes. Using \eqref{eq:IntegrationByPartsDerivative}, we can write for the second term in $A''_{-}$:
\begin{equation*}
    \begin{split}
    &\quad-\O_o(\langle \omega \rangle^{-1}) \kappa(\rho)\rho^{-\frac{5}{2}} (\rho-1)^{-\frac{1}{2}+\lambda}\int_{\rho}^{\infty} \frac{\kappa(s)s^{\frac{5}{2}}\big((3-\lambda)f_1'(s)+sf_1''(s)\big)}{(s-1)^{\frac{1}{2}+\lambda}}ds\\
    &=-\O_o(\langle \omega \rangle^{-1}) \kappa(\rho)\rho^{-\frac{5}{2}} (\rho-1)^{-\frac{1}{2}+\lambda}\Big[-\kappa(\rho)\rho^{\frac{5}{2}}f_1'(\rho)(\rho-1)^{\frac{1}{2}-\lambda} +\frac{5}{2}\int_{\rho}^{\infty} \frac{(\kappa(s)s^{\frac{3}{2}}f_1'(s)}{(s-1)^{\frac{1}{2}+\lambda}}ds \\
    &\quad \quad+\int_{\rho}^{\infty} \frac{(\kappa(s)s^{\frac{5}{2}}f_1''(s)}{(s-1)^{\frac{1}{2}+\lambda}}ds \Big].
    \end{split}
\end{equation*}
The first term above will not contribute since its negative counterpart from $A''_{+}$ will annul it. The operator that results from the second term above is estimated by
\begin{equation*}
    \kappa(\rho)\rho^{-\frac{5}{2}}(\rho-1)^{-\frac{1}{2}} \int_{\rho}^{\infty} \frac{\kappa(s)s^{\frac{3}{2}}f_1'(s)}{(s-1)^{\frac{1}{2}}}\langle \tau +\log(\rho-1)-\log(s-1)\rangle^{-2}ds
\end{equation*}
and it is standard, by now, to control the above in $L^{\infty}_{\tau}(\R_+)L^2_{\rho}(\B^c)$ by the $\dot{H}^2(\B^c)$ norm of $f_1$. The next term follows the same steps, and the simpler case of $A''_+$ is omitted.
    \end{proof}
    \end{lemma}

\textbf{$B''$ terms}: Next, we state the result for terms near $\rho=\infty$.
\begin{lemma}
    \label{lemma:TDotB''}
    Let $f_1\in \dot{H}^2(\R^6)$, and $f(\rho;\lambda):=(2-\lambda)f_1(\rho)+\rho f_1'(\rho)$, then for $i \in\{+,-\}$, the following estimate holds:
    \begin{equation}
        \|\dot{T}_{B''_{i}}(\cdot)f\|_{L^{\infty}_{\tau}(\R_+)L^2(\B^c)}\lesssim \|f_1\|_{\dot{H}^2(\R^6)}.
    \end{equation}
    \begin{proof}
    The terms are analogous to their $A''$ counterparts and have enough decay in $\rho$ to be controlled in the $L^{\infty}_{\tau}(\R_+)L^2_{\rho}(\B^c)$ norm.
    \end{proof}
    \end{lemma}

\textbf{$C''$ terms}: Next, we deal with the mixed terms $C_{ij}$.
\begin{lemma}
    \label{lemma:TDotC''}
    The following estimate holds
    \begin{equation}
        \|\dot{T}_{C''_{ij}}(\cdot)f\|_{L^{\infty}_{\tau}(\R_+)L^2(\B^c)}\lesssim \|f_1\|_{\dot{H}^2(\R^6)}.
    \end{equation}
    for  $f_1\in C_{c,rad}^{\infty}(\R^6)$, and $i,j \in\{+,-\}$.
    \begin{proof}
    We look at the first term in $C''_{--}$ and use \eqref{eq:IntegrationByPartsTwo} to obtain
    \begin{equation*}
        \begin{split}
        &\quad -\frac{5}{2}\O_o(\langle \omega \rangle^{-1}) (1-\kappa(\rho))\rho^{-\frac{5}{2}} (\rho-1)^{-\frac{1}{2}+\lambda} \int_{\rho}^{\infty} \frac{\kappa(s)s^{\frac{3}{2}}\big((2-\lambda)f_1(s)+sf_1'(s)\big)}{(s-1)^{\frac{1}{2}+\lambda}}ds\\
        &=-\frac{5}{2}\O_o(\langle \omega \rangle^{-1}) (1-\kappa(\rho))\rho^{-\frac{5}{2}} (\rho-1)^{-\frac{1}{2}+\lambda}\Big[-\kappa(\rho)\rho^{\frac{3}{2}}f_1(\rho)(\rho-1)^{\frac{1}{2}-\lambda} + \frac{3}{2}\int_{\rho}^{\infty} \frac{\kappa(s)s^{\frac{1}{2}}f_1(s)}{(s-1)^{\frac{1}{2}+\lambda}}ds\\
        &\quad + \int_{\rho}^{\infty} \frac{\kappa(s)s^{\frac{3}{2}}f_1'(s)}{(s-1)^{\frac{1}{2}+\lambda}}ds\Big].
        \end{split}
    \end{equation*}
    The first term in the above display has a compact support and $\O_o(\langle \omega \rangle^{-1}$ decay, and can be treated in a similar fashion to previous boundary terms. For the second term, the operator $\dot{T}$ reads
    \begin{equation*}
        \begin{split}
            (1-\kappa(\rho))\rho^{-\frac{5}{2}} (\rho-1)^{-\frac{1}{2}} \int_{\rho}^{\infty} \frac{\kappa(s)s^{\frac{1}{2}}f_1(s)}{(s-1)^{\frac{1}{2}}} \langle \tau +\log(\rho-1)-\log(s-1)\rangle^{-2}ds,
        \end{split}
    \end{equation*}
    which is of a similar form that we have dealt with before. The third term demands a similar handling, and we conclude the proof.
    \end{proof}
    \end{lemma}

    \begin{proposition}
\label{prop:ExteriorStrichartz}
    The semigroup $\Sf$ satisfies the estimates
    \begin{align*}
        \|[\Sf(\cdot)\ff]_1\|_{L^q_\tau(\R_+) L^r(\B^c_1)}&\lesssim  \|\ff\|_{\mathcal{H}^2}
    \end{align*}
    for all $q\in [2,\infty],~r\in[6,12]$ that satisfy \eqref{eq:AdmissibleExponents}
    and all $\ff\in C_{c,rad}^{\infty}(\R^6)\times  C_{c,rad}^{\infty}(\R^6)$.
    Additionally, we have for all $\ff\in C_{c,rad}^{\infty}(\R^6)\times  C_{c,rad}^{\infty}(\R^6)$, the estimate:
    \begin{align*}
        \|[\Sf(\cdot)\ff]_1\|_{L^\infty_\tau(\R_+) \dot{H}^2(\B^c_1)} &\lesssim  \|\ff\|_{\mathcal{H}^2}.
    \end{align*}
    \begin{proof}
        Using the representation \eqref{eq:IntegralRepresentation}, the proof is a combination of Proposition \ref{prop: free strichartz} and Lemmas 6.1-6.14.
    \end{proof}
\end{proposition}


\subsection{Energy estimate for the second component of the solution}
\label{section:SecondComponent}
In this section, our aim is to prove the following estimates for the second component $[\mb{S}(\tau)\mb{f}]_2$ of the semigroup:
\begin{equation*}
    \| [\mb{S}(\cdot)\mb{f}]_2\|_{L_{\tau}^{\infty}(\R_+)\dot{H}^1(\B^c)} \lesssim \|\mb{f}\|_{\mathcal{H}}.
\end{equation*}
Since the estimate inside the interior of the cone has been handled In Section 5,  it remains to deal with the second component on the exterior of the lightcone. We begin by noting that the derivative of the second component is given by
	\begin{equation}
		\begin{split}
        \label{eq:u2'}
			u_2'(\rho;\lambda) &= \big(\rho \tilde{u}_1''(\rho;\lambda) - (\lambda-2)\tilde{u}_1'(\rho;\lambda)\big) \int_{\rho}^{\infty} \frac{\tilde{u}_2(s;\lambda)s^5 F_{\lambda}(s)}{(s^2-1)^{\frac{3}{2}+\lambda}} ds\\
            &\quad -\big(\rho \tilde{u}_2''(\rho;\lambda) - (\lambda-2)\tilde{u}_2'(\rho;\lambda)\big) \int_{\rho}^{\infty} \frac{\tilde{u}_1(s;\lambda)s^5 F_{\lambda}(s)}{(s^2-1)^{\frac{3}{2}+\lambda}} ds
            \\
            &\quad-\frac{\rho F_\lambda(\rho)}{1-\rho^2}-f_1'(\rho)
		\end{split}
	\end{equation}
	Note that the terms $-\frac{F_\lambda(\rho)}{1-\rho^2}-f'(\rho)$ will cancel when we subtract the free reolvent. Thus, need not pay any attention to them and write out the leading order term in the expression $\rho \tilde{u}_1''(\rho;\lambda) - (\lambda-2)\tilde{u}_1'(\rho;\lambda)$:
	\begin{equation*}
		\begin{split}
			&\frac{1}{4\sqrt{3+2\lambda}}\kappa(\rho) (\rho+1)^{-\frac{1}{2}+\lambda} \rho^{-\frac{7}{2}} g_1(\rho;\lambda) \big(-2(3+2\lambda)(\lambda-2)\rho + 5(3+2\lambda)\big)\\
			&\quad +\frac{a_{1,3}(\lambda)}{4\sqrt{3+2\lambda}} (1-\kappa(\rho))(\rho+1)^{-\frac{1}{2}+\lambda}\rho^{-\frac{7}{2}} g_3(\rho;\lambda) \big(-2(3+2\lambda)(\lambda-2)\rho +5(3+2\lambda)\big)\\
			&\quad + \frac{a_{1,4}(\lambda)}{4\sqrt{3+2\lambda}} (1-\kappa(\rho))(\rho-1)^{-\frac{1}{2}+\lambda}\rho^{-\frac{7}{2}} g_4(\rho;\lambda) \big(2(3+2\lambda)(\lambda-2)\rho +5(3+2\lambda)\big),
		\end{split}
	\end{equation*}
	and similarly for $\rho \tilde{u}_2''(\rho;\lambda) - (\lambda-2)\tilde{u}_2'(\rho;\lambda)$:
	\begin{equation*}
		\begin{split}
			&\frac{1}{4\sqrt{3+2\lambda}}\kappa(\rho) (\rho-1)^{-\frac{1}{2}+\lambda} \rho^{-\frac{7}{2}} g_2(\rho;\lambda) \big(2(3+2\lambda)(\lambda-2)\rho + 5(3+2\lambda)\big)\\
			&\quad +\frac{a_{2,3}(\lambda)}{4\sqrt{3+2\lambda}} (1-\kappa(\rho))(\rho+1)^{-\frac{1}{2}+\lambda}\rho^{-\frac{7}{2}} g_3(\rho;\lambda) \big(-2(3+2\lambda)(\lambda-2)\rho +5(3+2\lambda)\big)\\
			&\quad + \frac{a_{2,4}(\lambda)}{4\sqrt{3+2\lambda}} (1-\kappa(\rho))(\rho-1)^{-\frac{1}{2}+\lambda}\rho^{-\frac{7}{2}} g_4(\rho;\lambda) \big(2(3+2\lambda)(\lambda-2)\rho +5(3+2\lambda)\big),
		\end{split}
	\end{equation*}
    We observe that when two derivatives hit $\tilde{u}_i, i=1,2$, the resulting behaviour is decay of order $-3+\lambda$. However, multiplication with $\rho$ makes the decay worse. The observation is that subtraction of the term $(\lambda-2)\tilde{u}_i'(\rho;\lambda)$ results in a cancellation in $\rho$ which leads us to the aforementioned behaviour. With this, we can write out the leading order terms in \eqref{eq:u2'}:
	\begin{equation}
		u_2'(\rho;\lambda)=\mathcal{A}_1 + \mathcal{A}_2 + \mathcal{B}_1 + \mathcal{B}_2 +\sum_{i=1}^8 \mathcal{C}_i,
	\end{equation}
	where 
	\begin{equation}
		\begin{split}
			\label{eq:u2'Around1}
		\mathcal{A}_1&=-\frac{\lambda-2}{2}\kappa(\rho)\rho^{-\frac{5}{2}}(\rho+1)^{-\frac{1}{2}+\lambda} g_1(\rho;\lambda) \int_{\rho}^{\infty} \frac{\kappa(s)s^{\frac{5}{2}}g_2(s;\lambda)F_{\lambda}(s)}{(s+1)^{\frac{3}{2}+\lambda}} ds\\
		\mathcal{A}_2&=-\frac{\lambda-2}{2}\kappa(\rho)\rho^{-\frac{5}{2}}(\rho-1)^{-\frac{1}{2}+\lambda} g_2(\rho;\lambda) \int_{\rho}^{\infty} \frac{\kappa(s)s^{\frac{5}{2}}g_1(s;\lambda)F_{\lambda}(s)}{(s-1)^{\frac{3}{2}+\lambda}} ds
		\end{split}
	\end{equation}
	
	\begin{equation}
		\begin{split}
			\label{eq:u2'Cross}
			\sum_{i=1}^8\mathcal{C}_i&=-\frac{\lambda-2}{2} a_{1,3}(\lambda) (1-\kappa(\rho))\rho^{-\frac{5}{2}} (\rho+1)^{-\frac{1}{2}+\lambda}g_3(\rho;\lambda) \int_{\rho}^{\infty} \frac{\kappa(s)s^{\frac{5}{2}} g_2(s;\lambda)F_{\lambda}(s)}{(s+1)^{\frac{3}{2}+\lambda}} ds\\
			& -\frac{\lambda-2}{2} a_{1,3}(\lambda) \kappa(\rho)\rho^{-\frac{5}{2}} (\rho-1)^{-\frac{1}{2}+\lambda}g_2(\rho;\lambda) \int_{\rho}^{\infty} \frac{(1-\kappa(s))s^{\frac{5}{2}} g_3(s;\lambda)F_{\lambda}(s)}{(s-1)^{\frac{3}{2}+\lambda}} ds\\
			&+\frac{\lambda-2}{2} a_{1,4}(\lambda) (1-\kappa(\rho))\rho^{-\frac{5}{2}} (\rho-1)^{-\frac{1}{2}+\lambda}g_4(\rho;\lambda) \int_{\rho}^{\infty} \frac{\kappa(s)s^{\frac{5}{2}} g_2(s;\lambda)F_{\lambda}(s)}{(s+1)^{\frac{3}{2}+\lambda}} ds\\
			&-\frac{\lambda-2}{2} a_{1,4}(\lambda) \kappa(\rho)\rho^{-\frac{5}{2}} (\rho-1)^{-\frac{1}{2}+\lambda}g_2(\rho;\lambda) \int_{\rho}^{\infty} \frac{(1-\kappa(s))s^{\frac{5}{2}} g_4(s;\lambda)F_{\lambda}(s)}{(s+1)^{\frac{3}{2}+\lambda}} ds\\
			&-\frac{\lambda-2}{2} a_{2,3}(\lambda) \kappa(\rho)\rho^{-\frac{5}{2}} (\rho+1)^{-\frac{1}{2}+\lambda}g_1(\rho;\lambda) \int_{\rho}^{\infty} \frac{(1-\kappa(s))s^{\frac{5}{2}} g_3(s;\lambda)F_{\lambda}(s)}{(s-1)^{\frac{3}{2}+\lambda}} ds\\
			&+\frac{\lambda-2}{2} a_{2,3}(\lambda) (1-\kappa(\rho))\rho^{-\frac{5}{2}} (\rho+1)^{-\frac{1}{2}+\lambda}g_3(\rho;\lambda) \int_{\rho}^{\infty} \frac{\kappa(s)s^{\frac{5}{2}} g_1(s;\lambda)F_{\lambda}(s)}{(s-1)^{\frac{3}{2}+\lambda}} ds\\
			&-\frac{\lambda-2}{2} a_{2,4}(\lambda) \kappa(\rho)\rho^{-\frac{5}{2}} (\rho+1)^{-\frac{1}{2}+\lambda}g_1(\rho;\lambda) \int_{\rho}^{\infty} \frac{(1-\kappa(s))s^{\frac{5}{2}} g_4(s;\lambda)F_{\lambda}(s)}{(s+1)^{\frac{3}{2}+\lambda}} ds\\
			&-\frac{\lambda-2}{2} a_{2,4}(\lambda) (1-\kappa(\rho))\rho^{-\frac{5}{2}} (\rho-1)^{-\frac{1}{2}+\lambda}g_4(\rho;\lambda) \int_{\rho}^{\infty} \frac{\kappa(s)s^{\frac{5}{2}} g_1(s;\lambda)F_{\lambda}(s)}{(s-1)^{\frac{3}{2}+\lambda}} ds,
			\end{split}
	\end{equation}
	
	and
	\begin{equation}
		\begin{split}
			\label{eq:u1'AtInfinity}
		\mathcal{B}_1&	=-\frac{\lambda-2}{2}  (1-\kappa(\rho))\rho^{-\frac{5}{2}} (\rho-1)^{-\frac{1}{2}+\lambda} g_4(\rho;\lambda)\int_{\rho}^{\infty}\frac{(1-\kappa(s))s^{\frac{5}{2}} g_3(s;\lambda)F_{\lambda}(s)}{(s-1)^{\frac{3}{2}+\lambda}} ds\\
		\mathcal{B}_2&=-\frac{\lambda-2}{2} (1-\kappa(\rho))\rho^{-\frac{5}{2}} (\rho+1)^{-\frac{1}{2}+\lambda} g_3(\rho;\lambda) \int_{\rho}^{\infty}\frac{(1-\kappa(s))s^{\frac{5}{2}} g_4(s;\lambda)F_{\lambda}(s)}{(s+1)^{\frac{3}{2}+\lambda}} ds.
					\end{split}
	\end{equation}
    After considering the free counterparts and using an integration by parts, these terms can be dealt with in the same way as the first component of the resolvent and we have the following result:
    \begin{proposition}
        \label{prop:SecondComponent}
        The following estimate holds for $[\mb{S}(\tau)\mb{f}]_2$ of the semigroup:
        \begin{equation}
            \label{eq:SecondComponent}
            \| [\mb{S}(\cdot)\mb{f}]_2\|_{L_{\tau}^{\infty}(\R_+)\dot{H}^1(\B^c)} \lesssim \|\mb{f}\|_{\mathcal{H}},
        \end{equation}
        for $\mb{f}=(f_1,f_2) \in C_{c,rad}^{\infty}(\R^6) \times C_{c,rad}^{\infty}(\R^6)$.
    \begin{proof}
        We note that the second component is given by
        \begin{equation*}
            \begin{split}
			[\mb{S}(\tau)\mb{f}]_2(\rho) &= [\mb{S}_0(\tau)\mb{f}]_2(\rho) + [\tilde{\mb{S}}(\tau)\mb{f}]_2(\rho)\\
			&=: [\mb{S}_0(\tau)\mb{f}]_2(\rho) + \frac{1}{2\pi i} \lim_{\varepsilon\to 0}\lim_{N \to \infty} \int_{\varepsilon -iN}^{\varepsilon+iN} e^{\lambda \tau} [(\lambda-1)\mathcal{R}(F_{\lambda})(\rho;\lambda) -\rho\partial_{\rho}\big(\mathcal{R}(F_{\lambda})(\rho;\lambda)\big)\\
            &\quad \quad -(\lambda-1)\mathcal{R}_{\mathrm{f}}(F_{\lambda})(\rho;\lambda) +\rho\partial_{\rho}\big(\mathcal{R}_{\mathrm{f}}(F_{\lambda})(\rho;\lambda)\big)d \lambda.
		\end{split}
        \end{equation*}
The first term in the above display can be estimated by Proposition \ref{prop: free strichartz}. A full-fledged expression for the derivative of the above can be written using \eqref{eq:u2'}. Since the analysis is the same as that for the first component of the semigroup, we briefly illustrate the idea using one of the terms in \eqref{eq:u2'Around1}. Consider the difference of the term $\mathcal{A}_2$ and its free counterpart $\mathcal{A}_{\mathrm{f}_2}$:
	\begin{equation}
    \label{eq:A2-Af2}
		\begin{split}
			\mathcal{A}_2 - \mathcal{A}_{\mathrm{f}_2}=&-\frac{\lambda-2}{2}\kappa(\rho) \rho^{-\frac{5}{2}} (\rho-1)^{-\frac{1}{2}+\lambda} \Big[\big(g_2(\rho;\lambda)-g_{\mathrm{f}_2}(\rho;\lambda) \big)\int_{\rho}^{\infty} \frac{\kappa(s)s^{\frac{5}{2}}g_1(s;\lambda)f(s)}{(s-1)^{\frac{3}{2}+\lambda}} ds\\
			&\quad  \quad + g_{\mathrm{f}_2}(\rho;\lambda) \int_{\rho}^{\infty} \frac{\kappa(s)s^{\frac{5}{2}} \big(g_1(s;\lambda)-g_{\mathrm{f}_1}(s;\lambda)\big)f(s)}{(s-1)^{\frac{3}{2}+\lambda}} ds\Big]=:\mathcal{I}_1 + \mathcal{I}_2
		\end{split}
	\end{equation}
	We consider the first term above since the second one can be estimated similarly. Using the definition of $g_2$ and $g_{\mathrm{f}_2}$ and an integration by parts, we have
	\begin{equation*}
		\begin{split}
\mathcal{I}_1 &= -\frac{\lambda-2}{1+2\lambda} (\rho-1)^{-1} \kappa(\rho)^2 F_{\lambda}(\rho)g_1(\rho;\lambda) \big(g_2(\rho;\lambda) - g_{\mathrm{f}_2}(\rho;\lambda)\big)\\
&\quad -\frac{\lambda-2}{1+2\lambda}  \kappa(\rho) \rho^{-\frac{5}{2}} (\rho-1)^{-\frac{1}{2}+\lambda} \big(g_2(\rho;\lambda) - g_{\mathrm{f}_2}(\rho;\lambda)\big) \int_{\rho}^{\infty} \frac{\big(\kappa(s)s^{\frac{5}{2}}g_1(s;\lambda) F_{\lambda}(s)\big)'}{(s-1)^{\frac{1}{2}+\lambda}}ds\\
&=\mathcal{I}_{11} + \mathcal{I}_{12}.
		\end{split}
	\end{equation*} 
	 We consider the leading order term in $\mathcal{I}_{12}$ when the derivative hits $s^{\frac{5}{2}}$. We note that the leading order coefficient is $\O_o(\langle \omega \rangle^{-1})$. Thus, the resulting operator $T$ is estimated by
	\begin{equation*}
		\begin{split}
			\kappa(\rho)\rho^{-\frac{5}{2}} (\rho-1)^{\frac{1}{2}} \int_{\rho}^{\infty} \frac{\kappa(s)s^{\frac{3}{2}}f_2(s)}{(s-1)^{\frac{1}{2}}}\langle \tau +\log(\rho-1)-\log(s-1)\rangle^{-2} ds,
						\end{split}
	\end{equation*}
	for which the $L^2(\B^c)$ is estimated by
	\begin{equation*}
		\begin{split}
			&\Big[ \int_1^{\infty} \kappa(\rho)^2 (\rho-1) \big( \int_{\rho}^{\infty} \frac{\kappa(s)s^{\frac{3}{2}}f_2(s)}{(s-1)^{\frac{1}{2}}} \langle \tau +\log(\rho-1)-\log(s-1)\rangle^{-2}ds\big)^2 d\rho \Big]^{\frac{1}{2}}\\
				&\quad \lesssim \int_{1}^{\infty} \frac{\kappa(s)s^{\frac{3}{2}}f_2(s)}{(s-1)^{\frac{1}{2}}} \langle \tau -\log(s-1)\rangle^{-2}ds.
		\end{split}
	\end{equation*}
	This is a form which can be estimated in the $L^{\infty}_{\tau}L^2_{\rho}(\B^c)$ norm along the same lines as before, and we omit the details. For the term when the derivative hits $F_{\lambda}$ in $\mathcal{I}_{12}$, the resulting operator is given by
	\begin{equation*}
		\begin{split}
			\kappa(\rho)\rho^{-\frac{5}{2}} (\rho-1)^{\frac{1}{2}} \int_{\rho}^{\infty} \frac{\kappa(s)s^{\frac{5}{2}}f_2'(s)}{(s-1)^{\frac{1}{2}}}\langle \tau +\log(\rho-1)-\log(s-1)\rangle^{-2} ds,
		\end{split}
	\end{equation*}
	whose $L^2(\B^c)$ norm is bounded by
	\begin{equation*}
\lesssim \int_{1}^{\infty} \frac{\kappa(s)s^{\frac{5}{2}}f_2'(s)}{(s-1)^{\frac{1}{2}}} \langle \tau -\log(s-1)\rangle^{-2}ds,
	\end{equation*}
	and the rest of the details follow analogously.
 To estimate $\mathcal{I}_{11}$, we also deal with the second term in \eqref{eq:A2-Af2}, and observe that the resulting boundary term is given by
 \begin{equation*}
     -\frac{\lambda-2}{1+2\lambda}\kappa(\rho)^2 (\rho-1)^{-1}F_{\lambda}(\rho)g_{\mathrm{f}_2}(\rho;\lambda)\big(g_1(\rho;\lambda) -g_{\mathrm{f}_1}(\rho;\lambda)\big).
 \end{equation*}
 This leads to the following boundary term that we have to control:
 \begin{equation*}
     \mathcal{I}_B = -\frac{\lambda-2}{1+2\lambda}\kappa(\rho)^2 (\rho-1)^{-1}F_{\lambda}(\rho) \big(g_1(\rho;\lambda)g_2(\rho;\lambda) - g_{\mathrm{f}_1}(\rho;\lambda)g_{\mathrm{f}_2}(\rho;\lambda)\big).
 \end{equation*}
From Lemma \ref{lem:solsnear1}, we observe that
\begin{equation*}
\begin{split}
    g_1(\rho;\lambda) - g_{\mathrm{f}_1}(\rho;\lambda)& = -\frac{1}{3+2\lambda} \int_1^{\rho} V(s)ds + \mathcal{O}((\rho-1)\langle \omega \rangle^{-2}),\\
 g_2(\rho;\lambda) - g_{\mathrm{f}_2}(\rho;\lambda) &= \frac{1}{3+2\lambda} \int_1^{\rho} V(s)ds + \mathcal{O}((\rho-1)\langle \omega \rangle^{-2}).
    \end{split}
\end{equation*}
This gives that the leading order behavior:
\begin{equation*}
    \mathcal{I}_B = -\frac{\lambda-2}{(1+2\lambda)(3+2\lambda)^2} \kappa(\rho)^2 F_{\lambda}(\rho).
\end{equation*}
For the resulting operator $T$, one has the decay of order $\langle \omega \rangle^{-2}$ while for $\dot{T}$, it is $\mathcal{O}_o(\langle \omega \rangle^{-1})$ which is sufficient.
	Next, we look at the operator $\dot{T}$ for $\mathcal{I}_{12}$. It suffices to deal with the following terms:
	\begin{equation*}
		\begin{split}
			&\frac{5}{2}\frac{\lambda-2}{(1+2\lambda)(3+2\lambda)} \kappa(\rho)\rho^{-\frac{5}{2}} (\rho-1)^{\frac{1}{2}+\lambda} \int_{\rho}^{\infty} \frac{\kappa(s)s^{\frac{3}{2}}f(s)}{(s-1)^{\frac{1}{2}+\lambda}} ds\\
			&\quad + \frac{\lambda-2}{(1+2\lambda)(3+2\lambda)} \kappa(\rho)\rho^{-\frac{5}{2}} (\rho-1)^{\frac{1}{2}+\lambda} \int_{\rho}^{\infty} \frac{\kappa(s)s^{\frac{5}{2}}f'(s)}{(s-1)^{\frac{1}{2}+\lambda}} ds\\
			&=:\mathcal{P}_1 + \mathcal{P}_2
		\end{split}
	\end{equation*}
	where $f(s):=(2-\lambda)f_1(s)+sf_1'(s)$. We consider $\mathcal{P}_2$, the treatment of $\mathcal{P}_1$ being similar. Using \eqref{eq:IntegrationByPartsDerivative}, we have
	\begin{equation*}
		\begin{split}
			\mathcal{P}_2&=\frac{\lambda-2}{(1+2\lambda)(3+2\lambda)} \kappa(\rho)\rho^{-\frac{5}{2}} (\rho-1)^{\frac{1}{2}+\lambda}\Big[-\kappa(\rho) \rho^{\frac{5}{2}} (\rho-1)^{\frac{1}{2}-\lambda} f_1'(\rho) \\
			&\quad \quad +\frac{5}{2}\int_{\rho}^{\infty} \frac{\kappa(s)s^{\frac{3}{2}} f_1'(s)}{(s-1)^{\frac{1}{2}+\lambda}} ds +\int_{\rho}^{\infty} \frac{\kappa(s)s^{\frac{5}{2}}f_1''(s)}{(s-1)^{\frac{1}{2}+\lambda}} ds\Big]\\
			&=\mathcal{P}_{21} +  \mathcal{P}_{22} + \mathcal{P}_{23}.
		\end{split}
	\end{equation*}
	The operator that results from $\mathcal{P}_{23}$ is given by
	\begin{equation*}
		\begin{split}
			&\quad \kappa(\rho)\rho^{-\frac{5}{2}} (\rho-1)^{\frac{1}{2}} \int_{\rho}^{\infty} \frac{\kappa(s)s^{\frac{5}{2}} f_1''(s)}{(s-1)^{\frac{1}{2}}} \langle \tau +\log(\rho-1)-\log(s-1)\rangle^{-2}ds
		\end{split}
	\end{equation*}
	The above can be controlled in $L_{\tau}^{\infty}(\R_+)L_{\rho}^2(\B^c)$ norm  by the $\dot{H}^2(\B^c)$ norm of $f_1$ by using the same change of variables as before. The other terms demand a similar treatment and the proof is complete.
    \end{proof}
    \end{proposition}

    
    To summarise, in Sections \ref{section:InteriorEstimates} and \ref{section:StrichartzEstimatesExterior}, we have proved the following:
    \begin{theorem}
	\label{thm:StrichartzFullSemigroup}
	 	Let $q\in[2,\infty]$ and $r\in[6,12]$ satisfy
        \begin{equation}
            \label{eq:AdmissiblePQ}
		\frac{1}{q} +\frac{6}{r} = 1.
	\end{equation} 
    and let
	 	$\mathbf P$ be the spectral projection onto the unstable spectrum of $\mathbf L$ and assume that $0$ is not a resonance of $V$. Then for all $\mathbf f\in\mathcal H^2$,
	 	\begin{equation*}
        \begin{split}
	 		\|[\mathbf S(\tau)(\mathbf I-\mathbf P)\mathbf f]_1\|_{L^q_\tau(\mathbb R_+)L^r(\mathbb R^6)}
	 		&\lesssim \|\mathbf f\|_{\mathcal H^2},\\
	 		\|[\mathbf S(\tau)(\mathbf I-\mathbf P)\mathbf f]_1\|_{L^2_\tau(\mathbb R_+)\dot W^{1,4}(\mathbb R^6)}
	 		&\lesssim \|\mathbf f\|_{\mathcal H^2},\\
	 		\|\mathbf S(\tau)(\mathbf I-\mathbf P)\mathbf f\|_{L^\infty_\tau(\mathbb R_+)\mathcal H^2}
	 		&\lesssim \|\mathbf f\|_{\mathcal H^2}.
            \end{split}
	 	\end{equation*}
	 	Moreover, for all $\mathbf h\in L^1_\tau(\mathbb R_+)\mathcal H^2$, the above estimates
	 	hold with $\mathbf S(\tau)(\mathbf I-\mathbf P)\mathbf f$ replaced by
	 	$$\displaystyle \int_0^\tau \mathbf S(\tau-\sigma)(\mathbf I-\mathbf P)\mathbf h(\sigma)\,d\sigma$$
	 	and $\|\mathbf f\|_{\mathcal H^2}$ replaced by $\|\mathbf h\|_{L^1_\tau(\mathbb R_+)\mathcal H^2}$.
\end{theorem}

\begin{theorem}[Exterior estimates]
\label{thm:Exterior}
		Let $V$ be the potential of an admissible self-similar solution, and let
		$q\in[2,\infty]$, $r\in[6,12]$ satisfy \eqref{eq:AdmissiblePQ}. Then, for all
		$\mb{f}\in\mathcal{H}^2$,
		\begin{align*}
			\|[\mb{S}(\tau)\mb{f}]_1\|_{L^q_\tau(\R_+)L^r(\B_1^c)}  & \lesssim \|\mb{f}\|_{\mathcal{H}^2}, \\
			\|[\mb{S}(\tau)\mb{f}]_1\|_{L^2_\tau(\R_+)\dot W^{1,4}(\B_1^c)} & \lesssim \|\mb{f}\|_{\mathcal{H}^2},\\
			\|[\mb{S}(\tau)\mb{f}]_1\|_{L^\infty_\tau(\R_+)\dot H^2(\B_1^c)}
			+\|[\mb{S}(\tau)\mb{f}]_2\|_{L^\infty_\tau(\R_+)\dot H^1(\B_1^c)} & \lesssim \|\mb{f}\|_{\mathcal{H}^2}.
		\end{align*}
		Moreover, for all $\mb{h}\in L^1_\tau(\R_+)\mathcal{H}^2$, the same estimates hold with
		$\mb{S}(\tau)\mb{f}$ replaced by $\displaystyle \int_0^\tau\mb{S}(\tau-\sigma)\mb{h}(\sigma)\,d\sigma$
		and $\|\mb{f}\|_{\mathcal{H}^2}$ replaced by $\displaystyle \|\mb{h}\|_{L^1_\tau(\R_+)\mathcal{H}^2}$.
    \begin{proof}
		Since elements of $rg \mb{P}$ are supported in $\overline{\B^6}$, we have
		$[\mb{S}(\tau)\mb{P}\mb{f}]|_{\B_1^c}=0$, so (4.19) represents
		$[\mb{S}(\tau)\mb{f}]_1$ on $\B_1^c$. For $\rho\geqslant 1$, the resolvent in Lemma  \ref{lemma:ResolventFull} is built from $\tilde u_1,\tilde u_2$ only and does not involve the connection coefficient $c_{0,1}$, so the contour can be moved to the imaginary axis for every admissible $V$. The estimates then follow from Proposition 2.5, Lemmas 6.1--6.14 and Proposition \ref{prop:SecondComponent}.
		The inhomogeneous bounds follow from the homogeneous bounds using Minkowski's inequality.
	\end{proof}
		\end{theorem}
	 

\section{Nonlinearity control and proof of the main results}
\label{sec:NonlinearityControl}
\subsection{Nonlinearity control} In this section, we control the nonlinearity for the cubic wave and the wave maps equations and finally prove our two main theorems.
\subsubsection{Cubic wave nonlinearity}
For the cubic wave equation \eqref{eq:cNLW}, we have
\begin{equation}
	\mathbf{N}_{\mathrm{NLW}}(\bm{\phi}) = \begin{bmatrix} 0\\
		\phi_1^3 + 3u_1^*\phi_1^2
		\end{bmatrix},
\end{equation}
where $u_1^*$ satisfies the estimate
$$
|u_1^*(\rho)|\lesssim \langle\rho\rangle^{-1}, \quad | \partial_{\rho}u_1^*(\rho)| \lesssim \langle \rho \rangle^{-2}.
$$

We define
\begin{equation*}
    \mathcal{X}:=\{(\phi_1,\phi_2) \in C_b([0,\infty);\mathcal{H}^2): \|(\phi_1, \phi_2)\|_{\mathcal{X}} <\infty\},
\end{equation*}
where
\begin{equation*}
    \| (\phi_1,\phi_2)\|_{\mathcal{X}}:= \| \phi_1\|_{L_{\tau}^4(\R_+)L_{\rho}^8(\R^6)} + \| \phi_1\|_{L_{\tau}^2(\R_+)\dot{W}^{1,4}(\R^6)} + \|\phi_1\|_{L^\infty_{\tau}(\R_+)\dot{H}^2(\R^6)} + \|\phi_2\|_{L_{\tau}^{\infty}(\R_+)\dot{H}^1(\R^6)}.
\end{equation*}

A simple application of H\"older's inequality and Hardy's inequality yield the following result:
\begin{lemma}
	\label{lemma:NLWNonlinearity}
	For $\bm{\phi}, \bm{\psi} \in \mathcal{X}$, the following estimates hold:
	\begin{equation*}
		\begin{split}
		\|\mathbf{N}_{\mathrm{NLW}}(\bm{\phi}) \|_{L_{\tau}^1(\R_+)\mathcal{H}^2} \lesssim \|\phi_1\|_{L_\tau^2(\R_+) \dot{W}^{1,4}(\R^6)}^2+\|\phi_1\|_{L_\tau^2(\R_+) \dot{W}^{1,4}(\R^6)} \|\phi_1\|_{L_\tau^4(\R_+)L^8(\R^6)}^2,
        \end{split}
        \end{equation*}
and
\begin{equation*}
\begin{split}
			&\|\mathbf{N}_{\mathrm{NLW}}(\bm{\phi}) - \mathbf{N}_{\mathrm{NLW}}(\bm{\psi}) \|_{L_{\tau}^1(\R_+)\mathcal{H}^2} \\
            &\quad \lesssim \Big(\| \phi_1-\psi_1\|_{L_{\tau}^4(\R_+)L_{\rho}^8(\R^6)} +  \| \phi_1-\psi_1\|_{L_{\tau}^2(\R_+)\dot{W}_{\rho}^{1,4}(\R^6)} \Big)\\
			&\quad \times \Big(\| \phi_1\|_{L_{\tau}^4(\R_+)L_{\rho}^8(\R^6)}^2 + \| \psi_1\|_{L_{\tau}^4(\R_+)L_{\rho}^8(\R^6)}^2 + \|\phi_1\|_{L_{\tau}^4(\R_+)L_{\rho}^8(\R^6)} \|\psi_1\|_{L_{\tau}^4(\R_+)L_{\rho}^8(\R^6)} \\
            &\quad \quad + \|\phi_1\|_{L_{\tau}^4(\R_+)L_{\rho}^8(\R^6)} \| \phi_1\|_{L_{\tau}^2(\R_+)\dot{W}_{\rho}^{1,4}(\R^6)}+ \| \psi_1\|_{L_{\tau}^2(\R_+)\dot{W}_{\rho}^{1,4}(\R^6)}\| \psi_1\|_{L_{\tau}^4(\R_+)L_{\rho}^{8}(\R^6)}\\
			&\quad \quad + \| \phi_1\|_{L_{\tau}^4(\R_+)L_{\rho}^{8}(\R^6)}  \| \psi_1\|_{L_{\tau}^2(\R_+)\dot{W}_{\rho}^{1,4}(\R^6)} + \| \psi_1\|_{L_{\tau}^4(\R_+)L_{\rho}^{8}(\R^6)} \| \phi_1\|_{L_{\tau}^2(\R_+)\dot{W}_{\rho}^{1,4}(\R^6)} \\
            &\quad \quad + \|\phi_1\|_{L_{\tau}^2(\R_+)\dot{W}_{\rho}^{1,4}(\R^6)} + \| \psi_1\|_{L_{\tau}^2(\R_+)\dot{W}_{\rho}^{1,4}(\R^6)}\Big).
		\end{split}
	\end{equation*}
	\begin{proof}
		Since the first component of the nonlinearity is zero, we control the second component of  $\mb{N}_{\mathrm{NLW}} (\bm{\phi})$, where $\bm{\phi} = (\phi_1,\phi_2)$ in $\dot{H}^1$ as follows:
		\begin{equation*}
			\|\phi_1^3 \|_{L_{\tau}^1(\R_+)\dot{H}^1(\R^6)} \lesssim \|\phi_1\|_{L_\tau^2(\R_+) \dot{W}^{1,4}(\R^6)} \|\phi_1\|_{L_\tau^4(\R_+)L^8(\R^6)}^2.
			\end{equation*}
            Moreover, 
            \begin{equation*}
            \begin{split}
                \|\langle \cdot \rangle^{-1} \phi_1^2\|_{L_{\tau}^1(\R_+)\dot{H}^1(\R^6)} &\lesssim \| \langle \cdot \rangle^{-1} \phi_1\|_{L_{\tau}^2(\R_+)L_{\rho}^4(\R^6)}^2 + \| \langle \cdot \rangle^{-1} \phi_1\|_{L_{\tau}^2(\R_+)L^4(\R^6)} \| \phi_1'\|_{L^2_{\tau}(\R_+)L^4_{\rho}(\R^6)}\\
                &\lesssim \| \phi_1\|_{L_\tau^2(\R_+)\dot{W}_{\rho}^{1,4}(\R^6)}^2,
                \end{split}
            \end{equation*}
            where we used the following version of Hardy's inequality
        \begin{equation}\label{Eq:Hardy4}
            \||.|^{-1} f\|_{L^4(\mathbb R^d)}\lesssim \|f\|_{\dot{W}^{1,4}(\mathbb R^d)}.
        \end{equation}
			For the difference (dropping the subscripts for simplicity), we have
	\begin{equation*}
				\begin{split}
			\| \phi^3 - \psi^3\|_{L_{\tau}^1(\R_+)\dot{H}^1(\R^6)} &\lesssim \| \phi -\psi\|_{L_{\tau}^4(\R_+)L_{\rho}^8(\R^6)} \Big( \| \phi\|_{L_{\tau}^2 \dot{W}^{1,4}(\R^6)} \|\phi\|_{L_{\tau}^4(\R_+)L_{\rho}^8(\R^6)} + \|\psi\|_{L_{\tau}^4(\R_+)L_{\rho}^8(\R^6)} \| \psi\|_{L_{\tau}^2 \dot{W}^{1,4}(\R^6)}\\
			&\quad \quad + \|\psi\|_{L_{\tau}^4(\R_+)L_{\rho}^8(\R^6)} \| \phi\|_{L_{\tau}^2 \dot{W}^{1,4}(\R^6)}
			  + \|\phi\|_{L_{\tau}^4(\R_+)L_{\rho}^8(\R^6)} \| \psi\|_{L_{\tau}^2 \dot{W}^{1,4}(\R^6)} \Big)\\
			  & \quad + \| \phi-\psi\|_{L_{\tau}^2 \dot{W}^{1,4}(\R^6)} \Big(\|\phi\|_{L_{\tau}^4(\R_+)L_{\rho}^8(\R^6)}^2 +  \|\psi\|_{L_{\tau}^4(\R_+)L_{\rho}^8(\R^6)}^2\\
			  &\quad \quad  + \|\phi\|_{L_{\tau}^4(\R_+)L_{\rho}^8(\R^6)}\|\psi\|_{L_{\tau}^4(\R_+)L_{\rho}^8(\R^6)} \Big).
			\end{split}
		\end{equation*}
	Lastly, by once more employing \eqref{Eq:Hardy4} one also concludes that
	\begin{equation*}
			\| u_1^* \phi^2 - u_1^*\psi^2 \|_{L_{\tau}^1(\R_+)\dot{H}^1(\R^6)}  \lesssim\|\phi_1-\psi_1\|_{L_\tau^2(\R_+) \dot{W}^{1,4}(\R^6)}\left(\|\phi_1\|_{L_\tau^2(\R_+) \dot{W}^{1,4}(\R^6)}+\|\psi_1\|_{L_\tau^2(\R_+) \dot{W}^{1,4}(\R^6)}\right).
	\end{equation*}	
    Summing up the terms gives us the stated estimate.
	\end{proof}
\end{lemma}


\subsubsection{Wave maps nonlinearity}
In this case, we recall that the nonlinearity is given by
\begin{equation}
        \mb{N}_{\mathrm{WM}}(\bm{\phi})=\begin{bmatrix}
        0\\
        -\frac{(d-3)}{2\rho^3} \big[\sin\big(2\rho (u_{\mathrm{WM}_1}^*(\rho) + \phi_1)\big) - \sin(2\rho u_{\mathrm{WM}_1}^*(\rho)) \big] +\frac{d-3}{\rho^2}\cos(2\rho u_{\mathrm{WM}_1}^*(\rho))\phi_1
    \end{bmatrix}.
\end{equation}
where $\bm{\phi}=(\phi_1,\phi_2)$.


\begin{lemma}
\label{lemma:WMNonlinearity}
	For $\bm{\phi}, \bm{\psi} \in \mathcal{X}$, the following estimates hold:
	\begin{equation*}
		\begin{split}
			&\|\mb{N}_{\mathrm{WM}}(\bm{\phi})\|_{L_{\tau}^1(\R_+)\mathcal{H}^2} \lesssim \|\phi_1\|_{L_{\tau}^2(\R_+)\dot{W}^{1,4}(\R^6)} \|\phi_1\|_{L_{\tau}^4(\R_+)L_{\rho}^8(\R^6)}^2 + \|\phi_1\|_{L_{\tau}^4(\R_+)L_{\rho}^{8}(\R^6)}^4 + \|\phi_1\|_{L_{\tau}^2(\R_+)\dot{W}_{\rho}^{1,4}(\R^6)}^2,\\
			\text{and}\\
			&\|\mb{N}_{\mathrm{WM}}(\bm{\phi}) - \mb{N}_{\mathrm{WM}}(\bm{\psi}) \|_{L_{\tau}^1(\R_+)\mathcal{H}^2} \\
			&\quad\lesssim \Big[\|\phi_1-\psi_1\|_{L_{\tau}^4(\R_+)L_{\rho}^8(\R^6)}
\Big(\|\phi_1\|_{L_{\tau}^4(\R_+)L_{\rho}^8(\R^6)} + \|\psi_1\|_{L_{\tau}^4(\R_+)L_{\rho}^8(\R^6)}\Big)
+ \|\phi_1-\psi_1\|_{L_{\tau}^2(\R_+)\dot{W}_{\rho}^{1,4}(\R^6)}\Big] \\
&\qquad\times\Big(\|\phi_1\|_{L_{\tau}^2(\R_+)\dot{W}_{\rho}^{1,4}(\R^6)} + \|\psi_1\|_{L_{\tau}^2(\R_+)\dot{W}_{\rho}^{1,4}(\R^6)}
+ \|\phi_1\|_{L_{\tau}^4(\R_+)L_{\rho}^8(\R^6)}^2 + \|\psi_1\|_{L_{\tau}^4(\R_+)L_{\rho}^8(\R^6)}^2\Big).
		\end{split}
	\end{equation*}
	\begin{proof}
    We define an auxiliary function $\eta: (0,\infty)\times \R \to \R$ by
    \begin{equation*}
        \eta(\rho;x):=-\frac{d-3}{2\rho^3} \big[\sin(2\rho u_{\mathrm{WM}_1}^*(\rho)+2\rho x) - \sin(2\rho u_{\mathrm{WM}_1}^*(\rho))\big] +\frac{(d-3)}{\rho^2} \cos(2 \rho u_{\mathrm{WM}_1}^*(\rho))x.
    \end{equation*}
    For the ease of notation, we denote $u_{\mathrm{WM}_1}^*$ by $u^*$. Then, we have, using the integral form of Taylor's theorem,
    \begin{equation*}
    \eta(\rho;x) = \frac{(d-3)\sin(2\rho u^*(\rho))}{\rho}x^2 + 2(d-3)\int_0^x (x-s)^2 \cos(2\rho u^*(\rho) +2\rho s)ds,
    \end{equation*}
    which implies
    \begin{equation*}
        |\eta(\rho;x)| \lesssim \langle \rho \rangle^{-1}|x|^2 + |x|^3,
    \end{equation*}
    for all $\rho \in (0,\infty)$. Note that the first term is smooth at $\rho=0$. The derivative of $\eta$ in $\rho$ is given by
    \begin{equation*}
    \begin{split}
        \partial_{\rho}\eta(\rho;x)& = \frac{d-3}{\rho^2} \big[-\sin(2\rho u^*(\rho)) +2\rho \cos(2\rho u^*(\rho)) \big(u^*(\rho) +\rho \partial_{\rho}u^*(\rho)\big) \big]x^2 \\
        &\quad -4(d-3)\int_0^x (x-s)^2 \sin(2\rho u^*(\rho) +2\rho s)\big(u^*(\rho) +\rho \partial_{\rho}u^*(\rho) + s \big)ds,
        \end{split}
    \end{equation*}
    which is again smooth at $\rho=0$ as can be seen by its Taylor expansion. Furthermore,
    \begin{equation*}
    \begin{split}
    &|\partial_{\rho}\eta(\rho;x)| \lesssim \langle \rho \rangle^{-2}|x|^2+|x|^4,\\
    &|\partial_{x}\eta(\rho;x)| \lesssim \langle \rho \rangle^{-1}|x|+|x|^2.\\
    &|\partial_x\partial_{\rho} \eta(\rho;x)| \lesssim \langle \rho\rangle^{-2}|x| + |x|^3\\
    &|\partial_x^{2}\eta(\rho;x)| \lesssim \langle \rho\rangle^{-1} +|x|.
    \end{split}
    \end{equation*}
		Now, the proof easily follows from H\"older's inequality and Hardy's inequality \eqref{Eq:Hardy4}.
	\end{proof}
    
\end{lemma}
\begin{remark}
    For $d=6$, by employing 
    \begin{equation*}
        \| |\cdot|^{-1}f\|_{L^4(\B^c)} \lesssim \|f\|_{\dot{W}^{1,4}(\B^c)},
    \end{equation*}
    one notes that the estimates in Lemmas \ref{lemma:NLWNonlinearity} and \ref{lemma:WMNonlinearity} are also valid when the spatial domain in the function spaces is replaced by $\B_1^c$.
\end{remark}



\subsection{Cubic wave equation}
 We recall the nonlinearity is given by
\begin{equation*}
    \mb{N}(\Phi):=\mb{N}_{\mathrm{NLW}}(\Phi) = (0, \phi_1^3 + 3u_1^*\phi_1^2)^t,
\end{equation*}
where $\Phi:=(\phi_1,\phi_2)$, $u_1^*$ is an admissible self-similar solution.

\subsubsection{Proof of Theorem \ref{Thm: NLW2}} 
We define the map:
	\begin{equation}
		\mb{K}_{\mb u_0}(\Phi)(\tau)=\mb{S}(\tau)(\mb{I}-\mb{P})\mb{u_0} +\int_0^{\tau} \mb{S} (\tau-\sigma) (\mb{I}-\mb{P}) \mb{N}(\Phi(\sigma))d\sigma -\int_{\tau}^{\infty} \mb{S}(\tau-\sigma)\mb{P}\mb{N}(\Phi(\sigma)) d\sigma,
	\end{equation}
	and show that it is a contraction on a closed ball in $\mathcal{X}$. Furthermore, we recall the assumption that $0$ is not a resonance of the potential $V=3(u_1^*)^2$.

    \begin{remark}
	\label{remark:rgP}
    \begin{enumerate}
\item 	The proof of Proposition \ref{Prop:SpectralCharacterization} shows that eigenfunctions of $\mathbf{L}$ with eigenvalue $\lambda_j>0$ are smooth in $\B_1^6$, vanish outside $\overline{\B_1^6}$ and
	have first component of the form $c(1-\rho)^{\lambda_j+\frac32}h(1-\rho)$ near the light cone, for smooth $h$. For generalised eigenfunctions, interior elliptic regularity in $\B^6$ and the embedding of radial $\dot H^2\times\dot H^1$ functions into $C^1\times C$ away from the origin give the same conclusion by induction. \\
    If $\lambda_j$ is not algebraically simple, the generalised eigenfunctions may contain factors $\log^k(1-\rho)$ near the light cone. More precisely, their first components are of the form
$(1-\rho)^{\lambda_j+\frac{3}{2}}\sum_{k=0}^{m-1}h_k(1-\rho)\log^k(1-\rho)$ with $h_k$ smooth, where $m$ is the length of the Jordan chain. Since elliptic regularity is not available at $\rho=1$, where the equation
degenerates, we argue directly. As $\lambda_j+\frac{3}{2}>\frac{3}{2}$, the first component and its first derivative vanish at $\rho=1$, so the first component is $C^1$ across $\rho=1$. The second component, which involves only the first component and its first derivative, is continuous and
also vanishes at $\rho=1$.\\
    Hence $rg \mb{P}\subset C^1_{c,\mathrm{rad}}(\R^6)\times C_{c,\mathrm{rad}}(\R^6)$,
	and in particular the first components belong to $L^8(\R^6)\cap\dot W^{1,4}(\R^6)$:
	\begin{equation*}
		\|u_1\|_{L^8(\R^6)} + \|u_1\|_{\dot {W}^{1,4}(\R^6)} \lesssim \|\mb{u}\|_{\mathcal{H}^2}
		\qquad\text{for all } \mb{u} \in rg \mb{P}.
	\end{equation*}
	Since $V$ and $\lambda_j$ are real, $rg \mb{P}$ has a basis of
	real-valued functions.
\item Note that $\mb{P}$ commutes with $\mb{S}(\tau)$ and its range, $rg \mb{P}$ is finite dimensional, and we denote $\mb{L}_\mb{P}:=\mb{L}\big|_{rg \mb{P}}$. Since $\sigma(\mb L_{\mb P})=\{\lambda_1,...,\lambda_N\}\subset(0,\infty)$, the family $\mb S(\tau)\mb P=e^{\tau\mb L_{\mathbf P}}\mb P$ extends to all $\tau\in\mathbb R$, and for any fixed $0<\nu<\min_j\lambda_j$,
\begin{equation}
\label{eq:UnstableDecay}
\begin{split}
	\|\mb S(\tau)\mb P\mb f\|_{\mathcal H^2}\lesssim_{\mb P,\nu} e^{\nu\tau}\|\mb f\|_{\mathcal H^2}\quad\text{for all }\tau\leqslant  0,\ \mb{f} \in\mathcal H^2,\\
    \|\mb [\mb S(\tau)\mb P\mb f]_1\|_{L^8(\R^6) \cap \dot{W}^{1,4}(\R^6)}\lesssim_{\mb P,\nu} e^{\nu\tau}\|\mb f\|_{\mathcal H^2}\quad\text{for all }\tau\leqslant  0,\ \mb{f} \in\mathcal H^2,
    \end{split}
\end{equation}
 since $\rg \mb{P}$ is finite-dimensional and contained in $L^8 \cap \dot{W}^{1,4}$.
 \end{enumerate}
 \end{remark}

	With this, we prove that $\mb{K}_{\mb{u_0}}$ satisfies the conditions for Banach's fixed point argument:
	\begin{proposition}
		\label{prop:FixedPointFullSpace}
		Let $\mb{u}_0 \in \mathcal{H}^2$ and $\Phi, \Psi \in  \mathcal{X}$. Then,
$\mathbf{K}_{\mathbf{u}_0}(\Phi) \in \mathcal{X}$. Moreover, the following estimates hold:
		\begin{equation*}
			\begin{split}
				\| \mb{K}_{\mb{u_0}}(\Phi)\|_{\mathcal{X}} &\lesssim \|\mb{u_0}\|_{\mathcal{H}^2} + \|\Phi\|_{\mathcal{X}}^2+\|\Phi\|_{\mathcal{X}}^3, \\
					\| \mb{K}_{\mb{u}_0} (\Phi) - \mb{K}_{\mb{u}_0} (\Psi)\|_{\mathcal{X}} &\lesssim \|\Phi-\Psi\|_{\mathcal{X}} \sum_{i=1}^{2} \big(\|\Phi\|_{\mathcal{X}}^i + \|\Psi\|_{\mathcal{X}}^i \big).
							\end{split}
		\end{equation*}
		\begin{proof}
			Let $Z$ be any of the spatial space in the norms defining $\mathcal{X}$. For $r\in [1,\infty]$, we  have, using \eqref{eq:UnstableDecay}
			\begin{equation*}
				\begin{split}
					\Big \| \int_{\tau}^{\infty} \mb{S}(\tau-\sigma)\mb{P}\mb{N}(\Phi(\sigma)) d\sigma \Big\|_{Z} &\lesssim \int_{\tau}^{\infty} e^{\nu(\tau-\sigma)} \|\mb{N}(\Phi(\sigma))\|_{\mathcal{H}^2} d\sigma,
				\end{split}
			\end{equation*}
			which implies
			\begin{equation*}
				\begin{split}
					\Big\| \int_{\tau}^{\infty} \mb{S}(\tau-\sigma)\mb{P}\mb{N}(\Phi(\sigma)) d\sigma \Big\|_{L_{\tau}^r(\R_+)Z} &\lesssim \big\| e^{\nu(\cdot)} \mb{1}_{(-\infty,0]}(\cdot) \ast \| \mb{N}(\Phi(\sigma)) \|_{\mathcal{H}^2} \big\|_{L_{\tau}^r(\R_+)}\\
					& \lesssim \| e^{\nu(\cdot)}\mb{1}_{(-\infty,0]}\|_{L^r(\R)}  \| \mb{N}(\Phi)\|_{L^1(\R_+)\mathcal{H}^2}\\
					&\lesssim \| \mb{N}(\Phi)\|_{L^1(\R_+)\mathcal{H}^2}.
				\end{split}
			\end{equation*}
			Using Theorem \ref{thm:StrichartzFullSemigroup} and the above observation, we have
			\begin{equation*}
				\begin{split}
				\|\mb{K}_{\mb{u_0}}(\Phi)\|_{\mathcal{X}} &\lesssim \| \mb{S}(\tau)(\mb{I} -\mb{P})\mb{u_0} \|_{\mathcal{X}} + \Big \| \int_0^{\tau} \mb{S}(\tau-\sigma)(\mb{I}-\mb{P})\mb{N}(\Phi)(\sigma)d\sigma \Big\|_{\mathcal{X}}\\
				&\quad + \Big\| \int_{\tau}^{\infty} \mb{S}(\tau-\sigma)\mb{PN}(\Phi(\sigma)) d\sigma \Big\|_{\mathcal{X}}\\
				&\lesssim \|\mb{u_0}\|_{\mathcal{H}^2} + \| \mb{N}(\Phi)\|_{L_{\tau}^1(\R_+)\mathcal{H}^2}\\
				&\lesssim \|\mb{u_0}\|_{\mathcal{H}^2}  + \|\Phi\|_{\mathcal{X}}^2 + \| \Phi\|_{\mathcal{X}}^3.
							\end{split}
			\end{equation*}
			In a similar vein, we have
			\begin{equation*}
				\begin{split}
					\| \mb{K}_{\mb{u_0}}(\Phi) - \mb{K}_{\mb{u_0}}(\Psi) \|_{\mathcal{X}} &\lesssim \Big\| \int_0^{\tau} \mb{S}(\tau-\sigma)(\mb{I}-\mb{P})(\mb{N}(\Phi(\sigma)) - \mb{N}(\Psi(\sigma))) d\sigma \Big\|_{\mathcal{X}}\\
					&\quad + \Big\| \int_{\tau}^{\infty} \mb{S}(\tau-\sigma)\mb{P}(\mb{N}(\Phi(\sigma)) - \mb{N}(\Psi(\sigma))) d\sigma \Big\|_{\mathcal{X}} \\
					&\lesssim \|\mb{N}(\Phi) -\mb{N}(\Psi)\|_{L_{\tau}^1(\R_+)\mathcal{H}^2}\\
					&\lesssim \| \Phi-\Psi\|_{\mathcal{X}} \Big( \|\Phi\|_{\mathcal{X}} + \| \Psi\|_{\mathcal{X}} + \|\Phi\|_{\mathcal{X}}^2 + \| \Psi\|_{\mathcal{X}}^2 \Big)
				\end{split}
			\end{equation*}
			 \end{proof}
			 	\end{proposition}

			 With the above result, the following readily follows by means of Banach's fixed point theorem:
			 \begin{lemma}
             \label{lemma:KContraction}
			 	There exist $C, \delta_0>0$ such that the following holds for all $\delta\in(0,\delta_0]$ and all $\mb u_0\in\mathcal H^2$ with $\|\mb u_0\|_{\mathcal H^2}\leqslant \frac{\delta}{C}$: $\mb{K}_{\mb{u_0}}$ admits a unique fixed point in $\Phi \in \{ \Phi :\|\Phi\|_{\mathcal{X}} \leqslant \delta\}$. Furthermore, $\Phi \in C([0,\infty);\mathcal{H}^2)$ and solves:
			 	  \begin{equation*}
			 		\Phi(\tau)=\mb S(\tau)\big(\mb u_0-\mb C(\Phi,\mb u_0)\big)+\int_0^\tau\mb S(\tau-\sigma)\mb N(\Phi(\sigma))\,d\sigma, \text{ for } \tau \geqslant 0,
			 	\end{equation*}
                where
		 \begin{equation*}
		 	\mb{C}(\Phi,\mb{u_0}):=\mb{Pu_0} +\int_0^{\infty} \mb{S}(-\sigma) \mb{PN}(\Phi(\sigma)) d\sigma \in rg \mb{P}.
		 \end{equation*}
			 \begin{proof}
		 The first claim is a straightforward application of Banach's fixed point  theorem on the closed ball $\{\Phi: \|\Phi\|_{\mathcal{X}}\leqslant \delta\}$. For the second one, we note that 
		 \begin{equation*}
		 	\mb K_{\mb{u_0}}(\Phi)(\tau) = \mb{S}(\tau)(\mb{u_0} - \mb{C}(\Phi,\mb{u_0})) +\int_0^{\tau} \mb{S}(\tau-\sigma) \mb{N}(\Phi(\sigma))d\sigma.
		 \end{equation*}

         \end{proof}
	\end{lemma}
		 
		 \begin{lemma}
	Let $C,\delta_0$ be as in Lemma \ref{lemma:KContraction}, $\delta \in (0,\delta_0]$ and $\|\mb{u_0}\|_{\mathcal{H}^2} \leqslant C\delta$. Let $\Phi$ be the unique fioxed point of $\mb{K}_{\mb{u_0}}$ in the ball $\{\|\Phi\|_{\mathcal{X}} \leqslant \delta\}$. Let $\{\hf_1, \hf_2, ..., \hf_{N'}\}$ be a real basis of $rg \mb{P}$. Then, there exist unique $\mu_1, \mu_2, ..., \mu_{N'} \in \R$ such that
	\begin{equation*}
		\mb{C}(\Phi, \mb{u_0}) = \sum_{j=1}^{N'} \mu_j \hf_j.
		\end{equation*}
		Define  $\mb{\tilde{u_0}} = \mb{u_0} - \sum_{j=1}^{N'} \mu_j \hf_j$, then $\mb{K}_{\mb{\tilde{u_0}}} = \mb{K}_{\mb{u_0}}$ and $\mb{C}(\Phi, \mb{\tilde{u}_0}) = 0$.
        \begin{proof}
            We have
            \begin{equation*}
                \mb{C}(\Phi, \mb{u_0}) = \mb{P}\mb{u_0} + \int_0^{\infty} \mb{S}(-\sigma) \mb{P N}(\Phi(\sigma))d\sigma \in \rg\mb{P},
            \end{equation*}
            where the integral converges and $\mb{N}(\Phi) \in L_{\tau}^1\mathcal{H}^2$. Since $rg\mb{P}$ is finite-dimensional with basis $\{\hf_1, ..., \hf_{N'}\}$, the coefficients $\mu_j$ exist and are unique. That $\mb{K}_{\mb{u_0}} = \mb{K}_{\mb{\tilde{u_0}}}$ follows immediately from $\mb{P}\hf_j = \hf_j$:
            \begin{equation*}
                (1-\mb{P})(\mb{\tilde{u_0}}) = (1-\mb{P})(\mb{u_0} - \sum_{j=1}^{N'}\mu_j \hf_j) = \mb{u_0} - \sum_{j=1}^{N'}\mu_j \hf_j - \mb{P} \mb{u_0} +\sum_{j=1}^{N'}\mu_j \hf_j = (1-\mb{P})\mb{u_0}.
            \end{equation*}
            Similarly, 
            \begin{equation*}
                \mb{C}(\Phi, \mb{\tilde{u_0}}) = \mb{P}\mb{u_0} + \int_0^{\infty} \mb{S}(-\sigma) \mb{P N}(\Phi(\sigma))d\sigma - \sum_{j=1}^{N'}\mu_j \hf_j = \mb{C}(\Phi,\mb{u_0}) - \sum_{j=1}^{N'}\mu_j \hf_j =0.
            \end{equation*}
        \end{proof}
\end{lemma}
		
Now, we can prove Theorem \ref{Thm: NLW2}.
    \begin{proof}[Proof of Theorem \ref{Thm: NLW2}]
    Let $C,\delta_0$ be as in Lemma \ref{lemma:KContraction}. Fix $R \in (0,\delta_0]$ and set $\delta :=\frac{R}{C}, \delta':= R$.
Let  $\|\mb{u}_0\|_{\mathcal{H}^2}=\|(\varphi_0,\varphi_1)\|_{\mathcal{H}^2} \leqslant \delta$. Let $\Phi=(\phi_1,\phi_2)$ be the solution from Lemma \ref{lemma:KContraction} and set
$\mathbf v := \mathbf u^* + \Phi$. Undoing the change of variables,
we define, for $t\geqslant 1$ and $x\in\mathbb R^6$,
\begin{equation*}
\begin{split}
u_1(t,x)&:= 
\frac{1}{t} v_1(\log(t) ,\frac{x}{t}) =u_1^*(t,x)+ \frac{1}{t} \phi_1(\log(t),\frac{x}{t}),\\
u_2(t,x)&:= \frac{1}{t^2} v_2(\log(t),\frac{x}{t}) =\partial_tu_1^*(t,x)+\frac{1}{t^2} \phi_2( \log(t),\frac{x}{t}).
\end{split}
\end{equation*}
Then $u:=u_1$ is a solution of \eqref{eq:cNLW} with $\partial_t u_1=u_2$ and since
$\Phi(0) = \mb{u_0} - \mb C(\Phi,\mb{u_0}) = (\tilde\varphi_0,\tilde\varphi_1)$, it satisfies the initial condition from Theorem \ref{Thm: NLW2}.  From Lemma \ref{lemma:KContraction}, we know that there exists a solution $\Phi \in \mathcal{X}$ corresponding to this initial data. Now consider,
        \begin{equation*}
            \begin{split}
               \sup_{t\in[1,\infty)} \|u_1(t,\cdot) - u_1^*(t,\cdot)\|_{\dot{H}^2(\R^6)}&=\sup_{t\in[1,\infty)} \Big\| \frac{1}{t} v_1(\log(t), \frac{\cdot}{t}) - \frac{1}{t} u_1^*(\frac{\cdot}{t})\Big\|_{\dot{H}^2(\R^6)}\\
               &=\sup_{t\in[1,\infty)} \|v_1(\log(t),\cdot) - u_1^*(\cdot)\|_{\dot{H}^2(\R^6)}\\
               &=\|\phi_1\|_{L^\infty(\R_+)\dot{H}^2(\R^6)}
            \end{split}
        \end{equation*}
        Similarly, for the second component, we have
            \begin{equation*}
            \begin{split}
               \sup_{t\in[1,\infty)} \|u_2(t,\cdot) - \partial_tu_1^*(t,\cdot)\|_{\dot{H}^1(\R^6)}&=\sup_{t\in[1,\infty)} \Big\| \frac{1}{t^2} v_2(\log(t), \frac{\cdot}{t}) - \frac{1}{t^2} \partial_t u_1^*(\frac{\cdot}{t})\Big\|_{\dot{H}^1(\R^6)}\\
               &=\sup_{t\in[1,\infty)} \|v_2(\log(t),\cdot) - \partial_t u_1^*(\cdot)\|_{\dot{H}^1(\R^6)}\\
               &=\|\phi_2\|_{L^\infty(\R_+)\dot{H}^1(\R^6)}
            \end{split}
        \end{equation*}
        In conclusion,
        \begin{equation*}
            \sup_{t\in[1,\infty)} \|u_1(t,\cdot) - u_1^*(t,\cdot)\|_{\dot{H}^2(\R^6)} + \sup_{t\in[1,\infty)} \|u_2(t,\cdot) - \partial_t u_1^*(t,\cdot)\|_{\dot{H}^1(\R^6)} =\|\phi_1\|_{L^\infty(\R_+)\dot{H}^2(\R^6)} + \|\phi_2\|_{L^\infty(\R_+)\dot{H}^1(\R^6)} \leqslant \delta'.
        \end{equation*}
        \end{proof}

\subsubsection{The exterior problem} 
For radial functions on $\B_1^c$ we set
\begin{equation*}
	\|f\|_{\dot H^2(\B_1^c)} := \|D^2 f\|_{L^2(\B_1^c)}, \qquad
	\|f\|_{\dot H^1(\B_1^c)} := \|\nabla f\|_{L^2(\B_1^c)}, \qquad
	\|f\|_{\dot W^{1,4}(\B_1^c)} := \|\nabla f\|_{L^4(\B_1^c)},
\end{equation*}
and define $\mathcal{H}^{ext}$ as the completion of $\{\mb{f}\big|_{\B^c}: \mb{f} \in C_{c,rad}^{\infty}(\R^6) \times C_{c,rad}^{\infty}(\R^6)\}$ with respect to the norm:
\begin{equation*}
    \|\mb{f}\|_{\mathcal{H}^{ext}}:=\|(f_1,f_2)\|_{\mathcal{H}^{ext}} =\|f_1\|_{\dot{H}^2(\B^c)} + \|f_2\|_{\dot{H}^1(\B^c)},
\end{equation*}
and 
$\mathcal{X}^{ext} := \{(\phi_1,\phi_2)  \in C_b([0,\infty); \mathcal{H}^{\mathrm{ext}}): \|(\phi_1,\phi_2)\|_{\mathcal{X}^{ext}} <\infty\}$
where
\begin{equation*}
	\|(\phi_1,\phi_2)\|_{\mathcal{X}^{ext}} = \| \phi_1\|_{L_{\tau}^4(\R_+)L_{\rho}^8(\B_1^c)} + \| \phi_1\|_{L_{\tau}^2(\R_+)\dot{W}^{1,4}(\B_1^c)} + \|\phi_1\|_{L^\infty_{\tau}(\R_+)\dot{H}^2(\B_1^c)} + \|\phi_2\|_{L_{\tau}^{\infty}(\R_+)\dot{H}^1(\B_1^c)}.
\end{equation*}


\begin{lemma}
			\label{lemma:Extension}
			There exists a bounded linear operator
			$\mb{E}: \mathcal{H}^{ext} \to \mathcal{H}^2$ such that
		 $(\mb{E}\mb{g})|_{\B_1^c}=\mb{g}$ for all $\mb{g}\in\mathcal{H}^{ext}$.
		\begin{proof}
			Let $\mb{g}=(g_1,g_2)=\mb{f}|_{\B_1^c}$ with
			$\mb{f}\in C^\infty_{c,rad}(\R^6)\times C^\infty_{c,rad}(\R^6)$ and set
			\begin{equation*}
				(\mb{E}\mb{g})_1(\rho):=
				\begin{cases} g_1(\rho), & \rho \geqslant 1,\\ g_1(1)+\frac{1}{2}(\rho^2-1)g_1'(1), & \rho<1,
				\end{cases}
				\qquad
				(\mathbf{E}\mathbf{g})_2(\rho):=
				\begin{cases} g_2(\rho), & \rho \geqslant 1,\\ g_2(1), & \rho<1.
				\end{cases}
			\end{equation*}
			Then $(\mb{Eg})_1$ is $C^1$ and $(\mb{Eg})_2$ is continuous across $\rho=1$, so their weak derivatives of second and first order, respectively, are given by
			the piecewise classical ones. In $\B^6$ we have $D^2(\mb{Eg})_1=g_1'(1)I_6$ and $\nabla(\mb{Eg})_2=0$. Since
			\begin{equation*}
				|g_1'(1)| = \Big|\int_1^\infty g_1''(\rho) d\rho \Big|
				\leqslant \Big(\int_1^\infty\rho^{-5}d\rho\Big)^{\frac{1}{2}}\Big(\int_1^\infty|g_1''(\rho)|^2\rho^5 d\rho\Big)^{\frac{1}{2}}
				\lesssim\|g_1\|_{\dot H^2(\B_1^c)},
			\end{equation*}
			we obtain
			\begin{equation*}
				\|(\mb{Eg})_1\|_{\dot H^2(\R^6)}=\|D^2(\mb{Eg})_1\|_{L^2(\R^6)}
				\lesssim \|g_1\|_{\dot H^2(\B_1^c)}, \qquad
				\|(\mb{Eg})_2\|_{\dot H^1(\R^6)}=\|g_2\|_{\dot H^1(\B_1^c)}.
			\end{equation*}
			Hence $\|\mb{Eg}\|_{\mathcal{H}^2} \lesssim \|\mb{g}\|_{\mathcal{H}^{ext}}$,
			and $\mb{E}$ extends to $\mathcal{H}^{ext}$ by density.
		\end{proof}
	\end{lemma}

For $\mb g :=\mb{f}\big|_{\B^1_c} \in \mathcal{H}^{ext}$, we define the exterior semigroup $\mb{S}^{ext}$ by
\begin{equation*}
	\mb{S}^{ext}(\tau)\mb{g}:=\big(\mb{S}(\tau)\mb{Eg}\big)\big|_{\B_1^c}.
\end{equation*}
By finite speed of propagation, if $\mb{f}\in\mathcal{H}^2$ vanishes on $\B_1^c$, then
$\mb{S}(\tau)\mb{f}$ vanishes on $\B_1^c$ for all $\tau\geqslant 0$. Consequently,
$\mb{S}^{ext}(\tau)\mb{g}=(\mb{S}(\tau)\mb{f})|_{\B_1^c}$ for every
$\mb{f}\in\mathcal{H}^2$ with $\mb{f}|_{\B_1^c}=\mb{g}$, and
$(\mb{S}^{ext}(\tau))_{\tau \geqslant 0}$ is a strongly continuous semigroup on $\mathcal{H}^{ext}$.\\


\begin{lemma}
    \label{lem:SExterior}
	For all $\mb{g}\in\mathcal{H}^{ext}$ and
	$\mb{h}\in L^1(\R_+;\mathcal{H}^{ext})$ we have
    \begin{equation*}
		\|\mb{S}^{ext}(\cdot)\mb{g}\|_{\mathcal{X}^{ext}}
		\lesssim \|\mb{g}\|_{\mathcal{H}^{ext}}, \qquad
	       \Big\|\int_0^\tau\mb{S}^{ext}(\tau-\sigma)\mb{h}(\sigma) d\sigma\Big\|_{\mathcal{X}^{ext}}
			\lesssim\|\mb{h}\|_{L^1_\tau(\R_+)\mathcal{H}^{ext}}.
	\end{equation*}
	\begin{proof}
			Apply Theorem \ref{thm:Exterior} to $\mb{f}=\mb{Eg}$ and to
			$\sigma \mapsto \mb{E}\mb{h}(\sigma)$, and use Lemma \ref{lemma:Extension}.
	\end{proof}
\end{lemma}

Now, we are ready to prove Theorem \ref{thm:MainTheoremNLW}
\begin{proof}[Proof of Theorem 1.3]
Let $\mb{u}_0=(\varphi_0,\varphi_1)\in\mathcal{H}^2$ and set
$\mb{g}_0:=\mb{u}_0|_{\B_1^c}$, so that
$\|\mb{g}_0\|_{\mathcal{H}^{ext}} \leqslant \|\mb{u}_0\|_{\mathcal{H}^2}$. For
$\Phi\in\mathcal{X}^{ext}$ we define
\begin{equation*}
	\Gamma^{ext}_{\mb{u}_0}(\Phi)(\tau):=\mb{S}^{ext}(\tau)\mb{g}_0
	+\int_0^\tau \mb{S}^{ext}(\tau-\sigma) \mb{N}(\Phi(\sigma)) d\sigma .
\end{equation*}
By Lemma \ref{lem:SExterior} and Remark \ref{remark:rgP}, there is a constant $C\geqslant 1$ such that
\begin{align*}
	\|\Gamma^{ext}_{\mb{u}_0}(\Phi)\|_{\mathcal{X}^{ext}}
	&\leqslant  C\Big(\|\mb{g}_0\|_{\mathcal{H}^{ext}}
	+\|\Phi\|^2_{\mathcal{X}^{ext}}+\|\Phi\|^3_{\mathcal{X}^{ext}}\Big) \\
	\|\Gamma^{ext}_{\mb{u}_0}(\Phi)-\Gamma^{ext}_{\mb{u}_0}(\Psi)\|_{\mathcal{X}^{ext}}
	&\leqslant  C \|\Phi-\Psi\|_{\mathcal{X}^{ext}}
	\sum_{i=1}^{2}\Big(\|\Phi\|^i_{\mathcal{X}^{ext}} + \|\Psi\|^i_{\mathcal{X}^{ext}}\Big)
\end{align*}
for all $\Phi,\Psi \in \mathcal{X}^{ext}$. Choose $R_0>0$ so small that
$C(R_0+R_0^2)\leqslant \frac{1}{4}$, fix $R\in(0,R_0]$ and set
\begin{equation*}
	\delta:=\frac{R}{2C}, \qquad \delta':=R. 
\end{equation*}
Let $\|\mb{u}_0\|_{\mathcal{H}^2} \leqslant \delta$ and let
$B_R:=\{\Phi\in\mathcal{X}^{ext}:\|\Phi\|_{\mathcal{X}^{ext}}\leqslant  R\}$.
For $\Phi \in B_R$ the first estimate gives
\begin{equation*}
	\|\Gamma^{ext}_{\mb{u}_0}(\Phi)\|_{\mathcal{X}^{ext}}
	\leqslant  C\delta+CR\big(R+R^2\big)\leqslant \frac{R}{2}+\frac{R}{4}\leqslant  R,
\end{equation*}
so $\Gamma^{ext}_{\mb{u}_0}$ maps $B_R$ into itself, and the second
estimate gives
\begin{equation*}
	\|\Gamma^{ext}_{\mb{u}_0}(\Phi)-\Gamma^{ext}_{\mb{u}_0}(\Psi)\|_{\mathcal{X}^{ext}}
	\leqslant  2C\big(R+R^2\big)\|\Phi-\Psi\|_{\mathcal{X}^{ext}}
	\leqslant \frac{1}{2} \|\Phi-\Psi\|_{\mathcal{X}^{ext}}
\end{equation*}
for $\Phi, \Psi \in B_R$. Since $B_R$ is a closed subset of the Banach space
$\mathcal{X}^{ext}$, Banach's fixed point theorem yields a unique fixed point
$\mb{w}=(w_1,w_2)\in  B_R$. The final assertion follows by undoing the change of variables and defining:
 \begin{equation*}
     u(t,x) = u_1^*(t,x) +\frac{1}{t}w_1\Big(\log(t),\frac{x}{t}\Big), \quad \partial_t u(t,x) = \partial_t u_1^*(t,x) +\frac{1}{t^2}w_2\Big(\log(t),\frac{x}{t}\Big),
 \end{equation*}
 to arrive at
 \begin{equation*}
    \sup_{t\in [1,\infty)}\|(u,\partial_t u)-(u_1^*, \partial_t u_1^*)\|_{\dot{H}^2(\R^6\setminus B_t(0)) \times \dot{H}^1(\R^6\setminus B_t(0))} = \|(w_1,w_2)\|_{L_{\tau}^{\infty}\dot{H}^2(\B^c)\times L_{\tau}^{\infty}\dot{H}^1(\B^c)} \leqslant \delta'.
 \end{equation*}
 For uniqueness, let $u,v\in C^\infty(\Omega)$ be two solutions of Eq.~\ref{eq:cNLW} in $\Omega$ satisfying \eqref{NLW3:pert_data}, and set $w:=u-v$. Fix $(t_0,x_0) \in \Omega$ and consider the truncated backward light cone
\begin{equation*}
	K:=\{(t,x): 1\leqslant t \leqslant  t_0, |x-x_0| \leqslant  t_0-t\}.
\end{equation*}
For $(t,x)\in K$ we have $|x| \geqslant  |x_0|-(t_0-t)=t+(|x_0|-t_0)>t$, hence $K\subset\Omega$. On $K$, the difference satisfies
\begin{equation*}
	\partial_t^2w-\Delta w=(u^2+uv+v^2)\,w ,\qquad w(1,\cdot)=\partial_tw(1,\cdot)=0  \text{ on } B_{t_0-1}(x_0),
\end{equation*}
and $M:=\sup_K|u^2+uv+v^2|<\infty$ since $K$ is compact. Here, $B_{t_0-1}(x_0)$ denotes a ball of radius $t_0-1$ centred at $x_0$. For
\begin{equation*}
	E(t):=\frac{1}{2} \int_{B_{t_0-t}(x_0)}\big(|\partial_t w|^2+|\nabla w|^2+|w|^2\big)(t,x) dx
\end{equation*}
the standard energy identity on the cone, in which the boundary term is $\leqslant 0$, yields $E'(t)\leqslant  (M+1)E(t)$ for $t \in [1,t_0)$. Since $E(1)=0$,
Gr\"onwall's inequality implies $E \equiv 0$, so $w=0$ on $K$ and in particular
$u(t_0,x_0)=v(t_0,x_0)$.\\
Furthermore, the fixed point obtained above is a solution that lies in $C([0,\infty);\mathcal{H}^{ext})$. Standard persistence of regularity argument provides the claimed smoothness.
\end{proof}

            \subsection{Wave maps equation} We recall that the nonlinearity is given by
\begin{equation}
        \mb{N}_{\mathrm{WM}}(\bm{\phi})=\begin{bmatrix}
        0\\
        -\frac{(d-3)}{2\rho^3} \big[\sin\big(2\rho (u_{\mathrm{WM}_1}^*(\rho) + \phi_1)\big) - \sin(2\rho u_{\mathrm{WM}_1}^*(\rho)) \big] +\frac{d-3}{\rho^2}\cos(2\rho u_{\mathrm{WM}_1}^*(\rho))\phi_1
    \end{bmatrix}.
\end{equation}
where $\bm{\phi}=(\phi_1,\phi_2)$.

            \begin{proof}[Proof of Theorem \ref{thm:MainTheoremWM}]
                The proof is analogous to the proof of Theorem \ref{thm:MainTheoremNLW} and follows by replacing the nonlinearity and the nonlinear estimates from Lemma \ref{lemma:WMNonlinearity} accordingly.
            \end{proof}

            \begin{proof}[Proof of Theorem \ref{thm:MainTheoremWM2}]
                The proof is analogous to the proof of Theorem \ref{Thm: NLW2} and follows by replacing the nonlinearity and the nonlinear estimates from Lemma \ref{lemma:WMNonlinearity} accordingly. 
            \end{proof}

            \begin{proof}[Proof of Corollary \ref{thm:MainTheoremWM2_stable}]
               In this particular case, when $N''=0$ the proof follows by setting $\mb{P}=0$ in the proof of Theorem \ref{thm:MainTheoremWM2}, hence giving orbital stability of \eqref{eq:WaveMapsSSS} under small perturbations.
            \end{proof}

\section{Appendix}

\subsection{On the admissibility of self-similar solutions to the cubic wave equation and the wave maps equation}
\begin{lemma}
\label{lem: wmadmissible}
Let $u\in C^\infty((0,\infty)\times \R^6)$ be a self-similar solution to 
\begin{align}\label{eq:wavemaps again}
	\big(\partial_t^2 -\partial_r^2 -\frac{5}{r}\partial_r\big)u(t,r) + \frac{3\sin(2ru(t,r)) - 6ru(t,r)}{2r^3} =0.
\end{align}
Then, its profile $U$, defined by the relation $u(t,x)=\displaystyle\frac{1}{t}U(\frac{|x|}{t})$ satisfies the symbol type behaviour
$$
U(\rho)=\O(\rho^{-1})
$$
as $\rho \to \infty$. Furthermore, the function $F(s):=U(s^{-1})$ extends to an analytic function in a neighbourhood of $0$.
Consequently, also the potential 
$V(\rho):=-\frac{3}{\rho^2}\big[\cos(2\rho U(\rho))-1]$ 
satisfies
$$
V(\rho)=\O(\rho^{-2})
$$
and $W(s):=V(s^{-1})$ extends to an analytic function in a neighbourhood of $0$.
\end{lemma}
\begin{proof}
Assume that $u$ is a smooth self-similar solution of the co-rotational wave maps equation. 
We set $v(t,r)=r u(t,r).$ A direct computation shows that $v$ then solves 
\begin{align*}
    	\big(\partial_t^2 -\partial_r^2 -\frac{3}{r}\partial_r\big)v(t,r) + \frac{3\sin(2v(t,r))}{2r^2} =0.
\end{align*}
We now insert the self-similar ansatz $v(t,r)=F(\frac{r}{t})$ to arrive at the equation
\begin{align*}
  \partial_\rho^2 F(\rho)+ (\rho^2-1)^{-1}(-\frac{3}{\rho}+2\rho) \partial_\rho F(\rho)=- \frac{3}{2}\frac{\sin(2F(\rho))}{(\rho^2-1)\rho^2}
\end{align*}
Note that we can find an integrating factor. Namely, we can rewrite the equation as
\begin{align}\label{Eq: symbol type behavior on solution}
\partial_\rho    \left[\frac{\rho^3\partial_\rho F(\rho)}{\sqrt{\rho^2-1}} \right]= -\frac{3\rho\sin(2F(\rho))}{2(\rho^2-1)^{\frac{3}{2}}}.
\end{align}
Hence, we see that
\begin{align*}
\partial_\rho    \left[\frac{\rho^3\partial_\rho F(\rho)}{\sqrt{\rho^2-1}} \right]= O(\rho^{-2}).
\end{align*}
Thus, it follows
\begin{align*}
\lim_{\rho \to \infty}   \frac{\rho^3\partial_\rho F(\rho)}{\sqrt{\rho^2-1}} 
\end{align*} 
exists and is finite.
This in turn implies that 
$\partial_\rho F(\rho)= O(\rho^{-2})$ and we conclude that  $\displaystyle\lim_{\rho\to \infty}F(\rho)$ does indeed exist and is finite. Consequently, as the profile $U$ satisfies $U(\rho)=\rho^{-1} F(\rho)$ the claim follows by differentiating \eqref{Eq: symbol type behavior on solution}.
To show that $U$ is analytic, we define $s=\rho^{-1}$ and $G(s)=F(s^{-1}).$ Then, $G$ solves the equation 
\[
(1-s^2)G''(s)+sG'(s)+\frac32\sin(2G(s))=0,
\]
which we recast as
\[
\left(\frac{G'(s)}{\sqrt{1-s^2}}\right)'
=-\frac{3\sin(2G(s))}{2(1-s^2)^{3/2}}.
\]
Let now 
\[
a:=\lim_{s\downarrow0}G(s),
\qquad
b:=\lim_{s\downarrow0}G'(s)
\]
which exist from previous considerations and 
rewrite the equation as the first-order system
\[
\frac{d}{ds}
\begin{pmatrix}G\\ H\end{pmatrix}
=
\begin{pmatrix}
H\\
-\dfrac{sH+\frac32\sin (2G)}{1-s^2}
\end{pmatrix}.
\]
The right-hand side is analytic near $(s,G,H)=(0,a,b)$.
Hence the initial-value problem with $G(0)=a$ and $H(0)=b$
has a unique local analytic solution. By uniqueness, this
solution agrees with the original solution for small positive
$s$. Thus $G$ extends analytically across $s=0$.
\end{proof}
In the same fashion, one proves the corresponding result for the cubic wave equation.
\begin{lemma}\label{lem: cubicadmissible}
Let $u\in C^\infty((0,\infty)\times \R^6)$ be an admissible self-similar solution to the equation 
\begin{align}
(\partial_t^2-\Delta)u(t,x)=u^3(t,x)
\end{align}
Then, its profile $U$, defined by the relation $u(t,x)=\displaystyle\frac{1}{t}U(\frac{|x|}{t})$ satisfies the symbol type behaviour
$$
U(\rho)=\O(\rho^{-1})
$$
as $\rho \to \infty$. Furthermore, the function $F(s):=U(s^{-1})$ extends to an analytic function in a neighborhood of $0$.
Consequently, the potential 
$V:= 3U^2$
satisfies
$$
V(\rho)=\O(\rho^{-2})
$$
and $W(s):=V(s^{-1})$ extends to an analytic function in a neighbourhood of $0$.
\end{lemma}

	\bibliographystyle{plain} 
\bibliography{biblio.bib}
\end{document}